\documentclass[11pt,reqno]{amsart}
\usepackage[foot]{amsaddr}
\usepackage[utf8]{inputenc}
\usepackage[T1]{fontenc}
\usepackage[margin=1.5in]{geometry}
\usepackage{amsfonts}
\usepackage[all]{xy}
\usepackage[title,titletoc]{appendix}
\usepackage{amssymb}
\usepackage{amsthm}
\usepackage{mathtools}
\usepackage{comment}
\usepackage{esint}
\usepackage{hyperref}
\usepackage{xcolor}
\hypersetup{
   colorlinks,
   linkcolor={red!80!black},
   citecolor={blue!70!black},
   urlcolor={blue!90!black}
}
\usepackage{microtype}
\usepackage{dutchcal}
\usepackage{upgreek}
\newcommand{\wwidehat}[1]{\widehat{#1}}
\newcommand{\cR}{\mathcal{R}}

\usepackage{tikz}
\usepackage{fourier}
\usepackage[mathscr]{euscript}

\usepackage{pgfornament}

\renewcommand{\H}{\mathbb{H}}
\newcommand{\N}{\mathbb{N}}
\newcommand{\R}{\mathbb{R}}
\newcommand{\Z}{\mathbb{Z}}

\newcommand{\cC}{\mathcal{C}}
\newcommand{\cD}{\mathcal{D}}
\newcommand{\cE}{\mathcal{E}}

\newcommand{\cH}{\mathcal{H}}
\newcommand{\cI}{\mathcal{I}}

\newcommand{\cK}{\mathcal{K}}
\newcommand{\cL}{\mathcal{L}}
\newcommand{\cM}{\mathcal{M}}
\newcommand{\cN}{\mathcal{N}}
\newcommand{\cP}{\mathcal{P}}
\newcommand{\cS}{\mathcal{S}}
\newcommand{\cT}{\mathcal{T}}
\newcommand{\cU}{\mathcal{U}}
\newcommand{\cV}{\mathcal{V}}

\renewcommand{\epsilon}{\varepsilon}

\newcommand{\cc}{{\mathsf{c}}}

\newcommand{\vv}{\mathsf{v}}
\newcommand{\HH}{\mathsf{H}}
\newcommand{\VV}{\mathsf{V}}
\newcommand{\XX}{\mathsf{X}}
\newcommand{\YY}{\mathsf{Y}}
\newcommand{\e}{\varepsilon}
\renewcommand{\approx}{\asymp}
\newcommand{\bbM}{\mathbb{M}}
\DeclareMathOperator{\trace}{\bf Trace}
\newcommand{\ud}[0]{\,\mathrm{d}}

\makeatletter
\def\resetMathstrut@{%
  \setbox\z@\hbox{%
    \mathchardef\@tempa\mathcode`\(\relax
    \def\@tempb##1"##2##3{\the\textfont"##3\char"}%
    \expandafter\@tempb\meaning\@tempa \relax
  }%
  \ht\Mathstrutbox@1.2\ht\z@ \dp\Mathstrutbox@1.2\dp\z@
}
\makeatother

\newcommand{\0}{\mathbf{0}}
\renewcommand{\1}{\mathbf 1}
\newcommand{\one}{\mathbf{1}}

\newcommand{\from}{\colon}
\newcommand{\symdiff}{\mathbin{\triangle}}
\renewcommand{\mid}{:}
\renewcommand{\supset}{\supseteq}
\renewcommand{\subset}{\subseteq}
\newcommand{\eqdef}{\stackrel{\mathrm{def}}{=}}

\renewcommand{\ge}{\geqslant}

\renewcommand{\le}{\leqslant}
\renewcommand{\setminus}{\smallsetminus}

\newcommand{\n}{\{1,\ldots,n\}}

\renewcommand{\emptyset}{\varnothing}

\newtheorem{thm}{Theorem}[section]

\newtheorem{question}[thm]{Question}
\newtheorem{lemma}[thm]{Lemma}
\newtheorem{prop}[thm]{Proposition}
\newtheorem{cor}[thm]{Corollary}
\newtheorem{defn}[thm]{Definition}
\theoremstyle{remark}
\newtheorem{remark}[thm]{Remark}

\newtheorem{conjecture}[thm]{Conjecture}

\newcommand{\vpfl}[1]{{\overline{\mathsf{v}}_{\!#1}}}
\newcommand{\vpP}[1]{{\overline{\mathsf{v}}^{P}_{\!#1}}}

\newcommand{\cDv}{{\mathcal{D}_{\mathsf{V}}}}
\newcommand{\bJ}{\mathbf{J}}

\DeclareMathOperator{\NM}{NM}
\DeclareMathOperator{\ENM}{ENM}
\DeclareMathOperator{\id}{id}
\DeclareMathOperator{\diam}{diam}
\DeclareMathOperator{\Lip}{Lip}

\DeclareMathOperator{\inter}{int}

\DeclareMathOperator{\supp}{supp}

\DeclareMathOperator{\Per}{Per}

\newcommand{\pd}[2]{\frac{\partial #2}{\partial #1}}

\newcommand{\cVv}{{\mathcal{V}_{\mathsf{V}}}}
\newcommand{\cVh}{{\mathcal{V}_{\mathsf{H}}}}

\title{Foliated corona decompositions}
\dedicatory{Dedicated to the memory of Louis Nirenberg}

\author[Assaf Naor]{Assaf Naor}
\address{(A.N.) Princeton University,
Department of Mathematics,
Fine Hall, Washington Road,
Princeton, NJ 08544-1000, USA. E-mail  address: \tt{naor@math.princeton.edu}.}

\author[Robert Young]{Robert Young}
\address{(R.Y.) New York University, Courant Institute of Mathematical Sciences, 251 Mercer Street, New York, NY 10012, USA. E-mail  address: \tt{ryoung@cims.nyu.edu}.}

\thanks{A.N.~was supported by the BSF, the Packard Foundation and the Simons Foundation.  R.Y.~was supported by NSF grant 1612061 and the Sloan Foundation. The research that is presented here was conducted under the auspices of the Simons Algorithms and Geometry (A\&G) Think Tank.}

\begin{document}
\begin{abstract}
We prove that the $L_4$ norm of the vertical perimeter of any measurable subset of the $3$--dimensional Heisenberg group $\H$ is at most a universal constant multiple of the (Heisenberg) perimeter of the subset. We show that this isoperimetric-type inequality is optimal in the sense that there are sets for which it fails to hold with the $L_4$ norm replaced by the $L_q$ norm for any $q<4$. This  is in contrast to the $5$--dimensional setting, where the above result holds with the $L_4$ norm replaced by the $L_2$ norm.

The proof of the aforementioned isoperimetric inequality introduces a new structural methodology for understanding the geometry of surfaces in $\H$. In previous work (2017) we showed how to obtain a hierarchical decomposition of Ahlfors-regular surfaces into pieces that are approximately intrinsic Lipschitz graphs.  Here we prove that any such graph admits a {\em foliated corona decomposition}, which is a family of nested partitions into pieces that are close to ruled surfaces.

Apart from the intrinsic geometric and analytic significance of these results, which settle questions posed by  Cheeger--Kleiner--Naor (2009) and Lafforgue--Naor (2012), they have several noteworthy implications. We deduce that the $L_1$ distortion of a word-ball of radius $n\ge 2$ in the discrete $3$--dimensional Heisenberg group is bounded above and below by universal constant multiples of $\sqrt[4]{\log n}$; this is in contrast to  higher dimensional Heisenberg groups, where our previous work (2017) showed that the distortion of a word-ball of radius $n\ge 2$ is of order $\sqrt{\log n}$. We also show that for any $p>2$ there is a metric space that embeds into both $\ell_1$ and $\ell_p$, yet not into a Hilbert space.  This answers the classical question of whether there is a metric analogue of the Kadec--Pe{\l}czy\'nski theorem (1962), which implies that a normed space that embeds into both $L_p$ and $L_q$ for $p<2<q$ is isomorphic to a Hilbert space. Another consequence is that for any $p>2$ there is a Lipschitz function $f\from\ell_p\to \ell_1$ that cannot be factored through a subset of a Hilbert space using Lipschitz functions, i.e., there are no Lipschitz functions $g\from\ell_p\to \ell_2$ and $h\from g(\ell_p)\to \ell_1$ such that $f=h\circ g$; this answers the question, first broached by Johnson--Lindenstrauss (1983), whether there is an analogue of Maurey's theorem  (1974)  that such a factorization  exists if $f$ is linear. Finally, we obtain conceptually new examples that demonstrate  the failure of the Johnson--Lindenstrauss dimension reduction lemma (1983) for subsets of $\ell_1$; these are markedly different from the previously available examples (Brinkman--Charikar, 2003) which do not embed into any uniformly convex normed space, while for any $p>2$ we obtain subsets of $\ell_1$ for which the Johnson--Lindenstrauss lemma fails, yet they  embed into $\ell_p$.
\end{abstract}

\maketitle
\setcounter{tocdepth}{4}
\tableofcontents


\section{Introduction}\label{sec:intro}

Since our main theorem (Theorem~\ref{thm:XYD} below) can be stated without the need to recall any specialized background, we will start by formulating it. After doing so, we will explain its significance and context, as well as geometric applications that answer longstanding open questions. We will then describe our  main conceptual contribution, called a {\em foliated corona decomposition}, which is a new structural methodology that we introduce in the {\em proof of} this theorem; see Remark~\ref{rem:overview} and mainly Section~\ref{sec:overview}  for an overview.

For a smooth function $f\from\R^3\to \R$ define $\XX f,\YY f\from\R^3\to \R$ by setting for $h=(x,y,z)\in \R^3$,
\begin{equation}\label{eq:def XY}
 \XX f(h)\eqdef \frac{\partial f}{\partial x}(h)+\frac{1}{2}y\frac{\partial f}{\partial z}(h)\qquad\mathrm{and}\qquad  \YY f(h)\eqdef \frac{\partial f}{\partial y}(h)-\frac{1}{2}x\frac{\partial f}{\partial z}(h).
\end{equation}
Also,  for $t\in (0,\infty)$ define $D_\vv^tf\from\R^3\to \R$ by setting for $h=(x,y,z)\in \R^3$,
\begin{equation}\label{eq:def Dv}
{D}_{\vv}^t(h)\eqdef \frac{f(x,y,z+t)-f(h)}{\sqrt{t}}.
\end{equation}

\begin{thm}\label{thm:XYD} Every compactly supported smooth function $f\from\R^3\to \R$ satisfies\footnote{We will use throughout the following (standard) asymptotic notation. For $a,b\in (0,\infty)$, the notations
$a\lesssim b$ and $b\gtrsim a$  mean that $a\le Cb$ for some
universal constant $C\in (0,\infty)$. The notation $a\asymp b$
stands for $(a\lesssim b) \wedge  (b\lesssim a)$. If we need to allow for dependence on parameters, we indicate this by subscripts. For example, in the presence of an auxiliary parameter $q$, the notation $a\lesssim_qb$ means that $a\le C(q)b$, where $C(q)\in (0,\infty)$ is allowed to depend only on $q$, and analogously for the notations $a\gtrsim_q b$ and $a\asymp_q b$.}
\begin{equation}\label{eq:our main in intro}
\left(\int_0^\infty\left(\int_{\R^3} |D_{\vv}^tf(h)|\ud h\right)^4\!\frac{\ud t}{t}\right)^{\frac14}\lesssim \int_{\R^3} \left(|\XX f(h)|+|\YY f(h)|\right)\ud h.
\end{equation}
Moreover, one cannot replace the $L_4(\frac{\ud t}{t})$ norm above by an $L_q(\frac{\ud t}{t})$ norm for any $0<q<4$.
\end{thm}
The second assertion (sharpness) of Theorem~\ref{thm:XYD} resolves negatively the conjecture of~\cite{LafforgueNaor} that~\eqref{eq:our main in intro} holds with the $L_4(\frac{\ud t}{t})$ norm in the left hand side replaced by the $L_2(\frac{\ud t}{t})$ norm. Notwithstanding the optimality of~\eqref{eq:our main in intro}, it should be noted that  it was previously unknown whether such a bound holds true merely for {\em some} finite exponent, namely that there exists $0<p<\infty$  such that in the setting of Theorem~\ref{thm:XYD} we have
\begin{equation}\label{eq:our main in intro q version}
\left(\int_0^\infty\left(\int_{\R^3} |D_{\vv}^tf(h)|\ud h\right)^p\!\frac{\ud t}{t}\right)^{\frac1{p}}\lesssim \int_{\R^3} \left(|\mathbf{\XX} f(h)|+|\YY f(h)|\right)\ud h.
\end{equation}
It is simple to justify (see~\cite[Remark~4]{NY18}) that if~\eqref{eq:our main in intro q version} holds, then the analogous bound  holds for any larger exponent $P>p$.

\begin{remark}\label{rem:overview}
To briefly indicate what goes into Theorem~\ref{thm:XYD}, we first note that the functional inequality~\eqref{eq:our main in intro} is equivalent to a certain  isoperimetric-type inequality (see~\eqref{eq:coarea}) for sufficiently smooth surfaces in $\R^3$. By~\cite{NY18}, it turns out  that it suffices to prove this isoperimetric-type inequality  for  a more restricted class of surfaces (intrinsic Lipschitz graphs; see Section~\ref{sec:intrinsic graphs}).  Such surfaces can still be very complicated, as one can see in Figure~\ref{fig:Lip10}. However, notice that the example in Figure~\ref{fig:Lip10} has an anisotropic texture, with features of many different scales that line up along a one-dimensional foliation.

\begin{figure}
  \begin{center}\includegraphics[width=\textwidth]{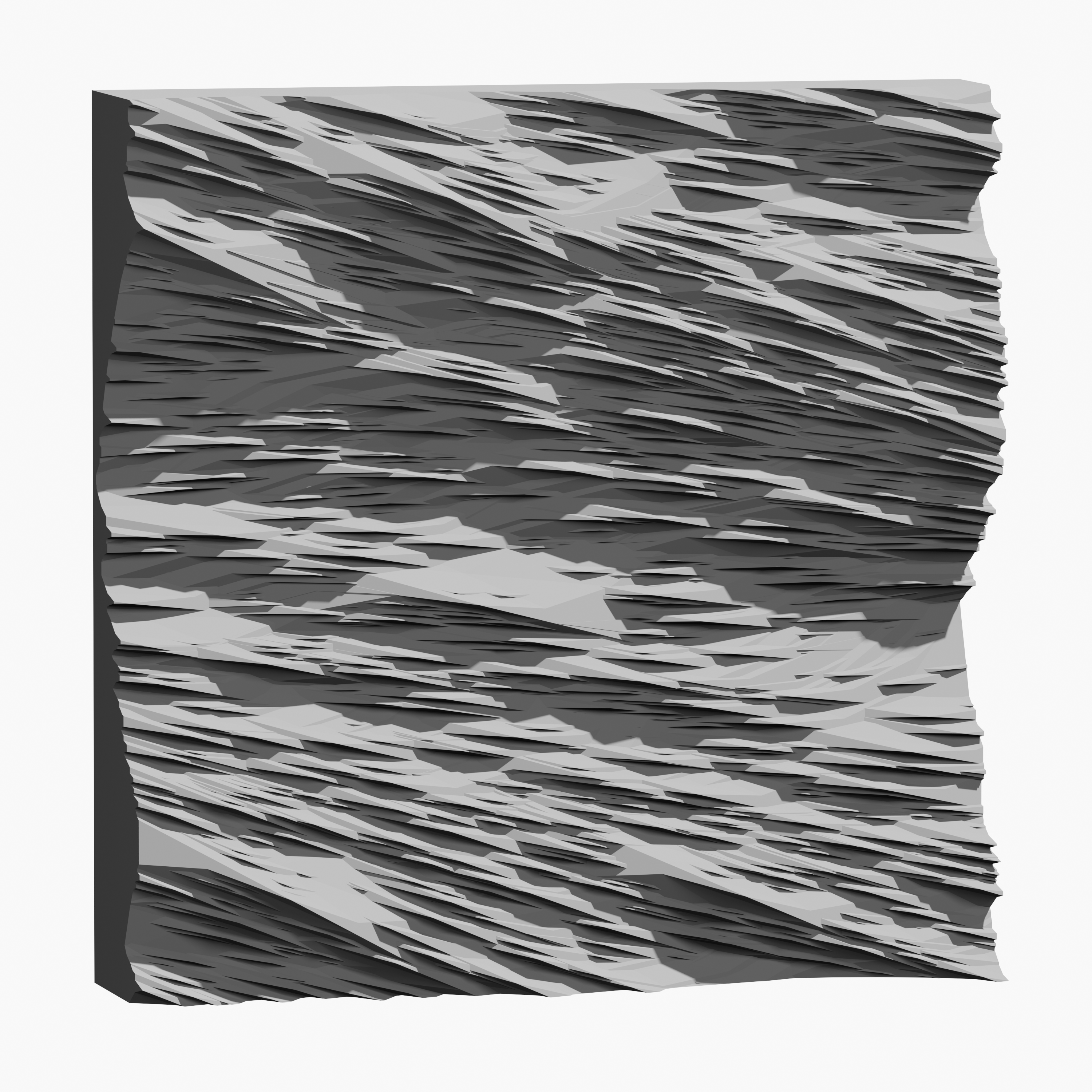}
\end{center}
\caption{\label{fig:Lip10}
 \small \em An example of an intrinsic Lipschitz graph.}
\end{figure}

We prove the desired isoperimetric-type inequality by showing that the texture of {\em any} intrinsic Lipschitz graph can be encoded as a foliated corona decomposition, which is a multi-scale hierarchical partition of the surface.  The pieces of this decomposition are roughly rectangular regions that mimic the dimensions and orientation of the features of the surface. Crucially, we can control the number and size of these pieces.  The desired inequality holds locally on each piece up to suitably controlled error, and the full inequality is obtained by summing the resulting estimates.  This process is illustrated in Figure~\ref{fig:foliated corona of bump} and Figure~\ref{fig:bumpy surface}, and a more detailed overview  can be found in Section~\ref{sec:overview}.
\end{remark}

In contrast to Theorem~\ref{thm:XYD}, we have the following theorem, the case $p=2$ of which is due to~\cite{AusNaoTes} and the case $p\in (1,2]$ of which is due (via a different proof) to~\cite{LafforgueNaor}.
\begin{thm}\label{thm:LN-XYD} If $f\from\R^3\to \R$ is smooth and compactly supported, then for every $p\in (1,2]$,
\begin{equation}\label{eq:quote LN q-1}
\bigg(\int_0^\infty\left(\int_{\R^3} |D_{\vv}^tf(h)|^p\ud h\right)^{\frac{2}{p}}\!\frac{\ud t}{t}\bigg)^{\frac12}\lesssim \frac{1}{\sqrt{p-1}}\bigg(\int_{\R^3} \left(|\mathbf{\XX} f(h)|^p+|\YY f(h)|^p\right)\ud h\bigg)^{\frac{1}{p}}.
\end{equation}
\end{thm}
See~\cite{LafforgueNaor} for a variant of Theorem~\ref{thm:LN-XYD} when $p>2$.  The pertinent point of comparison to~\eqref{eq:our main in intro}  is as $p\to 1^+$, namely there is a jump discontinuity at the endpoint $p=1$.

It should be noted that the dependence on $p$ in the right hand side of~\eqref{eq:quote LN q-1} is not specified  in~\cite{LafforgueNaor}, but one  obtains~\eqref{eq:quote LN q-1} in the form stated above by tracking the dependence on $p$ in the proof of~\cite{LafforgueNaor}; we explain how to do so in Appendix~\ref{sec:littlewood paley} below. We conjecture that the following  bound holds, which is better than~\eqref{eq:quote LN q-1} only in terms of the dependence on $p$; its geometric ramifications will be derived later (see Remark~\ref{rem:sharp in p}), at which point it will become clear why we need to record an explicit (power-type) dependence as  $p\to 1^+$ in~\eqref{eq:quote LN q-1} rather than using the implicit $\lesssim_p$ notation as done in~\cite{LafforgueNaor}.
\begin{conjecture}\label{conj:function of p} In the setting of Theorem~\ref{thm:LN-XYD} we have
\begin{equation}\label{eq:ask sqrt4}
\bigg(\int_0^\infty\left(\int_{\R^3} |D_{\vv}^tf(h)|^p\ud h\right)^{\frac{2}{p}}\!\frac{\ud t}{t}\bigg)^{\frac12}\lesssim \frac{1}{\sqrt[4]{p-1}}\bigg(\int_{\R^3} \left(|\mathbf{\XX} f(h)|^p+|\YY f(h)|^p\right)\ud h\bigg)^{\frac{1}{p}}.
\end{equation}
\end{conjecture}

Another key point of comparison between Theorem~\ref{thm:XYD} and the literature is with its higher-dimensional counterpart due to~\cite{NY18}. For a smooth function $f\from\R^5\to \R$,  denote  in analogy to~\eqref{eq:def XY} and~\eqref{eq:def Dv} for every $h=(x_1,y_1,x_2,y_2,z)\in \R^5$ and $t\in (0,\infty)$,
$$
\XX_1f(h)\eqdef \frac{\partial f}{\partial x_1}(h)-\frac{1}{2}y_1\frac{\partial f}{\partial z}(h),\qquad \XX_2f(h)\eqdef \frac{\partial f}{\partial x_2}(h)-\frac{1}{2}y_2\frac{\partial f}{\partial z}(h),
$$
$$
\YY_1f(h)\eqdef \frac{\partial f}{\partial y_1}(h)+\frac{1}{2}x_1\frac{\partial f}{\partial z}(h),\qquad \YY_2f(h)\eqdef \frac{\partial f}{\partial y_2}(h)+\frac{1}{2}x_2\frac{\partial f}{\partial z}(h),
$$
and
$$
{D}_{\vv}^t(h)\eqdef \frac{f(x_1,y_1,x_2,y_2,z+t)-f(h)}{\sqrt{t}}.
$$
We then have the following theorem (it holds with $\R^5$ replaced mutatis mutandis by $\R^{2k+1}$ for every $k\ge 2$; we are focusing only on $\R^5$ because the crucial qualitative  difference that we establish here is between dimension $3$ and all the larger odd dimensions).
\begin{thm}\label{thm:XYD-NY17}  If $f\from\R^5\to \R$ is smooth and compactly supported, then for every $p\in [1,2]$,
\begin{align}\label{eq:quote NY interpolated}
\begin{split}
\bigg(\int_0^\infty\bigg(\int_{\R^5}& |D_{\vv}^tf(h)|^p\ud h\bigg)^{\frac{2}{p}}\!\frac{\ud t}{t}\bigg)^{\frac12}\\&\lesssim \bigg(\int_{\R^5} \left(|\mathbf{\XX}_1 f(h)|^p+|\YY_1 f(h)|^p+|\mathbf{\XX}_2 f(h)|^q+|\YY_2 f(h)|^p\right)\ud h\bigg)^{\frac{1}{p}}.
\end{split}
\end{align}
\end{thm}
The case $p=2$ of Theorem~\ref{thm:XYD-NY17} is from~\cite{AusNaoTes} and in the range $p\in (1,2]$ the bound~\eqref{eq:quote NY interpolated} but with $\lesssim$ replaced by $\lesssim_p$ is from~\cite{LafforgueNaor}. The case $p=1$ of Theorem~\ref{thm:XYD-NY17} is from~\cite{NY18}. Inequality~\eqref{eq:quote NY interpolated} as stated above, i.e., with the right hand side multiplied by a universal constant rather than a constant that depends on $p$ as in~\cite{LafforgueNaor},  follows by interpolating between the cases $p=1$ and $p=2$ of~\cite{NY18} and~\cite{AusNaoTes}, respectively. Indeed, \eqref{eq:quote NY interpolated} asserts the boundedness of a linear operator, the $L_2(L_p)$ norms in the left hand side of~\eqref{eq:quote NY interpolated}  are an interpolation family by classical interpolation theory~\cite{BL76}, and the Sobolev $W^{1,p}$  norms  in the right hand side of~\eqref{eq:quote NY interpolated} are  an interpolation family by~\cite[Theorem~8.8]{Bad09}.

\subsection{Geometric implications}  Let $\H$ be the $3$--dimensional Heisenberg group with real coefficients. As a set,  $\H$ is identified  with $\R^{3}$, and the group structure on $\H$ is  given by
\begin{equation}\label{eq:def Heisenberg product}
\forall g=(x,y,z),h=(\chi, \upsilon,\zeta)\in \R^3,\qquad g h\eqdef \Big(x+\chi,y+\upsilon,z+\zeta+\frac{1}{2}(x\upsilon- y\chi)\Big).
\end{equation}
The identity element is $\0=(0,0,0)$ and the inverse of $g=(x,y,z)$  is $g^{-1}=(-x,-y,-z)$. The center of $\H$ is $\{0\}\times \{0\}\times \R$ and if we let $\H_\Z$ be the discrete subgroup of $\H$ that is generated by $(1,0,0)$ and $(0,1,0)$, then we have
$$\H_\Z=\bigg\{\Big(x,y,z+\frac{xy}{2}\Big)\mid x,y,z\in \Z\bigg\}\subset \Z\times \Z\times \frac{\Z}{2}.$$

Let $d_W\from\H_\Z\times \H_\Z\to \N\cup\{0\}$ be the left-invariant word metric on $\H_\Z$ that is induced by the symmetric set of generators $\{(-1,0,0),(1,0,0),(0,-1,0),(0,1,0)\}$. It is well-known (and elementary to verify) that for every $g=(x,y,z),h=(\chi, \upsilon,\zeta)\in \H_\Z$ we have
\begin{equation}\label{eq:equiv to Kor}
 d_W(g,h)\asymp |x-\chi|+|y-\upsilon|+\sqrt{|2z-2\zeta-x\upsilon+y\chi|}
\end{equation}
In fact, an exact formula for $d_W(g,h)$, which directly implies~\eqref{eq:equiv to Kor}, is derived in~\cite{Bla03}. For every $n\in \N$, denote   the word-ball of radius $n$   centered at the identity element by
\begin{equation}\label{eq:word ball}
\mathcal{B}_n\eqdef \left\{g\in \H_\Z:\ d_W(g,\0)\le n\right\}.
\end{equation}

\subsubsection{Embeddings} Recall that a metric space $(M,d_M)$ is said to admit a bi-Lipschitz embedding into a Banach space $(X,\|\cdot\|_X)$ if there exist $D\in [1,\infty)$ and $\phi\from M\to X$ such that
\begin{equation}\label{eq:def D embedding}
\forall x,y\in M,\qquad d_M(x,y)\le \|\phi(x)-\phi(y)\|_X\le Dd_M(x,y).
\end{equation}
The infimum over those $D\in [1,\infty)$ for which this holds is called the $X$--distortion of $M$ and is denoted $\cc_X(M)$. If no such $D$ exists, then one writes $\cc_X(M)=\infty$.

Theorem~\ref{thm:sqrt4} below is a sharp asymptotic evaluation of $\cc_{\ell_1}(\mathcal{B}_n)$. It answers a question posed in~\cite{LN06,CK10,CheegerKleinerMetricDiff,CKN09,CKN,Nao10,Pan13,LafforgueNaor}; these references ask for the asymptotic evaluation of $\cc_{\ell_1}(\mathcal{B}_n)$, but most of them also conjecture that $\cc_{\ell_1}(\mathcal{B}_n)\asymp \sqrt{\log n}$, so Theorem~\ref{thm:sqrt4} constitutes both a resolution of an open problem, and an unexpected answer. The fact that $\lim_{n\to \infty}\cc_{\ell_1}(\mathcal{B}_n)=\infty$ is due to~\cite{CK10}, the previously best known upper bound~\cite{Ass83} was  $\cc_{\ell_1}(\mathcal{B}_n)\lesssim \sqrt{\log n}$ and the previously best-known lower bound~\cite{CKN} was $\cc_{\ell_1}(\mathcal{B}_n)\ge (\log n)^\delta$ for some positive but very small universal constant $\delta$; thus both the upper and the lower bounds of Theorem~\ref{thm:sqrt4} are new.

\begin{thm}\label{thm:sqrt4} $\cc_{\ell_1}(\mathcal{B}_n)\asymp \sqrt[4]{\log n}$ for every integer $n\ge 2$.
\end{thm}

In contrast, the word-ball of radius $n\ge 2$ in the $5$--dimensional Heisenberg group has $\ell_1$\nobreakdash--distortion of order $\sqrt{\log n}$; this was proved in~\cite{NY18} using Theorem~\ref{thm:XYD-NY17}.

The statement of Theorem~\ref{thm:sqrt4} has two parts. While the lower bound $\cc_{\ell_1}(\mathcal{B}_n)\gtrsim \sqrt[4]{\log n}$ is framed above as a ``negative result'' (impossibility of embedding), it encapsulates a ``positive result,'' namely the aforementioned new structural information on surfaces in $\H$, to which most of this article is devoted. The upper bound $\cc_{\ell_1}(\mathcal{B}_n)\lesssim \sqrt[4]{\log n}$ is a ``positive result,'' namely a new geometric realization of $\mathcal{B}_n$, but we will soon see that it has ramifications for counterexamples to  natural geometric questions.

The estimate~\eqref{eq:our main in intro} of Theorem~\ref{thm:XYD} implies the lower bound $\cc_{\ell_1}(\mathcal{B}_n)\gtrsim \sqrt[4]{\log n}$. In fact, such {\em vertical-versus-horizontal Poincar\'e inequalities }were originally envisaged as obstructions to embeddings of  $\mathcal{B}_n$ into various spaces; see~\cite{AusNaoTes,NN12,LafforgueNaor,NY-STOC}, and most pertinently Section~3 of~\cite{NY18}, where we treated such matters in greater generality than what is needed here; in particular, for any $p\ge 1$,  if every compactly supported smooth function $f\from\R^3\to \R$ satisfies the inequality
\begin{equation}\label{eq:q main in intro}
\left(\int_0^\infty\left(\int_{\R^3} |D_{\vv}^tf(h)|\ud h\right)^p\!\frac{\ud t}{t}\right)^{\frac1{p}}\lesssim \int_{\R^3} \left(|\mathbf{\XX} f(h)|+|\YY f(h)|\right)\ud h,
\end{equation}
then by~\cite[\S~3]{NY18} and the reasoning in~\cite[\S~1.3]{NY18} we have $\cc_{\ell_1}(\mathcal{B}_n)\gtrsim (\log n)^{\frac{1}{p}}$.

Thus, $\cc_{\ell_1}(\mathcal{B}_n)\gtrsim \sqrt[4]{\log n}$, since  Theorem~\ref{thm:XYD}  asserts that~\eqref{eq:q main in intro} holds for $p=4$. This also demonstrates that the matching upper bound $\cc_{\ell_1}(\mathcal{B}_n)\lesssim \sqrt[4]{\log n}$ of Theorem~\ref{thm:sqrt4} implies the second assertion of Theorem~\ref{thm:XYD}, namely the optimality of the $L_4(\frac{\ud t}{t})$ norm in the left hand side of~\eqref{eq:our main in intro}. Here we prove the following more refined embedding statement which we formulate as a separate theorem because it has further noteworthy applications.

\begin{thm}\label{thm:twist the center} For every $\vartheta\ge  \frac14$ and every integer $n\ge 2$ there exists $\phi=\phi_{n,\vartheta}\from\H_\Z\to \ell_1$ with respect to which every two points $g=(x,y,z),h=(\chi,\upsilon,\zeta)\in \H_\Z$ with $d_W(g,h)\le 2n$ satisfy
\begin{equation}\label{eq:twisted z}
\|\phi(g)-\phi(h)\|_{\ell_1}\asymp |x-\chi|+|y-\upsilon|+\frac{\sqrt{|2z-2\zeta-x\upsilon+y\chi|}}{(\log n)^\vartheta}.
\end{equation}
\end{thm}
By~\eqref{eq:equiv to Kor} and the case $\vartheta=\frac14$ of Theorem~\ref{thm:twist the center}, the following weakening of~\eqref{eq:twisted z} holds.
$$
\forall g,h\in \mathcal{B}_n,\qquad \frac{d_W(g,h)}{\sqrt[4]{\log n}} \lesssim \|\phi(g)-\phi(h)\|_{\ell_1}\lesssim d_W(g,h).
$$
So, the upper bound $\cc_{\ell_1}(\mathcal{B}_n)\lesssim \sqrt[4]{\log n}$  of Theorem~\ref{thm:sqrt4} follows from Theorem~\ref{thm:twist the center}. However, Theorem~\ref{thm:twist the center} is of further use thanks to the following embedding result of~\cite{LN14}. At present, the fact that both  our embedding and that of~\cite{LN14} yield the same  expression (up to universal constant factors) for the metric in the image seems to be a fortunate and consequential coincidence; it would be valuable, if possible, to explain conceptually why those formulas  coincided   (e.g.~is this inevitable due to underlying symmetries?).

\begin{thm}\label{thm:twist the center Lp} For any $p>2$, any $\vartheta\ge \frac{1}{p}$ and any integer $n\ge 2$ there is $\psi=\psi_{n,p,\vartheta}\from\H_\Z\to \ell_p$ such that every $g=(x,y,z),h=(\chi,\upsilon,\zeta)\in \H_\Z$ with $d_W(g,h)\le 2n$ satisfy
\begin{equation}\label{eq:twisted z L4}
\|\psi(g)-\psi(h)\|_{\ell_p}\asymp |x-\chi|+|y-\upsilon|+\frac{\sqrt{|2z-2\zeta-x\upsilon+y\chi|}}{(\log n)^\vartheta}.
\end{equation}
\end{thm}
Theorem~\ref{thm:twist the center Lp} is not formulated explicitly in~\cite{LN14}, but  it is a direct consequence of  Lemma~3.1 in~\cite{LN14} combined with the finite-determinacy theorem of~\cite{Ost12}, which together imply that for every $\e\in (0,\frac12]$ there exists an embedding $\sigma=\sigma_{\e,p}\from\H\to \ell_p$ for which  every $g=(x,y,z),h=(\chi,\upsilon,\zeta)\in \H_\Z$ satisfy
\begin{equation}\label{eq:psi'}
\|\sigma(g)-\sigma(h)\|_{\ell_p}\asymp |x-\chi|^{1-\e}+|y-\upsilon|^{1-\e}+\e^{\frac{1}{p}}|2z-2\zeta-x\upsilon+y\chi|^{\frac{1-\e}{2}}.
\end{equation}
(Without reference to~\cite{Ost12}, Lemma~3.1 in~\cite{LN14} asserts the existence of such an embedding into $L_p$ rather than into $\ell_p$.)
To derive Theorem~\ref{thm:twist the center Lp} from~\eqref{eq:psi'}, let $\uppi\from \H\to \R$ be the map that is given by setting  $\uppi(x,y,z)= (x,y)$ for  $(x,y,z)\in \H$ and choose
\begin{equation}\label{eq:psi parametrs}\e=\frac{1}{\log n}\qquad  \mathrm{and}\qquad  \psi=\frac{\sigma}{(\log n)^{\vartheta-\frac{1}{p}}}\oplus \uppi\from\H_\Z\to \ell_p\oplus \R^2\cong \ell_p.
\end{equation}

\subsubsection{Aspects of the Ribe program}\label{sec:aspects of ribe} Inspired by a fundamental rigidity theorem of~\cite{Rib76} and first put forth in~\cite{Bou86}, the Ribe program is a web of conjectures and analogies whose goal is to transfer linear phenomena in the geometry of Banach spaces to questions about metric spaces, where  Lipschitz mappings take the role of bounded linear operators; see e.g.~the surveys~\cite{Kal08-survey,Nao12,Bal13,Ost13,Nao18}. We will next explain how the above results answer natural questions in this area.

Theorem~\ref{thm:two Lps} below follows from Theorem~\ref{thm:twist the center}, Theorem~\ref{thm:twist the center Lp} and~\cite{AusNaoTes,LafforgueNaor}. It  answers a  longstanding  question in metric embedding theory; even though (to the best of our knowledge) this question never appeared in published\footnote{We have seen it appear in writing  only in  grant proposals, and it was posed verbally among experts. In particular, we are indebted to Gideon Schechtman for valuable discussions on this matter over the years.} texts, it was a  folklore open problem. To briefly explain the context, the classical work~\cite{KP61} (together with a differentiation argument of~\cite{Man72}) implies that for  $1\le p< r< q<\infty$, if a Banach space $X$ admits a bi-Lipschitz embedding into both $L_p$ and $L_q$, then $X$ also admits a bi-Lipschitz embedding into $L_r$. The case $r=2$ of this statement is that if $X$ embeds into $L_p$ for two finite values of $p$ that lie on both sides of $2$, then $X$ must embed into (hence, by~\cite{Enf70},  be linearly isomorphic to) a Hilbert space; a different proof of the latter statement, as a special case of a much more general phenomenon, follows from~\cite{Kwa72}. In light of these facts about the geometry of Banach spaces, one is naturally led to ask if a metric space $M$ that embeds bi-Lipschitzly into $L_p$ for two finite values of $p$ that lie on both sides of $2$ must admit a bi-Lipschitz embedding into a Hilbert space.

\begin{thm}\label{thm:two Lps} For any $2<p\le 4$ there  is a metric space $M$  that admits a bi-Lipschitz embedding into $\ell_1$ and into $\ell_r$ for all $r\ge p$, yet $M$ does not admit a bi-Lipschitz embedding into $L_q$ for any $1<q<p$. More generally, $M$ does not admit a bi-Lipschitz embedding into a Banach space whose modulus of uniform convexity has power type $q$ for  $2\le q<p$.
\end{thm}
For the statement of Theorem~\ref{thm:two Lps}, recall that a Banach space $(X,\|\cdot\|)$ has modulus of uniform convexity of power type $q$ if there is $C>0$ such that the sharpened triangle inequality $\|x+y\|\le 2-C\|x-y\|^q$ holds for any unit vectors $x,y\in X$. By~\cite{Cla36,Han56},  for $1<q<\infty$ any $L_q(\mu)$ space has modulus of uniform convexity of power type $\max\{q,2\}$.

\begin{proof}[Proof of Theorem~\ref{thm:two Lps} assuming Theorem~\ref{thm:twist the center} and Theorem~\ref{thm:twist the center Lp}]  For every $n\in \N$, we define  $M_n=\phi_{n,\vartheta}(\mathcal{B}_n)\subset \ell_1$, where  $\phi_{n,\vartheta}$ is as in Theorem~\ref{thm:twist the center} applied with $\vartheta=\frac{1}{p}\ge \frac14$.

By considering the union of sufficiently widely-spaced translations in $\ell_1$ of the finite sets $\{M_n\}_{n=1}^\infty$, we see that there is $M\subset \ell_1$ such that $\sup_{n\in \N} \cc_M(M_n)<\infty$.

For every $r\ge p$, consider $\psi_{n,r,\vartheta}(\mathcal{B}_n)\subset \ell_r$, where $\psi_{n,r,\vartheta}$ is as in  Theorem~\ref{thm:twist the center Lp}. Theorem~\ref{thm:twist the center} and Theorem~\ref{thm:twist the center Lp} show that $\psi_{n,r,\vartheta}(\mathcal{B}_n)$ is bi-Lipschitz equivalent with $O(1)$ distortion to $M_n$. Hence, by  considering a suitable union of translations in $\ell_r$ of the finite sets $\{\psi_{n,r,\vartheta}(\mathcal{B}_n)\}_{n=1}^\infty$, we see that $\cc_{\ell_r}(M)<\infty$. Let $X$ be a Banach space whose modulus of uniform convexity has power type $q$ for  $2\le q<p$.  By~\cite{LafforgueNaor} we have
$$
(\log n)^{\frac{1}{q}}\lesssim_X \cc_{X}(\mathcal{B}_n) \lesssim (\log n)^\vartheta\cc_{X}(M_n)=(\log n)^{\frac{1}{p}}\cc_{X}(M_n),
$$
where the penultimate step holds because, due to~\eqref{eq:twisted z}, $M_n$ and $\mathcal{B}_n$ are bi-Lipschitz equivalent with distortion $O((\log n)^\vartheta)$.  Therefore, since $q<p$,
$$
\cc_{X}(M_n)\gtrsim_X (\log n)^{\frac{1}{q}-\frac{1}{p}} \xrightarrow[n\to \infty]{} \infty.
$$
Hence, $\cc_X(M)=\infty$, as required. For future reference we record in passing that we obtained the following bound when $X=L_q$ and $1<q<p$.
\begin{equation}\label{eq:Lq case}
\cc_{L_q}(M_n)\gtrsim_q (\log n)^{\frac{1}{\max\{q,2\}}-\frac{1}{p}}.
\end{equation}

Note that the  bound in~\cite{AusNaoTes}, which is asymptotically weaker than that of~\cite{LafforgueNaor},  suffices for the qualitative conclusion $\cc_X(M)=\infty$ of Theorem~\ref{thm:two Lps}. The above estimates seem to be the best that one could achieve using available methods; it would be very interesting to determine the optimal behavior, e.g.~if an $n$--point metric space $W$ embeds with $O(1)$ distortion into  $\ell_1$ and also into $\ell_p$ for some $p>2$, how large can $\cc_{\ell_2}(W)$ be?
\end{proof}

\begin{remark} With more care it is possible to ensure that the metric space $M$ of Theorem~\ref{thm:two Lps} is  a left-invariant metric $\updelta=\updelta_p$ on $\H_\Z$; see  Theorem~\ref{thm:invariant metric on two sides}. Concretely, for $p=4$ the metric $\updelta_4$ can be taken to satisfy the following bounds for any $(a,b,c)\in \H_\Z$ with $|c|\ge 3$.
$$
\delta_4\big(\0,(a,b,c)\big)\asymp |a|+|b|+\frac{\sqrt{|c|}}{\sqrt[4]{\log |c|}\cdot (\log\log |c|)^2}.
$$
By the reasoning in~\cite[Section~9]{NP11}, since $\H_\Z$ is amenable, it follows that $(\H_\Z,\updelta)$ admits a bi-Lipschitz embedding into $L_1$ and $L_r$ for all $r\ge p$ which is also equivariant (with respect to an action of $\H_\Z$ on, respectively, $L_1$ and $L_p$ by affine isometries); we did not investigate if this holds for equivariant embeddings into the sequence spaces $\ell_1$ and $\ell_r$.

The natural question how the embeddability of a group into $L_p$ depends on $p$ was also studied in the literature; see~\cite{CDH10,Czu17}, and especially the recent solution of this question in~\cite{MS20}, where it is proved that the phenomenon of Theorem~\ref{thm:two Lps} does not hold for equivariant coarse embeddings (namely, for such embeddings the corresponding set of $p$ is always an interval). Note that for coarse embeddings that need not be equivariant, the statement of~\cite{MS20} was previously known as a direct consequence of~\cite[Remark~5.10]{MN04} (from here, using~\cite{NP11}, one gets the full equivariant statement of~\cite{MS20}  for amenable groups). Theorem~\ref{thm:invariant metric on two sides} shows that the situation is markedly different if one considers bi-Lipschitz embeddings rather than coarse embeddings.
\end{remark}

The following question arises naturally from Theorem~\ref{thm:two Lps} and seems quite difficult.

\begin{question}\label{Q:crazy set}
For a metric space $M$, how complicated can the following set be?
$$\big\{1\le p<\infty: \cc_{L_p}(M)<\infty\big\}.$$
\end{question}

Theorem~\ref{thm:two Lps} leaves the possibility that there is better behavior in the reflexive range, i.e., that if a metric space $M$ embeds bi-Lipschitzly into  $\ell_p$ and $\ell_q$ for $1<q<2<p<\infty$, then $M$  embeds bi-Lipschitzly into a Hilbert space. If true, this would be an excellent theorem, but due to Theorem~\ref{thm:two Lps} we speculate that the answer is negative. A substantial new idea seems to be needed here. Less ambitiously, does the above assumption (even allowing $q=1$) imply that $M$ embeds into a Hilbert space with finite average distortion (see~\cite{Nao19} for the relevant definition)? Does this imply that every $n$--point subset of $M$ embeds into a Hilbert space with bi-Lipschitz distortion $o(\log n)$, i.e., asymptotically better than the distortion that is guaranteed by the general embedding theorem of~\cite{Bou85}?

\medskip

The above reasoning also leads to Theorem~\ref{thm:factorization} below, which answers another natural question arising in the Ribe program, on the factorization of Lipschitz functions.

We first briefly make preparatory observations that will be also useful elsewhere. Recall that for $K\in \N$ a metric space $X$ is said to be $K$-{\em doubling} if for every $r>0$, any ball $B\subset X$ of radius $r$ can be covered by $K$ balls of radius $r/2$. $X$ is doubling if it is $K$--doubling for some $K\in \N$. The metric space $M$ of Theorem~\ref{thm:two Lps} can be taken to be doubling. Indeed, fix $p>2$ and $n\in \N$. As in the proof of Theorem~\ref{thm:two Lps}, write $\vartheta=1/\min\{p,4\}$. It was shown in~\cite{LN14} that $\psi_{n,p,\vartheta}(\H_\Z)$ is a $O(1)$--doubling subset of $\ell_p$. Let $S\subset \ell_p$ be the disjoint union of  translates in $\ell_p$ of the finite sets $\{\psi_{n,p,\vartheta}(\mathcal{B}_n)\}_{n=1}^\infty$ that are sufficiently widely-spaced so as to ensure that $S$ is a doubling subset of $\ell_p$, and $\sup_{n\in \N}\cc_S(M_n)<\infty$. As in the proof of Theorem~\ref{thm:two Lps}, using Theorem~\ref{thm:twist the center} we get an embedding $\varphi\from S\to \ell_1$ satisfying $\|\varphi(x)-\varphi(y)\|_{\ell_1}\asymp \|x-y\|_{\ell_p}$ for all $x,y\in S$. Thus $\varphi(S)=M$ is a doubling subset of $\ell_1$.

Since $S$ is doubling, by~\cite{LN05} we can extend $\varphi$ to a Lipschitz function $f\from\ell_p\to \ell_1$. If there were Lipschitz mappings  $g\from\ell_p\to \ell_2$ and $h\from g(\ell_p)\to \ell_1$ such that $f=h\circ g$, then it would follow that for all $x,y\in S$ we have
$$
\|x-y\|_{\ell_p}\asymp \|\varphi(x)-\varphi(y)\|_{\ell_1}=\left\|h\big(g(x)\big)-h\big(g(y)\big)\right\|_{\ell_1}\lesssim \|g(x)-g(y)\|_{\ell_2}\lesssim \|x-y\|_{\ell_p}.
$$
Therefore, $g\circ\varphi^{-1}$ would be a bi-Lipschitz embedding of $M$ into $\ell_2$, which we proved above was impossible. We thus arrive at the following statement.

\begin{thm}\label{thm:factorization} For any $2<p<\infty$ there is a Lipschitz mapping $f\from \ell_p\to \ell_1$ that cannot be factored through a subset of a Hilbert space using Lipschitz mappings. Namely, there do not exist Lipschitz mappings $g\from \ell_p\to \ell_2$ and $h\from g(\ell_p)\to \ell_1$ such that $f=h\circ g$. More generally, $f$ cannot be factored  using Lipschitz mappings through a subset of a Banach space whose modulus of uniform convexity has power type $q$ for $2\le q<\min\{4,p\}$.
\end{thm}

By~\cite[Theorem~5.2]{LP68}, for $p\ge 2$ any linear operator from $\ell_p$ to $\ell_1$ factors through $\ell_2$ (the factorization is via linear operators, though by~\cite{JMS09} this is equivalent to factorization using Lipschitz functions as above). Theorem~\ref{thm:factorization} demonstrates that there is no analogue of this factorization phenomenon for Lipschitz mappings.

Such investigations arose in the Ribe program  in the seminal work~\cite{JL84} which had a major influence on the subsequent fruitful efforts by many mathematicians in search of metric analogues of the extension and factorization paradigm of~\cite{Mau74}.  This search is itself intimately intertwined with the search for metric theories of type and cotype.

We refer to the survey~\cite{Mau03} for an exposition of the powerful and deep theory of type and cotype of Banach spaces; it suffices to say here that one can define linear invariants of Banach spaces that are called {\em type $2$} and {\em cotype $2$}, such that  $L_p$ has type $2$ if $2\le p<\infty$ and  cotype $2$ if $1\le p\le 2$, and such that the following extension and factorization phenomenon~\cite{Mau74} holds.

Suppose that $Y$ is a Banach space of type $2$ and that $Z$ is a Banach space of  cotype $2$. Let $X$ be a linear subspace of $Y$ and let $\tau\from X\to Z$ be a bounded linear operator. Then there exist a bounded linear operator $T\from Y\to Z$ that extends $\tau$, a Hilbert space $H$ and bounded linear operators $A\from Y\to H$,  $B\from A(Y)\to Z$ with $T=BA$.

\cite{JL84} raised the question of when the analogous statement holds  in the metric setting. Namely, now $Y,Z$ are metric spaces, $X$ is an arbitrary subset of $Y$, $f\from X\to Z$ is a Lipschitz mapping, and we ask for the same extension and factorization through a Hilbert space $H$, i.e., to establish the existence of Lipschitz mappings $F\from Y\to Z$, $\alpha\from Y\to H$ and $\beta\from \alpha(Y)\to Z$, such that the following diagram commutes.
\begin{equation}\label{eq:extension and factorization}
\begin{split}
\xymatrix  {
Y \ar[r]^\alpha  \ar[dr]^F &
\alpha(Y)\ar@{^{(}->}[r]^{\subset} \ar[d]^\beta & H\\
\strut   X \ar@{^{(}->}[u]^{\subset} \ar[r]^{f} & Z &
}
\end{split}
\end{equation}
An implicit but central part of this endeavor encompasses the important issue of how to define useful notions of type 2 and cotype 2 for metric spaces so that, at the very least, $\ell_p$ has type $2$ for $2\le p<\infty$ and cotype $2$ for $1\le p\le 2$. Clearly~\eqref{eq:extension and factorization} has two components. The first is if $f$ admits the Lipschitz extension $F$. The second is if $F$ can be factored through a subset of a  Hilbert space.  While these questions come hand-in-hand in the linear theory of~\cite{Mau74} (see also~\cite{Pis86}), they are different issues in the metric setting.

The main focus of~\cite{JL84} was the Lipschitz extension problem, so it highlighted the first component above. At the time, the metric version of the extension problem was a bold and speculative question, but~\cite{Bal92} introduced metric notions of type $2$ and cotype $2$ and obtained a powerful extension result for maps from spaces of Markov type $2$ to spaces of Markov cotype $2$.  Combined with~\cite{NPSS06}, this provides a quite satisfactory understanding of the extension component of~\eqref{eq:extension and factorization} when the target space is $\ell_p$, $1<p<2$.  However, this understanding is currently confined to the reflexive range, and the question remains a major open problem when the target space is $\ell_1$ (see~\cite{Kal12,MN13-bary} for a partial negative answer, and~\cite{MM16} for an intriguing algorithmic reformulation).

In contrast to the achievement of~\cite{Bal92}, Theorem~\ref{thm:factorization} demonstrates that there is \emph{no} way to define notions of type 2 and cotype 2 for metric spaces so that any map from a space of type $2$ to a space of cotype $2$ factors through Hilbert space and such that $\ell_p$ has type 2 when $2< p<\infty$ and cotype $2$ when $p=1$.  Though this resolves the factorization question when the target is $\ell_1$, it remains a fascinating open problem to see if a factorization theory analogous to~\cite{Bal92} can be developed when the target is $\ell_q$ for $1<q<2$.

It is instructive to examine the dual interpretation of Theorem~\ref{thm:factorization}. Just as the dual formulation of the linear factorization and extension problems was key to~\cite{Mau74}, duality also plays an important role in the nonlinear theory.  The duality lemma that was found in~\cite{Bal92} for Lipschitz extension\footnote{Quoting what \cite{Bal92} says about this crucial duality step: ``{\em This lemma is a variant of one used by Maurey. A related lemma was
found earlier by Johnson, Lindenstrauss and Schechtman: their result actually characterises extensions which factor through subsets of
Hilbert space, a problem much closer to Maurey's argument. Their lemma provided much of the stimulus for the present work.}'' Unfortunately, it seems that the work of Johnson, Lindenstrauss and Schechtman that is mentioned in~\cite{Bal92} was never published.} does not shed light on Lipschitz factorization, but the factorization issue was broached in~\cite{FJ09,CD14}. One can deduce from~\cite{CD14} the following factorization criterion. Given $\Phi>0$, metric spaces $(X,d_X)$, $(Z,d_Z)$ and $f\from X\to Z$, there exists a Hilbert space $H$ and a factorization $f=\beta\circ \alpha$ for some Lipschitz mappings $\alpha\from X\to H$ and $\beta\from \alpha(X)\to Z$ with $\|\alpha\|_{\Lip}\|\beta\|_{\Lip}\le \Phi$ if and only if
for all $n\in \N$ and $x_1,\ldots,x_n\in X$, any two symmetric stochastic matrices  $\mathsf{A}=(a_{ij}),\mathsf{B}=(b_{ij})\in \mathsf{M}_n(\R)$ such that $\mathsf{A}-\mathsf{B}$ is positive semidefinite satisfy the following quadratic  inequality.
\begin{equation}\label{eq:quadratic dual}
\sum_{i=1}^n\sum_{j=1}^n a_{ij} d_Z\big(f(x_i),f(x_j)\big)^2\le \Phi^2 \sum_{i=1}^n\sum_{j=1}^n b_{ij} d_X(x_i,x_j)^2.
\end{equation}
Theorem~\ref{thm:factorization} yields the first example of a Lipschitz mapping $f\from \ell_p\to \ell_1$ for $2<p<\infty$ that fails to satisfy~\eqref{eq:quadratic dual} for any $\Phi>0$, despite the fact that if $f$ were  a linear operator, then by~\cite{Mau74} it would automatically  satisfy~\eqref{eq:quadratic dual} with $\Phi\lesssim_p \|f\|_{\Lip}$.

\begin{remark}\label{rem:J}
  Another counterexample to the nonlinear version of~\cite{Mau74} arises from an embedding of the Laakso graphs into a non-classical Banach space.  Let $\{\Lambda_n\}_{n=1}^\infty$ be the Laakso graphs~\cite{Laa00,Laa02}, indexed so that $|\Lambda_n|=n$; these are series-parallel (hence planar) graphs that are $O(1)$--doubling when equipped with their shortest-path metric.

  On one hand, the Laakso graphs do not admit a bi-Lipschitz embedding into a Hilbert space.  In fact, by~\cite{Laa00,LP01}, we have $\cc_{\ell_2}(\Lambda_n)\gtrsim \sqrt{\log n}$ (this is sharp by the general embedding theorem of~\cite{Rao99}).  Moreover, by~\cite{MN-SOCG}, for every uniformly convex Banach space $X$ we have $\lim_{n\to \infty}\cc_X(\Lambda_n)=\infty$.

  On the other hand, by~\cite{GNRS04}, we have $\sup_{n\in \N}\cc_{\ell_1}(\Lambda_n)<\infty$, and by~\cite{JS09}, we have $\sup_{n\in \N}\cc_{Y}(\Lambda_n)<\infty$ when $Y$ is a Banach space that is not reflexive.  By considering translates of the images of the embeddings in $\ell_1$ that are sufficiently widely spaced, we obtain a doubling subset $\Lambda\subset \ell_1$ such that $\cc_Y(\Lambda)<\infty$ for any nonreflexive Banach space $Y$ and $\cc_{X}(\Lambda)=\infty$ for any uniformly convex Banach space $X$.

  By~\cite{Jam78}, there exists a Banach space $\mathbb{J}$ that has  type $2$, yet $\mathbb{J}$ is not reflexive; a different construction of such a Banach space was found in~\cite{PX87}.  So, $\Lambda$ embeds bi-Lipschitzly into both the cotype $2$ space $\ell_1$ and the type $2$ space $\mathbb{J}$, yet not into a Hilbert space.  This is impossible in the linear setting; by~\cite{Kwa72}  a Banach space of type $2$ and cotype $2$ is isomorphic to a Hilbert space (this is a far reaching generalization of the aforementioned consequence of~\cite{KP61} that motivates Theorem~\ref{thm:two Lps}).   This reasoning also produces a stronger asymptotic estimate than \eqref{eq:Lq case}, since $\cc_{\ell_2}(\Lambda_n)\gtrsim \sqrt{\log n}$, but it cannot shed light on the $\ell_p$ setting of~\eqref{eq:Lq case} because it relies precisely on the non-reflexivity of $\mathbb{J}$ (through the use of~\cite{JS09}) to deduce that $\sup_{n\in \N} \cc_{\mathbb{J}}(\Lambda_n)<\infty$.

  The Laakso graphs also lead to a counterexample to the metric version of~\cite{Mau74}.  Let $\varphi\from \Lambda\to \mathbb{J}$ be a bilipschitz embedding.  Since $\Lambda$ is a doubling subset of $\ell_1$, one can use~\cite{LN05} to construct a Lipschitz map $f\from  \ell_1\to \mathbb{J}$ that extends $\varphi$.  As above, $f$ cannot factor through a Hilbert space (or even through any uniformly convex Banach space $X$) by Lipschitz maps, because such a factorization would produce a bilipschitz embedding of $\Lambda$ into a Hilbert space (respectively, into $X$).

  This discussion shows that if one is allowed to replace $\ell_p$ in Theorem~\ref{thm:two Lps}  and Theorem~\ref{thm:factorization}  by non-classical (indeed, ``exotic'' and hard to come by) Banach spaces such as $\mathbb{J}$, then it is possible to demonstrate the failure of the metric space version of~\cite{Mau74} and its important precursor~\cite{Kwa72} using well-known examples.

  Part of the impetus for the search for definitions of metric space notions of type $2$ and cotype $2$  was the hope of obtaining a metric version of the theorem of~\cite{Kwa72}, but it was well-known to experts that the metric definitions of type $2$ and cotype $2$ found over the past decades are not suitable for this purpose (see e.g.~the discussion in~\cite{DLP13}).  The above discussion demonstrates conclusively that it is impossible to define metric space notions of type $2$ and cotype $2$  that are bi-Lipschitz invariant, pass to subsets, coincide for Banach spaces with  type $2$ and cotype $2$, and for which~\cite{Kwa72} holds for doubling metric spaces, i.e., any doubling space that has both type $2$ and cotype $2$ admits a bi-Lipschitz embedding into a Hilbert space (the corresponding statement with $\Lambda$ replaced by a metric space that is not doubling follows by using~\cite{Bou86} instead of the Laakso graphs in the above reasoning; in fact, using the improvement~\cite{Bau07} of~\cite{Bou86}, the infinite binary tree embeds bilipschitzly into both $\ell_1$ and $\mathbb{J}$, but not into a Hilbert space).   Theorem~\ref{thm:two Lps}  shows that this is so even if one restricts attention to subsets of $\ell_p$ for $p>2$.
\end{remark}

\subsubsection{Dimension reduction}\label{sec:dim reduction} By a highly influential lemma of~\cite{JL84}, any finite subset $S$ of a Hilbert space embeds with bi-Lipschitz distortion $O(1)$ into a $k$--dimensional Hilbert space for $k\lesssim \log |S|$; see~\cite{Nao18} for an indication of the significance of this statement. The question whether this phenomenon holds with Hilbert space replaced by $\ell_1$ was a prominent open problem until it was resolved negatively in~\cite{BC05}, where it was shown that for  arbitrarily large $n\in \N$ there is an $n$--point subset $D_n$ of $\ell_1$ such that if $D_n$ embeds with bi-Lipschitz distortion $O(1)$ into $\ell_1^k$, then necessarily $k\ge n^c$ for some universal constant $c>0$. In~\cite{LMN05} it was shown that $D_n$ can be taken   to be $O(1)$--doubling, and in~\cite{NPS18} it was shown that $\ell_1^k$ can be replaced by an arbitrary $k$--dimensional subspace of the Schatten--von Neumann trace class $\mathsf{S}_1$; both of these enhancements hold without changing the conclusion (other than perhaps values of universal constants).

The examples $\{D_n\}_{n=1}^\infty$ of~\cite{BC05} are the diamond graphs~\cite{NR03}, while their aforementioned doubling counterparts in~\cite{LMN05}  are the Laakso graphs $\{\Lambda_n\}_{n=1}^\infty$ that we discussed in Remark~\ref{rem:J}. By~\cite{MN-SOCG,JS09} we have $\sup_{n\in \N}\cc_X(D_n)=\sup_{n\in \N}\cc_X(\Lambda_n)=\infty$ for every uniformly convex Banach space $X$. In fact, by~\cite{JS09} the converse of this statement holds true (though we do not need it below), namely $X$ admits an equivalent uniformly convex norm if and only if $\sup_{n\in \N}\cc_X(D_n)=\infty$ or $\sup_{n\in \N}\cc_X(\Lambda_n)=\infty$. Theorem~\ref{thm:dim reduction}  below obtains new examples that demonstrate the failure of dimension reduction in $\ell_1$ à la~\cite{JL84}, which are qualitatively different than the previously known examples, since our examples do admit a bi-Lipschitz embedding into a uniformly  convex Banach space (specifically, into $\ell_p$ for any $p>2$). At present, this comes with a worse lower bound on the target dimension, but see Remark~\ref{rem:sharp in p} below which explains how Conjecture~\ref{conj:function of p} would remedy this (for the very same example that we consider in Theorem~\ref{thm:dim reduction}).

\begin{thm}\label{thm:dim reduction} There is a universal constant $c>0$ with the following property.  For every $n\in \N$ and $2<p\le 4$ there exists a  $O(1)$--doubling subset $\mathscr{H}_n=\mathscr{H}_n(p)$ of $\ell_1$ with $|\mathscr{H}_n|\le n$ such that $\cc_{\ell_q}(\mathscr{H}_n)\lesssim 1$ for all $q\ge p$, and for every $D\ge 1$, if $X$ is a finite-dimensional subspace of the Schatten--von Neumann trace class $\mathsf{S}_1$ for which $\cc_X(\mathscr{H}_n)\le D$, then necessarily
\begin{equation}\label{eq:dim X}
\dim(X)\ge \exp \bigg(\frac{c}{D^2}(\log n)^{1-\frac{2}{p}}\bigg).
\end{equation}
\end{thm}

In the statement of Theorem~\ref{thm:dim reduction}, recall that for $p\ge 1$ the Schatten--von Neumann trace class $\mathsf{S}_p$ is the Banach space of all the compact operators $T\from \ell_2\to \ell_2$ that satisfy
\begin{equation*}
\|T\|_{\mathsf{S}_p}\eqdef\Big(\trace\big[(T^*T)^{\frac{p}{2}}\big]\Big)^{\frac{1}{p}}<\infty.
\end{equation*}
Note that $\ell_p$ is the subspace of $\mathsf{S}_p$ consisting of the diagonal operators. Thus, the dimension reduction lower bound~\eqref{eq:dim X} holds in particular for any subspace $X$ of $\ell_1$.

The proof of Theorem~\ref{thm:dim reduction} is short (modulo previously stated results and the available literature), so we present the quick derivation now instead of postponing it to a later section; it mimics the reasoning of~\cite{LN04} while combining it with~\cite{LafforgueNaor},  Theorem~\ref{thm:twist the center} and Theorem~\ref{thm:twist the center Lp}, as well as structural information on subspaces of $\mathsf{S}_1$ from~\cite{NPS18}.

\begin{proof}[Proof of Theorem~\ref{thm:dim reduction}]By~\eqref{eq:equiv to Kor} we have $|\mathcal{B}_m|\asymp m^4$ for all $m\in \N$. So, fix $m\in \N$ with $m\asymp{\sqrt[4]{n}}$ such that $n\lesssim |\mathcal{B}_m|\le n$. Using the mapping $\phi_{m,\frac{1}{p}}\from \H_\Z\to\ell_1$ of Theorem~\ref{thm:twist the center}, define
$$
\mathscr{H}_n\eqdef \phi_{m,\frac{1}{p}}(\mathcal{B}_m).
$$
By combining Theorem~\ref{thm:twist the center} and Theorem~\ref{thm:twist the center Lp}, we indeed have  $\cc_{\ell_q}(\mathscr{H}_n)\lesssim 1$ for all $q\ge p$.

Let $X$ be a finite-dimensional subspace of $\mathsf{S}_1$. Fix $1<r\le 2$  whose value will be specified later so as to optimize the ensuing reasoning. By~\cite[Theorem~12]{NPS18}, we have\footnote{If one only wishes to rule out embeddings into low-dimensional subspaces of $\ell_1$ rather than of $\mathsf{S}_1$, then it suffices to use here~\cite[Theorem~1.2]{LT80}, which yields an embedding into $\ell_r$ rather than $\mathsf{S}_r$.}
$$
\cc_{\mathsf{S}_r}(X)\le \dim(X)^{1-\frac{1}{r}}.
$$
Hence, if $\cc_X(\mathscr{H}_n)\le D$, then, since $\cc_{\mathscr{H}_n}(\mathcal{B}_m)\lesssim (\log n)^{\frac{1}{p}}$ by Theorem~\ref{thm:twist the center}, we have
$$
\cc_{\mathsf{S}_r}(\mathcal{B}_m)\lesssim (\log n)^{\frac{1}{p}}\cc_{\mathsf{S}_r}(\mathscr{H}_n)\le (\log n)^{\frac{1}{p}}D\cc_{\mathsf{S}_r}(X)\le (\log n)^{\frac{1}{p}}D\dim(X)^{1-\frac{1}{r}}.
$$
At the same time, by~\cite{LafforgueNaor} we have\footnote{As in the discussion before Conjecture~\ref{conj:function of p}, the dependence on $r$ in this estimate is not stated in~\cite{LafforgueNaor}, while it is crucial for us here; a justification why the reasoning in~\cite{LafforgueNaor} implies   this appears in Appendix~\ref{sec:littlewood paley}.} $\cc_{\mathsf{S}_r}(\mathcal{B}_m)\gtrsim \sqrt{(r-1)\log n}$, so we conclude that
$$
\inf_{1<r\le 2}\frac{\dim(X)^{1-\frac{1}{r}}}{\sqrt{r-1}}\gtrsim \frac{(\log n)^{\frac12-\frac1{p}}}{D}.
$$
This gives the desired bound~\eqref{eq:dim X} by choosing $r-1\asymp 1/\log(\dim(X))$.
\end{proof}

\begin{remark}\label{rem:sharp in p} By substituting~\eqref{eq:ask sqrt4} into the reasoning of~\cite{LafforgueNaor}, a positive resolution of Conjecture~\ref{conj:function of p} would imply that for every  $r\in (1,2]$ and $n\in \N$ we have
\begin{equation}\label{eq:fourth dist ellp}
\cc_{\ell_r}(\mathcal{B}_n)\gtrsim \sqrt[4]{r-1}\cdot\sqrt{\log n}.
\end{equation}
An incorporation  of this improved distortion lower bound into the above proof of Theorem~\ref{thm:dim reduction} (while using~\cite{LT80} in place of~\cite{NPS18} since we are in the simpler $\ell_p$ setting) would imply that for any finite-dimensional subspace $X$ of $\ell_1$, if $\cc_X(\mathscr{H}_n(p))\le D$, then the following improvement over~\eqref{eq:dim X} holds true.
\begin{equation}\label{eq:D4}
\dim(X)\ge \exp \bigg(\frac{c}{D^4}(\log n)^{2-\frac{4}{p}}\bigg).
\end{equation}
Notably, for $p=4$ this would be an improvement from $\dim(X)\ge \exp \left(\frac{c}{D^2}\sqrt{\log n}\right)$ to
\begin{equation}\label{eq:power type D4}
\dim(X)\ge n^{\frac{c}{D^4}},
\end{equation}
namely a power-type dimension reduction lower bound as in~\cite{BC05}. Understanding what is the correct behavior as $p\to 2^+$ remains an intriguing open question; some deterioration of the lower bound as in~\eqref{eq:dim X} or~\eqref{eq:D4} must occur because by~\cite{JL84} logarithmic dimension reduction is possible for finite subsets of a Hilbert space.

Another question that this discussion obviously raises is if~\eqref{eq:fourth dist ellp} could be enhanced to
\begin{equation}\label{eq:fourth dist Sp}
\cc_{\mathsf{S}_r}(\mathcal{B}_n)\gtrsim \sqrt[4]{r-1}\cdot\sqrt{\log n}.
\end{equation}
If so, then~\eqref{eq:power type D4} would hold when $X$ is a subspace of $\mathsf{S}_1$ rather than $\ell_1$. More substantially, this would resolve a difficult open question (see the discussion following Question~13 in~\cite{NY18}) by showing that $\H_\Z$ does not admit a bi-Lipschitz embedding into $\mathsf{S}_1$. In fact, for the latter conclusion it would suffice to establish the weaker property
\begin{equation}\label{eq:eta p}
\lim_{n\to \infty} \cc_{\mathsf{S}_{1+\frac{1}{\log n}}}(\mathcal{B}_n)=\infty.
\end{equation}
Indeed, by~\cite{NPS18} we have $\cc_{\mathsf{S}_1}(\mathcal{B}_n)\gtrsim \cc_{\mathsf{S}_r}(\mathcal{B}_n)$ when $r=1+1/\log n$.  Due to its  significant consequences, we expect that proving~\eqref{eq:eta p}, and all the more so its stronger version~\eqref{eq:fourth dist Sp}, would  require a major and  conceptually new idea.
\end{remark}

We end this discussion on dimension reduction by noting that~\cite{Tao19} shows that one could embed $\mathcal{B}_n$ with optimal distortion (up to universal constant factors) into Euclidean space of dimension $O(1)$.  Theorem~\ref{thm:dim reduction} shows that this fails badly  if one aims for optimal $\ell_1$--distortion embedding of $\mathcal{B}_n$ into a bounded dimensional subspace of $\ell_1$.

\subsubsection{Permanence of compression rates of groups} Suppose that $(M,d_M)$ is a metric and $(X,\|\cdot\|_X)$ is a Banach space. The {\em compression rate} of a Lipschitz mapping $f\from M\to X$ is the non-decreasing function $\omega_f\from  [0,\infty)\to [0,\infty)$  that is defined~\cite{Gro93} by
\begin{equation}\label{eq:compression rate}
\forall\, s\ge 0,\qquad \omega_f(s)\eqdef \inf_{\substack{x,y\in M\\ d_M(x,y)\ge s}} \|f(x)-f(y)\|_X.
\end{equation}
Equivalently, $\omega_f$ is the largest non-decreasing function from $[0,\infty)$ to $[0,\infty)$ such that $$\forall\, x,y\in M,\qquad \|f(x)-f(y)\|_X\ge \omega_f\big(d_M(x,y)\big).$$

There is a great deal of  interest in determining the largest possible compression rate of $1$--Lipschitz mappings from a finitely generated group $G$ (equipped with a word metric that is induced by some finite generating set) to certain Banach spaces, notable and useful examples of which are Hilbert space and $L_1$. The  literature on this topic is too extensive to discuss here, and we only mention that a substantial part of it is devoted to understanding the extent to which compression rates are preserved under various group operations (e.g.~various semidirect products). Theorem~\ref{thm:permanence} below provides a new example of the lack of such permanence which does not seem to be accessible using previously available methods. It leverages the fact that we establish here a marked difference between the $L_1$ embeddability of Heisenberg groups of dimension $3$ and dimension $5$.

\begin{thm}\label{thm:permanence} There exists a finitely group $G$  that has two finitely generated normal subgroups $H,K\triangleleft G$ such that the following properties hold true.
\begin{enumerate}
    \item Any $h\in H$ and $k\in K$ commute.
    \item $H\cap K$ is the center of $G$.
    \item $H$ and $K$ are isomorphic.
    \item $H$ and $K$ are undistorted in $G$; in fact, they admit generating sets $S_H$ and $S_K$ such that $S_H\cup S_K$ generates $G$ and the word metric on $G$ that is induced by $S_H\cup S_K$ restricts to the word metrics on $H$ and $K$ that are induced by $S_H$ and $S_K$, respectively.
    \item The $L_1$ compression of $G$ is asymptotically smaller than that of $H$ (hence also of $K\cong H$). Concretely, there exists a Lipschitz mapping $f\from H\to \ell_1$ that satisfies
        \begin{equation}\label{eq:H3 compression}
        \forall\, s\ge 3,\qquad \omega_{f}(s)\gtrsim \frac{s}{\sqrt[4]{\log s}\cdot (\log \log s)^2},
        \end{equation}
        yet for any Lipschitz mapping $F\from G\to L_1$ there are arbitrarily large $s\ge 4$ for which
        \begin{equation}\label{eq:H5 compression}
        \omega_F(s)\le \frac{s}{\sqrt{(\log s)\log\log s}}.
        \end{equation}
\end{enumerate}
\end{thm}

\begin{proof} Let $G_\R$ be the $5$--dimensional Heisenberg group, i.e., $\R^5$ with the group operation
\begin{multline*}
(x_1,y_1,x_2,y_2,z)(x_1',y_1',x_2',y_2',z')\\=\Big(x_1+x_1',y_1+y_1',x_2+x_2',y_2+y_2',z+z'+\frac{1}{2}(x_1y_1'+x_2y_2'- y_1x_1'-y_2x_2')\Big)
\end{multline*}
for $(x_1,y_1,x_2,y_2,z), (x_1',y_1',x_2',y_2',z')\in \R^5$.  Let $G$ be the $5$--dimensional integer Heisenberg group, which is the subgroup $G=\big\{(x_1,y_1,x_2,y_2,z+(x_1y_1+x_2y_2)/2)  : x_1,x_2,y_1,y_2,z\in \Z\big\}.$
The subgroups $H,K$ are natural copies of $\H_\Z$ in $G$, namely $$H=\{(x_1,y_1,x_2,y_2,z)\in G: x_2=y_2=0\} \qquad \mathrm{and} \qquad K=\{(x_1,y_1,x_2,y_2,z)\in G: x_1=y_1=0\}.$$ One directly checks the first four assertions of Theorem~\ref{thm:permanence}. The bound~\eqref{eq:H3 compression} follows by considering the mapping $f\from \H_\Z\to \ell_1(\ell_1)\cong\ell_1$ that is given by
$$
f\eqdef \bigoplus_{n=1}^\infty \frac{1}{n^2}\phi_{2^{2^n},\frac14},
$$
where the mappings that are being concatenated are those of Theorem~\ref{thm:twist the center}. The final assertion~\eqref{eq:H5 compression} of Theorem~\ref{thm:permanence} follows from~\cite[Theorem~9]{NY18}.
\end{proof}

\begin{remark} The term $\log\log s$ in~\eqref{eq:H3 compression} and~\eqref{eq:H5 compression} can be improved slightly; for~\eqref{eq:H3 compression} this follows by examining the above proof, and  for~\eqref{eq:H5 compression} this is explained by~\cite[Theorem~9]{NY18}. However, some unbounded lower-order correction is necessary in~\eqref{eq:H3 compression} for the specific groups that we used in the proof of Theorem~\ref{thm:permanence}; see Remark~\ref{rem:lower order term needed}.
\end{remark}

Obviously, Theorem~\ref{thm:permanence}  raises the  question if a similar phenomenon could occur for embeddings into a Hilbert space rather than into $L_1$. Also, in Theorem~\ref{thm:permanence} the compression rate of the subgroups $H,K$ grows roughly (suppressing lower-order factors) like $s/\sqrt[4]{\log s}$ as $s\to \infty$, while the compression rate of $G$  grows slower than $s/\sqrt{\log s}$. What are the possible asymptotic profiles of the compression rates that exhibit such phenomena?

\subsection{Decomposing surfaces into approximately ruled pieces} \label{sec:overview} In the previous sections, we discussed consequences of Theorem~\ref{thm:XYD} (and the refined version of its second part in Theorem~\ref{thm:twist the center}).  In this section, we will give an overview of the concepts involved in the proof of Theorem~\ref{thm:XYD}, especially our main contribution, which is a new way to describe the structure of surfaces in $\H$.

The statement of Theorem~\ref{thm:XYD} is in terms of smooth functions $f\from \H\to\R$, but the main bound~\eqref{eq:our main in intro} has an equivalent formulation in terms of surfaces in $\H$; see~\eqref{eq:use previous corona} below.  We will prove it by showing that surfaces in $\H$ admit a multi-scale hierarchical decomposition into pieces that are close to ruled surfaces (unions of horizontal lines) and that most of these pieces (in a quantitative sense) are long and narrow, giving the decomposition the appearance of a Venetian blind with many narrow slats; see Figure~\ref{fig:foliated corona of bump} and Figure~\ref{fig:bumpy surface} for examples. For reasons that will be clarified soon, we call the above structure a {\em foliated corona decomposition}.  This decomposition is conceptually central to this work, and the most involved part of this paper is to formulate this decomposition, prove its existence, and demonstrate its utility for the aforementioned applications (more are forthcoming).

The defining feature of this decomposition is that its pieces, which we call \emph{pseudoquads}, have widely varying aspect ratios.  Each pseudoquad is roughly rectangular, and we define the {\em aspect ratio} of a pseudoquad to be its width divided by its height; long, narrow rectangles have large aspect ratios, while tall, skinny rectangles have small aspect ratios.  The fact that the pieces of the decomposition (the slats of the Venetian blind) can have unbounded aspect ratios allows the decomposition to have additional symmetries and ultimately leads to the exponent $4$ in Theorem~\ref{thm:XYD}.

Specifically, in order to work with long, narrow pieces, we must prove results on the geometry of $\H$ that are invariant not only under the usual scaling automorphisms, but also under automorphisms that stretch and shear $\H$.  The resulting automorphism-invariant bounds allow us to produce a decomposition that is likewise invariant under rescaling, stretching, and shearing.  Furthermore, the overlap of the pieces of our decomposition is controlled by a coercive quantity that scales like the fourth power of the aspect ratio under automorphisms.  This leads to a new {\em weighted Carleson packing condition} in which overlaps are normalized by the fourth power of the aspect ratio; this condition leads directly to the exponent $4$ in the bound~\eqref{eq:our main in intro} of Theorem~\ref{thm:XYD}.

Proving the optimality of Theorem~\ref{thm:XYD} entails finding a surface for which~\eqref{eq:use previous corona} is sharp.  Part of the construction of such a surface can be seen in Figure~\ref{fig:bumpy surface}.  The surface in the figure can be viewed as a surface with a foliated corona decomposition for which the weighted Carleson packing condition is sharp.  For this reason, it is pedagogically beneficial to describe that  construction after describing foliated corona decompositions. In truth, the general decomposition methodology and the construction that demonstrates its optimality are intertwined:  limitations of such a construction indicate what decomposition to look for.  We therefore suggest to also consider the alternative route of first examining the construction of the specific (sharp) example prior to considering the task of decomposing general surfaces; the proofs in the rest of this article follow the latter (``reverse'') route as this leads to a more gradual introduction of notations and concepts.

The ensuing considerations belong firmly to the setting of the continuous Heisenberg group and its Carnot--Carath\'eodory geometry. They therefore assume some familiarity with notions from that setting; the pertinent background appears in Section~\ref{sec:prem} below.

\subsubsection{Fractal Venetian blinds abound}\label{sec:corona intro} In what follows, for any $s>0$ the  Hausdorff measure $\cH^s$ on $\H$ will be with respect to the Carnot--Carath\'eodory metric $d$ on $\H$. We denote the standard generators of $\H$ by $X=(1,0,0), Y=(0,1,0)$, and $Z=(0,0,1)$.

For $\Omega\subset \H$ and $a\in \R$, consider the symmetric difference
\begin{equation}\label{eq:def D set}
  \mathsf{D}_a\Omega\eqdef  \Omega \symdiff \Omega Z^{2^{-2a}}= \bigl(\Omega\setminus \Omega Z^{2^{-2a}}\bigr) \cup \bigl(\Omega Z^{2^{-2a}}\setminus \Omega\bigr)
\end{equation}
If $\Omega,U\subset \H$ are measurable, then, following~\cite{LafforgueNaor,NY18}, we define  $\vpfl{U}(\Omega)\from \R\to \R$ by
\begin{equation}\label{eq:def vert per}
\forall\, a\in \R,\qquad \vpfl{U}(\Omega)(a)\eqdef 2^a \cH^4\left(U\cap \mathsf{D}_a\Omega\right)=2^a\int_U\big|\1_\Omega(h)-\1_\Omega\bigl(hZ^{-2^{-2a}}\bigr)\big|\ud\cH^4(h).
\end{equation}
Thus, $\vpfl{U}(\Omega)(a)$ is a (normalized) measurement of the amount that $\Omega$ changes within $U$ when translated up and down by the specified (Carnot--Carath\'eodory) distance $2^{-a}$.

By~\cite[Lemma~38]{NY18}, in order to prove the first part of Theorem~\ref{thm:XYD}, namely inequality~\eqref{eq:our main in intro} for any compactly supported smooth function $f\from \H\to \R$, it suffices to prove that every measurable subset $\Omega\subset \H$ satisfies the following isoperimetric-type inequality.
\begin{equation}\label{eq:coarea}
\big\|\vpfl{\H}(\Omega)\big\|_{L_4(\R)}\lesssim \cH^3(\partial \Omega).
\end{equation}
This amounts in essence to an application of the coarea formula (e.g.~\cite{Amb01}).

A central step of~\cite{NY18} is a further reduction of~\eqref{eq:coarea} to the special case that $\Omega$ is (a piece of) an   {\em intrinsic Lipschitz epigraph}.  An intrinsic Lipschitz epigraph $\Gamma^+$ is a region of $\H$ that is bounded by an intrinsic Lipschitz graph $\Gamma$.  The notion of an intrinsic Lipschitz graph was introduced in~\cite{FSSC06} and all of the relevant background is explained in Section~\ref{sec:intrinsic graphs} below.  The intrinsic Lipschitz condition is parametrized by an intrinsic Lipschitz constant $\lambda\in (0,1)$.  By combining Proposition~55, Theorem~57 and Lemma~58 of~\cite{NY18} (see the deduction on page~232 of~\cite{NY18}) it follows that to prove~\eqref{eq:coarea} it suffices to show that for every  $0<\lambda<1$ the vertical perimeter of any intrinsic $\lambda$--Lipschitz epigraph $\Gamma^+\subset \H$ satisfies the growth bound
\begin{equation}\label{eq:use previous corona}
\forall\, r>0,\qquad \big\|\vpfl{B_r(\0)}\big(\Gamma^+\big)\big\|_{L_4(\R)}\lesssim_\lambda r^3,
\end{equation}
where  $B_r(\0)$ denotes the (Carnot--Carath\'eodory) ball   of radius $r$  centered at $\0=(0,0,0)$.

The structural information that underlies the reduction of~\eqref{eq:coarea} to~\eqref{eq:use previous corona}  is that for any $0<\lambda<1$, any (sufficiently nice; see~\cite{NY18} for precise assumptions) surface in $\H$ has a multi-scale hierarchical decomposition into pieces that are close to intrinsic $\lambda$--Lipschitz graphs, and moreover that decomposition has controlled overlap in the sense that it satisfies  a $O(1)$--Carleson packing condition. As such, this decomposition is an intrinsic Heisenberg analog of the {\em corona decompositions} that were introduced and developed for subsets of Euclidean space in~\cite{DavidSemmesSingular} and have since led to a variety of powerful applications in harmonic analysis (see also the monograph~\cite{DSAnalysis}).

The corona decomposition of~\cite{NY18} is in some respects a Heisenberg variant of a ``vanilla'' corona decomposition.  Like corona decompositions in $\R^n$, it is a hierarchical partition of a surface into pieces of bounded aspect ratio, and the Carleson packing condition governing overlaps of pieces depends only on the diameter of the pieces. Nevertheless, there are  key differences, including the fact that the proof in~\cite{NY18} relies on a new ``stopping rule'' (based on the quantitative nonmonotonicity of~\cite{CKN}) that yields, in fact, a different proof of the existence of corona decompositions even in Euclidean space (though, for less general sets than those that~\cite{DavidSemmesSingular} treats).  In addition, while ``vanilla'' Euclidean corona decompositions cover a surface in $\R^n$ by pieces that are approximately graphs of Lipschitz functions, the approximating graphs in~\cite{NY18} are intrinsic Lipschitz, like the surface depicted in Figure~\ref{fig:Lip10}. While Lipschitz graphs in Euclidean space vary slowly in all directions, intrinsic Lipschitz graphs vary slowly in horizontal directions but can vary quickly in vertical directions and can have Hausdorff dimension 2.5 with respect to the Euclidean metric~\cite{KirSC04}.
This can make these graphs difficult to analyze, and even after the decomposition step of~\cite{NY18}, the challenge of establishing estimates such as~\eqref{eq:use previous corona}  remains.

In~\cite{NY18}, we addressed this challenge for the $5$--dimensional Heisenberg group $\H^5$, but our techniques do not shed light on the $3$--dimensional setting of Theorem~\ref{thm:XYD}.  An intrinsic Lipschitz graph in $\H^5$ is the intrinsic graph of a function $\psi$ that is defined on a $4$--dimensional vertical hyperplane $V_0$. An inspection of the intrinsic Lipschitz condition shows that  the restriction of $\psi$ to any coset of $\H$ that is contained in $V_0$ is Lipschitz with respect to the Carnot--Carath\'eodory metric on $\H$. In~\cite{NY18}, we applied a representation-theoretic functional inequality of~\cite{AusNaoTes} to each of these restrictions, yielding a bound on the vertical variation of $\psi$. The desired control on the vertical perimeter of intrinsic Lipschitz graphs in $\H^5$ followed by  integrating this bound over the cosets of $\H$ in $V_0$.

In the  $3$--dimensional setting of the present work, the intrinsic graph $\Gamma$ in~\eqref{eq:use previous corona} corresponds to an intrinsic Lipschitz function $\psi\from V_0\to \R$, where $V_0$ is a $2$--dimensional vertical plane in $\H$. For concreteness, assume in what follows that $V_0=\{(x,0,z):\ x,z\in \R\}$ is the $xz$--plane.  The reasoning of~\cite{NY18} is irrelevant to proving~\eqref{eq:use previous corona}: one cannot restrict $\psi$ to cosets of a lower-dimensional Heisenberg group, as there is no such group!

Our strategy here is therefore entirely different from that of~\cite{NY18}. We will prove~\eqref{eq:use previous corona} by finding a new structural  description of intrinsic Lipschitz graphs in $\H$.  Specifically, we will prove that they admit a hierarchical family of partitions into pieces that are approximately ruled surfaces and bound the total error of these approximations.

We call this description of $\Gamma$ a \emph{foliated corona decomposition}.  It is a sequence of nested partitions of $\Gamma$ into approximately rectangular regions, called \emph{pseudoquads}, of varying heights and widths.  On each pseudoquad, $\Gamma$ is close to a vertical plane, and these vertical planes can be glued together to form a collection of ruled surfaces such that at most locations and scales, $\Gamma$ is approximated by one of the ruled surfaces; see Remark~\ref{rem:ruled-surfaces}.  Furthermore, the decomposition satisfies a new {\em weighted} variant of the classical  \emph{Carleson packing condition}.  Namely, we bound the weighted sum of the measures of the pseudoquads in the decomposition, where the measure of each pseudoquad is normalized by the fourth power of its aspect ratio.  We will see that the occurrence of the fourth power here is {\em dictated} by the requirement that this decomposition should be invariant under certain automorphisms of $\H$ (scaling, stretch, and shear automorphisms).

\begin{thm}\label{thm:ilg admits fcd intro}
  Any intrinsic Lipschitz graph in $\H$ has a foliated corona decomposition.
\end{thm}
The above description of foliated corona decompositions and the statement of Theorem~\ref{thm:ilg admits fcd intro} clearly lack rigorous definitions, but they convey the essence of what is achieved here. The necessary  technical matters are treated  in Section~\ref{sec:decompose here} below, where a precise formulation of  Theorem~\ref{thm:ilg admits fcd intro}  appears as Theorem~\ref{thm:corona graph formulated}. The justification that Theorem~\ref{thm:ilg admits fcd intro} can be used to achieve our goal~\eqref{eq:use previous corona} is carried out in Section~\ref{sec:fcd bounds vper} below; the groundwork of constructing a foliated corona decomposition makes this deduction quite mechanical.

We will next cover a few technical details necessary to describe foliated corona decompositions and the subdivision mechanism that produces them.  Recall that $V_0\subset \H$ is the $xz$--plane. Fix $0<\lambda<1$  and let $\Gamma$ be an intrinsic $\lambda$--Lipschitz graph that is the intrinsic graph of $\psi\from V_0\to \R$.  That is, $\Gamma=\Psi(V_0)$, where $\Psi(v)=vY^{\psi(v)}$ for all $v\in V_0$.  The function $\psi$ satisfies the intrinsic Lipschitz condition (Definition~\ref{def:intrinsic lipschitz graph}); the nonlinear nature of this condition is the source of subtleties that ensue (and the reason why basic questions on the rectifiability properties of intrinsic Lipschitz graphs remain open; see e.g.~\cite{DFO20}).

For any $p\in \Gamma$, there is a horizontal curve $\gamma$ contained in $\Gamma$ that passes through $p$, so $\Gamma$ is the union of all such curves. It is often convenient to work in $V_0$ instead of $\Gamma$.  To this end, let $\Pi\from \H\to V_0$ be the  projection to $V_0$, so $\Pi(\Psi(v))=v$ for  $v\in V_0$.  The projected curve $\Pi\circ \gamma$ is a curve in $V_0$ which we call a \emph{characteristic curve}; see Section~\ref{sec:char} for a detailed discussion.  Parametrize $\gamma$ so that $\Pi(\gamma(t))=(t,0,g(t))$ for some continuous function $g$.  This function is a solution of the differential equation $g'(t)=-\psi(t,0,g(t))$, and conversely, each solution gives  a characteristic curve.  If $\Gamma$ is a vertical plane, then $\psi(x,0,z)=ax+b$ for some $a,b\in \R$, in which case the characteristic curves are parallel parabolas.

Since horizontal curves pass through every point of $\Gamma$, there is a characteristic curve through every point of $V_0$, so one can reconstruct $\Gamma$ from its set of characteristic curves.  Note that the characteristic curve through $p$ is not necessarily unique: when $\psi$ is not smooth, these curves can split and rejoin~\cite{BigolinCaravennaSerraCassano}.  When $\psi$ is smooth, the characteristic curves foliate $V_0$, so there is a coordinate system on $V_0$ such that the foliation forms one set of coordinate lines.  However,  it is  difficult to use this coordinate system to study the geometry of $\Gamma$ because the distance between two characteristic curves can vary wildly. Foliated corona decompositions provide a way to overcome this difficulty.

A \emph{pseudoquad} for $\Gamma$ is a region in $V_0$ that is bounded by characteristic curves above and below and by vertical line segments on either side.  We call a pseudoquad $Q$ \emph{rectilinear} if its top and bottom boundaries approximate two parallel parabolas; if the top and bottom boundaries of $Q$ are exactly two parallel parabolas, we call $Q$ a \emph{parabolic rectangle}.  Parabolic rectangles are the projections to $V_0$ of rectangles in $\H$ bounded by two horizontal line segments and two vertical line segments.  The width $\delta_x(Q)$ and height $\delta_z(Q)$ of such a pseudoquad are defined to be, respectively, the width and height of its approximating parabolic rectangle; see Section~\ref{sec:pseudoquads}. The  aspect ratio of $Q$ is $\alpha(Q)=\delta_x(Q)/\sqrt{\delta_z(Q)}$.

Let $Q_0\subset V_0$ be a rectilinear pseudoquad.  A foliated corona decomposition for $\Gamma$ with root at $Q_0$ is a sequence of nested partitions of $Q_0$ into rectilinear pseudoquads. We construct such a decomposition using the following {\em subdivision algorithm} which, importantly, outputs pseudoquads that can be divided into two sets $\cVv$ and $\cVh$, called, respectively, the {\em vertically cut} pseudoquads and {\em horizontally cut} pseudoquads. The algorithm repeatedly cuts pseudoquads into halves.  Let $Q$ be a pseudoquad in the decomposition.  If $\Psi(Q)$ is a region in $\Gamma$ that is sufficiently close to a vertical plane $V_Q$ and if the characteristic curves through $Q$ are close to characteristic curves for $V_Q$, then  cut $Q$ in half along one of the characteristic curves of $\Gamma$.  In this case, say that $Q$ is horizontally cut and add it to $\cVh$.  Otherwise, cut $Q$ in half along a vertical line through its center, say that $Q$ is vertically cut, and add it to $\cVv$.  By applying this procedure iteratively, we obtain a sequence of nested partitions of $Q_0$; see Figure~\ref{fig:foliated corona of bump}.

\begin{figure}
  \begin{centering}
    \includegraphics[width=.8\textwidth]{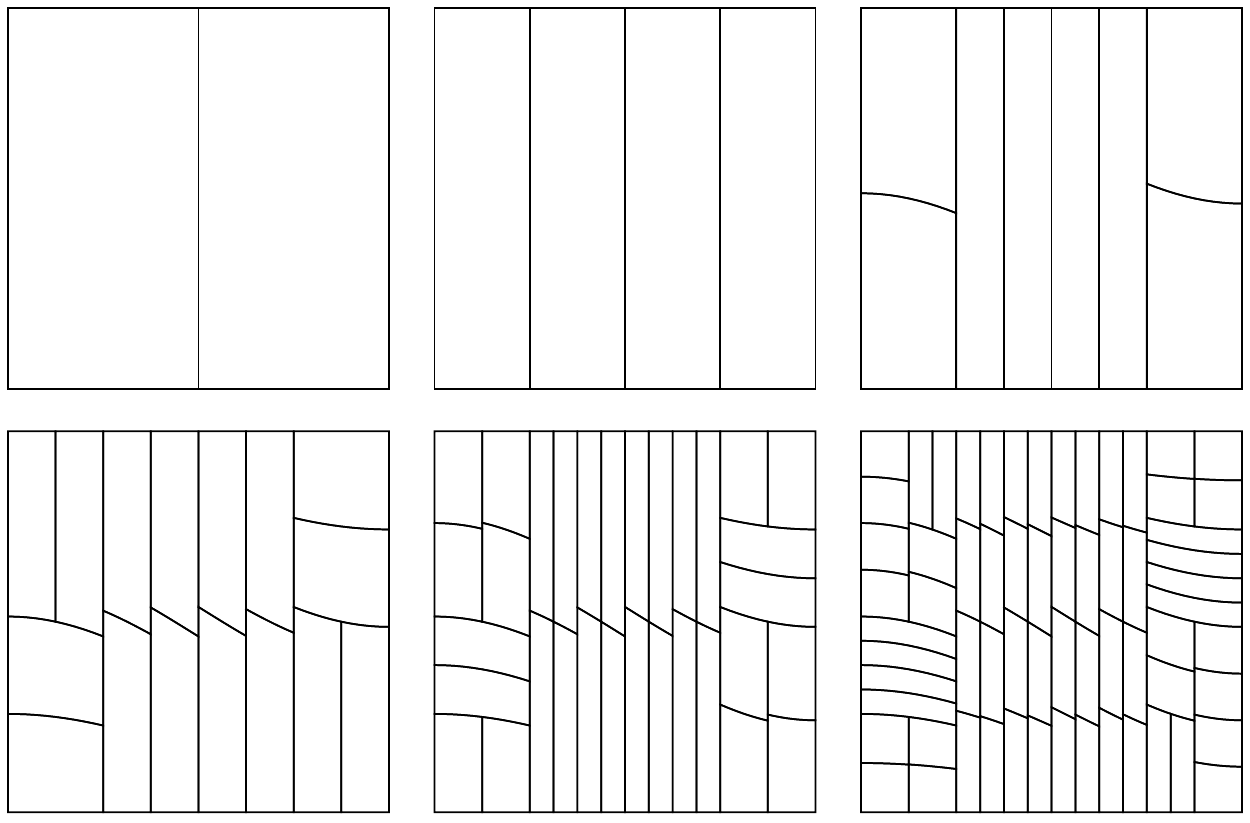}\\
    \includegraphics[width=.6\textwidth]{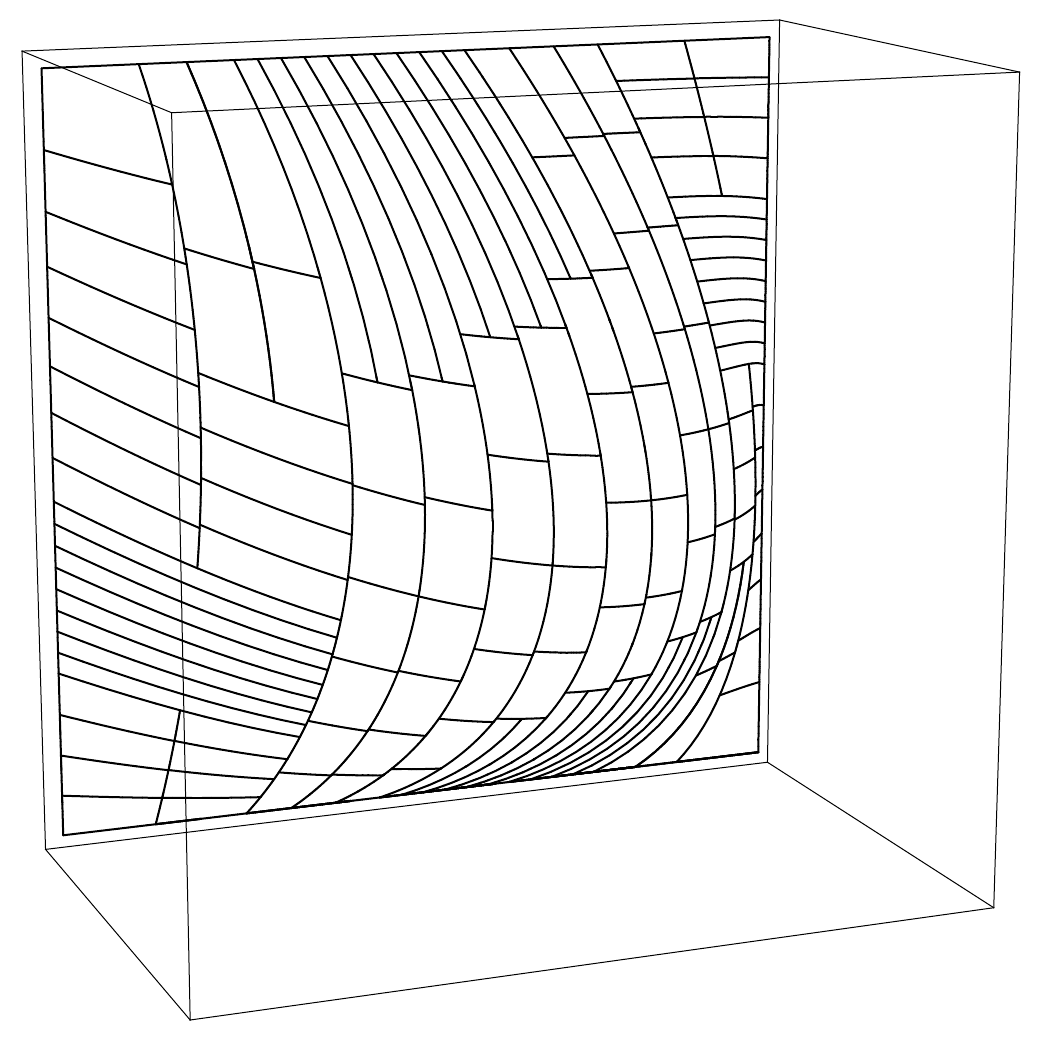}
  \end{centering}
  \caption{\label{fig:foliated corona of bump}{\small \em Stages in the construction of a foliated corona decomposition for a bump function as in the top row of Figure~\ref{fig:bumpy surface}.  The aspect ratio of the regions in the decomposition varies widely.  On the sides, where the surface is close to a vertical plane, the aspect ratio is large and the regions are short and wide; near the top and bottom, where it is further from a plane, the regions are tall and narrow.}}
\end{figure}

A crucial part of the algorithm is the mechanism determining whether to cut the  pseudoquad horizontally or vertically.  We stated qualitatively how this step depends on the geometry of $\Psi(Q)$, but we implement it quantitatively by introducing a coercive quantity called \emph{$R$--extended nonmonotonicity}.  This is a family of measures $\Omega^P_{\Gamma^+,R}$ on the vertical plane $V_0$, parametrized by  $R>0$; see Section~\ref{sec:parametric monotonicity}. These are inspired by the quantitative nonmonotonicity of~\cite{CKN}, but there are key differences. For instance, while the nonmonotonicity of $\Gamma$ on a subset $U\subset \H$ measures how lines intersect $\Gamma$ inside $U$, the $R$--extended nonmonotonicity of $\Gamma$ on a subset $W\subset V_0$ measures how lines intersect $\Gamma$ inside an $R$--neighborhood of $\Psi(W)$.
We refer to Section~\ref{sec:parametric monotonicity} for the details, in particular to Lemma~\ref{lem:sums of OmegaP} which shows that for any measurable $U\subset V_0$,
\begin{equation}\label{eq:kinematic intro}
    \sum_{i\in \Z} \Omega^P_{\Gamma^+,2^{-i}}(U)\lesssim_{\lambda} |U|,
\end{equation}
where $|U|$ is the area of $U$ and $\lambda$ is the intrinsic Lipschitz constant of $\psi$.

Analogously to~\cite{CKN}, extended nonmonotonicity is coercive in the following sense.  Let $U=[0,1]\times\{0\}\times [0,1]\subset V_0$ and for $r>0$, let $rU$ be the square of side $r$ concentric with $U$.  There is a universal constant $r>1$ such that if $\delta$ is sufficiently small, $R$ is sufficiently large, $\psi(0)$ is bounded, and $\Omega_{\Gamma^+,R}(rU)<\delta$, then $\Psi(U)$ is close to a vertical plane and the characteristic curves that pass through $U$ are close to characteristic curves of that vertical plane (i.e., parabolas).  The proof of this geometric statement (whose precise formulation appears as Proposition~\ref{prop:Omega control})  is the most technically involved part of this work; it is outlined in Section~\ref{sec:omega control outline} and carried out in Section~\ref{sec:extended monotone} and Section~\ref{sec:l1 and characteristic}.

By translation, rescaling, and applying a shear automorphism, a similar coercive property applies to any pseudoquad of aspect ratio $1$, but for the subdivision algorithm, we need a coercive property for pseudoquads of arbitrary aspect ratio.  If $Q$ is a pseudoquad of aspect ratio $\alpha(Q)$, the \emph{stretch automorphism} $s(x,y,z)=(\alpha(Q)^{-1}x, \alpha(Q) y, z)$ sends $Q$ to a pseudoquad of aspect ratio $1$.  The extended nonmonotonicity of $s(Q)$ scales like $\alpha(Q)^4$, so if the extended nonmonotonicity of $\Gamma^+$ on $Q$ is at most $\delta |Q|/\alpha(Q)^{4}$, then $\Psi(Q)$ is close to a vertical plane and the characteristic curves that pass through $Q$ are close to characteristic curves of that vertical plane.

Therefore, in the subdivision algorithm above, there is  $\delta>0$ such that we cut $Q$ horizontally if and only if the extended nonmonotonicity of $\Gamma^+$ on $Q$ is at most $\delta |Q|/\alpha(Q)^{4}$.  This criterion, combined with \eqref{eq:kinematic intro}, leads to a crucial bound on the total pseudoquads that have been vertically cut by the subdivision algorithm.  Specifically, if $Q$ is a pseudoquad of the decomposition and $\mathcal{D}_{\mathsf{V}}(Q)$ is the set of vertically cut pseudoquads $Q'$ in the decomposition that are contained in $Q$, then
\begin{equation}\label{eq:Carleson weighted intro}
\sum_{Q'\in \mathcal{D}_{\mathsf{V}}(Q)} \frac{|Q'|}{\alpha(Q')^{4}}\lesssim_\lambda |Q|.
\end{equation}
The condition~\eqref{eq:Carleson weighted intro} is the aforementioned weighted Carleson packing condition, and the $L_4$ norm that appears in Theorem~\ref{thm:XYD} arises directly from the exponent $4$ in \eqref{eq:Carleson weighted intro}.

Thus, the $L_4$ norm in Theorem~\ref{thm:XYD} is ultimately dictated by having to prove a coercive property for intrinsic Lipschitz graphs that is invariant under stretch automorphisms.  This stretch-invariance has multiple effects.  On one hand, stretch-invariance means that it suffices to prove the coercive property for pseudoquads of aspect ratio $1$; indeed, it is enough to consider pseudoquads that approximate the unit square.  On the other hand, it induces a substantial complication in the proofs: since the intrinsic Lipschitz constant is not invariant under stretch automorphisms, the coercivity must be independent of the intrinsic Lipschitz constant.

\subsubsection{A maximally bumpy surface}\label{sec:maximally bumpy}
The optimality part of Theorem~\ref{thm:XYD} corresponds to constructing (in Section~\ref{sec:counterexampleConstruction}) an intrinsic Lipschitz graph for which the $L_4(\R)$ norm in~\eqref{eq:use previous corona} cannot be replaced by the $L_q(\R)$ norm for any $0<q<4$. Theorem~\ref{thm:twist the center} is deduced in Section~\ref{sec:proof of ctr embedding} by analyzing this construction; the level sets of the resulting embedding into $L_1$ are a superposition of certain random rotations, scalings and translations of this surface.

We will show that for any sufficiently small $\epsilon>0$, there are intrinsic Lipschitz surfaces in $\H$ of bounded (Heisenberg) perimeter that are $\epsilon$--far from planes at $\epsilon^{-4}$ different scales, many more than the $\epsilon^{-2}$ different scales that are possible (by~\cite{NY18}) for such surfaces in the $5$--dimensional Heisenberg group $\H^5$ (or, for that matter, in $\R^n$, by the Jones travelling salesman theorem~\cite{Jon90} and the higher-dimensional analogues thereof~\cite{DavidSemmesSingular}).

We construct these surfaces by adding bumps to a vertical plane.  While  surfaces that demonstrate that the bound of~\cite{NY18} for $\H^5$ is optimal can be constructed by adding round bumps with equal width and height, it is more natural in $\H$ to add oblong bumps with width (horizontal size) $w$, depth $d$ (size perpendicular to the surface), and height $h$ (vertical size).  The automorphisms of the Heisenberg group preserve the ratio $dw/h$, so we can construct a family of bump functions by applying automorphisms to a prototype bump with $d=w=h=1$.  The resulting bumps have $h=dw$, and we define the aspect ratio $\alpha$ of such a bump to be
$$\alpha=\frac{w}{\sqrt{h}}=\frac{w}{\sqrt{d w}}=\sqrt{\frac{w}{d}}.$$
A horizontal curve connecting one side of the bump to its other side has slope roughly $d/w=\alpha^{-2}$, so adding a layer of bumps with aspect ratio $\alpha\ge 1$ to a surface multiplies its perimeter by roughly $1+\alpha^{-4}$. Thus, we can start with a unit square, then add $\epsilon^{-4}$ layers of bumps of width $\epsilon^{-1} r_i$, depth $\epsilon r_i$, and height $r_i^2$, for $r_1 \gg \dots \gg r_{\epsilon^{-4}}$.  These bumps all have aspect ratio $\epsilon^{-1}$, so the resulting surface $\Sigma$ has bounded perimeter, and for any $x\in \Sigma$, the intersections $B_{r_i}(x)\cap \Sigma$ are each $\epsilon r_i$--far away from any plane.  So, $\Sigma$ is $\epsilon$--far from planes at $\epsilon^{-4}$ different scales. The implementation of this strategy in Section~\ref{sec:counterexampleConstruction} is in essence an example of a foliated corona decomposition. At each stage we use the characteristic curves of the surface that was obtained in the previous stage to guide us where to glue the next layer of bumps.   Figure~\ref{fig:bumpy surface} shows a sketch of the construction.
\begin{figure}
  \begin{center}\includegraphics[width=.33\textwidth]{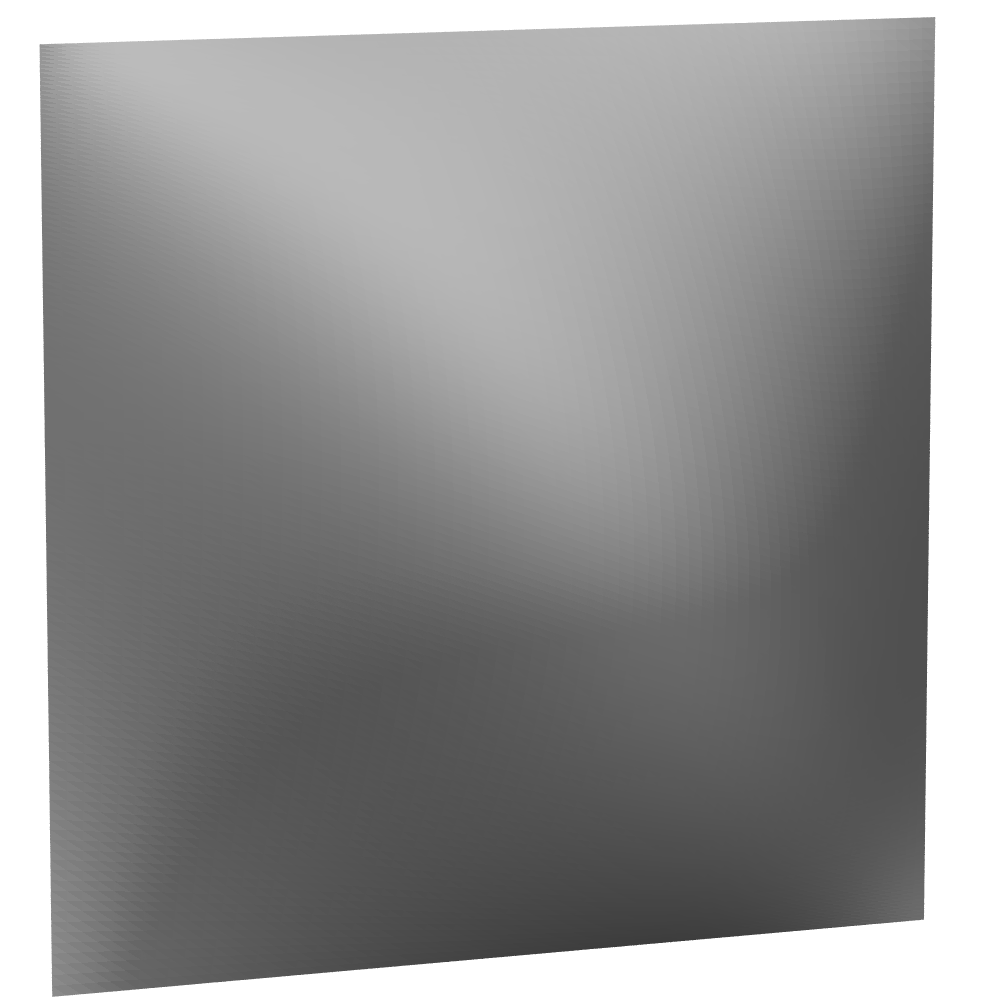} \includegraphics[width=.33\textwidth]{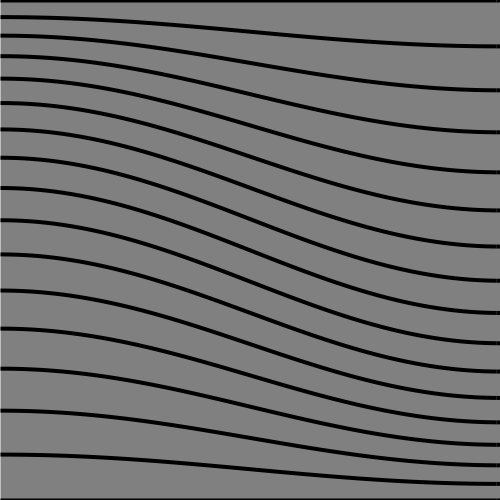}\includegraphics[width=.33\textwidth]{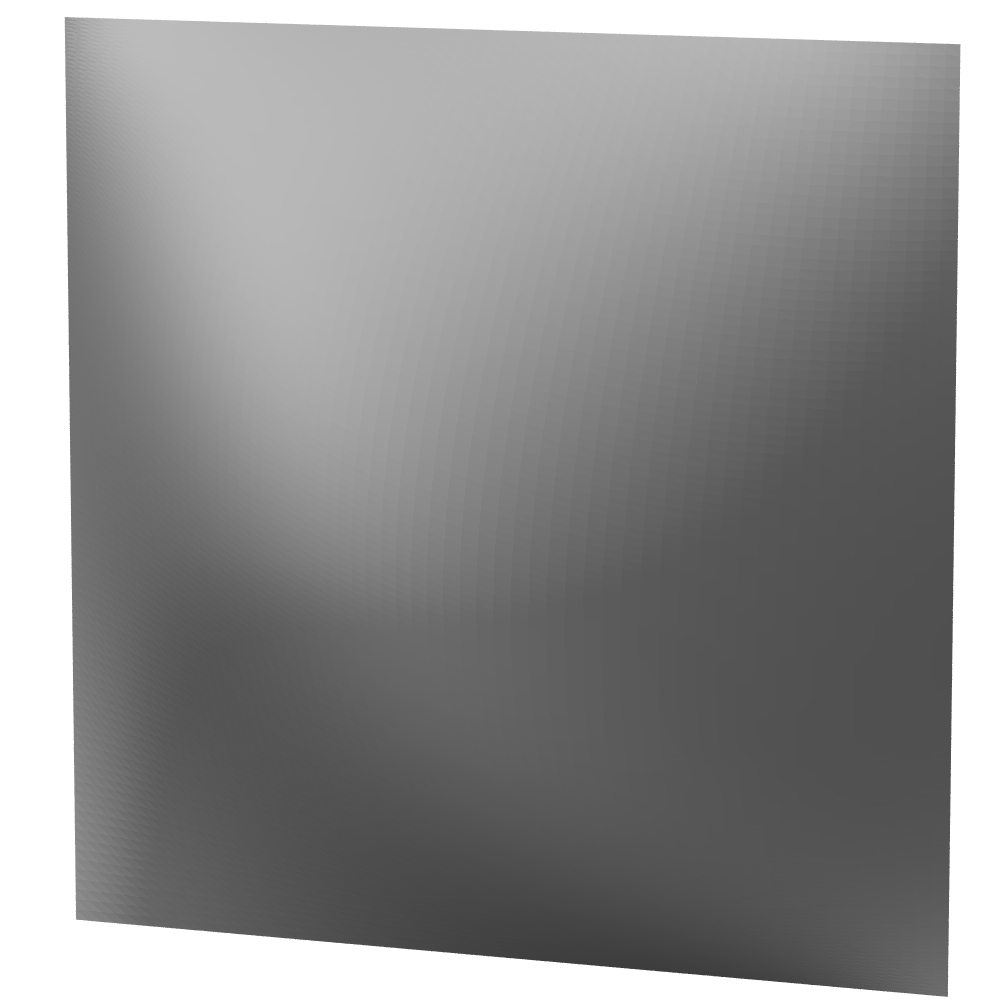}

  \vspace{.2in}

  \includegraphics[width=.33\textwidth]{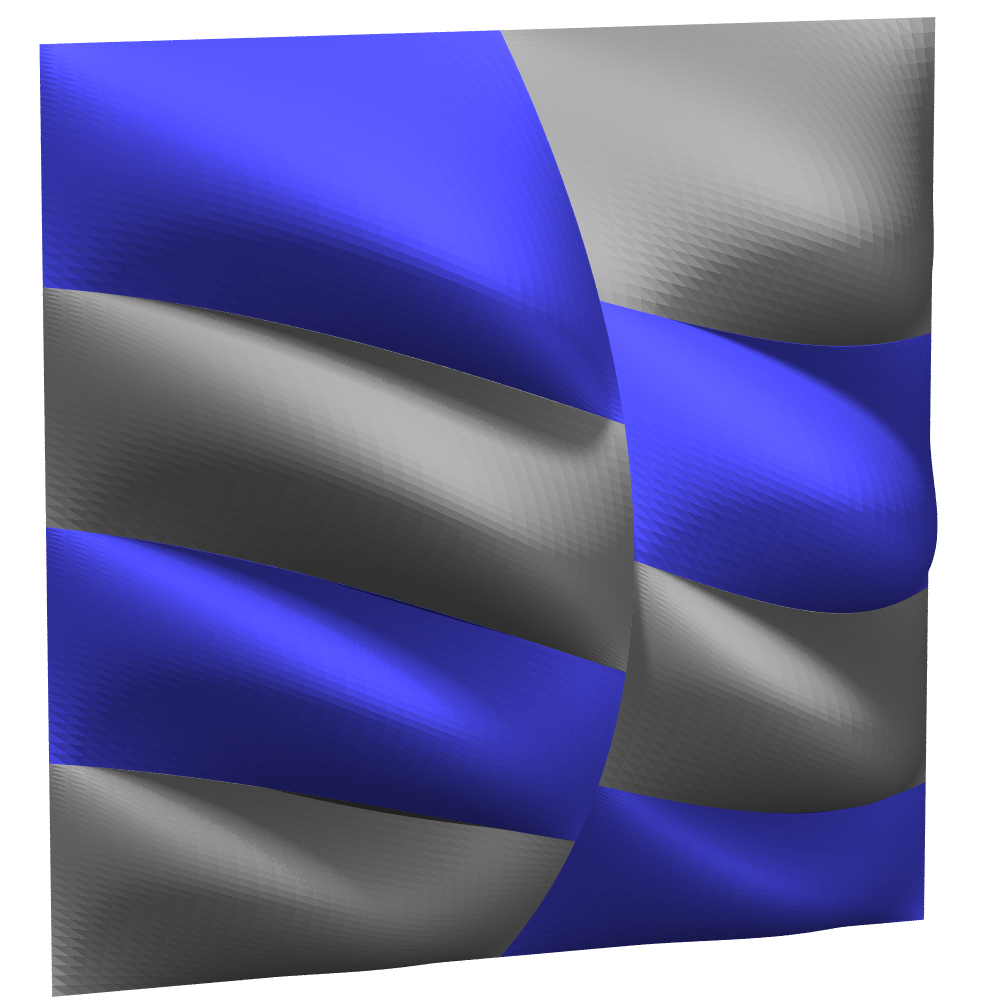} \includegraphics[width=.33\textwidth]{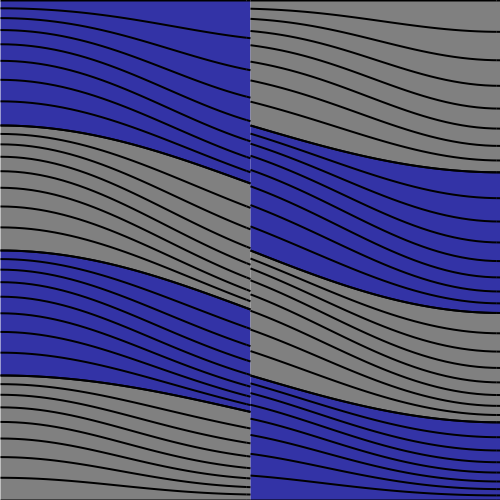}\includegraphics[width=.33\textwidth]{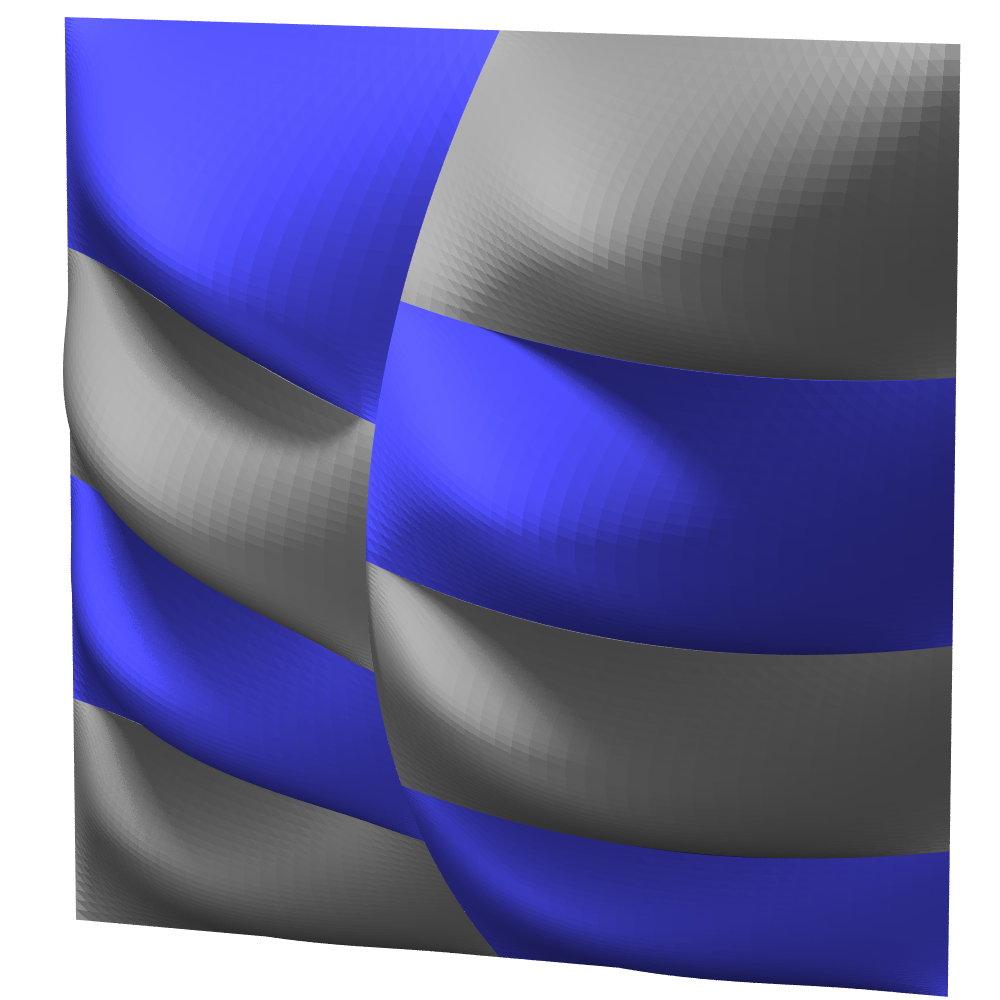}

  \vspace{.2in}

  \includegraphics[width=.33\textwidth]{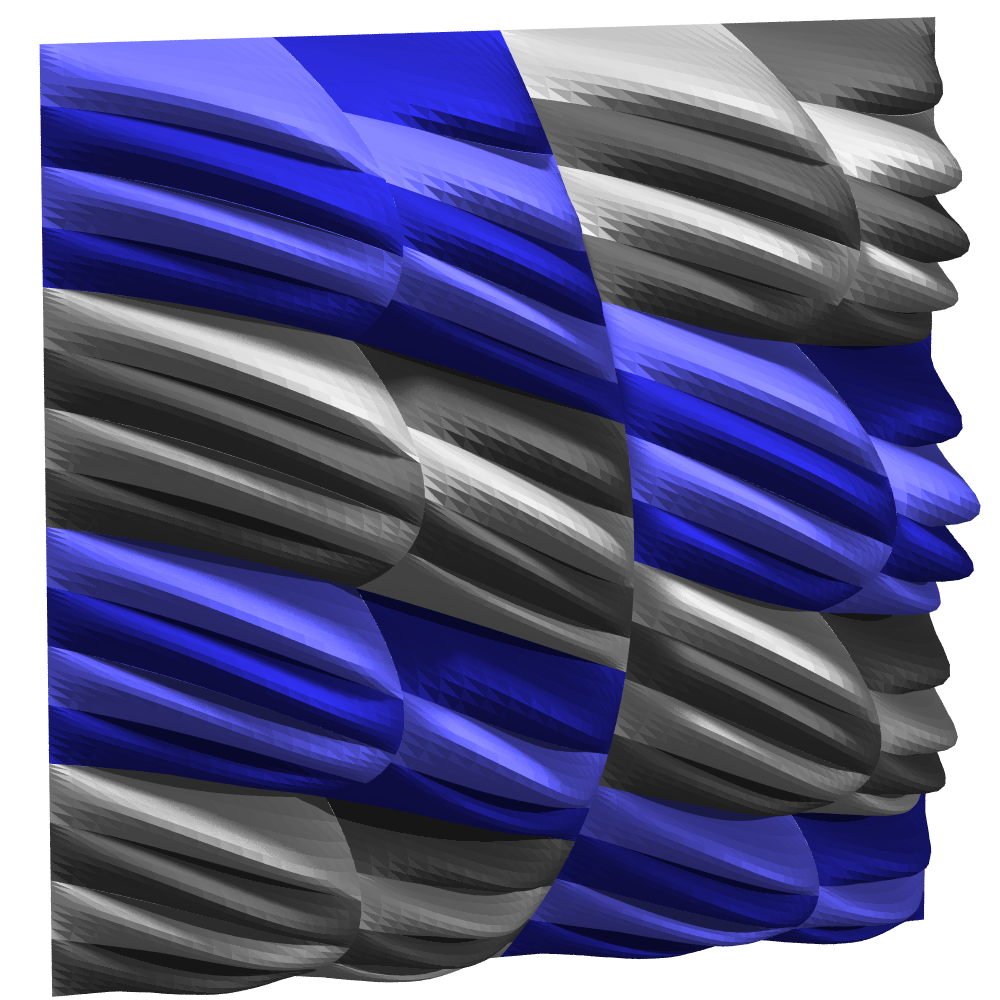} \includegraphics[width=.33\textwidth]{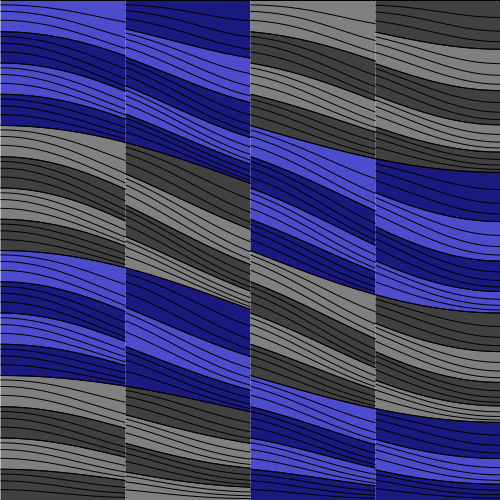}\includegraphics[width=.33\textwidth]{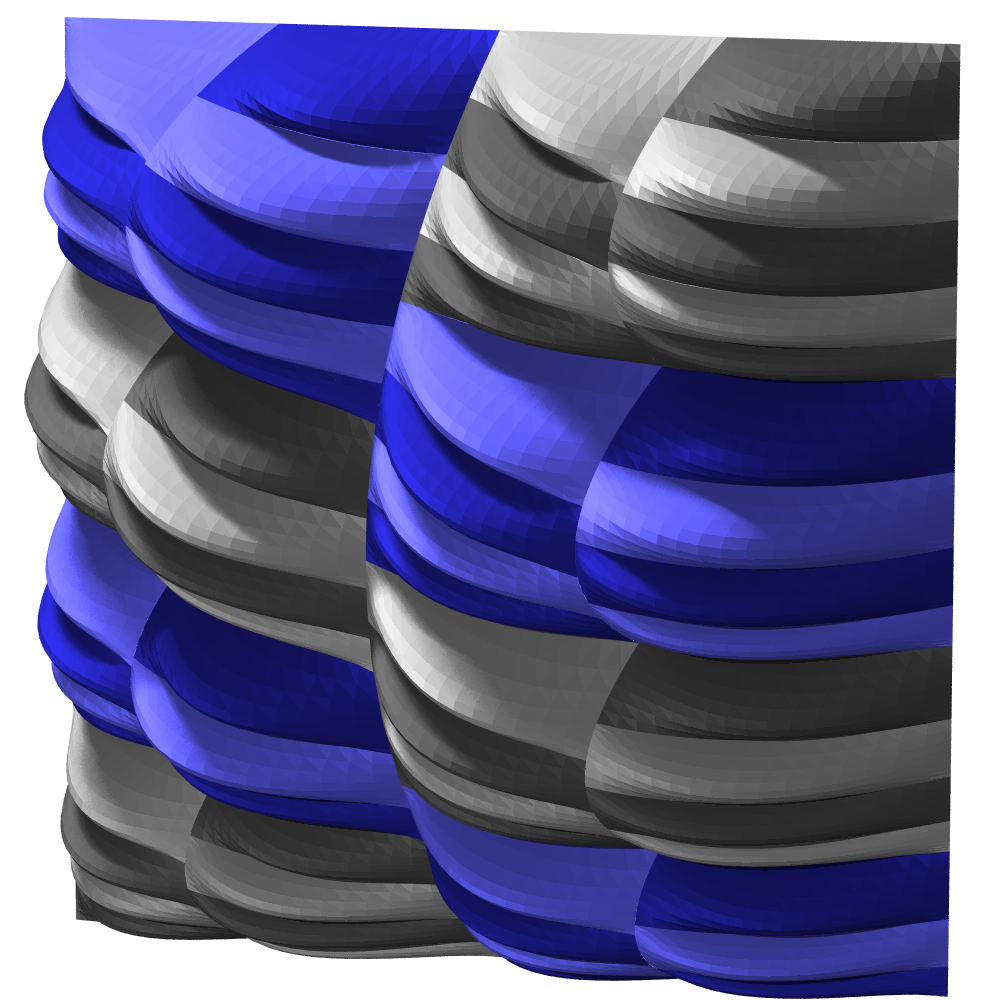}
\end{center}
\caption{\label{fig:bumpy surface}
  {\small\em The first three steps of the construction of a maximally rough surface in $\H$.  The left and right column show the same surface from two different angles.  The center column shows a projection of the surface to the plane, with characteristic curves marked.  Since the second derivatives of these curves are small, the Heisenberg area of the surface is bounded, but the surface can be made $\epsilon$--far from a plane at $\epsilon^{-4}$ different scales --- much more than what is possible in $\H^5$.}}
\end{figure}

It is highly informative to examine why  this construction does not work in $\H^5$.  Bumps on a surface in $\H^5$ have five dimensions, which we denote $w_1, w_2, d_1, d_2$, and $h$, so that $h$ is vertical, the other four dimensions are horizontal, and $d_2$ is normal to the surface.  The automorphisms of $\H^5$ preserve the ratios $d_1w_1/ (d_2w_2)$,  $d_1w_1/h$, and $d_2w_2/h$.  If $\beta$ is a bump with $d_1w_1=d_2w_2=h$ and $d_2\le  w_2$, then the slopes of $\beta$ in the three horizontal directions are roughly $d_2/w_1$, $d_2/w_2$, and $d_2/d_1$.  So, adding $\beta$ to a vertical rectangle with dimensions $w_1\times w_2\times d_1\times h$ increases the volume of the rectangle by a factor of roughly
$$\nu(w_1, w_2, d_1, d_2, h)\eqdef 1+ \max\left\{\frac{d^2_2}{w^2_1}, \frac{d^2_2}{w^2_2}, \frac{d^2_2}{d^2_1}\right\},$$
and the resulting bump is roughly $d_2/\sqrt{h}$--far from a $4$--dimensional hyperplane at scale $\sqrt{h}$.  If $d_2/\sqrt{h}=\epsilon$, then $d_1 w_1=h=\epsilon^{-2}d_2^2$, and
$$\nu(w_1, w_2, d_1, d_2, h) \ge 1+ \frac{d^2_2}{\max\{ d_1^2,w^2_1\}} \ge 1+\frac{d_2^2}{d_1 w_1} = 1+ \epsilon^{2}.$$
Hence, this construction results, at best, in a surface that is $\epsilon$--far from planes at $\epsilon^{-2}$ different scales. One may also consider bumps where $d_1 w_1$, $d_2 w_2$, and $h$ are not proportional, such as bumps with $d_1=w_1=w_2=d_2^{-1} =r \gg 1=h$.  This is more subtle than it might initially  seem.  Indeed, because the $d_1$-- and $w_1$--directions do not commute, there are no $r\times r \times r\times r^{-1}\times 1$ boxes in $\H^5$ that stay close to horizontal.  Consequently, a bump of these dimensions behaves similarly to a collection of smaller bumps with $d_1 w_1=d_2 w_2=h$, which are governed by the previous reasoning.

\subsection{Roadmap}

In Section~\ref{sec:prem}, we present notation for working with the Heisenberg group and some definitions and results related to intrinsic graphs and characteristic curves.  In Section~\ref{sec:counterexampleConstruction}, we construct an intrinsic graph with large vertical perimeter and use it to construct the embeddings used in Theorem~\ref{thm:twist the center} and its consequences.

The rest of the paper is devoted to defining and constructing foliated corona decompositions and using them to prove equation~\eqref{eq:use previous corona} bounding the vertical perimeter of an intrinsic Lipschitz graph.  In Sections~\ref{sec:pseudoquads}, we define a rectilinear foliated patchwork, which decomposes an intrinsic Lipschitz graph into rectilinear pseudoquads, and in Section~\ref{sec:decompose here}, we define the weighted Carleson packing condition required for such a patchwork to be a foliated corona decomposition.  Then, in Section~\ref{sec:fcd bounds vper}, we show that an intrinsic Lipschitz graph that admits a foliated corona decomposition satisfies equation~\eqref{eq:use previous corona}.

It remains to show that every intrinsic Lipschitz graph admits a foliated corona decomposition.  We produce foliated corona decompositions by the subdivision algorithm described in Section~\ref{sec:subdiv algorithm}.  The fact that the patchworks produced by this algorithm satisfy the weighted Carleson packing condition relies on careful analysis of a coercive quantity, the extended parametric nonmonotonicity, defined in Section~\ref{sec:parametric monotonicity}.  When this coercive quantity is small, the graph satisfies strong geometric bounds, detailed in Proposition~\ref{prop:Omega control}.  Assuming Proposition~\ref{prop:Omega control}, we prove the weighted Carleson condition in Section~\ref{sec:weighted Carleson}.  In Section~\ref{sec:omega control outline}, we outline the proof of Proposition~\ref{prop:Omega control}, and in Sections~\ref{sec:extended monotone}--\ref{sec:l1 and characteristic}, we prove it.

\section{Preliminaries}\label{sec:prem}

Most of this section presents initial facts about the Heisenberg group that will be used throughout what follows. However, we will start by briefly setting  notation for measure theoretical boundaries and interiors that are best described in greater generality (though they will be applied below only to either the Heisenberg group or the real line).

Let $(\bbM,d_\bbM,\mu)$ be  a non-degenerate metric measure space, i.e., $(\bbM,d_\bbM)$ is a metric space and $\mu$ is a Borel measure on $\bbM$ such that $\mu(B_\bbM(x,r))>0$ for all $x\in \bbM$ and $r>0$, where $B_\bbM(x,r)=\{y\in \bbM:\ d_\bbM(x,y)\le r\}$ is the closed $d_\bbM$--ball of radius $r$ centered at $x$.

Given a subset $S\subset \bbM$, we define the {\em measure-theoretic support} $\supp_{\mu}(S)$ of $S$ to be the usual measure-theoretic support of the indicator function $\1_S\from \bbM\to \{0,1\}$, namely
\begin{equation}\label{eq:def supp mu}
\supp_{\mu}(S)\eqdef \bigcap_{r>0} \big\{x\in \bbM:\ \mu(B_\bbM(x,r)\cap S)>0\big\}.
\end{equation}
The {\em measure-theoretic boundary} of $S$ is defined as
\begin{equation}\label{eq:def partial mu}
\partial_\mu S\eqdef \supp_\mu(S) \cap \supp_\mu(\bbM\setminus S)= \bigcap_{r>0} \Big\{x\in \bbM:\ 0<\frac{\mu(B_\bbM(x,r)\cap S)}{\mu(B_\bbM(x,r))}<1\Big\}.
\end{equation}
The {\em measure-theoretic interior} of $S$ is defined as
\begin{equation}\label{eq:def inter mu}
\inter_\mu (S)\eqdef \bbM\setminus \supp_\mu(\bbM\setminus S)= \bigcup_{r>0}\big\{x\in \bbM: \ \mu(B_\bbM(x,r)\setminus S)=0\big\}.
\end{equation}
These definitions are nonstandard; other works define the measure-theoretic boundary as the set of points where the density of $S$ is not $0$ or $1$.  The advantage of our definition is that one may check that  $\inter_\mu (S)$ is open in $\bbM$ and its (topological) boundary $\partial \inter_\mu (S)$ is contained in $\partial_\mu S$. The sets $\inter_\mu (S)$, $\inter_\mu (\bbM\setminus S)$, $\partial_\mu S$ are disjoint and their union is $\bbM$, i.e.,
\begin{equation}\label{eq:3disjoint union}
\bbM=\inter_\mu (S)\bigsqcup \big(\inter_\mu (\bbM\setminus S)\big)\bigsqcup \partial_\mu S.
\end{equation}

\subsection{The Heisenberg group}\label{sec:Prelim heis} Here we summarize basic notation and terminology related to the Heisenberg group.

Throughout what follows, $\|\cdot\|\from \R^3\to \R$ will denote the Euclidean norm on $\R^3$, namely  $\|(a,b,c)\|=\sqrt{a^2+b^2+c^2}$ for all $a,b,c\in \R$.  Let $$X\eqdef(1,0,0),\ Y\eqdef(0,1,0),\ Z\eqdef (0,0,1)$$ be the standard basis of $\R^3$, and let $x,y,z\from \R^3\to \R$ be the coordinate functions.  Namely, for $u=(a,b,c)\in \R^3$ we set $x(u)=a$, $y(u)=b$ and $z(u)=c$.  With this notation, the Heisenberg group operation~\eqref{eq:def Heisenberg product} can be written as
\begin{equation}\label{eq:group operation}
\forall  u,v\in \H=\R^3,\qquad   u v = u+v+ \frac{x(u)y(v)-y(u)x(v)}{2} Z.
\end{equation}

The linear span of a set of vectors $S\subset \R^3$ will be denoted $\langle S\rangle$. The plane $\mathsf{H}\eqdef \langle X,Y\rangle$ is called the space of \emph{horizontal vectors}.  Let $\uppi\from \R^3\to \mathsf{H}$ be the orthogonal projection.  A \emph{horizontal line} in $\H$ is a coset of the form $w\langle h\rangle\subset \H$ for some $w\in \H$ and $h\in \mathsf{H}$.  

The union of the horizontal lines passing through a point $u\in \H$ is the  plane $u\mathsf{H}$, which we denote $\mathsf{H}_u$ and call the \emph{horizontal plane centered at $u$}.  Every plane $P\subset \R^3$ either contains a coset of $\langle Z\rangle$ (a \emph{vertical line}), in which case we call $P$ a \emph{vertical plane}, or can be written $P=\mathsf{H}_u$ for some unique $u\in \H$.

If $I\subset \R$ is an interval and $\gamma\from I\to \H$ is a curve such that $x\circ \gamma, y\circ \gamma, z\circ\gamma\from I\to \R$ are  Lipschitz, then $\gamma'(t)$ is defined for almost all $t\in I$.  One then says that $\gamma$ is a \emph{horizontal curve} if $\gamma$ is tangent to $\mathsf{H}_{\gamma(t)}$ at $\gamma(t)$ for almost all $t\in I$, i.e., for almost all $t\in I$ we have
$$
\left.\frac{\ud}{\ud s} \left(\gamma(t)^{-1}\gamma(s)\right)\right|_{s=t}\in \mathsf{H}.
$$
Note that horizontality is left-invariant; if $\gamma$ is a horizontal curve and $g\in \H$, then $g\cdot \gamma$ is also a horizontal curve.
If $\gamma(t)=(\gamma_x(t),\gamma_y(t),\gamma_z(t))$, then this requirement  is equivalent to the differential equation $2\gamma_z'(t)=\gamma_x(t)\gamma_y'(t)-\gamma_y(t)\gamma_x'(t)$.

Define
$$\ell(\gamma)\eqdef \int_I \|\uppi(\gamma'(t))\|\ud t.$$
The sub-Riemannian or Carnot--Carath\'eodory metric $d\from \H\times \H\to [0,\infty)$ is defined by letting $d(v,w)$ be the infimum of $\ell(\gamma)$ over all horizontal curves $\gamma$ connecting $v\in \H$ to $w\in \H$.  This metric is left-invariant, i.e., $d(ga,gb)=d(a,b)$ for all $a,b,g\in \H$.

If $\gamma$ is a horizontal curve connecting $v$ to $w$, then $\uppi\circ \gamma$ is a curve in $\R^2$ of the same length connecting $\uppi(v)$ to $\uppi(w)$, so $d(v,w)\ge \|\uppi(v)-\uppi(w)\|$.  Consequently, any horizontal line in $\H$ is a geodesic.  Also, $d$ satisfies (e.g.~\cite{BR96,Gro96,Mon02}) the ball-box inequality
\begin{equation}
  \label{eq:metric approximation}
  \forall  h=(x,y,z)\in \H,\qquad   d(\0,h) \le |x|+ |y|+4\sqrt{|z|} \le 2 d(\0,h)  + 4\cdot \frac{d(\0,h)}{\sqrt{2\pi}} \le 4 d(\0,h).
\end{equation}

For $h\in \H$ and $r\ge 0$  we let $B_r(h)=\{g\in \H:\ d(g,h)\le r\}= hB_r(\0)$ denote the closed ball of radius $r$ centered at $h$ with respect to the sub-Riemannian metric $d$ on $\H$; throughout what follows we will not use this notation for balls with respect to any other metric.

For $\sigma>0$ denote by $\cH^\sigma$ the $\sigma$--dimensional Hausdorff measure that $d$ induces on $\H$. Thus $\cH^4$ is the Lebesgue measure on $\R^3$, which is also the Haar measure on $\H$. Given a measurable subset $E\subset \H$, the associated {\em perimeter measure} that is induced by $d$ will be denoted by $\Per_E(\cdot)$; we refer to~\cite{FSSCRectifiability} for background on this fundamental notion, noting  only that there exists $\eta>0$ such that if $E\subset \H$ has a piecewise smooth boundary, then  $\Per_E(U)=\eta \cH^3(U\cap \partial E)$ for every open subset $U\subset \H$.

It is  also beneficial to describe the group operation on $\H$ in terms of a symplectic form.  Let $\omega_{\R^2}\from \R^2\times \R^2\to \R$ be the standard symplectic form, i.e.,
$$
\forall (a,b),(\alpha,\beta)\in \R^2,\qquad \omega_{\R^2}\big((a,b),(\alpha,\beta)\big)\eqdef a\beta-b\alpha=\det \begin{pmatrix} a & b \\
           \alpha & \beta
                       \end{pmatrix}.
$$
Under this notation,  \eqref{eq:group operation} can be written as follows.
\begin{equation}\label{eq:group operation symplectic}
\forall  u,v\in \H,\qquad   u v = u+v+ \frac{\omega_{\R^2}(\uppi(u),\uppi(v))}{2} Z.
\end{equation}

This lets us define automorphisms of $\H$.
Let $A\from \R^2\to \R^2$ be an invertible linear map with determinant $J\in \R \setminus \{0\}$, so that $\omega_{\R^2}(A(v),A(w))=J\omega_{\R^2}(v,w)$ for any $v,w\in \R^2$.  It follows from \eqref{eq:group operation symplectic} that the map $\tilde{A}\from \H\to \H$ that is defined by
\begin{equation}\label{eq:def tilde A}
\forall (x,y,z)\in \H,\qquad \tilde{A}(x,y,z)\eqdef(A(x,y), J z)
\end{equation}
is an automorphism of $\H$ which,   since $\tilde{A}(\mathsf{H})=\mathsf{H}$,  sends horizontal curves to horizontal curves and is thus Lipschitz with respect to the sub-Riemannian metric on $\H$.  If $A$ is an orthogonal matrix, then $\tilde{A}$ is an isometry.  As a notable special case, for $a,b>0$, we define
\begin{equation}\label{eq:def s a b}
\forall (x,y,z)\in \H,\qquad s_{a,b}(x,y,z)\eqdef (ax,by,abz),
\end{equation}
which we call a \emph{stretch map}.
When $a=b=t$, $s_{t,t}$ is the usual scaling automorphism of $\H$, which scales the sub-Riemannian metric on $\H$ by a factor of $t$.  For simplicity, in what follows we will sometimes write $s_{t,t}=s_t$.

\subsection{Intrinsic graphs and intrinsic Lipschitz graphs}\label{sec:intrinsic graphs} Throughout what follows, we denote the $xz$--plane by $V_0$, namely
$$
V_0\eqdef \{(x,y,z)\in \H:\ y=0\}=\R\times \{0\}\times \R\subset \H.
$$
Note that the restriction of $\cH^3$ to $V_0$ is proportional to the Lebesgue measure on $V_0$.

Fix $U\subset V_0$.  The intrinsic graph of a function $\psi\from U\to \R$ is defined in~\cite{FSSC06} to be
\begin{equation}\label{eq:def intrinsic graph}
\Gamma_\psi\eqdef \left\{v Y^{\psi(v)} \mid v\in U\right\}=\Big\{\big(x(v),\psi(v),z(v)+\frac12x(v)\psi(v)\big):\ v\in U\Big\}\subset \H,
\end{equation}
where in~\eqref{eq:def intrinsic graph}, as well as throughout  what follows, it is convenient to use the exponential notation  $u^t=tu=(tx(u),ty(u),tz(u))$ for $u\in \H$ and $t\in \R$.
Observe that any coset of $\langle Y\rangle$ that passes through $U$ intersects $\Gamma_\psi$ in exactly one point. We will also use the following notation for the intrinsic epigraph of $\psi$.
$$
\Gamma_\psi^+\eqdef  \left\{v Y^{t} \mid (v,t)\in U\times (\psi(v),\infty)\right\}.
$$

Suppose that $U\subset V_0$ is an open subset of $V_0$ and that $g\from U\to \R$ is smooth. For every $\psi \from U\to \R$ define a function $\partial_\psi g\from U\to \R$ by
\begin{equation}\label{eq: d psi}
\partial_\psi g\eqdef \frac{\partial g}{\partial x}- \psi\frac{\partial g}{\partial z}.
\end{equation}
If $\psi$ is smooth, then we define the \emph{horizontal derivative} of $\psi$ to be the function
\begin{equation}\label{eq:d psi psi}
  \partial_\psi \psi=\frac{\partial \psi}{\partial x}- \psi\frac{\partial \psi}{\partial z}.
\end{equation}

Let $v\in U$ and let $p\eqdef vY^{\psi(v)}\in \Gamma_\psi$.
One can interpret $\partial_\psi \psi$ by considering the horizontal plane $\mathsf{H}_p$. This plane locally intersects $\Gamma_\psi$ in a curve, and the tangent vector of this curve at $p$ is given by $X+\partial_\psi \psi(v)Y$. The horizontal derivative also determines the slope of the intrinsic tangent plane to $\Gamma_\psi$, where the slope of a vertical plane is the slope of its projection to $\mathsf{H}$.
As $r\to 0$, rescalings of the intersections $B_r(p)\cap \Gamma_\psi$ converge to a vertical tangent plane with slope $\partial_\psi\psi(v)$.

The following proposition is part of Theorem~1.2 of \cite{AVSCIntrinsic}. It expresses the area $\cH^{3}(\Gamma_\psi)$ of $\Gamma_\psi$, namely the 3-dimensional Hausdorff measure (with respect to the sub-Riemannian metric) of $\Gamma_\psi$, in terms of $\partial_\psi \psi$.
\begin{prop}[\cite{AVSCIntrinsic}]\label{prop:areaIntegral}
  There exists a constant  $c>0$ such that if $U\subset V_0$ is an open set and $\psi\from U\to \R$ is smooth, then
  \begin{equation}\label{eq:intrinsic area formula}
    \cH^3(\Gamma_\psi)\approx \cS^{3}(\Gamma_\psi)=c\int_U \sqrt{1+(\partial_\psi \psi)^2}\ud w\asymp \cH^3(U)+\|\partial_\psi \psi\|_{L_1(U)},
  \end{equation}
  where $\cS^3$ is the $3$--dimensional spherical Hausdorff measure on $\H$.
\end{prop}
Recent work \cite{JulNicVit} has shown that the spherical Hausdorff measure and the Hausdorff measure on $\Gamma_\psi$ are equal up to a multiplicative constant, so the first equivalence in \eqref{eq:intrinsic area formula} can be replaced with an equality up to a constant factor.

For $\lambda\in (0,1)$, define the double cone
$$
\mathrm{Cone}_\lambda\eqdef \left\{h\in \H \mid |y(h)|>\lambda d(\0,h)\right\}.
$$
This is a cone centered on the horizontal line $\langle Y\rangle$ which is scale-invariant, i.e., $$\forall  t>0,\qquad s_{t,t}(\mathrm{Cone}_\lambda)=\mathrm{Cone}_\lambda.$$
The intersection $\mathsf{H}\cap \mathrm{Cone}_\lambda$ is a double cone in $\mathsf{H}$ with angle depending on $\lambda$.  Specifically,
\begin{align}\label{eq:formula for intersection with cone}
\begin{split}
  \mathsf{H}\cap \mathrm{Cone}_\lambda
  &=\left\{(x,y,0)\in \H\mid |y| > \lambda \sqrt{x^2+y^2}\right\}\\
  &=\left\{(x,y,0)\in \H\mid |y| > \frac{\lambda}{\sqrt{1-\lambda^2}} |x|\right\}.
  \end{split}
\end{align}

\begin{defn}\label{def:intrinsic lipschitz graph}
  Let $U\subset V_0$ and let $\Gamma\subset \H$ be an intrinsic graph over $U$.  For  any $\lambda\in (0,1)$, we say that $\Gamma$ is an {\em intrinsic $\lambda$--Lipschitz graph} if $(h \mathrm{Cone}_\lambda(V_0))\cap \Gamma=\emptyset$ for every $h\in \Gamma$.  Equivalently, for every $p,q\in \Gamma$,
  $$|y(q)-y(p)| \le \lambda d(p,q).$$
  We say that $\Gamma$ is an \emph{intrinsic Lipschitz graph} if it is intrinsic $\lambda$--Lipschitz for some $\lambda\in (0,1)$.  If $\Gamma=\Gamma_\psi$ for some $\psi\from U\to \R$, then we say that $\psi$ is an \emph{intrinsic Lipschitz function}.
\end{defn}
Definition~\ref{def:intrinsic lipschitz graph} gives the same class of intrinsic Lipschitz graphs as the definition introduced in~\cite{FSSC06}, but it gives different classes of intrinsic $\lambda$--Lipschitz graphs; see Section~3.2 of \cite{RigotQuantitative} for a proof that the definitions are equivalent.

The following simple bound will be convenient later.
\begin{lemma}\label{lem:ilg line distance}
  Let $0\le \lambda\le 1$ and let $\Gamma=\Gamma_\psi$ be an intrinsic $\lambda$--Lipschitz graph of a function $\psi\from U\subset V_0\to \R$.  Let $v, w\in U$ and write $p=vY^{\psi(v)}\in \Gamma$ and $q=wY^{\psi(w)}\in \Gamma$.  Then
  $$|y(p)-y(q)|=|\psi(v)-\psi(w)|\le \frac{2}{1-\lambda} d(p, q\langle Y\rangle).$$
\end{lemma}
\begin{proof}
  Denote $m=d(p, w\langle Y\rangle)$. Let $c\in w\langle Y\rangle$ be a point such that $d(p,c)=m$.  By the intrinsic Lipschitz condition,
  $$|y(c)-y(q)|\le m+|y(p)-y(q)|\le m + \lambda d(p,q) \le m + \lambda (m + |y(c)-y(q)|).$$
 This simplifies to give
 $$
 |y(c)-y(q)|\le \frac{1+\lambda}{1-\lambda}m.
 $$
 Hence,
  \begin{equation*}
  |y(p)-y(q)| \le |y(p)-y(c)|+|y(c)-y(q)|\le \frac{2m}{1-\lambda}.\qedhere
  \end{equation*}
\end{proof}

Intrinsic Lipschitz graphs satisfy the following  version of Rademacher's differentiation theorem due to~\cite[Theorem~4.29]{FSSCDifferentiability}.
\begin{thm}[\cite{FSSCDifferentiability}]
  Let $0 < \lambda < 1$, let $U\subset V_0$ be an open set and let $f\from U\to \R$ be a function such that $\Gamma_\psi \subset \H$ is an intrinsic $\lambda$--Lipschitz graph.  Then for almost every $p\in U$, $\Gamma_\psi$ has an intrinsic tangent plane at $pY^{\psi(p)}$ whose slope satisfies
  \begin{equation}\label{eq:upper hor der}
  |\partial_\psi \psi(p)|\le \frac{\lambda}{\sqrt{1-\lambda^2}}.
\end{equation}
\end{thm}
We note that~\cite[Theorem~4.29]{FSSCDifferentiability} is concerned with the (almost everywhere) existential statement of horizontal derivatives. The upper bound in~\eqref{eq:upper hor der} follows from~\eqref{eq:formula for intersection with cone} and the fact that the intrinsic tangent plane at $pY^{\psi(p)}$ is disjoint from $p\mathrm{Cone}_\lambda$
(see also Lemma~\ref{lem:char curve z bounds}). This bound on the horizontal derivatives of an intrinsic Lipschitz graph leads to a bound on the perimeter measure.
The following result follows from Theorem~4.1 of \cite{FSSC-AreaFormula}, which proves a similar bound on the Hausdorff measure of
$\Gamma$, and the results of~\cite{FSSCRectifiability}, which imply that the Hausdorff measure of $\Gamma$ and the perimeter measure of $\Gamma^+$ differ by at most a multiplicative constant.  Let $\Pi\from \H \to V_0$ be the natural (nonlinear) projection to $V_0$ along cosets of $\langle Y\rangle$, i.e., $\Pi(v)=vY^{-y(v)}$ for every $v\in \H$. Equivalently,
\begin{equation}\label{eq:Pi formula}
\forall(x,y,z)\in \H,\qquad   \Pi(x,y,z)\eqdef \big(x,0,z-\frac{1}{2}xy\big).
\end{equation}
\begin{lemma}[{\cite{FSSC-AreaFormula}}]\label{lem:intrinsic Lipschitz perimeter}
  Fix $\lambda\in (0,1)$. Let $\psi\from V_0\to \R$ be $\lambda$--intrinsic Lipschitz.  The perimeter measure $\Per_{\Gamma_\psi^+}$ satisfies the following equivalence for measurable subsets $A\subset \Gamma_\psi$.
  $$\Per_{\Gamma_\psi^+}(A) \asymp_\lambda |\Pi(A)|,$$
  where here, and henceforth, $|\cdot|$ denotes the Haar measure on $V_0$, normalized to coincide with the usual $2$--dimensional area measure in $\R^3$.
\end{lemma}

\subsection{Characteristic curves}\label{sec:char}
Let $U\subset V_0$ be an open set and let $\psi\from U\to \R$ be a continuous function.  The differential operator $\partial_\psi$ given in~\eqref{eq: d psi} defines a  vector field on $V_0$ that is continuous and has $x$--coordinate $1$, so by the Peano existence theorem, there is at least one flow line of $\partial_\psi$ through every point of $U$, defined on an interval. These flow lines are the graphs of functions $g\from I\to \R$ satisfying
\begin{equation}\label{eq:horizontal curve eq}
\forall  t\in I,\qquad   g'(t)+\psi\big(t,0,g(t)\big)=0.
\end{equation}
We call these flow lines \emph{characteristic curves} of $\Gamma_\psi$.

The solution to \eqref{eq:horizontal curve eq} guaranteed by the Peano existence theorem is only local, but when $\psi$ is intrinsic Lipschitz, we can define $g$ on all of $\R$. Indeed, by the Peano existence theorem, if $S_r=[-1,1]\times \{0\}\times [-r,r]$ and $\sup_{q\in S_r}|\psi(q)| \le r$, then there exists a $g\from (-1,1)\to [-r,r]$ that solves \eqref{eq:horizontal curve eq} with initial condition $g(0)=0$.  Let $(x,0,z)\in S_r$. By Lemma~\ref{lem:ilg line distance} with $v=\0$, $w=(x,0,z)$, there is some $C=C_\psi>0$ such that
$$|\psi(x,0,z)| \le |\psi(\0)| + \frac{2}{1-\lambda} d\left(Y^{\psi(\0)}, (x,0,z)\right) \le C + C|x| + C\sqrt{|z|} \le 2 C + C\sqrt{r}.$$
If $r$ is sufficiently large, then $\sup_{q\in S_r}|\psi(q)| \le r$, so \eqref{eq:horizontal curve eq} can be solved on $(-1,1)$. More generally, for any $x_0,z_0$, there is a $g\from (x_0-1,x_0+1)\to \R$ that solves \eqref{eq:horizontal curve eq} with initial condition $g(x_0)=z_0$. By patching together such solutions, we obtain a global solution to \eqref{eq:horizontal curve eq}.

In this section, we will show that the characteristic curves of $\Gamma_\psi$ are the projections of horizontal curves in $\Gamma_\psi$ and use them to describe $\Gamma_\psi$.  In the next section, we will describe how characteristic curves transform under automorphisms of $\H$; later, we will use these curves to describe how horizontal lines intersect an intrinsic Lipschitz graph.

\begin{lemma}\label{lem:char}
  Let $\Gamma=\Gamma_\psi$.  The characteristic curves of $\Gamma$ are exactly the projections (under $\Pi$) of horizontal curves $\phi\from I\to \Gamma$ such that $x(\phi(t))=t$ for every $t\in I$.
\end{lemma}
Because characteristic curves can branch and rejoin (see~\cite{BigolinCaravennaSerraCassano} for such examples), there are intrinsic Lipschitz graphs with horizontal curves whose $x$--coordinate is not monotone.  Thus the condition $x(\phi(t))=t$ of Lemma~\ref{lem:char} cannot be dropped.
\begin{proof}[Proof of Lemma~\ref{lem:char}]

  First, we claim that if $\phi$ is a horizontal curve in $\Gamma$ with $x(\phi(t))=t$, then $\Pi\circ \phi$ is a characteristic curve of $\Gamma$.
  Write $\Gamma=\Gamma_\psi$ and let $\phi\from I\to \Gamma$ be a horizontal curve of the form $\phi(t)=X^t Y^{f(t)}Z^{g(t)}$.  Then $f$ and $g$ are Lipschitz, $\Pi(\phi(t))=(t,0,g(t))$, and, since $\phi(t)\in \Gamma$, we have $f(t)=\psi(t,0,g(t))$.  Since $\phi$ is horizontal,
  $$\frac{\ud}{\ud u} \phi(t)^{-1} \phi(t+u)\bigg|_{u=0}\in \mathsf{H}$$
  for almost every $t\in I$.  Observe that
  \begin{align*}
    \phi(t)^{-1} \phi(t+u)
    &=\bigl(X^{t}Y^{f(t)}Z^{g(t)}\bigr)^{-1} \bigl(X^{t+u}Y^{f(t+u)}Z^{g(t+u)}\bigr) \\
    &=X^{u}Y^{f(t+u)-f(t)}Z^{g(t+u)-g(t)+u f(t)}.
  \end{align*}
  Since $f$ and $g$ are Lipschitz, the following identity holds almost everywhere.
  \begin{equation}\label{eq:tangent to gamma}
    \frac{\ud}{\ud u} \phi(t)^{-1} \phi(t+u)\bigg|_{u=0}=X+f'(t) Y+(g'(t)+f(t))Z.
  \end{equation}
  That is, $g$ satisfies \eqref{eq:horizontal curve eq}.

  Conversely, suppose that $g$ is a solution of \eqref{eq:horizontal curve eq} and let $f(t)=\psi(t,0,g(t))$.  By Theorem~1.1 and Theorem~1.2 of \cite{BigolinCaravennaSerraCassano}, $f$ is Lipschitz.  Therefore, $\phi(t)=X^t Y^{f(t)}Z^{g(t)}$ is a Lipschitz curve in $\Gamma$ such that $\Pi(\phi(t))=(t,0,g(t))$ and such that $\phi$ satisfies \eqref{eq:tangent to gamma} almost everywhere.  In combination with~\eqref{eq:horizontal curve eq}, this implies that $\phi$ is horizontal.
\end{proof}

If $\psi$ is smooth, the characteristic curves of $\Gamma_\psi$ foliate $U$. If $\psi$ is merely intrinsic Lipschitz, characteristic curves can branch and rejoin, but if two characteristic curves pass through the same point, then they are tangent at that point; see Figure~1 of \cite{BigolinCaravennaSerraCassano} for an example of this phenomenon.

Characteristic curves satisfy bounds based on the intrinsic Lipschitz constant of $\Gamma$.
\begin{lemma}\label{lem:char curve z bounds} Fix $\lambda\in (0,1)$ and denote
$$
L\eqdef \frac{\lambda}{\sqrt{1-\lambda^2}}.
$$
  Let $\Gamma=\Gamma_\psi$ be an intrinsic $\lambda$--Lipschitz graph over an open set and let $\gamma\from I\to V_0$ be a characteristic curve for $\Gamma$ parametrized so that $x(\gamma(t))=t$ for all $t\in I$.   Then,
  \begin{equation}\label{eq:f gamma Lipschitz}
 \forall  s,t\in I,\qquad    |\psi(\gamma(s))-\psi(\gamma(t))|\le L|s-t| .
  \end{equation}
Also, if we denote $g(t)=z(\gamma(t))$, then
  \begin{equation}\label{eq:g Taylor}
  \forall  s,t\in I,\qquad     \big|g(t) - g(s) - g'(s) \cdot (t-s)\big|\le L\frac{(t-s)^2}{2}.
  \end{equation}
\end{lemma}
\begin{proof}
  Since $\gamma$ is characteristic, the curve
  $\phi(t)=\gamma(t)\cdot Y^{\psi(\gamma(t))}$
  is horizontal.  The intrinsic Lipschitz condition implies that
  \begin{equation}\label{eq:lambda upper for m}\forall  \delta\in \R\setminus \{0\},\qquad \frac{|y(\phi(t+\delta))-y(\phi(t))|}{d(\phi(t), \phi(t+\delta))} \le \lambda.\end{equation}

  By Pansu's theorem \cite{Pan89}, for almost every $t\in I$, there is a vector $h_t\in \mathsf{H}$ such that
  $$\lim_{\delta\to 0} \frac{d\left(\phi(t) h_t^{\delta},\phi(t+\delta)\right)}{\delta} = 0.$$
  Indeed, $h_t=(1, m,0)$, where $m=(\psi\circ \gamma)'(t)$.  Then
  $$\liminf_{\delta\to 0} \frac{|y(\phi(t+\delta))-y(\phi(t))|}{d(\phi(t), \phi(t+\delta))} \ge \liminf_{\delta\to 0} \frac{|\delta m| - d(\phi(t) h_t^{\delta},\phi(t+\delta))}{\delta\|h_t\|+d(\phi(t) h_t^{\delta},\phi(t+\delta))} = \frac{|m|}{\sqrt{1+m^2}}.$$
  By~\eqref{eq:lambda upper for m} it follows that $\frac{|m|}{\sqrt{1+m^2}} \le \lambda$, so for almost every $t\in I$,
  \begin{equation}\label{eq:m upper}|(\psi\circ \gamma)'(t)|=|m| \le L.\end{equation}
  This implies \eqref{eq:f gamma Lipschitz}. By \eqref{eq:horizontal curve eq}, $g'(t)=-\psi(\gamma(t))$, so it follows from~\eqref{eq:m upper} that $|g''(t)|\le L$  for almost every $t\in I$. The remaining bound \eqref{eq:g Taylor} is therefore justified as follows.
  \begin{equation*}
  \left|g(t) - \big(g(s) + g'(s) \cdot (t-s)\big)\right|=\bigg|\int_s^t  (t-u)g''(u)\ud u\bigg|\le L\bigg|\int_s^t  (t-u)\ud u\bigg|=L\frac{(t-s)^2}{2}.\qedhere
  \end{equation*}
\end{proof}

Since there is a characteristic curve through every point $p\in U$ and the derivative of such a curve at $p$ is $-\psi(p)$, an intrinsic graph $\Gamma$ can be reconstructed from its characteristic curves.  Indeed, one way to construct intrinsic Lipschitz graphs is to construct a foliation of $V_0$ by $C_1$ curves $\{z=g_\alpha(x)\}$, $\alpha\in A$ such that $\Lip(g_\alpha')\lesssim 1$ for every $\alpha\in A$.  Each such curve lifts to a horizontal curve, and one can show that the union of these lifts is an intrinsic Lipschitz graph.  (This is how the graphs in Figure~\ref{fig:bumpy surface} were constructed.)

For illustration, we consider planes in $\H$.  A vertical plane $V$ that is not orthogonal to $V_0$ is an intrinsic graph over $V_0$.  The horizontal curves in $V$ are parallel lines; let $L$ be one such line.  The image $\Pi(L)$ is a parabola in $V_0$, and the characteristic curves of $V$ are the parabolas parallel to $\Pi(L)$.  The second derivative of these parabolas depends on the angle between $V$ and $V_0$.

Let $v\in \H$.  The horizontal plane $\mathsf{H}_v$ centered at $v$ is not an intrinsic graph, but the horizontal line $v\langle Y\rangle$ divides $\mathsf{H}_v$ into two intrinsic graphs.  The horizontal lines in $\mathsf{H}_v$ all pass through $v$, and their projections to $V_0$ are parabolas through $\Pi(v)$.  Since they all intersect at $v$, their projections are all tangent at $\Pi(v)$.  These parabolas foliate the complement in $V_0$ of the vertical line through $\Pi(v)$.  They have unboundedly large second derivatives, so the two halves of $\mathsf{H}_v$ are locally intrinsic Lipschitz graphs, but not globally.

\subsection{Automorphisms and characteristic curves}\label{sec:automs and ccs} Recall that any invertible linear map $A\from \R^2\to \R^2$ induces an automorphism $\tilde A$ of $\H$ as in~\eqref{eq:def tilde A}.  We are particularly interested in the case that $Y$ is an eigenvector of $A$.  In this case, $\tilde{A}(\langle Y\rangle)=\langle Y\rangle$, so $\tilde{A}$ sends cosets of $\langle Y\rangle$ to cosets of $\langle Y\rangle$.  A set $\Gamma$ is an intrinsic graph if and only if it intersects each coset of $\langle Y\rangle$ at most once, so $\tilde{A}$ sends intrinsic graphs to intrinsic graphs.

One family of maps with this property are the stretch maps $s_{a,b}(x,y,z)=(ax,by,abz)$ defined in~\eqref{eq:def s a b}.  To construct a second family of maps with the above property, let $b\in \R$ and consider the linear map $A_b(x,y)=(x,y+bx)$, which is a shear of the plane $\R^2$. The induced map $\tilde{A}_b$, is an automorphism of $\H$ given by the formula
$$\forall (x,y,z)\in \H,\qquad \tilde{A}_b(x,y,z)=(x,y+bx,z),$$
and we call such maps \emph{shear maps}.  (Note that these are different from the shear maps considered in~\cite{Xie16}.)

Let $\Pi\from \H \to V_0$ be as in~\eqref{eq:Pi formula}, i.e., the projection to $V_0$ along cosets of $\langle Y\rangle$. The maps above preserve cosets of $\langle Y\rangle$, so composed with $\Pi$ they induce maps from $V_0$ to $V_0$.
\begin{lemma}\label{lem:v0 transforms}
Fix $h=(x_0,y_0,z_0)\in \H$ and $v=(x,0,z) \in V_0$. For any $a,b,t\in \R$ we have
  $$\Pi\big(s_{a,b}(v Y^t)\big)=s_{a,b}(v)=(ax,0,abz),$$
  $$\Pi\big(\tilde{A}_{b}(vY^t)\big)=\big(x,0,z-\frac12 bx^2\big),$$
  and
  $$\Pi(h vY^t)=\big(x+x_0, 0, z+z_0-xy_0-\frac12x_0y_0\big).$$
\end{lemma}
\begin{proof}
  $\Pi(gY^t)=\Pi(g)$ for all $g\in \H$ and $t\in \R$.  Since $s_{a,b}$ and $\tilde{A}_b$ are homomorphisms,
  $$\Pi\big(s_{a,b}(v  Y^t)\big)=\Pi\big(s_{a,b}(v) Y^{bt}\big)=s_{a,b}(v)=(ax,0,abz),$$
  and
  $$\Pi\big(\tilde{A}_b(v Y^t)\big)=\Pi\big(\tilde{A}_b(v) Y^{t}\big)=(x,bx,z)Y^{-bx}=\big(x,0,z-\frac12bx^2\big).$$
  Finally,
  \begin{align*}\Pi(h v Y^t)=\Pi(h v)&=\big(x_0+x,y_0,z_0+z-\frac12xy_0\big)Y^{-y_0}\\&=\big(x_0+x,0,z_0+z-\frac12xy_0-\frac12(x_0+x)y_0\big).\qedhere
  \end{align*}
\end{proof}

We next describe how these maps affect characteristic curves and intrinsic graphs.
\begin{lemma}\label{lem:curve transforms} Fix $U\subset V_0$ and $\psi\from U\to \R$ be a continuous function. Write $\Gamma=\Gamma_\psi$. Let $C=\{(x,0,z)\in V_0:\ z=g(x)\}$ be a characteristic curve of $\Gamma$.  Let $q\from \H\to \H$ be a stretch map, shear map, or left translation, and let $\hat{q}\from V_0\to V_0$, $\hat{q}(v)=\Pi(q(v))$ be the map that $q$ induces on $V_0$.  Then $q(\Gamma)$ is the intrinsic graph of a function $\hat{\psi}\from \hat{q}(U)\to \R$ and $\hat{q}(C)$ is a characteristic curve of $q(\Gamma)$. Also,
\begin{itemize}
\item  If $a, b\in \R\setminus \{0\}$ and $q=s_{a,b}$, then $\hat{\psi}(\hat{q}(v))=b \psi(v)$ for all $v\in U$.

\item  If $b\in \R$ and $q=\tilde{A}_b$, then $\hat{\psi}(\hat{q}(v))=\psi(v)+bx(v)$ for all $v\in U$.

 \item  If $h\in \H$ and $q(p)=hp$ for all $p\in \H$, then $\hat{\psi}(\hat{q}(v))=\psi(v)+y(h)$ for all $v\in U$.
  \end{itemize}
\end{lemma}
\begin{proof}
  Any coset of $\langle Y\rangle$ intersects $q(\Gamma)$ at most once, so $q(\Gamma)$ is an intrinsic graph with domain $\Pi(q(\Gamma))=\hat{q}(\Gamma)$.

  Let $\gamma\subset \Gamma$ be the horizontal curve such that $\Pi(\gamma)=C$.  Then $q(\gamma)$ is a horizontal curve in $q(\Gamma)$.  For all $g\in \H$ and $t\in \R$ we have $\Pi(gY^t)=\Pi(g)$. Consequently, we have $\Pi(q(\gamma))=\Pi(q(C))=\hat{q}(C)$, and $\hat{q}(C)$ is characteristic for $q(\Gamma)$.

  For any $v\in U$, we have $q(vY^{\psi(v)})\in q(\Gamma)$, and since $q(\Gamma)$ is an intrinsic graph, we must have $q(vY^{\psi(v)})=\hat{q}(v) Y^{\hat{\psi}(\hat{q}(v))}$.  The claimed expressions for $\hat{\psi}$ follow directly.
\end{proof}

Observe that if $q\from \H \to \H$ preserves cosets of $\langle Y\rangle$, then
\begin{equation}\label{eq:preserve-cosets}
  q(\Pi(p))\in q(p \langle Y \rangle) = q(p) \langle Y \rangle,
\end{equation}
so $\Pi \circ q =\Pi \circ q\circ \Pi$. In particular, if $q_1$ and $q_2$ are stretch maps, shear maps, or left translations, then
$$\hat{q_1}\circ \hat{q_2} = \Pi \circ q_1\circ \Pi \circ q_2 = \Pi\circ q_1\circ q_2= \widehat{q_1\circ q_2}.$$
Consequently, if $a,b,c\in \R$ and $q(v)=Y^{b} Z^{-c} \tilde{A}_{2a}(v)$ for all $v\in \H$, then $\hat{q}(x,0,z)=(x,0,z - ax^2-bx-c)$.  That is, for any quadratic function $f$, there is a map $q\from \H\to \H$ so that the characteristic curves of $q(\Gamma)$ are the characteristic curves of $\Gamma$ translated by $f$.  

Finally, stretch maps and shear maps send intrinsic Lipschitz graphs to intrinsic Lipschitz graphs (with a possible change in the Lipschitz constant).
\begin{lemma}\label{lem:automs preserve intrinsic lipschitz}
  Let $\Gamma$ be an intrinsic Lipschitz graph, and let $a,b\in \R\setminus \{0\}$.  Then $s_{a,b}(\Gamma)$ and $\tilde{A}_b(\Gamma)$ are intrinsic Lipschitz graphs, with an intrinsic Lipschitz constant depending on $a,b$, and the intrinsic Lipschitz constant of $\Gamma$.
\end{lemma}
\begin{proof}
  Let $q=s_{a,b}$ or $q=\tilde{A}_b$.  As $\Gamma$ is an intrinsic Lipschitz graph, there is a scale-invariant double cone $C\subset \H$ containing a neighborhood of $Y$ such that $pC\cap \Gamma=\emptyset$ for all $p\in \Gamma$.  The image $q(C)$ is a scale-invariant double cone containing a neighborhood of $Y$. Since
  $$\bigcap_{\lambda\in (0,1)} \mathrm{Cone}_\lambda = \langle Y\rangle\setminus \{\0\},$$ there is a $0<\lambda<1$ such that $\mathrm{Cone}_\lambda\subset q(C)$. For all $p\in \Gamma$,
  $$q(p) \mathrm{Cone}_\lambda \cap q(\Gamma) \subset q(p)q(C) \cap q(\Gamma)=q(pC\cap \Gamma)=\emptyset,$$
  so $q(\Gamma)$ is intrinsic $\lambda$--Lipschitz.
\end{proof}

\subsection{Measures on lines and the kinematic formula}\label{sec:kinematic}

Let $\cL$ be the space of horizontal lines in $\H$.  For $U\subset \H$, denote    the set of horizontal lines that intersect $U$ by
$$\cL(U)\eqdef \{L\in \cL:\ L\cap U\neq \emptyset\}.
$$
Let $\cN$ be the unique (up to constants) measure on $\cL$ that is invariant under the action of the isometry group of $\H$.  Scalings of horizontal lines are horizontal lines, so scaling automorphisms of $\H$ act on $\cL$, and $\cL(s_{t,t}(M))=t^3 \cL(M)$ for all $t>0$.  Henceforth $\cN$ will be normalized  so that $\cN(\cL(B_r(x)))=r^3$ for every $r>0$ and $x\in \H$.

The Heisenberg group satisfies the following kinematic formula, which we record here for ease of later use (see~\cite{Mon05} or equation (6.1) in~\cite{CKN}).  There exists a constant $c>0$ such that for any finite-perimeter set $E\subset \H$ and any open subset $U\subset \H$,
\begin{equation}\label{eq:kinematic}
  \Per_E(U)=c\int_{\cL} \Per_{E\cap L}(U\cap L) \ud\cN(L).
\end{equation}

Consider also the set $\cL^\#\eqdef \{(L,p):\ L\in \cL\ \wedge \ p\in L\}$ of {\em pointed horizontal lines}. Associate to each measurable subset $K\subset \cL^\#$ the following two quantities.
\begin{equation}\label{eq:first measure}
\int_\cL \cH^1\big(\{p\in L:\ (L,p)\in K\}\big) \ud \cN(L),
\end{equation}
and
\begin{equation}\label{eq:second measure}
\int_0^{2\pi}\int_{\H} \one_{K}\big(p\langle \cos (\theta )X+\sin(\theta)Y\rangle ,p\big) \ud\cH^4(p) \ud\theta.
\end{equation}
Both of the expressions in~\eqref{eq:first measure} and~\eqref{eq:second measure}   define measures on $\cL^\#$ that are invariant under the isometry group of $\H$, which acts transitively on $\cL^\#$. Therefore, they are proportional, and there is a constant $C>0$ such that for every measurable $K\subset \cL^\#$,
\begin{equation}\label{eq:pointed line identity}
\int_\cL \cH^1(K_L) \ud \cN(L) = C \int_{0}^{2\pi} \int_{\H} \one_{K}(L_{p,\theta},p) \ud\cH^4(p) \ud\theta,
\end{equation}
where we use the following notations for every $L\in \cL$, $p\in \H$ and $\theta\in [0,2\pi]$.
\begin{equation}\label{eq:L p theta}
K_L\eqdef \{p\in L\mid (L,p)\in K\}\subset L\qquad\mathrm{and}\qquad L_{p,\theta}\eqdef p\langle \cos (\theta )X+\sin(\theta)Y\rangle  \in \cL.
\end{equation}

\subsection{Vertical perimeter and parametric vertical perimeter} Given a measurable subset $E\subset V_0$, a measurable function  $\psi\from V_0\to \R$ and (a scale)  $a\in \R$, we define the {\em (normalized) parametric vertical perimeter at scale $a$  of $\psi$ on $E$} by
\begin{equation}\label{eq:def par vert perimeter}
\vpP{E,\psi}(a)\eqdef \frac{\int_E \big|\psi(v)-\psi\big(v Z^{-2^{-2a}}\big)\big|\ud \cH^3(v)}{2^{-a}}.
\end{equation}
This notion relates to the usual vertical perimeter~\eqref{eq:def vert per}  of the epigraph of $\psi$ as follows.
\begin{lemma}[parametric vertical perimeter versus vertical perimeter of epigraph]\label{lem:vpP and vpfl}
  For any measurable subset $E\subset V_0$, any measurable function $\psi\from V_0\to \R$, and any $a\in \R$, $$\vpP{E,\psi}(a)=\vpfl{\Pi^{-1}(E)}\big(\Gamma_\psi^+\big)(a).$$
\end{lemma}
\begin{proof}
  Recalling~\eqref{eq:def D set},  for $\Omega\subset \H$ and $a\in \R$ we denote $\mathsf{D}_a\Omega= \Omega \symdiff \Omega Z^{2^{-2a}}$.  Then
    \begin{equation*}
    \mathsf{D}_a\Gamma_\psi^+ = \big\{v Y^t \mid v\in V_0\ \wedge\  \psi(v)< t \le \psi\big(vZ^{-2^{-2a}}\big)\big\}
    \bigcup \big\{v Y^t \mid v\in V_0 \ \wedge\  \psi\big(vZ^{-2^{-2a}}\big) < t \le \psi(v)\big\},
  \end{equation*}
  since, by definition, $\Gamma_\psi^+ Z^{2^{-2a}}=\big\{v Y^t \mid v\in V_0\ \wedge\  \psi\big(vZ^{-2^{-2a}}\big)<t\big\}$.  Therefore,
  \begin{align*}
    \vpfl{\Pi^{-1}(E)}\big(\Gamma_\psi^+\big)(a)
    = \frac{\cH^4\big(\Pi^{-1}(E)\cap \mathsf{D}_a\Gamma^+_\psi\big)}{2^{-a}}
     =\frac{\int_{E} \big|\psi(v)-\psi\big(v Z^{-2^{-2a}}\big)\big|\ud \cH^3(v)}{2^{-a}}
    =\vpP{E,\psi}(a),
  \end{align*}
  where the second equality uses the fact that the map $(x,y,z)\mapsto (x,0,z)\cdot Y^y = (x,y,z+\frac{xy}{2})$ has constant Jacobian $1$.
\end{proof}

An advantage of the parametric vertical perimeter is that it increases or decreases by a constant factor under a stretch map or a shear map, as computed in the following lemma.
\begin{lemma}\label{lem:pvp is invariant}
  Let $\psi\from V_0\to \R$ and  $E\subset V_0$  be  measurable.  Let $q\from \H\to \H$, $\hat{q}\from V_0\to V_0$, and $\hat{\psi}\from V_0\to \R$ be as in Lemma~\ref{lem:curve transforms}, i.e., $q$ is a stretch map or a shear map, $\hat{q}$ is the map induced on $V_0$, and $\hat{\psi}$ is the function such that $q(\Gamma_\psi)=\Gamma_{\hat{\psi}}$. Then for all $t\in \R$ we have
\begin{itemize}
 \item  If $a, b\in \R\setminus \{0\}$ and $q=s_{a,b}$, then
  $\vpP{\hat{q}(E),\hat{\psi}}(t)=|ab|^{\frac{3}{2}}\cdot\vpP{E,\psi}\Big(t+\log_2\sqrt{|ab|}\Big).$
\item   If $b\in \R\setminus\{0\}$ and $q=\tilde{A}_b$, then
  $\vpP{\hat{q}(E),\hat{\psi}}(t)=\vpP{E,\psi}(t).$
  \end{itemize}
\end{lemma}
\begin{proof}
  If   $q=s_{a,b}$ for some $a, b\in \R\setminus \{0\}$, then   $\hat{q}(x,0,z)=(ax,0, abz)$ and $\hat{\psi}(\hat{q}(v))=b \psi(v)$ for every $v=(x,0,z)\in V_0$. So,
  \begin{align*}
    \vpP{\hat{q}(E),\hat{\psi}}(t)
    &=2^t\int_{\hat{q}(E)} \big|\hat{\psi}(v)-\hat{\psi}\big(v Z^{-2^{-2t}}\big)\big| \ud \cH^3(v)\\
    &=2^t |b| \int_{\hat{q}(E)} \big|\psi\big(\hat{q}^{-1}(v)\big)-\psi\big(\hat{q}^{-1}(v) Z^{-(ab)^{-1}2^{-2t}}\big)\big| \ud \cH^3(v)\\
    &=2^t a^2b^2 \int_{E} \big|\psi(v)-\psi\big(v Z^{-(ab)^{-1}2^{-2t}}\big)\big| \ud \cH^3(v)\\
    &=|ab|^{\frac{3}{2}}\cdot\vpP{E,\psi}\Big(t+\log_2\sqrt{|ab|}\Big).
  \end{align*}

 Next, if $q=\tilde{A}_b$ for some $b\in \R\setminus \{0\}$, then $\hat{\psi}(\hat{q}(v))=\psi(v)+bx(v)$ for all $v=(x,0,z)\in E$, and by Lemma~\ref{lem:v0 transforms} we have
  $$\hat{q}(v)=\Pi\big(q(x,0,z)\big)=\big(x,0,z-\frac12 bx^2\big).$$
 So, $\hat{\psi}(vZ^{-2t})=\psi(\hat{q}^{-1}(vZ^{-2^{-2t}}))+bx(\hat{q}^{-1}(vZ^{-2^{-2t}}))=\psi(\hat{q}^{-1}(v)Z^{-2^{-2t}})+bx(v)$, and hence
  \begin{multline*}
    \vpP{\hat{q}(E),\hat{\psi}}(t)
    =2^t \int_{\hat{q}(E)} \big|\psi\big(\hat{q}^{-1}(v)\big)-\psi\big(\hat{q}^{-1}(v) Z^{-2^{-2t}}\big)\big| \ud \cH^3(v)\\
    =2^t \int_{E} \big|\psi(v)-\psi\big(v Z^{-2^{-2t}}\big)\big| \ud\cH^3(v)
    =\vpP{E,\psi}(t).\tag*{\qedhere}
  \end{multline*}
\end{proof}

We end this section by recording  a straightforward \emph{a priori} upper bound on $\vpP{E,\psi}(a)$.

\begin{lemma}\label{lem:vpP exp decay} Suppose that $E\subset V_0$ is measurable and $\psi\from V_0\to \R$ is smooth. Then
  $$\forall  a\in \R,\qquad \vpP{E, \psi}(a)\le \min \left\{ 2^{a+1} \|\psi\|_{L_\infty(V_0)}, 2^{-a} \left\|\pd{z}{\psi}\right\|_{L_\infty(V_0)}\right\}\cH^3(E).$$
\end{lemma}
\begin{proof}
  For all $v=(x,0,z)\in E$, we (trivially) have
  $$
 \big|\psi(v)-\psi\big(v Z^{-2^{-2a}}\big)\big|= |\psi(x,0,z)-\psi(x,0,z-2^{-2a})|\le 2\|\psi\|_{L_\infty(V_0)},
  $$
  and
  $$\big|\psi(v)-\psi\big(v Z^{-2^{-2a}}\big)\big|= |\psi(x,0,z)-\psi(x,0,z-2^{-2a})|\le 2^{-2a} \left\|\pd{z}{\psi}\right\|_{L_\infty(V_0)}.$$
  Recalling the definition~\eqref{eq:def par vert perimeter}, we obtain the desired inequality by integrating over $E$.
\end{proof}


\section{Constructing surfaces and embeddings} \label{sec:counterexampleConstruction}

In this section, we will prove Proposition~\ref{prop:counterexample function}, following the reasoning sketched in Section~\ref{sec:maximally bumpy}, to construct surfaces that are $\alpha$--far from planes at $\alpha^{-4}$ different scales.  We use these surfaces to prove the following theorem.
\begin{thm}\label{thm:counterexample embedding}
  For any $k>1$, there is a left-invariant metric $\Delta=\Delta_k\from \H\times \H\to [0,\infty)$ on $\H$ and a measure space $(\mathscr{S},\mu)$ such that $(\H,\Delta)$ embeds isometrically in $L_1(\mu)$ and such that for any $h=(a,b,c)\in \H$ we have
 \begin{equation}\label{eq:for all h}|a|+|b|\lesssim \Delta(\mathbf{0},h)\lesssim |a|+|b|+\frac{\min\left\{\sqrt{|c|},k\right\}}{\sqrt[4]{\log k}}.\end{equation}
If moreover $1\le |c|\le k^2$, then, in fact
\begin{equation}\label{eq:with restriction on c}
\Delta(\mathbf{0},h)\asymp |a|+|b|+\frac{\sqrt{|c|}}{\sqrt[4]{\log k}}.
\end{equation}
\end{thm}

We will prove Theorem~\ref{thm:counterexample embedding} in Section~\ref{sec:proof of ctr embedding} after deriving  two of its applications, and stating Proposition~\ref{prop:counterexample function}.  The first application of Theorem~\ref{thm:counterexample embedding} is the proof of Theorem~\ref{thm:twist the center}.

\begin{proof}[Proof of Theorem~\ref{thm:twist the center} assuming Theorem~\ref{thm:counterexample embedding}] Letting $\Delta$ and $(\mathscr{S},\mu)$  be as in Theorem~\ref{thm:counterexample embedding}, fix $\xi\from \H\to L_1(\mu)$ such that $\|\xi(g)-\xi(h)\|_{L_1(\mu)}=\Delta(g,h)$ for all $g,h\in \H$. Also, using~\cite{Ass83}, fix $m\in \N$ and $\varphi\from \H\to \R^m$ such that $\|\varphi(g)-\varphi(h)\|_{\ell_1^m}\asymp \sqrt{d(g,h)}$ for all $g,h\in \H$.

Suppose that $\vartheta\ge \frac14$. Consider the function $\tau\from \H\to L_1(\mu)\oplus \R^2\oplus \R^m\cong L_1(\nu)$ (for a suitable measure $\nu$) that is given by
\begin{equation}\label{eq:tau formula}
\tau\eqdef \frac{\xi}{(\log k)^{\vartheta-\frac14}}\oplus \uppi\oplus \frac{\varphi}{(\log k)^\vartheta}.
\end{equation}
Since $\Delta$ is left-invariant, every $g=(x,y,z),h=(\chi,\upsilon,\zeta)\in \H$ with $1\le d(g,h)\le k$ satisfy
\begin{equation}\label{eq:tau theta}
\|\tau(g)-\tau(h)\|_{L_1(\nu)}\asymp |x-\chi|+|y-\upsilon|+\frac{\sqrt{|2z-2\zeta-x\upsilon+y\chi|}}{(\log k)^\vartheta},
\end{equation}
using~\eqref{eq:metric approximation}  and Theorem~\ref{thm:counterexample embedding}. While~\eqref{eq:tau theta} would hold even without the third component of $\tau$ in~\eqref{eq:tau formula}, thanks to that component $\tau(\H_\Z)$ is a locally-finite subset of $L_1(\nu)$. Every finite subset of $L_1(\nu)$ embeds with distortion $O(1)$ in $\ell_1$ (by approximating by simple functions), so by~\cite{Ost12}, it follows that $\tau(\H_\Z)$ admits a bi-Lipschitz embedding into $\ell_1$ of distortion  $O(1)$. As the word metric $d_W$ on $\H_\Z$ is bounded above and below by universal constant multiples of $d$, this gives Theorem~\ref{thm:twist the center} provided $k$ is a large enough universal constant multiple of $n$.
\end{proof}

A second application of Theorem~\ref{thm:counterexample embedding} is to construct a left-invariant metric on $\H_\Z$ with the properties of Theorem~\ref{thm:two Lps}, at the cost of losing an iterated logarithm in the associated distortion bounds that we derived in the proof of Theorem~\ref{thm:two Lps}. While the power of the iterated logarithm can be improved by taking more care in the ensuing reasoning, some  unbounded lower-order  loss {\em must} be incurred here; see Remark~\ref{rem:lower order term needed}.

\begin{thm}\label{thm:invariant metric on two sides} For any $2<p\le 4$ there is a left-invariant metric $\updelta=\updelta_p$ on $\H_\Z$ that admits a bi-Lipschitz embedding into both $\ell_1$ and $\ell_q$ for all $q\ge p$, yet not into any Banach space whose modulus of uniform convexity has power-type $r$ for  $2\le r<p$ (in particular, $(\H_\Z,\updelta)$ does not admit a bi-Lipschitz embedding into a Hilbert space or $\ell_s$ for $1<s<p$). Moreover, if we denote $\vartheta=1/p$, then  for every $h=(a,b,c)\in \H_\Z$ with $|c|\ge 3$ we have
\begin{equation}\label{eq:updelta metric}
\updelta(\0,h)\asymp |a|+|b|+\frac{\sqrt{|c|}}{(\log |c|)^\vartheta (\log\log |c|)^2}.
\end{equation}
\end{thm}

\begin{proof} Define a left-invariant metric $\updelta\from \H_\Z\times \H_\Z\to [0,\infty)$ as  a superposition of the metrics $\{\Delta_k\}_{k>0}$ of Theorem~\ref{thm:counterexample embedding}, by setting for every $h=(a,b,c)\in \H_\Z$,
\begin{equation}\label{eq:def updelta}
\updelta(\0,h)\eqdef \sum_{n=1}^\infty \frac{1}{n^2e^{(4\vartheta-1)n}} \Delta_{e^{e^{4n}}}(\0,h).
\end{equation}
We will first verify~\eqref{eq:updelta metric}, which in particular implies that the sum defining $\updelta$ converges, and hence by Theorem~\ref{thm:counterexample embedding} we would know  that $\updelta$ is indeed a left-invariant metric on $\H_\Z$, and that $(\H_\Z,\updelta)$  admits an isometric embedding into $\ell_1(L_1(\mu))$. By~\cite{Ost12}, it follows from this that  $(\H_\Z,\updelta)$ also admits a bi-Lipschitz embedding into the sequence space $\ell_1$.

Fix $h=(a,b,c)\in \H_\Z$ with $|c|\ge e^{e^4}$ and choose $m=m(c)\in \N$ such that
\begin{equation}\label{eq:choose mc}
e^{e^{4m}}\le \sqrt{|c|}<e^{e^{4(m+1)}}.
\end{equation}
 Then,
\begin{multline*}
\updelta(\0,h)\stackrel{\eqref{eq:for all h}}{\lesssim} |a|+|b|+\sum_{n=1}^\infty \frac{\min\left\{\sqrt{|c|},e^{e^{4n}}\right\}}{n^2e^{4\vartheta n}}\lesssim |a|+|b|+\sum_{n=1}^m \frac{e^{e^{4n}}}{n^2e^{4\vartheta n}} +\sum_{n=m+1}^\infty \frac{\sqrt{|c|}}{n^2e^{4\vartheta n}}\\
\asymp |a|+|b|+\frac{e^{e^{4m}}}{m^2e^{4\vartheta m}} +\frac{\sqrt{|c|}}{m^2e^{4\vartheta m}}\stackrel{\eqref{eq:choose mc}}{\lesssim} |a|+|b|+\frac{\sqrt{|c|}}{(\log |c|)^\vartheta (\log\log |c|)^2}.
\end{multline*}
Conversely, since the sum in~\eqref{eq:def updelta} is at least its summands for $n=1$ and $n=m+1$,
$$
\updelta(\0,h)\gtrsim |a|+|b|+\frac{\sqrt{|c|}}{(m+1)^2e^{4\vartheta (m+1)}}\stackrel{\eqref{eq:choose mc}}{\gtrsim} |a|+|b|+\frac{\sqrt{|c|}}{(\log |c|)^\vartheta (\log\log |c|)^2}.
$$
This is~\eqref{eq:updelta metric} if $|c|\ge e^{e^4}$, but then~\eqref{eq:updelta metric} follows formally in the remaining range  $3\le |c|<e^{e^4}$ (simply use  the triangle inequality to reduce  the upper bound to the case of large enough $|c|$ that we just proved, and take only the $n=1$ summand in~\eqref{eq:def updelta} for the lower bound).

By contrasting~\eqref{eq:updelta metric} with~\eqref{eq:equiv to Kor} we see that for every integer $n\ge 3$,
\begin{equation}\label{eq:loglog1}
\cc_{(\mathcal{B}_n,\updelta)}(\mathcal{B}_n,d_W)\lesssim (\log n)^\vartheta(\log\log n)^2.
\end{equation}
At the same time, if $2\le r<p$ and $X$ is a Banach space whose modulus of uniform convexity has power-type $r$, then by~\cite{LafforgueNaor} we have
\begin{equation}\label{eq:loglog2}
\cc_X(\mathcal{B}_n,d_W)\gtrsim_X (\log n)^{\frac{1}{r}}.
\end{equation}
By combining~\eqref{eq:loglog1} and~\eqref{eq:loglog2} we deduce that
$$
\cc_X(\mathcal{B}_n,\delta)\gtrsim_X \frac{(\log n)^{\frac{1}{r}-\vartheta}}{(\log\log n)^2}= \frac{(\log n)^{\frac{1}{r}-\frac{1}{p}}}{(\log\log n)^2}\xrightarrow[n\to \infty]{} \infty.
$$
Consequently, $(\H_\Z,\updelta)$ does not admit a bi-Lipschitz embedding into $X$.

It remains to show that $(\H_\Z,\updelta)$ admits a bi-Lipschitz embedding into $\ell_q$ for any $q\ge p$. As before, finite subsets of $L_q$ embed with distortion $O(1)$ in $\ell_p$ (by approximating by simple functions). Thus, due to~\cite{Ost12}, since $(\H_\Z,\updelta)$ is locally finite, it suffices to show  that $(\H_\Z,\updelta)$ admits a bi-Lipschitz embedding into $L_q$. By~\cite[Lemma~3.1]{LN14}, for any $0<\e<\frac{1}{2}$, there exists a left-invariant metric $\uprho_\e$ on $\H_\Z$ such that $(\H_\Z,\uprho_\e)$ embeds isometrically into $L_q$, and
\begin{equation}\label{eq:quote snowflake again}
\forall h=(a,b,c)\in \H_\Z,\qquad \uprho_\e(\0,h)\asymp |a|^{1-\e}+|b|^{1-\e}+\e^{\frac{1}{q}} |c|^{\frac{1-\e}{2}}.
\end{equation}
Define a left-invariant metric $\uprho\from \H_\Z\times \H_\Z\to [0,\infty)$  by setting for every $h=(a,b,c)\in \H_\Z$,
$$
\uprho(\0,h)\eqdef \bigg(|a|^q+|b|^q+\sum_{n=1}^\infty \frac{1}{n^{2q}e^{n(q\vartheta-1) }}\uprho_{2e^{-n}}(\0,h)^q \bigg)^{\frac{1}{q}}.
$$
By design, $(\H_\Z,\uprho)$ embeds isometrically into $\ell_q(L_q)$. So, the proof of Theorem~\ref{thm:invariant metric on two sides} will be complete if we show that $\updelta(\0,h)\asymp \uprho(\0,h)$ for all $h=(a,b,c)\in \H_\Z$ with, say, $|c|\ge 300$. To see this, by combining~\eqref{eq:updelta metric} and~\eqref{eq:quote snowflake again} it suffices to show that
\begin{equation}\label{eq:second sum to split}
\bigg(\sum_{n=1}^\infty \frac{1}{n^{2q}e^{nq\vartheta}|c|^{qe^{-n}}}\bigg)^{\frac{1}{q}}\asymp \frac{1}{(\log |c|)^\vartheta (\log\log |c|)^2}.
\end{equation}
Fix $s=s(c)\in \N$ such that $2e^{s}\le \log |c|<2e^{s+1}$ (this is possible because $|c|\ge 300$). Then,
\begin{multline*}
\bigg(\sum_{n=1}^\infty \frac{1}{n^{2q}e^{nq\vartheta}|c|^{qe^{-n}}}\bigg)^{\frac{1}{q}}\lesssim \bigg(\sum_{k=0}^{s-1}\frac{1}{(s-k)^{2q}e^{(s-k)q\vartheta}|c|^{qe^{-(s-k)}}}\bigg)^{\frac{1}{q}} +\bigg(\sum_{n=s+1}^\infty \frac{1}{n^{2q}e^{nq\vartheta}}\bigg)^{\frac{1}{q}}\\\asymp \frac{1}{e^{s\vartheta}}\bigg(\sum_{k=0}^{s-1}\frac{e^{qk\vartheta}}{(s-k)^{2q}(|c|^{e^{-s}})^{qe^k}}\bigg)^{\frac{1}{q}} +\frac{1}{s^2e^{s\vartheta}}\asymp \frac{1}{s^2e^{s\vartheta}}\asymp \frac{1}{(\log |c|)^\vartheta (\log\log |c|)^2},
\end{multline*}
where the final step holds by our choice of $s$, and the penultimate step holds as $|c|^{e^{-s}}\ge 2e$ by our choice of $s$, and therefore the sum in question is dominated by its $k=0$ summand. This proves half of the equivalence~\eqref{eq:second sum to split}, and the remaining direction of~\eqref{eq:second sum to split} follows by bounding from below the sum in the left hand side of~\eqref{eq:second sum to split} by its $n=s$ summand.
\end{proof}

\begin{remark}\label{rem:lower order term needed} It is evident from the above proof of Theorem~\ref{thm:invariant metric on two sides} that the power $2$ of $\log\log |c|$ in~\eqref{eq:updelta metric} can be improved to any fixed power that is strictly larger than $1$. However, the lower order term cannot be removed altogether. Specifically, suppose that $\mathfrak{d}$ is a left-invariant metric on $\H_\Z$ such that  every $h=(a,b,c)\in \H_\Z$ with $|c|\ge 3$ satisfies
\begin{equation}\label{eq:mathfrak d metric}
\mathfrak{d}(\0,h)\asymp |a|+|b|+\frac{\sqrt{|c|}}{\sqrt[4]{\log|c|}}.
\end{equation}
We claim that neither $\ell_1$ nor $\ell_4$ contains a bi-Lipschitz copy of $(\H_\Z,\mathfrak{d})$. In fact, we will next show that for every integer $n\ge 3$ the word-ball $\mathcal{B}_n\subset \H_\Z$ satisfies the distortion bounds
\begin{equation}\label{eq:ell1 loglog n}
\sqrt[4]{\log\log n}\lesssim \cc_{\ell_1}(\mathcal{B}_n,\mathfrak{d})\lesssim\log\log n,
\end{equation}
and,
\begin{equation}\label{eq:ell4 loglog n}
\cc_{\ell_4}(\mathcal{B}_n,\mathfrak{d})\asymp \sqrt[4]{\log\log n}.
\end{equation}
We conjecture that the first inequality in~\eqref{eq:ell1 loglog n} is sharp.

To prove~\eqref{eq:ell1 loglog n}, by substituting Theorem~\ref{thm:XYD} into~\cite[Lemma~33]{NY18}, and then substituting the resulting inequality into~\cite[Lemma~30]{NY18}, we get that there is a universal constant $\kappa\ge 5$ such that for every integer $n\ge 3$, every function $f:\H_\Z\to \ell_1$ satisfies
\begin{equation}\label{eq:local inequality mathfrak}
\bigg(\sum_{c=1}^{n^2} \frac{1}{c^3} \Big( \sum_{h\in \mathcal{B}_n} \|f(hZ^c)-f(h)\|_{\ell_1} \Big)^4\bigg)^{\frac14} \lesssim
\sum_{h\in \mathcal{B}_{\kappa n}} \big( \|f(hX)-f(h)\|_{\ell_1}+\|f(hY)-f(h)\|_{\ell_1}\big).
\end{equation}
Suppose that $D\ge 1$ is such that $\mathfrak{d}(g,h)\le \|f(g)-f(h)\|_{\ell_1}\le D\mathfrak{d}(g,h)$ for all $g,h\in \mathcal{B}_{2\kappa n}$. Then, by~\eqref{eq:mathfrak d metric} and~\eqref{eq:local inequality mathfrak} we have
\begin{equation}\label{eq:D log log}
D\gtrsim \bigg(\sum_{c=3}^{n^2} \frac{1}{c^3} \Big( \frac{\sqrt{c}}{\sqrt[4]{\log c}} \Big)^4\bigg)^{\frac14}=\bigg(\sum_{c=3}^{n^2} \frac{1}{c\log c} \bigg)^{\frac14}\asymp\sqrt[4]{\log\log n}.
\end{equation}
This proves the first inequality in~\eqref{eq:ell1 loglog n}. For the second inequality in~\eqref{eq:ell1 loglog n} consider the sum
$$
\mathfrak{d}_{1,n}= \sum_{j=0}^{5\lceil \log\log n\rceil}\Delta_{2^{2^j}}
$$
of metrics from Theorem~\ref{thm:counterexample embedding}. Then, by Theorem~\ref{thm:counterexample embedding} the metric space $(\H_\Z,\mathfrak{d}_{1,n})$ embeds isometrically into $\ell_1$ and $\mathfrak{d}\lesssim \mathfrak{d}_{1,n} \lesssim (\log\log n)\mathfrak{d}$ on $\mathcal{B}_n\times \mathcal{B}_n$.

The proof of~\eqref{eq:ell4 loglog n} is analogous. For the lower bound on $\cc_{\ell_4}(\mathcal{B}_n,\mathfrak{d})$  use (the case $q=4$ of) Theorem~1.1 in~\cite{LafforgueNaor} to get the following estimate for any function $f:\H_\Z\to\ell_4$.
$$
\sum_{c=1}^{n^2} \frac{1}{c^3} \Big( \sum_{h\in \mathcal{B}_n} \|f(hZ^c)-f(h)\|_{\ell_4} \Big)^4 \lesssim
\sum_{h\in \mathcal{B}_{21 n}} \big( \|f(hX)-f(h)\|_{\ell_4}^2+\|f(hY)-f(h)\|_{\ell_4}^4\big).
$$
With this inequality at hand, the desired lower bound follows as in~\eqref{eq:D log log}.  For the upper bound on $\cc_{\ell_4}(\mathcal{B}_n,\mathfrak{d})$, use the following metric on $\H_\Z$ which embeds isometrically into $\ell_4$.
$$
\mathfrak{d}_{4,n}= \bigg(\sum_{j=0}^{5\lceil \log\log n\rceil}\Delta_{2^{2^j}}^4\bigg)^{\frac14}.
$$

The above reasoning also shows mutatis mutandis that an unbounded lower-order factor loss is needed in the compression bound~\eqref{eq:H3 compression} of Theorem~\ref{thm:permanence}. Specifically, there is no mapping $f:\H_\Z\to \ell_1$ that is Lipschitz with respect to the word metric on $\H_\Z$ and whose compression rate (recall~\eqref{eq:compression rate}) satisfies $\omega_f(s)\gtrsim s/\sqrt[4]{\log s}$ when $s\ge 2$. It would be worthwhile to obtain a characterization of the possible compression rates of embeddings of $\H_\Z$ into $\ell_1$ in the spirit of~\cite[Theorem~9]{NY18}, but this would require more work. Specifically, one would need to replace the use in~\cite{NY18} of~\cite[Corollary~5]{Tes08} by a better embedding of $\H_\Z$ into $\ell_1$; we expect that the existence of such an embedding could could be deduced using the ideas of the present section, but we did not attempt to carry this out.
\end{remark}

The main ingredient in the proof of Theorem~\ref{thm:counterexample embedding} is the following proposition, which is proved in Section~\ref{sec:start construction}.  It constructs a function $\psi\from V_0\to \R$ whose intrinsic graph has small horizontal perimeter but large vertical perimeter due to bumps at many different scales.
Here and throughout the rest of this section, we denote the unit square in $V_0$ by $U$, i.e.,
$$U\eqdef [0,1]\times \{0\}\times [0,1]\subset V_0.$$

\begin{prop}\label{prop:counterexample function}
  There are universal constants $\rho,R,r\in \R$ with $R>r$ and $\rho>2^{2(R-r)}$ such that for any $\alpha\in \N$, there is a smooth function $\psi\from V_0\to \R$ that has the following properties.
  \begin{enumerate}
  \item\label{it:ctr periodic}
    $\psi$ is periodic with respect to the integer lattice $\Z\times \{0\}\times \Z$ of $V_0$.
  \item\label{it:ctr horiz bound}
    $\|\partial_\psi \psi\|_{L_2(U)}\lesssim 1$.
  \item \label{it:ctr linf}
    $\|\psi\|_{L_\infty(V_0)}\le \frac{1}{\alpha^2}$.
  \item \label{it:ctr interval lower}
    $\vpP{U,\psi}(a)\gtrsim \frac{1}{\alpha}$ for any integer $0\le n<\alpha^4$ and any $a\in I+  \log_2(\alpha\rho^{n})$, where $I=[r,R]$. Hence,
    $$\left\|\vpP{U,\psi}\right\|_{L_1([\log_2(\alpha\rho^{n})+r,\log_2(\alpha\rho^{n})+R])}\gtrsim \frac{1}{\alpha},$$
  \item \label{it:ctr lp vper}
    For any $q>0$, we have
    \begin{equation*}\label{eq:q=4 critical}\left\|\vpP{U,\psi}\right\|_{L_q(\R)} \gtrsim  \alpha^{\frac{4}{q}-1}.\end{equation*}
    \item $\vpP{U,\psi}(a)\lesssim \min\left\{\frac{1}{\alpha},\frac{2^a}{\alpha^2}\right\}$ for any $a\in \R$.
  \end{enumerate}
\end{prop}

By Proposition~\ref{prop:areaIntegral}, the second assertion of Proposition~\ref{prop:counterexample function} implies that $\cH^3(\partial E)\lesssim 1$, where $E$ is the epigraph of the restriction of $\psi$ to the unit square $U\subset V_0$.  In combination with Proposition~\ref{prop:counterexample function}.(\ref{it:ctr lp vper}), since $\alpha$ can be arbitrarily large, this shows that the $L_4(\R)$ norm in~\eqref{eq:coarea} cannot be replaced by $L_q(\R)$ for any $q\in (0,4)$; as explained in the introduction, this also implies the optimality of Theorem~\ref{thm:XYD}. Furthermore, since~\eqref{eq:coarea} is a consequence of~\eqref{eq:use previous corona}, Proposition~\ref{prop:counterexample function} also implies that for any $q\in (0,4)$, there is a $\lambda\in (0,1)$ such that for any $c>0$, there is an intrinsic $\lambda$--Lipschitz graph $\Gamma$ satisfying
$$\big\|\vpfl{B_1(\0)}\big(\Gamma^+\big)\big\|_{L_4(\R)}\ge c.$$
We expect that the construction in Section~\ref{sec:start construction} can be modified to produce an intrinsic Lipschitz graph directly (for instance, by stopping the construction early in regions where $\nabla_\psi \psi$ gets too large), but this is not needed here, so we leave the details to future work.

 Proposition~\ref{prop:counterexample function}.(\ref{it:ctr lp vper}) follows directly from Proposition~\ref{prop:counterexample function}.(\ref{it:ctr interval lower}). Indeed, since $\rho>2^{2(R-r)}$, the intervals $\{[\log_2(\alpha\rho^{n})+r,\log_2(\alpha\rho^{n})+R]\}_{n\in \Z}$ are disjoint. Consequently,
\begin{align}\label{eq:4 from 3}
\begin{split}
\left\|\vpP{U,\psi}\right\|_{L_q(\R)}^q&\ge \sum_{n=0}^{\alpha^4-1}  \left\|\vpP{U,\psi}\right\|_{L_q([\log_2(\alpha\rho^{n})+r,\log_2(\alpha\rho^{n})+R])}^q\\&\ge \sum_{n=0}^{\alpha^4-1}  \frac{1}{(R-r)^{q-1}}\left\|\vpP{U,\psi}\right\|_{L_1([\log_2(\alpha\rho^{n})+r,\log_2(\alpha\rho^{n})+R])}^q\gtrsim \alpha^{4-q}.
\end{split}
\end{align}
where the penultimate step is an application of Jensen's inequality and the final step holds because $R-r>0$ is a constant and by Proposition~\ref{prop:counterexample function}.(\ref{it:ctr interval lower}), each of the summands is at least a universal constant multiple of $\alpha^{-q}$.

\subsection{Obtaining an embedding from an intrinsic graph}\label{sec:proof of ctr embedding} Here we show how Theorem~\ref{thm:counterexample embedding} follows from Proposition~\ref{prop:counterexample function}.
Let $\rho,r, R>0$ be the universal constants of Proposition~\ref{prop:counterexample function}. Without loss of generality, we may take $k> 8$.  Let $\alpha\in \N$ be the unique integer satisfying
\begin{equation}\label{eq:alpha choice}
\sqrt[4]{\log_\rho \left(\frac{k}{8}\right)}\le \alpha <1+\sqrt[4]{\log_\rho \left(\frac{k}{8}\right)}.
\end{equation}

Let $\psi=\psi_\alpha$ be the function produced by Proposition~\ref{prop:counterexample function}.  Write $\Gamma=\Gamma_\psi$ and $\Gamma^+=\Gamma_\psi^+$. Denote by $A\subset V_0\cap \H_\Z$ the discrete subgroup that is generated by $X$ and $Z$, so that as a subset of $\R^3$ we have $A=\Z\times \{0\}\times \Z$.  For every $p\in \H$ define
$$
\forall  h_1,h_2\in \H,\qquad \lambda_p(h_1,h_2)\eqdef \left|\1_{p^{-1}\Gamma^+}(h_1)-\1_{p^{-1}\Gamma^+}(h_2)\right|=\left\{\begin{array}{ll}1 &\mathrm{if}\ |\{ph_1,ph_2\}\cap \Gamma^+|=1,\\
0&\mathrm{otherwise.}\end{array}\right.
$$
By the $A$--periodicity of $\psi$ we have $a\Gamma=\Gamma$ and  $\lambda_{ap}(h_1,h_2)=\lambda_p(h_1,h_2)$ for all $a\in A$ and $p,h_1,h_2\in \H$. We can therefore  define $\lambda_p$ also when $p$ is an equivalence class in the quotient $A\backslash \H$. Consider the following fundamental domain for $A$.
$$
P\eqdef \big\{X^aZ^cY^b:\ a,c\in [0,1)\ \mathrm{and}\ b\in \R\big\}=\big\{\big(a,b,c+\frac12 ab\big):\ (a,b,c)\in [0,1)\times \R\times [0,1)\big\}.
$$
We may define $l\from \H\times \H\to [0,\infty)$ by
$$
l(h_1,h_2)\eqdef \int_{A\textbackslash \H}\lambda_p(h_1,h_2)\ud \cH^4(p)=\int_{P} \lambda_p(h_1,h_2)\ud \cH^4(p).
$$
Since $\H$ is a unimodular group (namely, one directly checks that the Lebesgue measure $\cH^4$ is a bi-invariant Haar measure on $\H$), and $\lambda_p(gh_1,gh_2)=\lambda_{pg}(h_1,h_2)$, we have
$$
\forall  g,h_1,h_2\in \H,\qquad l(gh_1,gh_2)=l(h_1,h_2),
$$
i.e., $l$ is a left-invariant semi-metric on $\H$.

\begin{lemma}\label{lem:l is vper}
  For every $a\in \R$ we have
  $l\big(\mathbf{0},Z^{2^{-2a}}\big)=2^{-a} \cdot \vpP{U,\psi}(a).$
\end{lemma}
\begin{proof}
  For  $v\in V_0$ and $b\in \R$, we have $v Y^b \in \Gamma^+$ if and only if $b>\psi(v)$.  So, for any $c>0$,
  \begin{align*}
    \lambda_{vY^b}(\mathbf{0},Z^c)
    =\lambda_{\mathbf{0}}(vY^{b}, vZ^c Y^{b})
    = \begin{cases}
      1 & \psi(v)< b \le \psi(vZ^c)\text{ or } \psi(vZ^c) < b \le \psi(v).\\
      0 & \text{otherwise.}
    \end{cases}
  \end{align*}
  Consequently,
  $$\int_{\R} \lambda_{v Y^{b}}(\mathbf{0},Z^c)\ud b = |\psi(vZ^c)-\psi(v)|.$$
  Therefore, fixing $a\in \R$ and denoting $c=2^{-2a}$, we see that
  \begin{multline*}
    l\big(\mathbf{0},Z^{2^{-2a}}\big)=\int_{P}\lambda_{p}(\mathbf{0},Z^{c})\ud p
     =\int_{\R} \int_U \lambda_{vY^b}(\mathbf{0},Z^{c}) \ud v \ud b
   \\ =\int_U |\psi(vZ^{c})-\psi(v)| \ud v
     =2^{-a} \cdot \vpP{U,\psi}(a).\tag*{\qedhere}
  \end{multline*}
\end{proof}

For every $\theta\in [0,2\pi)$ let $R_\theta\from \H\to \H$ be  rotation around the $z$--axis by angle $\theta$.  Define the following left-invariant semi-metric on $\H$, which is also (by design) invariant under the family of $\{R_\theta: \theta\in [0,2\pi)\}$ automorphisms of $\H$.
$$
\forall  h_1,h_2\in \H,\qquad M(h_1,h_2)\eqdef \int_0^{2\pi} l\big(R_\theta(h_1),R_\theta(h_2)\big)\ud \theta.
$$

\begin{lemma}\label{lem:ctr M lipschitz}
  For every $w\in \mathsf{H}$ we have $M(\0,w)\lesssim \|w\|$.
\end{lemma}

\begin{proof}
  By the rotation-invariance of $M$, it suffices to show that $M(\0,X^t)\lesssim |t|$ for all $t$.  In fact, by the left-invariance of $M$ and the triangle inequality, it suffices to prove that $M(X^t,X^{-t})\lesssim t$ for $0<t<\frac14$.

  Let $L_0=\langle X\rangle\subset \H$ be the $x$--axis.  Recall that $L_{p,\theta}=pR_\theta(L_0)$ for  $p\in \H$ and $\theta\in [0,2\pi)$.  The map $(p,\theta)\mapsto (L_{p, \theta},p)$ is a bijection between $\H\times [0,\pi)$ and the set of pointed lines $\cL^\#=\{(L,p)\mid L\in \cL\ \wedge\  p\in L\}$.

  By the above definitions, we have
  $$M(X^{-t},X^t) = \int_{0}^{2\pi} \int_{P} \lambda_{p}\big(R_\theta(X^{-t}),R_\theta(X^t)\big) \ud\cH^4(p) \ud\theta.$$
  Let $K\subset P \times [0,2\pi)$ be the set of pairs $(p,\theta)$ such that $L_{p,\theta}$ intersects $\Gamma$ transversally, i.e., $L_{p,\theta}$ crosses the tangent plane of $\Gamma$ at every intersection. Since $\Gamma$ is smooth, the complement of $K$ has measure zero.

  Let $U'=B_{8}(\0)\cap V_0$.
  Let $(p,\theta)\in K$ be such that $\lambda_{p}(R_\theta(X^{-t}),R_\theta(X^t))\ne 0$. Then the line segment from $pR_\theta(X^{-t})$ to $pR_\theta(X^t)$ crosses $\Gamma$ at some point $g\in \Gamma$; we claim that $\Pi(g)\in U'$.

  By Proposition~\ref{prop:counterexample function}.(\ref{it:ctr linf}), we have $\|\psi\|_{L_\infty(V_0)}\le 1$, so $|y(g)|\le 1$ and $|y(p)|\le |y(g)|+t\le 2$. Since $p\in P$, there are $a, b,c\in \R$ such that $p=X^aZ^cY^b$, and these satisfy $|a|\le 1$, $|c|\le 1$, and $|b|=|y(p)|\le 2$. By \eqref{eq:metric approximation},
  $$d(\0, g)\le |a| + 4\sqrt{|c|} + |b| \le 7$$
  and
  $$d(\0, \Pi(g))\le d(\0, g) + |y(g)|\le 8,$$
  so $\Pi(g)\in U'$.

  Let $\Gamma(U')=\Gamma\cap \Pi^{-1}(U')=\Gamma_{\psi|_{U'}}$ and for $L\in \cL$, let
  $$I_{L}=\big\{p\in L\mid d\big(p,L \cap \Gamma(U')\big)\le t\big\}.$$
  We have seen above that if $(p,\theta)\in K$ and $\lambda_{p}(R_\theta(X^{-t}),R_\theta(X^t))\ne 0$, then there is some $g\in L_{p,\theta}\cap \Gamma(U')$ such that $d(p,g)\le t$. That is, $p\in I_{L_{p,\theta}}$. Furthermore, if $L$ intersects $\Gamma$ transversally, then
  $$\cH^1(I_L) \le 2t |L \cap \Gamma(U')| = 2t \Per_{\Gamma^+ \cap L}\big(\Pi^{-1}(U')\big).$$
  Hence,
  \begin{eqnarray*}
    M(X^{-t},X^t)&=&\int_{0}^{2\pi} \int_{P} \lambda_p\big(R_\theta(X^{-t}),R_\theta(X^t)\big) \ud\cH^4(p) \ud\theta\\
                 &\stackrel{\eqref{eq:pointed line identity}}{\lesssim}& \int_\cL \cH^1(I_L) \ud \cN(L) \\
                 &\lesssim& t \int_\cL \Per_{\Gamma^+ \cap L}\big(\Pi^{-1}(U')\big) \ud \cN(L) \\
                 &\stackrel{\eqref{eq:kinematic}}{\asymp}& t \Per_{\Gamma^+}\big(\Pi^{-1}(U')\big)\lesssim t.
  \end{eqnarray*}
  where $\Per_{\Gamma^+}\big(\Pi^{-1}(U')\big)\lesssim 1$ by Proposition~\ref{prop:areaIntegral} and Proposition~\ref{prop:counterexample function}.(\ref{it:ctr horiz bound}).
\end{proof}
  Next, define a left-invariant semi-metric $\Lambda$ on $\H$ by
  $$
  \forall  h_1,h_2\in \H,\qquad \Lambda(h_1,h_2)\eqdef \int_{r-\log_2\rho}^{R+\log_2\rho} 2^{a} M\big(s_{2^{-a}}(h_1), s_{2^{-a}}(h_2)\big) \ud a.
  $$

  \begin{lemma}\label{lem:Lambda on center} For all $c>0$ we have  $$\Lambda(\0,Z^c)=\Lambda(\0,Z^{-c})\lesssim \min\left\{\frac{\sqrt{c}}{\alpha},\frac{1}{\alpha^2}\right\}.$$
 Also, for all   $\frac{1}{\alpha^{2}\rho^{2\alpha^4}}\le c\le \frac{1}{\alpha^{2}}$ we have
   $$
   \Lambda(\0,Z^c)=\Lambda(\0,Z^{-c})\gtrsim \frac{\sqrt{c}}{\alpha}.$$
  \end{lemma}

  \begin{proof} Write $c=2^{-2t}$ for some $t\in \R$. Since $\Lambda$ is a left-invariant metric, $\Lambda(\0,Z^c)=\Lambda(\0,Z^{-c})$.
    By Lemma~\ref{lem:l is vper} we have the following identity.
\begin{equation}\label{eq:Lambda idenity}
  \Lambda(\0,Z^c)=2\pi \int_{r-\log_2\rho}^{R+\log_2\rho} 2^{a} l\left(\0,Z^{2^{-2(t+a)}}\right)\ud a=2\pi 2^{-t} \int_{r-\log_2\rho}^{R+\log_2\rho} \vpP{U,\psi}(t+a)\ud a.
  \end{equation}
So, $\Lambda(\0,Z^c)\lesssim \min\left\{\frac{\sqrt{c}}{\alpha},\frac{1}{\alpha^2}\right\}$ for  $c\in (0,\infty)$ by~\eqref{eq:Lambda idenity} and  the final assertion of Proposition~\ref{prop:counterexample function}.

  If $\frac{1}{\alpha^{2}\rho^{2\alpha^4}}\le c\le \frac{1}{\alpha^{2}}$, then  $t\in [\log_2(\alpha\rho^n),\log_2(\alpha\rho^{n+1})]$ for some  integer $0\le n< \alpha^4$. Hence, $$[t-\log_2\rho+r,t+\log_2\rho+R]\supset [\log_2(\alpha\rho^n)+r,\log_2(\alpha \rho^n)+R],$$ so~\eqref{eq:Lambda idenity} implies  that
  $$
  \Lambda(\0,Z^c)\ge 2\pi 2^{-t} \left\|\vpP{U,\psi}\right\|_{L_1([\log_2(\alpha\rho^{n})+r,\log_2(\alpha\rho^{n})+R])}\gtrsim \frac{\sqrt{c}}{\alpha},
  $$
  where the final step is the third assertion of Proposition~\ref{prop:counterexample function} (and the definition of $t$).
  \end{proof}

  \begin{lemma}\label{lem:Lambda is Lipschitz} $\Lambda(h_1,h_2)\lesssim d(h_1,h_2)$ for all $h_1,h_2\in \H$.
  \end{lemma}

  \begin{proof}
    By Lemma~\ref{lem:ctr M lipschitz} we have $M(\0,X^t)\lesssim |t|$ for any $t\in \R$, so
    \begin{equation*}
      \Lambda(\0,X^t)=\int_{r-\log_2\rho}^{R+\log_2\rho} 2^a M\big(\0,X^{2^{-a}t}\big)\ud a\lesssim t (R-r+2\log_2\rho)\lesssim |t|.
    \end{equation*}
    Therefore also $\Lambda(\0,Y^t)=\Lambda(\0,X^t)\lesssim |t|$, by the rotation-invariance of $\Lambda$.  Since $\Lambda$ is left-invariant it suffices to show that $\Lambda(\0,h)\lesssim d(\0,h)$ for all $h\in \H$. Any $h\in \H$ can be written as $h=X^aY^b[X^c,Y^c]$ for  $a,b,c\in \R$ satisfying $|a|,|b|,|c|\lesssim d(\0,h)$, so
    \begin{equation*}
      \Lambda(\0,h)\le \Lambda(\0,X^a) + \Lambda(\0,Y^b) + 2 \Lambda(\0,X^c) + 2 \Lambda(\0,Y^c)\lesssim d(\0,h).\tag*{\qedhere}
    \end{equation*}
  \end{proof}

  \begin{proof}[Proof of Theorem~\ref{thm:counterexample embedding}] Define a semi-metric $\Delta$ on $\H$ by setting for every $h_1,h_2\in \H$,
 \begin{equation}\label{eq:def Delta}
  \Delta(h_1,h_2)\eqdef k\alpha \Lambda\big(s_{\frac{1}{k\alpha}}(h_1),s_{\frac{1}{k\alpha}}(h_2)\big)+\sqrt{\big(x(h_1)-x(h_2)\big)^2+\big(y(h_1)-y(h_2)\big)^2}.
\end{equation}
Observe that $(\H,\Delta)$ embeds isometrically into $L_1$ because $\Lambda$ is an integral of so-called cut semimetrics (see e.g.~\cite[4.1]{DL97} for the definition). Such semimetrics embed isometrically into $\R$, so an integral of cut semimetrics embeds isometrically in $L_1$. By construction, $\Lambda$ is both left-invariant and invariant under the rotations $\{R_\theta:\ \theta\in [0,2\pi]\}$.

 Suppose that $v=(a,b,c)\in \H$ and let $w=(a,b,0)$ so that $w\in \mathsf{H}$ and $v=wZ^c$.  By Lemma~\ref{lem:ctr M lipschitz} and the second part of Lemma~\ref{lem:Lambda on center}, we have
 $$
|a|+|b|\lesssim \Delta(\0,v)\le \Delta(\0,w)+\Delta(\0,Z^c)\lesssim |a|+|b|+\frac{\min\left\{\sqrt{|c|},k\right\}}{\alpha}.
 $$
 Recalling that $\alpha\asymp \sqrt[4]{\log k}$ is given in~\eqref{eq:alpha choice}, this establishes~\eqref{eq:for all h}.

 To prove Theorem~\ref{thm:counterexample embedding}, it therefore remains to establish~\eqref{eq:with restriction on c}, i.e.,
 \begin{equation}\label{eq:goal for distance in range}
   1\le |c|\le k^2\implies  \Delta(\0,v)\gtrsim  |a|+|b|+\frac{\sqrt{|c|}}{\alpha}.
 \end{equation}
 By Lemma~\ref{lem:Lambda is Lipschitz}, there is  $L>0$ such that $\Delta(h_1,h_2)\le L d(h_1,h_2)$ for any $h_1,h_2\in \H$. By the first part of Lemma~\ref{lem:Lambda on center}, there is  $C>0$ such that $\Delta(\0,Z^c)\ge \frac{C\sqrt{c}}{\alpha}$ for all $1\le c \le k^2$.  On one hand, if $\|w\| \ge \frac{C\sqrt{|c|}}{2 L \alpha}$, then $\Delta(\0,v)\ge \|w\| \approx |a|+|b|+\frac{\sqrt{|c|}}{\alpha}$.  On the other hand, if $\|w\| < \frac{C \sqrt{|c|}}{2 L \alpha}$, then
 $$\Delta(\0,v)\ge \Delta(\0,Z^c)-\Delta(\0,w) \ge \frac{C\sqrt{c}}{\alpha} - L\|w\| \ge \frac{C \sqrt{|c|}}{2\alpha}\gtrsim |a|+|b|+\frac{\sqrt{|c|}}{\alpha}.$$
 In either case, \eqref{eq:goal for distance in range} holds.
\end{proof}

\subsection{Constructing a bumpy intrinsic graph}\label{sec:start construction}

In this section, we prove Proposition~\ref{prop:counterexample function}.  We start with a brief overview of our strategy. As sketched in Section~\ref{sec:maximally bumpy}, we will prove Proposition~\ref{prop:counterexample function} by constructing a smooth function $\psi\from V_0\to \R$ whose intrinsic graph is roughly $\alpha^{-1}$--far from a vertical plane at $\alpha^4$ different scales.  Specifically, for a suitable choice of universal constant $\rho>1$ we will  construct $\psi$ as a sum $\psi=\sum_{i=0}^{\alpha^4-1}\beta_i$.  Each of the summands $\beta_i\from V_0\to \R$ will itself be a sum of smooth bump functions of amplitude $\|\beta_i\|_{L_\infty(V_0)}\approx \alpha^{-2} \rho^{-i}$ that are supported on regions whose  width ($x$--coordinate) is $\rho^{-i}$ and whose height ($z$--coordinate) is roughly $\alpha^{-2}\rho^{-2i}$; their aspect ratio is therefore roughly $$\frac{\rho^{-i}}{\sqrt{\alpha^{-2}\rho^{-2i}}}\approx \alpha.$$  These regions cover $V_0$ and have disjoint interiors.  We will see that the bumpiness of $\beta_i$ at scale $\alpha^{-2}\rho^{-2i}$ implies the desired lower bounds on $\vpP{U,\psi}(t)$ when $t$ is near $\log_2(\alpha\rho^i)$.

In order to ensure that $\|\partial_\psi \psi\|_{L_2(U)}$ is bounded, we construct $\beta_i$ iteratively.  For $i\in \N$, we denote $\psi_{i}=\sum_{j=0}^{i-1}\beta_j$ and align the long axis of the bump functions making up $\beta_i$ with the characteristic curves  of $\Gamma_{\psi_{i}}$. This ensures that the characteristic curves of $\Gamma_{\psi}$ cross the bumps from left to right.  Since $\partial_{\psi} f$ measures the change in $f\from V_0\to \R$ along the characteristic curves of $\Gamma_\psi$ and each bump has amplitude roughly $\alpha^{-2} \rho^{-i}$ and width $\rho^{-i}$, we have $|\partial_{\psi} \beta_i| \lesssim  \alpha^{-2} \rho^{-i}/\rho^{-i}\asymp \alpha^{-2}$.

This iterative procedure is one of the motivations for the definition of a foliated corona decomposition.  A foliated corona decomposition of an arbitrary intrinsic graph $\Gamma$ can be viewed as a sequence of partitions of $V_0$ into regions as above, where the pieces of the partition are aligned with the characteristic curves of $\Gamma$.  One can use these partitions to reconstruct $\Gamma$ as a sum of perturbations, just as we constructed $\psi$ as a sum of bump functions.  Theorem~\ref{thm:ilg admits fcd intro} then states that any intrinsic Lipschitz graph can be constructed by such a process.

This construction also demonstrates the importance of the aspect ratio.  If the construction is modified so that the bump functions making up $\beta_i$ are supported on regions of aspect ratio $\alpha_i$, then $\|\partial_\psi \beta_i\|_{L_2(U)}\approx \alpha_i^{-2}$.  If the scales of the bump functions are sufficiently separated, then $\{\partial_\psi \beta_i\}_{i\ge 0}$ are roughly orthogonal in $L_2(U)$ and
$$\|\partial_\psi \psi\|_{L_2(U)}^2\approx \sum_{i\ge 0} \|\partial_\psi \beta_i\|_{L_2(U)}^2\approx \sum_{i\ge 0} \alpha_i^{-4}.$$
For $\psi$ to be intrinsic $\lambda$--Lipschitz, we must have $\|\partial_\psi \psi\|_{L_2(U)}^2\lesssim_\lambda 1$, which necessitates that $\sum_i \alpha_{i\ge 0}^{-4}\lesssim_\lambda 1$. This motivates the $\alpha(Q)^{-4}$ factor in the weighted Carleson condition~\eqref{eq:Carleson weighted intro}.

We next set some notation in preparation for the proof of Proposition~\ref{prop:counterexample function}.  If $\psi\from V_0\to \R$ is smooth, then the vector field $$M_\psi\eqdef \frac{\partial}{\partial x}-\psi \frac{\partial}{\partial z}$$ corresponding to $\partial_\psi$ is smooth (recall the definitions in Section~\ref{sec:intrinsic graphs}).  The flow lines of $M_\psi$ are the characteristic curves of $\Gamma_\psi$, which foliate $V_0$ (recall the terminology in Section~\ref{sec:char}).  For $s\in \R$, let $\Phi(\psi)_s\from V_0\to V_0$ be the flow of $M_\psi$, so that $\Phi(\psi)_0=\id_{V_0}$ and such that for any $v\in V_0$, the curve $s\mapsto \Phi(\psi)_s(v)$ is a characteristic curve of $\Gamma_\psi$.

Denote $\psi_0\equiv 0$ and let $\Gamma_0=\Gamma_{\psi_0}=V_0$. This function and graph are periodic with respect to $\Z\times \{0\}\times \Z$ and $\psi_0$ is zero on $\partial U$. Suppose that $i\ge 0$ and that $\psi_i\from V_0\to \R$ is smooth, periodic with respect to $\Z\times \{0\}\times \Z$, and zero on $\partial U$. We construct $\psi_{i+1}\from V_0\to \R$ as follows. Let
  \begin{equation}\label{eq:grid Gi}
  G_i\eqdef \big\{(m\rho^{-i},0,n\alpha^{-2}\rho^{-2i}):\ m,n\in \Z\big\}\subset V_0.
  \end{equation}
  Label the points in $G_i$ arbitrarily as $v_{i,1},v_{i,2},\ldots$ and note that the points $U\cap \{v_{i,1},v_{i,2},\ldots \}$ form a $\rho^i\times \alpha^2\rho^{2i}$ grid in $U$.  For each $j\in \N$ and $s,t\in \R$ define
\begin{equation}\label{eq:def Rij}
R_{i,j}(s,t)\eqdef \Phi(\psi_i)_s(v_{i,j}Z^t)\in V_0.
\end{equation}

Each $R_{i,j}$ is a diffeomorphism from $\R^2$ to $V_0$.
For any $s_0, t_0\in \R$, the image $R_{i,j}(s_0\times \R)$ is a vertical line and $R_{i,j}(\R \times t_0)$ is a characteristic curve of $\Gamma_{\psi_i}$.  Using the terminology of foliated patchworks that we will introduce in Section~\ref{sec:pseudoquads}, the map $R_{i,j}$ sends rectangles in $V_0$ to pseudoquads of $\Gamma_{\psi_i}$ (regions in $V_0$ that are bounded by characteristic curves of $\Gamma_{\psi_i}$  above and below and by vertical line segments on either side).  Denote
\begin{equation}\label{eq:def Qij}
Q_{i,j}\eqdef R_{i,j}\big([0,\rho^{-i}]\times [0,\alpha^{-2}\rho^{-2i}]\big)\subset V_0.
\end{equation}
Thus, $Q_{i,j}$ is a pseudoquad whose lower-left corner is $v_{i,j}$.  The sets $Q_{i,1},Q_{i,2},\ldots$ cover $V_0$ and have disjoint interiors.  They are obtained by cutting $V_0$ into vertical strips of width $\rho^{-i}$, then cutting each vertical strip along characteristic curves separated by $\alpha^{-2}\rho^{-2i}$.

Since $\psi_i$ is zero on $\partial U$, the top and bottom edges of $U$ are characteristic curves of $\Gamma_{\psi_i}$.  The bottom boundary of each $Q_{i,0}$ and the top boundary of $Q_{i,\alpha^2\rho^{2i}-1}$ thus lie in $\partial U$, and the $Q_{i,j}$'s partition $U$ (up to overlap on boundaries). In particular, the resulting partition of $V_0$ is periodic with respect to $\Z\times \{0\}\times \Z$.

Note, however, that the $Q_{i,j}$'s from one step in this construction generally do not partition the $Q_{i,j}$'s from another step. One can modify the construction so that the partitions in each step are nested, as in Figure~\ref{fig:foliated corona of bump}, but it requires some additional care.

Let $\beta\from V_0\to \R$ be a smooth function supported on the unit square $U$ such that $\beta$ is not identically zero and its partial derivatives of order at most $2$ are all in the interval $[-1,1]$. Fix also $\alpha,\rho\in \N$ with $\rho>1$.
Define $\beta_{i,j}\from V_0\to \R$ by setting it to be $0$ on $V_0\setminus Q_{i,j}$, and for all $R_{i,j}(s,t)\in Q_{i,j}$,
\begin{equation}\label{eq:def beta ij}
\beta_{i,j}\big(R_{i,j}(s,t)\big)\eqdef \alpha^{-2}\rho^{-i} \beta(\rho^is,0,\alpha^2\rho^{2i}t).
\end{equation}
Thus $\beta_{i,j}$ is a bump function supported on $Q_{i,j}$. Write
\begin{equation}\label{eq:def beta i}
\beta_i\eqdef \sum_{j=1}^\infty \beta_{i,j},
\end{equation}
and
\begin{equation}\label{eq:recurse psi}
\psi_{i+1}\eqdef \psi_i+\beta_i.
\end{equation}
Since $Q_{i,1},Q_{i,2},\ldots $ have disjoint interiors, $\|\psi_{i+1}\|_{L_\infty(V_0)}\le \|\psi_{i}\|_{L_\infty(V_0)} + \alpha^{-2}\rho^{-i}$, so by induction we have
\begin{equation}\label{eq:L infty psi i}
\|\psi_{i}\|_{L_\infty(V_0)}\le \frac{\alpha^{-2}}{\rho-1}\le \alpha^{-2}.
\end{equation}
Since the $Q_{i,j}$'s form a periodic partition of $V_0$, $\psi_{i+1}$ is periodic.
Since $\partial U$ is contained in the boundaries of the $Q_{i,j}$, we have $\psi_{i+1}|_{\partial U}=\psi_i|_{\partial U}=0$.

Thus, by induction, for any integer  $i\ge 0$, $\psi_i$ satisfies the first and third assertions (periodicity and $L_\infty$ boundedness) of Proposition~\ref{prop:counterexample function}.  We will show that if $\rho$ is large enough (depending only on $\beta$), then $\psi=\psi_{\alpha^4}$ satisfies the remaining assertions of Proposition~\ref{prop:counterexample function}, namely, the stated upper bounds on $\partial_\psi \psi$ and lower bounds on $\vpP{U,\psi}(a)$.

\subsubsection{The horizontal perimeter of $\Gamma_{\psi_i}$}\label{sec: per of example}

In this section, we prove the second assertion of Proposition~\ref{prop:counterexample function} by bounding $\|\partial_{\psi_i}\psi_i\|_{L_2(U)}$.  This bound, combined with Proposition~\ref{prop:areaIntegral}, gives an upper bound on $\cH^3(\Gamma_{\psi_i|_U})$.

Write for simplicity $\partial_i\eqdef \partial_{\psi_i}$ and let $D_i\eqdef \partial_{i+1} \psi_{i+1}-\partial_i \psi_i$. For $f,g\in L_2(U)$ we write
$$
\langle f,g\rangle_U\eqdef\int_U fg\ud \cH^3.
$$

\begin{lemma}\label{lem:alomost orthonormal} For every $\rho\ge 5$ and $\alpha\ge 1$,
$$
\forall  i\in \N,\qquad \|D_i\|_{L_\infty(V_0)}\lesssim \alpha^{-2},
$$
and
$$
\forall  m,n\in \N,\qquad |\langle D_m,D_n\rangle_U|\lesssim \alpha^{-4}\rho^{m-n}.
$$
\end{lemma}

Note that Lemma~\ref{lem:alomost orthonormal}  implies that for every $i\in \N$,
\begin{equation}\label{eq:sqrt i}
\|\partial_{\psi_i}\psi_i\|_{L_2(U)}\lesssim \frac{\sqrt{i}}{\alpha^2}.
\end{equation}
Thus, $\|\partial_{\psi_i}\psi_i\|_{L_2(U)}\lesssim 1$ for $i\lesssim \alpha^4$, i.e., the second assertion of Proposition~\ref{prop:counterexample function} holds true. To deduce~\eqref{eq:sqrt i} from  Lemma~\ref{lem:alomost orthonormal}  write
\begin{equation}\label{eq:telescoping D_m}
\partial_{\psi_i}\psi_i=\sum_{n=0}^{i-1}D_n,
\end{equation}
 and expand the squares to get
$$
\|\partial_{\psi_i}\psi_i\|_{L_2(U)}^2=\sum_{n=0}^{i-1} \|D_n\|_{L_2(U)}^2+2 \sum_{m=0}^{i-1}\sum_{n=m+1}^{i-1} \langle D_m,D_n\rangle\lesssim \sum_{n=0}^{i-1} \alpha^{-4}+\sum_{m=0}^{i-1}\sum_{k=1}^{\infty} \alpha^{-4}\rho^{-k}\asymp i\alpha^{-4},
$$
where the penultimate step is Lemma~\ref{lem:alomost orthonormal}  and the final step holds because $\rho\ge 2$.

Fix an integer $i\ge 0$ and note that
\begin{equation}\label{eq:brin in beta i}
D_i=\partial_{i+1} \psi_{i+1}-\partial_i \psi_i= (\partial_{i+1}-\partial_i) \psi_{i+1}+\partial_i \beta_i =-\beta_i \frac{\partial\psi_{i+1}}{\partial z}+\partial_i\beta_i.
\end{equation}
We will prove Lemma~\ref{lem:alomost orthonormal}  by bounding the terms in the right hand side of~\eqref{eq:brin in beta i} separately. To this end, it will be convenient to define as follows a system of {\em flow coordinates} on $Q_{i,j}$.

Fix  $i\in \N\cup \{0\}$ and $j\in \N$. Write for simplicity $(x_0,0,z_0)=v_{i,j}$, $Q=Q_{i,j}$ and $R=R_{i,j}$. Denote $R^{-1}=(s,t)\from Q\to \R^2$ and let $(x,0,z)\from Q\to \R^2$ be the standard coordinate system. Then $s$ and $t$ are functions of $x$ and $z$ and, conversely, $x$ and $z$ are functions of $s$ and $t$. Recalling the differential equation~\eqref{eq:horizontal curve eq} for characteristic curves, we have
$$
x=x_0+s\qquad\mathrm{and}\qquad z=z_0+t-\int_0^s \psi_i\big(R(\sigma,t)\big)\ud \sigma.
$$
Consequently,
\begin{equation}\label{eq:partials st xy}
  \begin{pmatrix}
    \pd{s}{x} & \pd{t}{x} \\
    \pd{s}{z} & \pd{t}{z}
  \end{pmatrix}=\begin{pmatrix}
    1 &  0 \\
    -\psi_i & 1-\int_0^s \frac{\partial \psi_i}{\partial t} (R(\sigma, t)) \ud \sigma
  \end{pmatrix},
\end{equation}
where for $f\from \R^2\to \R$, the partial derivatives $\pd{s}{f}$ and $\pd{t}{f}$ denote $\partial_s[f\circ R]$ and $\partial_t[f\circ R]$, respectively.
In particular, it follows that $\frac{\partial s}{\partial z}=0$ and $\frac{\partial z}{\partial t}\cdot \frac{\partial t}{\partial z}=1$.  Also,
\begin{equation}\label{eq:partial-s is nabla-i}
  \frac{\partial}{\partial s}=\frac{\partial}{\partial x}-\psi_i \frac{\partial}{\partial z}=\partial_i,
\end{equation}
so $\frac{\partial}{\partial s}$ does not depend on $j$.

Observe that by the definition of $\beta_i$, for all $s,t\in [0,\rho^{-i}]\times [0,\alpha^{-2} \rho^{-2i}]$, we have
$$\beta_i\big(R_{i,j}(s,t)\big)=\alpha^{-2} \rho^{-i} \beta(\rho^i s,0,\alpha^2 \rho^{2i} t).$$
It follows that for any $m,n\in \N\cup\{0\}$, we have
\begin{equation}\label{eq:derivs of beta}
  \left\|\frac{\partial^m}{\partial s^m}\frac{\partial^n}{\partial t^n} \beta_i\right\|_{L_\infty(Q_{i,j})}= \alpha^{-2} \rho^{-i}\rho^{mi} (\alpha^2 \rho^{2i})^n\left\|\frac{\partial^m}{\partial x^m}\frac{\partial^n}{\partial z^n}\beta\right\|_{L_\infty(U)}.
\end{equation}
This is especially useful when $m+n\le 2$, since in this case $\left\|\frac{\partial^m}{\partial x^m}\frac{\partial^n}{\partial z^n}\beta\right\|_{L_\infty(U)}\le 1$.  Thus,
\begin{equation}\label{eq:first two der betai}
 \left\|\frac{\partial \beta_i}{\partial t} \right\|_{L_\infty(Q_{i,j})}\le \rho^i\qquad\mathrm{and} \qquad  \left\|\frac{\partial^2\beta_i}{\partial t^2} \right\|_{L_\infty(Q_{i,j})}\le \alpha^2\rho^{3i}.
\end{equation}
Furthermore, since $\{Q_{i,j}\}_{j=1}^\infty$ cover $V_0$,
\begin{equation}\label{eq:bound nabla beta}
  \|\partial_i\beta_i\|_{L_\infty(V_0)}=\max_{j\in \N} \|\partial_i\beta_i\|_{L_\infty(Q_{i,j})}\stackrel{\eqref{eq:partial-s is nabla-i}}{=}\max_{j\in \N}\left\|\pd{s}{\beta_i}\right\|_{L_\infty(Q_{i,j})}\stackrel{\eqref{eq:derivs of beta}}{\le} \alpha^{-2}.
\end{equation}

The following lemma obtains  bounds on vertical derivatives that will be used later.

\begin{lemma}\label{lem:derivs of h}
  If $\rho\ge 8$, then for all $i\in \N\cup \{0\}$ we have
  \begin{equation}
    \label{eq:dhi/dz}
    \left\|\frac{\partial \psi_{i}}{\partial z}\right\|_{L_\infty(V_0)} \le 2\rho^{i-1},
    \end{equation}
    and
    \begin{equation}
    \label{eq:d2hi/dz2}
    \left\|\frac{\partial^2 \psi_{i}}{\partial z^2}\right\|_{L_\infty(V_0)} \le 2\alpha^2\rho^{3i-3}.
  \end{equation}
 Furthermore, if $(s,t)$ are the above flow coordinates on $Q_{i,j}$ for some $j\in \N$, then the following bound holds point-wise on $Q_{i,j}$.
  \begin{equation}
    \label{eq:dz/dt}
    \frac{3}{4}< e^{-2\rho^{-1}} \le \pd{z}{t}=\left(\pd{t}{z}\right)^{-1}\le e^{2\rho^{-1}}< \frac{4}{3}
  \end{equation}
\end{lemma}
\begin{proof} Denote for every integer $i\ge 0$,
\begin{equation}\label{eq:def m_i,mu_i}
m_i\eqdef \left\|\frac{\partial \psi_{i}}{\partial z}\right\|_{L_\infty(V_0)}\qquad\mathrm{and}\qquad \mu_i\eqdef \left\|\frac{\partial^2 \psi_{i}}{\partial z^2}\right\|_{L_\infty(V_0)}.
\end{equation}
Thus $m_0=\mu_0=0$. Fix $j\in \N$ and let $(s,t)$ be the flow coordinates on $Q_{i,j}$. We will first use the above identities to deduce bounds on vertical derivatives of $t$ in terms of $m_i,\mu_i$, and then bootstrap these bounds to deduce the desired bounds on $m_i,\mu_i$ themselves.

By \eqref{eq:partials st xy}, the following identity holds point-wise on $Q_{i,j}$.
$$\frac{\partial}{\partial s}\frac{\partial z}{\partial t}=-\frac{\partial \psi_i}{\partial t} = -\frac{\partial \psi_i}{\partial z}\frac{\partial z}{\partial t}-\frac{\partial \psi_i}{\partial x}\frac{\partial x}{\partial t}=-\frac{\partial \psi_i}{\partial z}\frac{\partial z}{\partial t}.
$$
  Consequently,
$$\pd{s}{}\left(\log \pd{t}{z}\right)=-\pd{z}{\psi_i}.$$
  Since $\pd{t}{z}=1$ when $s=0$, we integrate to get the identity
        \begin{equation}\label{eq:logarithmic integral}\frac{\partial z}{\partial t}=\exp \left(-\int_0^{s} \frac{\partial \psi_i}{\partial z} \big(R_{i,j}(\sigma, t)\big) \ud \sigma\right).\end{equation}
    And, by differentiating~\eqref{eq:logarithmic integral} we also get
          \begin{equation}\label{eq:logarithmic integral2}\frac{\partial^2 z}{\partial t^2}=-\frac{\partial z}{\partial t}\int_0^{s} \frac{\partial^2 \psi_i}{\partial z^2} \big(R_{i,j}(\sigma, t)\big) \frac{\partial z}{\partial t} \big(R_{i,j}(\sigma, t)\big)\ud \sigma.\end{equation}

       For points in $Q_{i,j}$, we have $|s|\le \rho^{-i}$, so it follows from~\eqref{eq:logarithmic integral} that $\bigl|\log \frac{\partial z}{\partial t}\bigr|\le \rho^{-i} m_i$, i.e.,
      \begin{equation}\label{eq:two sided with mi}
      e^{-\rho^{-i}m_i}\le \frac{\partial z}{\partial t}\le e^{\rho^{-i}m_i}.
      \end{equation}
       By substituting~\eqref{eq:two sided with mi} into~\eqref{eq:logarithmic integral2} we deduce that
       \begin{equation}\label{eq:second der of z in t}
       \left|\frac{\partial^2 z}{\partial t^2}\right|\le \rho^{-i}e^{2\rho^{-i}m_i}\mu_i.
       \end{equation}
    Since $\frac{\partial z}{\partial t}\cdot \frac{\partial t}{\partial z}=1$, it follows from~\eqref{eq:two sided with mi} that
   \begin{equation}\label{eq:two sided with mi inverse}
      e^{-\rho^{-i}m_i}\le  \frac{\partial t}{\partial z}\le e^{\rho^{-i}m_i},
      \end{equation}
    and also
    \begin{equation}\label{eq:second t with respect to z}
    \left|\frac{\partial^2 t}{\partial z^2}\right|=\left|-\left(\frac{\partial z}{\partial t}\right)^{-3}\frac{\partial^2 z}{\partial t^2}\right|\stackrel{\eqref{eq:two sided with mi}\wedge \eqref{eq:second der of z in t}}{\le} \rho^{-i}e^{5\rho^{-i}m_i}\mu_i.
    \end{equation}

 The bounds~\eqref{eq:two sided with mi inverse} and~\eqref{eq:second t with respect to z} on the vertical derivatives of the flow coordinate $t$ are in terms of the bounds $m_i,\mu_i$ on the vertical derivatives of $\psi_i$, but they imply as follows unconditional   bounds on $m_i,\mu_i$ (hence also, by~\eqref{eq:two sided with mi inverse} and~\eqref{eq:second t with respect to z} once more, unconditional bounds on the vertical derivatives of $t$). Firstly, observe that
 $$
 \left\|\frac{\partial \beta_i}{\partial z}\right\|_{L_\infty(Q_{i,j})}\le \left\|\frac{\partial \beta_i}{\partial t}\right\|_{L_\infty(Q_{i,j})}\left\|\frac{\partial t}{\partial z}\right\|_{L_\infty(Q_{i,j})}\stackrel{\eqref{eq:first two der betai}\wedge \eqref{eq:two sided with mi inverse} }{\le} \rho^ie^{\rho^{-i}m_i},
 $$
 and
\begin{multline*}
  \left\|\frac{\partial^2 \beta_i}{\partial z^2}\right\|_{L_\infty(Q_{i,j})}= \left\|\frac{\partial}{\partial z}\frac{\partial t}{\partial z}\frac{\partial \beta_i}{\partial t}\right\|_{L_\infty(Q_{i,j})}= \left\|\frac{\partial^2 t}{\partial z^2}\frac{\partial \beta_i}{\partial t}+\left(\frac{\partial t}{\partial z}\right)^2\frac{\partial^2\beta_i}{\partial t^2}\right\|_{L_\infty(Q_{i,j})}\\\stackrel{\eqref{eq:first two der betai}\wedge \eqref{eq:two sided with mi inverse}\wedge \eqref{eq:second t with respect to z} }{\le} \rho^{-i}e^{5\rho^{-i}m_i}\mu_i\rho^i+e^{2\rho^{-i}m_i}\alpha^2\rho^{3i}=e^{5\rho^{-i}m_i}\mu_i+e^{2\rho^{-i}m_i}\alpha^2\rho^{3i}.
 \end{multline*}
 Since $\{Q_{i,j}\}_{j=1}^\infty$ cover $V_0$, it follows that
 $$
  \left\|\frac{\partial \beta_i}{\partial z}\right\|_{L_\infty(V_0)}\le  \rho^ie^{\rho^{-i}m_i}\qquad\mathrm{and} \qquad   \left\|\frac{\partial^2 \beta_i}{\partial z^2}\right\|_{L_\infty(V_0)}\le e^{5\rho^{-i}m_i}\mu_i+e^{2\rho^{-i}m_i}\alpha^2\rho^{3i}.
 $$
 Since by~\eqref{eq:recurse psi} we have $\frac{\partial\psi_{i+1}}{\partial z}=\frac{\partial\psi_{i}}{\partial z}+\frac{\partial\beta_{i}}{\partial z}$ and $\frac{\partial^2\psi_{i+1}}{\partial z^2}=\frac{\partial^2\psi_{i}}{\partial z^2}+\frac{\partial^2\beta_{i}}{\partial z^2}$, we deduce that
\begin{equation}\label{m mu recursions}
 m_{i+1}\le m_i+ \rho^ie^{\rho^{-i}m_i}\qquad\mathrm{and}\qquad \mu_{i+1}\le \mu_i+ e^{5\rho^{-i}m_i}\mu_i+e^{2\rho^{-i}m_i}\alpha^2\rho^{3i}.
 \end{equation}
 By induction, we suppose that~\eqref{eq:dhi/dz} and~\eqref{eq:d2hi/dz2} hold for some integer $i\ge 0$, that is,
 \begin{equation}\label{eq:mi mui hypothesis}
 m_i\le 2 \rho^{i-1}\qquad \mathrm{and}\qquad  \mu_i\le 2\alpha^2\rho^{3i-3}.
   \end{equation}
   Since $\rho\ge 8$, it follows that
   $$
   m_{i+1}\stackrel{\eqref{m mu recursions}\wedge\eqref{eq:mi mui hypothesis}}{\le} \left(\frac{2}{\rho}+e^{2\rho^{-1}}\right)\rho^i\le \left(\frac14+\sqrt[4]{e}\right)\rho^i\le 2\rho^i.
   $$
 Thus~\eqref{eq:dhi/dz} holds for all integers $i\ge 0$.  Likewise,
 $$
\mu_{i+1}  \stackrel{\eqref{m mu recursions}\wedge\eqref{eq:mi mui hypothesis}}{\le} \left(\frac{2}{\rho^3}+\frac{2e^{10\rho^{-1}}}{\rho^3}+e^{4\rho^{-1}}\right)\alpha^2\rho^{3i}\le \left(\frac{2}{8^3}+\frac{2e^{\frac54}}{8^3}+\sqrt{e}\right)\alpha^2\rho^{3i}\le 2\alpha^2\rho^{3i},
 $$
 so~\eqref{eq:d2hi/dz2} also holds for all integers $i\ge 0$. The remaining assertion~\eqref{eq:dz/dt} follows by substituting the above bound on $m_i$ into~\eqref{eq:two sided with mi inverse}. \qedhere
 \end{proof}

Next, we will use the bounds of Lemma~\ref{lem:derivs of h} to bound  $\{D_i\}_{i=0}^\infty$ and their derivatives.

\begin{lemma}\label{lem:derivs of Di} Suppose that $\rho\ge 8$.
  For every integer $i\ge 0$ we have
  \begin{align}
    \label{eq:Di}
    \|D_i\|_{L_\infty(V_0)} &\le 3\alpha^{-2},\\
    \label{eq:dDi/dz}
    \left\|\frac{\partial D_i}{\partial z}\right\|_{L_\infty(V_0)}&\le 6 \rho^{2i},\\
    \label{eq:nabla Di}
    \|\partial_i D_i\|_{L_\infty(V_0)} &=5\alpha^{-2}\rho^{i}.
  \end{align}
\end{lemma}
\begin{proof}
  Fix $j\in \N$.  Let  $(s,t)$ be the flow coordinates on $Q_{i,j}$.  By \eqref{eq:brin in beta i} and \eqref{eq:partial-s is nabla-i}, we have
  \begin{equation}\label{eq:Di formula}
  D_i=-\beta_i \frac{\partial \psi_{i+1}}{\partial z}+\frac{\partial \beta_i}{\partial s}.
  \end{equation}
 Therefore, by Lemma~\ref{lem:derivs of h} and~\eqref{eq:derivs of beta} we have
  $$\|D_i\|_{L_\infty(Q_{i,j})} \le \alpha^{-2} \rho^{-i} \cdot 2\rho^{i}+\alpha^{-2}= 3\alpha^{-2}.$$
  This proves \eqref{eq:Di} because $\{Q_{i,j}\}_{j=1}^\infty$ cover $V_0$.

  Next, we consider $\pd{z}{D_i}$.  By differentiating~\eqref{eq:Di formula} we see that
  \begin{align*}
    \frac{\partial D_i}{\partial z}
    &=-\frac{\partial t}{\partial z}\cdot \frac{\partial \beta_i}{\partial t}\cdot \frac{\partial \psi_{i+1}}{\partial z}- \beta_i \frac{\partial^2 \psi_{i+1}}{\partial z^2}+\frac{\partial t}{\partial z}\cdot \frac{\partial^2 \beta_i}{\partial s\partial t}.
  \end{align*}
Hence,  by Lemma~\ref{lem:derivs of h} and~\eqref{eq:derivs of beta} we see that
  $$\left\|\frac{\partial D_i}{\partial z}\right\|_{L_\infty(Q_{i,j})} \le \frac{4}{3}\cdot \rho^{i}\cdot 2\rho^i +\alpha^{-2}\rho^{-i}\cdot 2\alpha^2 \rho^{3i} + \frac{4}{3}\cdot\rho^{2i}=6 \rho^{2i}.$$
  As before, this proves~\eqref{eq:dDi/dz} because $\{Q_{i,j}\}_{j=1}^\infty$ cover $V_0$.

  Finally, we consider $\partial_i D_i$. Note first that for any $m\in \N$,
  \begin{equation}\label{eq:d nabla hi/dz}
    \left\|\frac{\partial (\partial_{m}\psi_m)}{\partial z} \right\|_\infty \stackrel{\eqref{eq:telescoping D_m}}{\le} \sum_{n=0}^{m-1}\left\|\frac{\partial D_n}{\partial z}\right\|_\infty \stackrel{\eqref{eq:dDi/dz}}{\le} 6\frac{\rho^{2m}-1}{\rho^2-1}\le  7\rho^{2m-2},
  \end{equation}
where we used the assumption $\rho\ge 8$. Recalling~\eqref{eq:brin in beta i} and~\eqref{eq:partial-s is nabla-i}, we have
  \begin{equation*}
    \partial_i D_i
    =\frac{\partial}{\partial s}\left( -\beta_i\frac{\partial \psi_{i+1}}{\partial z}+\frac{\partial \beta_i}{\partial s}\right)
    =-\frac{\partial \beta_i}{\partial s} \cdot \frac{\partial \psi_{i+1}}{\partial z}-\beta_i\frac{\partial}{\partial s}\left(\frac{\partial \psi_{i+1}}{\partial z}\right)+\frac{\partial^2 \beta_i}{\partial s^2}.
  \end{equation*}
  Using Lemma~\ref{lem:derivs of h} and \eqref{eq:derivs of beta}, it follows that
  \begin{equation}\label{eq:for lie}
  \|\partial_iD_i\|_{L_\infty(Q_{i,j})}\le 3 \alpha^{-2}\rho^i+\alpha^{-2}\rho^{-i}\left\|\frac{\partial}{\partial s}\frac{\partial \psi_{i+1}}{\partial z}\right\|_{L_\infty(Q_{i,j})}.\end{equation}
To bound the last term in~\eqref{eq:for lie}, we first calculate the Lie bracket
  $$\left[\frac{\partial}{\partial z}, \frac{\partial}{\partial s}\right]=\left[\frac{\partial}{\partial z}, \frac{\partial}{\partial x}-\psi_i \frac{\partial}{\partial z}\right]=-\frac{\partial \psi_i}{\partial z}\cdot \frac{\partial}{\partial z}.$$
  This implies that
  \begin{multline*}
  \frac{\partial}{\partial s}\frac{\partial \psi_{i+1}}{\partial z}=\frac{\partial}{\partial z}\left(\frac{\partial \psi_{i+1}}{\partial s}\right)+\frac{\partial \psi_i}{\partial z}\cdot \frac{\partial \psi_{i+1}}{\partial z}\\=
  \frac{\partial}{\partial z}\left(\partial_{i}\psi_i+ \frac{\partial \beta_i}{\partial s}\right)+\frac{\partial \psi_i}{\partial z}\cdot \frac{\partial \psi_{i+1}}{\partial z}=  \frac{\partial (\partial_{i}\psi_i)}{\partial z} +\frac{\partial t}{\partial z}\cdot \frac{\partial^2\beta_i}{\partial s\partial t}+\frac{\partial \psi_i}{\partial z}\cdot \frac{\partial \psi_{i+1}}{\partial z}.
  \end{multline*}
Therefore,  by Lemma~\ref{lem:derivs of h}, \eqref{eq:derivs of beta}, \eqref{eq:dDi/dz}, and \eqref{eq:d nabla hi/dz}, we conclude that (since $\rho\ge 8$),
  \begin{equation*}
    \left\|\frac{\partial}{\partial s}\frac{\partial \psi_{i+1}}{\partial z}\right\|_{L_\infty(Q_{i,j})}
    \le 7\rho^{2i-2} +\frac{4}{3}\cdot \rho^{2i}+2 \rho^{i-1}\cdot 2 \rho^i
    \le 2\rho^{2i}.
  \end{equation*}
Due to~\eqref{eq:for lie},  this implies the final desired bound~\eqref{eq:nabla Di} of Lemma~\ref{lem:derivs of Di}.
\end{proof}

The first assertion~\eqref{eq:Di} of Lemma~\ref{lem:derivs of Di} gives the first assertion of Lemma~\ref{lem:alomost orthonormal}. To prove the second assertion of Lemma~\ref{lem:alomost orthonormal}, we first bound  the variation of $D_m$ on each of the pseudoquads $\{Q_{n,j}\}_{j=1}^\infty$ when $n\ge m$.
\begin{lemma}\label{lem:variation of Di} Fix two integers $n\ge m\ge 0$. For any $j\in \N$ and any $w,w'\in Q_{n,j}$, we have $$|D_m(w)-D_m(w')|\lesssim \alpha^{-2} \rho^{m-n}.$$
\end{lemma}
\begin{proof}
  Let $R=R_{n,j}$ and let $(s,t),(s',t')\in [0,\rho^{-n}]\times [0,\alpha^{-2}\rho^{-2n}]$ be such that $R(s,t)=w$ and $R(s',t')=w'$.  With respect to flow coordinates on $Q_{n,j}$, we have
  $$\pd{s}{D_m} = \partial_n D_m = \partial_m D_m+(\psi_m-\psi_n)\frac{\partial D_m}{\partial z}.$$
  Since $\|\psi_m-\psi_n\|_{L_\infty(V_0)} \le \alpha^{-2}\rho^{-m}+\alpha^{-2}\rho^{-n}\le 2\alpha^{-2}\rho^{-m}$, using Lemma~\ref{lem:derivs of Di} we get that
  \begin{equation*}
    \left\|\pd{s}{D_m}\right\|_{L_\infty(Q_{n,j})}  \le 5\alpha^{-2}\rho^{m}+2 \alpha^{-2}\rho^{-m}\cdot 6\rho^{2m}=17 \alpha^{-2}\rho^{m}.
  \end{equation*}
Hence, using Lemma~\ref{lem:derivs of h} and Lemma~\ref{lem:derivs of Di} we conclude that
  \begin{multline*}
    |D_m(w)-D_m(w')|\le \left\|\pd{s}{D_m}\right\|_{L_\infty(Q_{n,j})} |s-s'|  + \left\|\frac{\partial D_m }{\partial z}\right\|_{L_\infty(Q_{n,j})} \left\|\frac{\partial z }{\partial t}\right\|_{L_\infty(Q_{n,j})}|t-t'| \\
    \le 17 \alpha^{-2} \rho^{m-n}+ 6\rho^{2m}\cdot \frac{4}{3}\cdot \alpha^{-2}\rho^{-2n}\lesssim \alpha^{-2} \rho^{m-n}.\tag*{\qedhere}
  \end{multline*}
\end{proof}

Prior to proving Proposition~\ref{prop:counterexample function}, we record a quick consequence of Green's theorem.
\begin{lemma}\label{lem:Heisenberg Stokes}
  Let $M\subset V_0$ be a region bounded by a simple piecewise-smooth closed curve and let $f\from V_0\to \R$ be a smooth function.  Then
  $$\int_{M} \partial_f f \ud w=\int_{\partial M} \left(\frac{f^2}{2},f\right)\cdot \ud \mathbf{r}.$$
  In particular, if $g\from V_0\to \R$ is another smooth function such that $f=g$ on $\partial M$, then $$\int_{M} \partial_f f \ud w=\int_{M} \partial_g g \ud w.$$
\end{lemma}
\begin{proof}
  Since
  $$\nabla\times \left(\frac{f^2}{2},f\right)=\frac{\partial f}{\partial x}- f\frac{\partial f}{\partial z}=\partial_f f,$$
  the lemma follows from Green's Theorem.
\end{proof}

\begin{proof}[{Proof of Lemma~\ref{lem:alomost orthonormal}}] The first assertion of Lemma~\ref{lem:alomost orthonormal} was proved in  Lemma~\ref{lem:derivs of Di}, so here we treat its second assertion, namely that $\{D_n\}_{n=0}^\infty$ are almost-orthogonal.

Fix  $m,n\in \N\cup\{0\}$ with $n\ge m$ and $j\in \N$.   Since $\psi_{n+1}-\psi_{n}=\beta_n=0$ on $\partial Q_{n,j}$, Lemma~\ref{lem:Heisenberg Stokes} implies that $\int_{Q_{n,j}} D_n(w) \ud w=0$. So, fixing an arbitrary basepoint $w_0\in Q$, we have
\begin{equation*}
\biggl|\int_{Q_{n,j}} D_m(w) D_n(w) \ud w\biggr|=\biggl|\int_{Q_{n,j}} \big(D_m(w)-D_m(w_0)\big) D_n(w) \ud w\biggr|
\lesssim \alpha^{-4}\rho^{m-n}\cH^3(Q_{n,j}),
\end{equation*}
where in the final step we used~\eqref{eq:Di} and  Lemma~\ref{lem:variation of Di}. Hence, $ |\langle D_m,  D_n\rangle_U|$ is at most
  \begin{equation*}
 \sum_{\substack{j\in \N\\ Q_{n,j}\subset U}} \biggl|\int_{Q_{n,j}} D_m(w) D_n(w) \ud w\biggr|
    \le \sum_{\substack{j\in \N\\ Q_{n,j}\subset U}} \cH^3(Q_{n,j}) \alpha^{-4} \rho^{m-n} \asymp \alpha^{-4} \rho^{m-n}.\tag*{\qedhere}
  \end{equation*}
\end{proof}

\subsubsection{The vertical perimeter of $\Gamma_{\psi_i}$} Here we will complete the proof of Proposition~\ref{prop:counterexample function}.

 Define $\phi\from V_0\to \R$ to be the $A$--periodic extension of $\beta|_U$, i.e., $\phi(x,0,z)\eqdef\beta(\{x\},0,\{z\})$ for $(x,0,z)\in V_0$, where $\{a\}=a-\lfloor a\rfloor$ is the fractional part of $a\in \R$. Because the function
 $$
 \vpP{U,\phi}\from \R\to [0,\infty)
 $$
 is  continuous and not identically zero, there exist $\eta,R,r\in \R$ with $r<R$ such that
 \begin{equation}\label{eq:choose I}
 \forall  a\in I\eqdef [r,R],\qquad \vpP{U,\phi}(a)\ge \eta>0.
 \end{equation}
 We will show that if $\rho\in \N$ is large enough (depending only on the initial choice of bump function $\beta$), then the conclusion of Proposition~\ref{prop:counterexample function} holds for the above interval $I$. To this end, we will first establish the following point-wise bound on the vertical perimeter of each of the perturbations $\{\beta_i\}_{i=0}^\infty$ in terms of the vertical perimeter of $\phi$.

\begin{lemma}\label{lem:relate to phi} Suppose that $\rho>5$. For every $i\in \N\cup\{0\}$ and $a\in \R$ we have
$$
\vpP{U,\beta_i}(a)\ge \frac{1}{2\alpha}\vpP{U,\phi}\big(a-\log_2(\alpha\rho^i)\big)-\frac{3\rho^{i-1}}{2^{a}}.
$$
In particular, if $\rho\ge \frac{12}{2^r\eta}$ and $a\in I+\log_2(\alpha\rho^i)$, then $$\vpP{U,\beta_i}(a) \ge \frac{\eta}{2\alpha}-\frac{3\rho^{i-1}}{2^{r+\log_2(\alpha\rho^i)}} \ge \frac{\eta}{4\alpha}.$$
\end{lemma}

\begin{proof}[Proof of Proposition~\ref{prop:counterexample function} assuming Lemma~\ref{lem:relate to phi}] Fix an integer $\rho\ge \max\{12/(2^r\eta),8\}$ that will be specified later and let $\psi=\psi_{\alpha^4}$.  The first three assertions of Proposition~\ref{prop:counterexample function} were established in the construction of $\psi$ and in the discussion after Lemma~\ref{lem:alomost orthonormal}.  We will establish the last three by showing that
\begin{equation}\label{eq:last assertion}
 \forall a\in \R,\qquad \vpP{U,\psi}(a)\lesssim \min\left\{\frac{1}{\alpha},\frac{2^a}{\alpha^2}\right\},
\end{equation}
and
\begin{equation}\label{eq:assertion four}
 \forall a\in \bigcup_{n=0}^{\alpha^{4}-1}\big( I+  \log_2(\alpha\rho^{n})\big),\qquad \vpP{U,\psi}(a)\gtrsim \frac{1}{\alpha}.
\end{equation}

For every $i\in \N\cup \{0\}$, by the definition of $\beta_i$ and by~\eqref{eq:first two der betai} and~\eqref{eq:dz/dt}  we have $$\left\|\beta_i\right\|_{L_\infty(V_0)}\le \alpha^{-2}\rho^{-i}\qquad\mathrm{and}\qquad \left\|\frac{\partial \beta_i}{\partial z}\right\|_{L_\infty(V_0)}\le 2\rho^i.$$
Due to Lemma~\ref{lem:vpP exp decay}, for every $a\in \R$ we have
\begin{equation}\label{eq:a priori betai vert}
\vpP{U,\beta_i}(a)\le \min\left\{2^{a+1}\alpha^{-2}\rho^{-i},2^{-a+1}\rho^i\right\}=2\alpha^{-1}2^{-|a-\log_2(\alpha\rho^i)|}.
\end{equation}
Consequently,
$$
\vpP{U,\psi_{\alpha^4}}(a)=\vpP{U,\sum_{i=0}^{\alpha^4}\beta_i}(a)\le \sum_{i=0}^{\alpha^4} \vpP{U,\beta_i}(a)\stackrel{\eqref{eq:a priori betai vert}}{\le} \sum_{i=0}^\infty 2\alpha^{-1}2^{-|a-\log_2(\alpha\rho^i)|}\lesssim \alpha^{-1}.
$$
This proves \eqref{eq:last assertion}, because  by Lemma~\ref{lem:vpP exp decay}  we also have $$\vpP{U,\psi_{\alpha^4}}(a)\lesssim 2^a\|\psi_{\alpha^4}\|_{L_\infty(V_0)}\stackrel{\eqref{eq:L infty psi i}}{\le} 2^a\alpha^{-2}.$$

It remains to prove \eqref{eq:assertion four}, as we saw in~\eqref{eq:4 from 3} that this implies the remaining assertions of Proposition~\ref{prop:counterexample function}. Fix $n\in \{0,\ldots,\alpha^4-1\}$ and $a\in I+  \log_2(\alpha\rho^{n})$, so that $\vpP{U,\beta_n}(a) >\eta/(4\alpha)$ by Lemma~\ref{lem:relate to phi}. Let $s=\max\{|r|,|R|\}$, so that $|a-\log_2(\alpha\rho^n)|\le s$.  It follows from~\eqref{eq:a priori betai vert} that
\begin{equation}\label{eq:upper at i}
  \vpP{U,\beta_i}(a) \le 2\alpha^{-1} 2^{-|\log_2(\alpha\rho^n)-\log_2(\alpha\rho^i)| + |a-\log_2(\alpha\rho^n)|}\le 2\alpha^{-1}\rho^{-|n-i|} 2^s
\end{equation}
for any $i\in \N\cup \{0\}$.
Hence, by combining Lemma~\ref{lem:relate to phi} and~\eqref{eq:upper at i} we conclude that
\begin{multline*}
\vpP{U,\psi_{\alpha^4}}(a)=\vpP{U,\sum_{i=0}^{\alpha^4}\beta_i}(a)
\ge \vpP{U,\beta_n}(a)-\sum_{i=0}^{n-1}\vpP{U,\beta_i}(a)-\sum_{i=n+1}^{\alpha^4}\vpP{U,\beta_i}(a)\\\ge \frac{\eta}{4\alpha} - 2 \sum_{k=1}^{\infty}2\alpha^{-1}\rho^{-k} 2^{s}\ge \frac{\eta}{4\alpha} - \frac{5}{\alpha\rho} 2^{s}.
\end{multline*}
Choosing $\rho\eqdef \left\lceil \max\left\{8, \frac{12}{2^r\eta}, \frac{40 \cdot 2^s }{ \eta}\right\}\right\rceil$, this completes the proof of Proposition~\ref{prop:counterexample function}.\qedhere
\end{proof}

\begin{proof}[Proof of Lemma~\ref{lem:relate to phi}] We will start by introducing some (convenient, though ad hoc) notation and making some preliminary observations. For $i\in \N\cup \{0\}$ define a (discontinuous in the first variable) map $\mathfrak{S}_i\from \R^2\to \R^2$ as follows. If $s\in \R$, then  let $m\in \Z$ be the unique integer such that $s\in [m\rho^{-i},(m+1)\rho^{-i})$, and set for every $t\in \R$,
$$
\mathfrak{S}_i(s,t)\eqdef \Phi(\psi_i)_{s-m\rho^{-i}}(m\rho^{-i},0,t),
$$
where we recall the notation $\Phi(\cdot)_\cdot(\cdot)$ for characteristic curves that we set at the start of Section~\ref{sec:start construction}. Note that by design $x(\mathfrak{S}_i(s,t))=s$. Observe also that the lines $\R\times \{0\}\times \{0\}$ and $\R\times \{0\}\times \{1\}$ are characteristic curves for $\Gamma_{\psi_i}$, since  $\psi_i$ vanishes on those lines. Hence $\mathfrak{S}_i(s,0)=(s,0)$ and $\mathfrak{S}_i(s,1)=(s,1)$ for all $s\in [0,1]$. As $x(\mathfrak{S}(s,t))=s$ for all $t\in \R$, by the continuity of $\mathfrak{S}_i$ in the second variable, this implies that $\mathfrak{S}_i(s,[0,1])=\{s\}\times\{0\}\times [0,1]$. So,
\begin{equation}\label{eqLsquare to square}
\mathfrak{S}_i([0,1]^2)=U.
\end{equation}

The mapping $\mathfrak{S}_i$ is related  as follows to  the mappings $R_{i,1},R_{i,2},\ldots$ that are given in~\eqref{eq:def Rij}. Suppose as above  that $s\in [m\rho^{-i},(m+1)\rho^{-i})$ for some $m\in \Z$, and fix $t\in \R$ and $n\in \Z$. Recalling that $v_{i,1},v_{i,2},\ldots$ is an enumeration of the points in the grid $G_i$ that is given in~\eqref{eq:grid Gi}, let  $j\in \N$ be the index for which  $v_{i,j}=(m\rho^{-i},0,n\alpha^{-2}\rho^{-2i})$. Then
$$
\mathfrak{S}_i(s,t)=R_{i,j}(s-m\rho^{-i},t-n\alpha^{-2}\rho^{-2i}).
$$
Recalling the definition~\eqref{eq:def Qij} of the pseudo-quad $Q_{i,j}$, this implies that
$$
\overline{\mathfrak{S}_i\big([m\rho^{-i},(m+1)\rho^{-i})\times [n\alpha^{-2}\rho^{-2i},(n+1)\alpha^{-2}\rho^{-2i}]\big)}=Q_{i,j}.
$$
Also, recalling the definitions~\eqref{eq:def beta ij} and~\eqref{eq:def beta i},  it follows that if we define $\phi_i\from V_0\to \R$ by
\begin{equation}\label{eq:rescaled phi}
\forall  (s,t)\in \R^2,\qquad \phi_i(s,0,t)\eqdef \alpha^{-2}\rho^{-i}\phi(\rho^is,0,\alpha^2\rho^{2i}t),
\end{equation}
then
\begin{equation}\label{eq:slab flow}
\forall  (s,t)\in \R^2,\qquad \beta_i\big(\mathfrak{S}_i(s,t)\big)=\phi_i(s,0,t).
\end{equation}

Fix $i\in \N\cup \{0\}$, $a\in \R$  and  $(x,0,z)\in U$. Let $s=s(x,z),t=t(x,z),t'=t'(x,z,a)\in \R$ satisfy
\begin{equation}\label{eq:s,t,t'}
\mathfrak{S}_i(s,t)=(x,0,z)\qquad \mathrm{and}\qquad \mathfrak{S}_i(s,t')=(x,0,z-2^{-2a}).
\end{equation}
Due to~\eqref{eq:dz/dt} we have $e^{-2\rho^{-1}}2^{-2a}\le t-t'\le e^{2\rho^{-1}}2^{-2a}$. Hence,
\begin{equation}\label{eq:t' close to t}
|t'-(t-2^{-2a})|\le (e^{2\rho^{-1}}-1)2^{-2a}\le 3\rho^{-1}2^{-2a}.
\end{equation}
Now,
\begin{align*}\label{eq:change to stt'}
|\beta_i(x,0,z)&-\beta_i(x,0,z-2^{-2a})|\\
&\!\!\!\!\!\!\!\!\!\!\!\!\!\!\!\!\stackrel{\eqref{eq:slab flow}\wedge \eqref{eq:s,t,t'}}{=}|\phi_i(s,0,t)-\phi_i(s,0,t-2^{-2a})-\phi_i(s,0,t')+\phi_i(s,0,t-2^{-2a})|\\
&\ge |\phi_i(s,0,t)-\phi_i(s,0,t-2^{-2a})|-|\phi_i(s,0,t')-\phi_i(s,0,t-2^{-2a})|\\
&\!\!\!\stackrel{\eqref{eq:rescaled phi}}{\ge} |\phi_i(s,0,t)-\phi_i(s,0,t-2^{-2a})|-\rho^i|t'-(t-2^{-2a})|\\
&\!\!\!\stackrel{\eqref{eq:t' close to t}}{\ge} |\phi_i(s,0,t)-\phi_i(s,0,t-2^{-2a})|-3\rho^{i-1}2^{-2a}.
\end{align*}
In other words, we established the following point-wise estimate for the vertical difference quotients that occur in the definition~\eqref{eq:def par vert perimeter} of (parameterized) vertical perimeter.
$$
\frac{|\beta_i(x,0,z)-\beta_i(x,0,z-2^{-2a})|}{2^{-a}} \ge \frac{|\phi_i(s,0,t)-\phi_i(s,0,t-2^{-2a})|}{2^{-a}}-\frac{3\rho^{i-1}}{2^{a}}.
$$
By integrating this inequality over $U$ we get
\begin{align}
\begin{split}
\!\!\!\vpP{U,\beta_i}(a)&\stackrel{\eqref{eq:def par vert perimeter}}{\ge} \int_0^1\int_0^1 \frac{|\phi_i(s(x,z),0,t(x,z))-\phi_i(s(x,z),0,t(x,z)-2^{-2a})|}{2^{-a}}\ud x\ud z-\frac{3\rho^{i-1}}{2^{a}}\\
&\stackrel{\eqref{eq:partials st xy}}{=} \int_{\mathfrak{S}_i^{-1}(U)} \frac{|\phi_i(s,0,t)-\phi_i(s,0,t-2^{-2a})|}{2^{-a}}\left|\frac{\partial z}{\partial t}(s,t)\right|\ud s\ud t-\frac{3\rho^{i-1}}{2^{a}}\\
&\!\!\!\!\!\!\!\!\!\!\stackrel{\eqref{eq:dz/dt}\wedge \eqref{eqLsquare to square}}{\ge} \frac12 \int_0^1\int_0^1 \frac{|\phi_i(s,0,t)-\phi_i(s,0,t-2^{-2a})|}{2^{-a}}\ud s\ud t-\frac{3\rho^{i-1}}{2^{a}}.
\end{split}
\end{align}
It therefore remains to note the following identity.
\begin{align}
  \int_0^1\int_0^1&\frac{|\phi_i(s,0,t)-\phi_i(s,0,t-2^{-2a})|}{2^{-a}}\ud s\ud t\nonumber \\
  \label{eq:break on skew lattice} & =\int_{0}^{\rho^i} \int_{0}^{\alpha^2 \rho^{2i}}\frac{\alpha^{-4}\rho^{-4i}|\phi(\sigma,0,\tau)-\phi(\sigma,0,\tau-\alpha^2\rho^{2i}2^{-2a})|}{2^{-a}}\ud \sigma\ud \tau\\
& \label{eq:periodicity}=\frac{1}{\alpha}\int_U \frac{|\phi(v)-\phi(vZ^{-\alpha^2\rho^{2i}2^{-2a}})|}{\alpha \rho^i 2^{-a}}\ud v\\
& \label{eq:use def of vert per}= \frac{1}{\alpha}\vpP{U,\phi}\big(a-\log_2(\alpha\rho^i)\big),
\end{align}
where~\eqref{eq:break on skew lattice} uses the definition~\eqref{eq:rescaled phi} and the change of variables $(s,t)=(\rho^{-i}\sigma,\alpha^{-2}\rho^{-2i}\tau)$, \eqref{eq:periodicity} holds by the periodicity of $\phi$, and~\eqref{eq:use def of vert per} is a restatement of the definition~\eqref{eq:def par vert perimeter}.
\end{proof}

\section{Pseudoquads and foliated patchworks}\label{sec:pseudoquads}
Let $\Gamma$ be the intrinsic Lipschitz graph of $f\from V_0\to \R$.  A \emph{pseudoquad} $Q$ is a region of $V_0$ bounded by two vertical lines and two characteristic curves of $\Gamma$, i.e., a region of the form
$$Q=\big\{(x,0,z)\in V_0\mid x\in I\ \mathrm{and}\  g_1(x)\le z\le g_2(x)\big\},$$
where $I=[a,b]\subset \R$ is a closed, bounded interval and $g_1,g_2\from \R\to \R$ are functions whose graphs are characteristic.  We say that $I$ is the \emph{base} of $Q$ and we call $g_1$ and $g_2$ the \emph{lower} and \emph{upper bounds} of $Q$, respectively. The width of the pseudoquad $Q$  is just the length $\ell(I)=b-a$ of its base $I=[a,b]$. But, the height of $Q$ is not always well-behaved, since characteristic curves can join and split.  We therefore introduce \emph{rectilinear pseudoquads}, which approximate projections of rectangles in vertical planes.  If $\Gamma$ is a vertical plane, its characteristic curves are a family of parallel parabolas; conversely, any pseudoquad bounded by two parallel parabolas is the projection of a rectangle in $\H$ (a loop composed of two parallel horizontal lines and two vertical lines) to $V_0$.  Thus, if
$$R=\big\{(x,0,z)\in V_0\mid x\in I\ \mathrm{and}\  h_1(x)\le z\le h_2(x)\big\}$$
where $h_1,h_2\from \R\to \R$ are quadratic functions that differ by a constant, then we call $R$ a \emph{parabolic rectangle} with \emph{width} $$\delta_x(R)\eqdef \ell(I)$$ and \emph{height} $$\delta_z(R)\eqdef h_2-h_1.$$  For $r>0$ and an interval $I$, let $rI$ be the scaling of $I$ around its center by a factor of $r$, i.e.,
$$rI\eqdef \left[\frac{a+b}{2}-\frac{r \ell(I)}{2},\frac{a+b}{2}+\frac{r \ell(I)}{2}\right].$$
For $\rho>0$, let
\begin{align}\label{eq:def rho R}
\begin{split}
  \rho R
  &\eqdef \big\{(x,0,z)\in V_0\mid x\in \rho I\ \mathrm{and}\  z\in \rho^2 [h_1(x),h_2(x)]\big\}\\
  &=\left\{(x,0,z)\in V_0\mid x\in \rho I\  \mathrm{and}\  \left|z-\frac{h_1(x)+h_2(x)}{2}\right| \le \frac{\rho^2 \delta_z(R)}{2}\right\}.
  \end{split}
\end{align}

For $0<\mu\le \frac{1}{32}$, a \emph{$\mu$--rectilinear pseudoquad} is a pair $(Q,R)$, where $Q$ is a pseudoquad and $R$ is a parabolic rectangle with the same base $I$ as $Q$ such that, if $g_1$ and $g_2$ (respectively $h_1$ and $h_2$) are the lower and upper bounds of $Q$ (respectively $R$), then
\begin{equation}\label{eq:def rectilinear 1 and 2}
  \max\big \{\|g_1-h_1\|_{L_\infty(4I)},\|g_2-h_2\|_{L_\infty(4I)}\big\}\le \mu \delta_z(R).
\end{equation}
 We will frequently refer to a $\mu$--rectilinear pseudoquad $(Q,R)$ as simply $Q$, but we define its width and height to be the width and height of the associated parabolic rectangle, i.e., $\delta_x(Q)=\delta_x(R)$ and $\delta_z(Q)=\delta_z(R)$.  Likewise, for $\rho\ge 1$, we define $\rho Q=\rho R$.  Note that $Q$ need not be contained in $1 Q=R$, but the following lemma holds.
\begin{lemma}\label{lem:quad translates}
  Let $Q$ be a $\mu$--rectilinear pseudoquad.  Then $Q\subset 2Q$. In fact, for every $t\in \R$,
  $$ Q Z^{t \delta_z(Q)} \subset \sqrt{2|t|+2}\cdot Q.$$
\end{lemma}
\begin{proof}
  Let $R$, $g_1$, $g_2$, $h_1$, $h_2$ be as above.  Let $m_g=\frac{g_1+g_2}{2}$ and $m_h=\frac{h_1+h_2}{2}$.  Fix $(x,0,z)\in Q$, so that $g_1(x)\le z\le g_2(x)$.  For $i\in \{1,2\}$, we have
  $$|m_h(x)-g_i(x)|\le |m_h(x)-h_i(x)|+|h_i(x)-g_i(x)|\le \frac{\delta_z(Q)}{2}+\mu \delta_z(Q)\le \delta_z(Q),$$
  so
  $$\left|m_h(x)-\left(z+t \delta_z(Q)\right)\right|\le (1 + |t|) \delta_z(Q).$$
  Therefore, $(x,0,z+\delta_z(Q))\in \sqrt{2|t|+2}\cdot Q$.
\end{proof}
Continuing with the above notation, define  the \emph{aspect ratio} of $Q$ to be
\begin{equation}\label{eq:def aspect}
\alpha(Q)\eqdef \frac{\delta_x(Q)}{\sqrt{\delta_z(Q)}}.
\end{equation}
We use a square root here because the distance in the Heisenberg metric between the top and bottom of $Q$ is proportional to $\sqrt{\delta_z(Q)}$; thus this aspect ratio is invariant under the Heisenberg scaling.  Let $|Q|$ be the Lebesgue measure of $Q$ as a subset of $V_0\cong \R^2$.

The following lemma is a direct consequence of Lemma~\ref{lem:curve transforms}.
\begin{lemma}\label{lem:quad transforms}
  Let $a,b\in \R\setminus \{0\}$ and let $g=q\circ \rho_h\circ s_{a,b}\from \H\to \H$  be a composition of a shear map $q$, a left-translation by $h\in \H$, and a stretch map $s_{a,b}$.  Let $\hat{g}\from V_0\to V_0$ be the map induced on $V_0$, i.e., $\hat{g}(x)=\Pi(g(x))$ for all $x\in V_0$. Suppose that $(Q,R)$ is a $\mu$--rectilinear pseudoquad for an intrinsic graph $\Gamma$.  Then $(\hat{Q},\hat{R})=(\hat{g}(Q),\hat{g}(R))$ is a $\mu$--rectilinear pseudoquad for the intrinsic graph $\hat{g}(\Gamma)$, with the following parameters.
  \begin{align*}
    \delta_x(\hat{Q})=|a|\delta_x(Q),\qquad \delta_z(\hat{Q})=|ab| \delta_z(Q),\qquad
    |\hat{Q}|=|a^2b|\cdot |Q|,\qquad  \alpha(\hat{Q})  = \sqrt{\frac{|a|}{|b|}}\cdot  \alpha(Q).
  \end{align*}
\end{lemma}

\begin{remark}\label{rem:normalizing rectilinear}
  For any $\mu$--rectilinear pseudoquad $(Q,R)$, there is a transformation of $\H$ that sends $R$ to a square in $V_0$ and $Q$ to an approximation of the square. That is, if $a,b,c,d,x_0,w\in \R$ are such that
  $$R=\{(x,0,z)\in V_0 \mid |x-x_0|\le w \wedge |ax^2+bx+c -z| \le d\},$$
  $$h(v)=s_{w^{-1},wd^{-1}}\left(X^{-x_0} Y^{b}Z^{-c}  \tilde{A}_{2a}(v)\right),$$
  and $\hat{h} =\Pi \circ h$,
  then, by the remarks after Lemma~\ref{lem:curve transforms}, $\hat{h}(R)=[-1,1]\times \{0\}\times [-1,1].$

  By Lemma~\ref{lem:quad transforms}, $(\hat{h}(Q),\hat{h}(R))$ is $\mu$--rectilinear, so if $\hat{g}_1$ and $\hat{g}_2$ are the lower and upper bounds of $\hat{h}(Q)$, then $|\hat{g}_1(t)+1|<2\mu$ and $|\hat{g}_2(t)-1|<2\mu$ for all $t\in [-4,4].$
\end{remark}

We will prove Theorem~\ref{thm:ilg admits fcd intro} by constructing a collection of nested partitions of $V_0$ into pseudoquads.  We will describe these partitions by associating a rectilinear pseudoquad to each vertex of a rooted tree.  Let $(T,v_0)$ be a rooted tree with vertex set $\cV(T)$.  For $v\in \cV(T)$, we let $\cC(v)=\cC^1(v)$ denote the set of \emph{children} of $v$ and inductively for $n\ge 2$ let
$$\cC^n(v)=\bigcup_{w\in \cC^{n-1}(v)} \cC(w)$$
be the set of $n$'th generation descendants of $v$. Let $\cD(v)=\bigcup_{n=0}^\infty \cC^n(v)$ where $\cC^0(v)=\{v\}$.  For  $v\in \cV(T)\setminus \{v_0\}$, there is a unique \emph{parent} vertex $w$ such that $v\in \cC(w)$, and we denote this vertex by $\cP(v)$.  If $w\in \cD(v)$, we say that $w$ is a \emph{descendant} of $v$ or that $v$ is an \emph{ancestor} of $w$ and write $w\le v$.  This is a partial order with maximal element $v_0$.

\begin{defn}[rectilinear foliated patchwork]\label{def:rectilinear foliated patchwork}
  If $Q$ is a $\mu$--rectilinear pseudoquad, a \emph{\(\mu\)-rectilinear foliated patchwork} for $Q$ is a complete rooted binary tree \((\Delta,v_0)\) (i.e., every vertex has exactly two children) such that every vertex \(v\in \cV(\Delta)\) is associated to a $\mu$--rectilinear pseudoquad \((Q_v,R_v)\) with $Q_{v_0}=Q$.  Each vertex $v\in \cV(\Delta)$ is either \emph{vertically cut} or \emph{horizontally cut} in the following sense.

  Let $w$ and $w'$ be the children of $v$, let $I=[a,b]$ be the base of $Q_v$, and let $g_1$ and $g_2$ (respectively $h_1$ and $h_2$) be the lower and upper bounds of $Q_v$ (respectively $R_v$).
  \begin{enumerate}
  \item
    If $v$ is vertically cut, then $Q_w$ and $Q_{w'}$ are the left and right halves of $Q_v$, separated by the vertical line $x=\frac{a+b}{2}$.  That is,
    $$Q_{w}=\biggl\{(x,0,z)\in V_0\mid a\le x\le \frac{a+b}{2}\ \mathrm{and}\  g_1(x)\le z\le g_2(x)\biggr\},$$
    and
    $$Q_{w'}=\biggl\{(x,0,z)\in V_0\mid \frac{a+b}{2}\le x\le b \ \mathrm{and}\  g_1(x)\le z\le g_2(x)\biggr\}.$$
    Similarly,
    $$
    R_{w}=\left(\left[a, \frac{a+b}{2}\right] \times\{0\}\times  \R\right) \cap R_v\quad\mathrm{and}\quad  R_{w'}=\left(\left[\frac{a+b}{2}, b\right] \times\{0\}\times  \R\right) \cap R_v.
      $$ We therefore have $\delta_x(Q_w)=\delta_x(Q_{w'})=\frac{\delta_x(Q_v)}{2}$ and $\delta_z(Q_w)=\delta_z(Q_{w'})=\delta_z(Q_v)$.
  \item
    If $v$ is horizontally cut, then $Q_w$ and $Q_{w'}$ are the top and bottom halves of $Q_v$, separated by a characteristic curve.  That is, there is a function $c\from \R\to \R$ whose graph is characteristic, a quadratic function $k\from \R\to \R$, and $d\in (0,\infty)$ such that
    \begin{align*}
    Q_{w}&=\big\{(x,0,z)\in V_0\mid a\le x\le b\ \mathrm{and}\  g_1(x)\le z\le c(x)\big\},\\
    Q_{w'}&=\big\{(x,0,z)\in V_0\mid a\le x\le b\ \mathrm{and}\  c(x)\le z\le g_2(x)\big\},\\
    R_{w}&=\big\{(x,0,z)\in V_0\mid a\le x\le b\ \mathrm{and}\  k(x)-d\le z\le k(x)\big\},\\
    R_{w'}&=\big\{(x,0,z)\in V_0\mid a\le x\le b\ \mathrm{and}\  k(x)\le z\le k(x)+d\big\}.
    \end{align*}
    Then $\delta_x(Q_w)=\delta_x(Q_{w'})=\delta_x(Q_v)$ and $\delta_z(Q_w)=\delta_z(Q_{w'})=d$.  Furthermore, $Q_w$ and $Q_{w'}$ are assumed to be $\mu$--rectilinear; thus
    \begin{equation}\label{eq:horiz cut edge bounds}
      \max\left\{\|(k-d)-g_1\|_{L_\infty(4 I)}, \|k-c\|_{L_\infty(4 I)},\|(k+d)-g_2\|_{L_\infty(4 I)}\right\}\le \mu d.
    \end{equation}
  \end{enumerate}

  In either case, $Q_v=Q_{w}\cup Q_{w'}$ and $Q_w, Q_{w'}$ have disjoint interiors.  Let \(\cVv(\Delta)\subset \cV(\Delta)\) be the set of vertically cut vertices and let \(\cVh(\Delta)\subset \cV(\Delta)\) be the set of horizontally cut vertices.
\end{defn}

It follows from the above definition that $v\le w$ if and only if $Q_v\subset Q_w$.  Furthermore, if the interior of $Q_v$ intersects $Q_w$, then either $v\le w$ or $w\le v$.

\begin{lemma}\label{lem:soft cut props}
  For every $\e>0$ there exists $0<\mu=\mu(\e)\le \frac{1}{32}$ such that if $Q$ is a $\mu$--rectilinear pseudoquad, then
  \begin{equation}\label{eq:soft cut area}
    (1-\epsilon) \delta_x(Q) \delta_z(Q) \le |Q|\le (1+\epsilon) \delta_x(Q)  \delta_z(Q).
  \end{equation}
  If $Q$ is horizontally or vertically cut as in Definition~\ref{def:rectilinear foliated patchwork} and $Q'$ is a child of $Q$, then
  \begin{equation}\label{eq:soft cut split}
    \left(\frac{1}{2}-\epsilon\right) |Q|\le |Q'| \le \left(\frac{1}{2}+\epsilon\right) |Q|.
  \end{equation}
  If $Q$ is vertically cut, then $\delta_x(Q')=\frac{\delta_x(Q)}{2}$, $\delta_z(Q')=\delta_z(Q)$, and $\alpha(Q')=\frac{\alpha(Q)}{2}$.
  If $Q$ is horizontally cut, then $\delta_x(Q')=\delta_x(Q)$, and
  \begin{equation}\label{eq:soft horiz cut height}
    \left(\frac{1}{2}-2\mu \right)\delta_z(Q)\le  \delta_z(Q') \le \left(\frac{1}{2}+2\mu \right)\delta_z(Q).
  \end{equation}
Finally,
  \begin{equation}\label{eq:soft cut alpha}
    \left(\sqrt{2}-\epsilon\right)\alpha(Q) \le \alpha(Q') \le \left(\sqrt{2}+\epsilon\right)\alpha(Q).
  \end{equation}
  When $\epsilon=\frac{1}{4}$ we can take here $\mu=\frac{1}{32}$.
\end{lemma}
\begin{proof}
  Suppose that $\mu \le \frac{\epsilon}{8}$.  Let $(Q,R)$ be a $\mu$--rectilinear pseudoquad.  Suppose that $g_1$ and $g_2$ (respectively $h_1$ and $h_2$) be the lower and upper bounds of $Q$ (respectively $R$) and let $I$ be the base of $Q$.  Then $|R|=\delta_x(Q)\delta_z(Q)$ and
  \begin{equation}\label{eq:soft cut area mu}
    \bigl||Q|-\delta_x(Q)\delta_z(Q)\bigr|=\bigl||Q|-|R|\bigr| \le \int_I|g_1-h_1|\ud x+ \int_I|g_2-h_2|\ud x \le 2\mu \delta_x(Q) \delta_z(Q),
  \end{equation}
  so \eqref{eq:soft cut area} is satisfied.

  Let $Q'$ be a child of $Q$.  If $Q$ is vertically cut, then the formulas for $\delta_x(Q')$, $\delta_z(Q')$, and $\alpha(Q')$ follow from Definition~\ref{def:rectilinear foliated patchwork}.  As $Q'$ is $\mu$--rectilinear, \eqref{eq:soft cut area mu} implies that
  $$\left||Q'| - \frac{|Q|}{2}\right| \le \left||Q'| - \delta_x(Q')\delta_z(Q') \right| + \frac{1}{2} \big|\delta_x(Q)\delta_z(Q) - |Q|\big| \le 2 \mu \delta_x(Q)\delta_z(Q) \le  4 \mu |Q|,$$
  so $Q$ satisfies \eqref{eq:soft cut split} if $Q$ is vertically cut.

  If $Q$ is horizontally cut, then $\delta_x(Q')=\delta_x(Q)$ by Definition~\ref{def:rectilinear foliated patchwork}.  Let $c, k, d=\delta_z(Q')$ be as in Definition~\ref{def:rectilinear foliated patchwork} and let $t\in I$.  As $\|g_i-h_i\|_{L_\infty(I)} \le \mu \delta_z(Q)$ for $i\in \{1,2\}$ and $\delta_z(Q)=h_2-h_1$,
  $$(1 - 2\mu) \delta_z(Q) \le g_2(t) - g_1(t) \le (1 + 2\mu) \delta_z(Q).$$
  By \eqref{eq:horiz cut edge bounds},
  $$(1-\mu)\cdot 2d\le g_2(t) - g_1(t) \le (1+\mu)\cdot 2d.$$
  Then $d\le \frac{1+2\mu}{2-2\mu} \delta_z(Q) <\delta_z(Q)$, so
  $$|2d - \delta_z(Q)| \le |2d-(g_2(t) - g_1(t))|+ |(g_2(t) - g_1(t))-\delta_z(Q)| \le 4\mu \delta_z(Q),$$
  proving \eqref{eq:soft horiz cut height}.  This directly implies Equation~\eqref{eq:soft cut alpha}, and the horizontally cut case of \eqref{eq:soft cut split} follows from the above calculation and \eqref{eq:soft cut area mu}.
\end{proof}

The following two lemmas will be helpful later.
\begin{lemma}\label{lem:quadratic interval}
  For any quadratic function $q\from \R\to \R$ and any $t\in \R$,
  $$|q(t)|\le (1+t+2t^2)\|q\|_{L_\infty([-1,1])}.$$
\end{lemma}
\begin{proof} One only needs to note that, since $q$ is quadratic, for any $t\in \R$ we have
  \begin{equation*}
  q(t)=q(0)+ t\cdot \frac{q(1)-q(-1)}{2} + t^2\cdot \frac{q(-1)-2q(0)+q(1)}{2}.\tag*{\qedhere}
  \end{equation*}
\end{proof}

\begin{lemma}\label{lem:child hierarchy}
  For every $r\ge 2$ there is  $\mu=\mu(r)>0$ such that if $\Delta$ is a $\mu$--rectilinear foliated patchwork and $v,w\in \cV(\Delta)$ satisfy $w\le v$, then $rQ_w\subset rQ_v$.
\end{lemma}
\begin{proof}
  It suffices to consider the case that $w\in \cC(v)$.  If $v$ is vertically cut, this holds vacuously, so suppose that $v$ is horizontally cut.  Let $g_1$ and $g_2$ (respectively $h_1$ and $h_2$) be the lower and upper bounds of $Q_v$ (respectively $R_v$) and let $I$ be their base.  Denote $m_v=(h_1+h_2)/2$.  Then $r R_v$ is bounded by $m_v \pm r^2\delta_z(Q_v)/2$.

  Let $c, k, d=\delta_z(Q_w)$ be as in Definition~\ref{def:rectilinear foliated patchwork}.  By Lemma~\ref{lem:soft cut props}, we have $d\le \frac{3}{4} \delta_z(Q_v)$.  Then
  $$\|(k-d)-h_1\|_{L_\infty(4I)} \le \|k-d-g_1\|_{L_\infty(4I)} + \|g_1-h_1\|_{L_\infty(4I)} \le  \mu d + \mu \delta_z(Q_v) \le 2\mu \delta_z(Q_v).$$
  Likewise, $\|(k+d)-h_2\|_{L_\infty(4I)}\le 2\mu \delta_z(Q_v)$.  By Lemma~\ref{lem:quadratic interval}, since $k,h_1,h_2$ are quadratic functions, if $\mu$ is at most a sufficiciently small universal constant multiple of $r^{-2}$, then
  $$\max\{\|(k-d)-h_1\|_{L_\infty(r I)}, \|(k+d)-h_2\|_{L_\infty(r I)}\}\le \frac{\delta_z(Q_v)}{16}.$$
  By the triangle inequality,
  $$\|k-m_v\|_{L_\infty(r I)} \le \frac{\delta_z(Q_v)}{16}.$$

  Suppose that $Q_w$ is the lower half of $Q_v$, so that $Q_w$ is bounded by $g_1$ and $c$ and $R_w$ is bounded by $k-d$ and $k$.  Let $m_w=k-\frac{d}{2}$ so that $rQ_w$ is bounded by $m_w \pm \frac{r^2d}{2}$.
  For $x\in r I$,
  $$\left|m_v(x)-m_w(x)\right| \le \frac{\delta_z(Q_v)}{16} + \frac{d}{2} \le \frac{7}{16} \delta_z(Q_v) \le \frac{r^2\delta_z(Q_v)}{2} - \frac{r^2 d}{2},$$
  so
  $$\left[m_w(x)-\frac{r^2d}{2}, m_w(x) + \frac{r^2d}{2}\right]\subset \left[m_v(x)-\frac{r^2\delta_z(Q_v)}{2}, m_v(x)+\frac{r^2\delta_z(Q_v)}{2}\right].$$
  That is, $rQ_w\subset r Q_v$.  The case that $Q_w$ is the upper half of $Q_v$ is treated analogously.
\end{proof}

Let $\Delta$ be a $\mu$--rectilinear foliated patchwork for a $\mu$--rectilinear pseudoquad $Q$. For every subset of vertices $S\subset \cV(\Delta)$, define the \emph{weight} of $S$ to be
\begin{equation}\label{eq:def weight}
W(S)\eqdef \sum_{w\in S} \frac{|Q_w|}{\alpha(Q_w)^{4}}\stackrel{\eqref{eq:def aspect}}{=} \sum_{w\in S} \frac{\delta_z(Q_w)^{2}}{\delta_x(Q_w)^{4}}|Q_w|.
\end{equation}
We will use this to define a weighted Carleson condition which is a variant of the Carleson packing condition that is used in the theory of uniform rectifiability~\cite{DSAnalysis}.

\begin{defn}[weighted Carleson packing condition]\label{def:weighted Carleson} Suppose that $\Delta$ is a $\mu$--rectilinear foliated patchwork for a $\mu$--rectilinear pseudoquad $Q$. We say that \(\Delta\) satisfies a \emph{weighted Carleson packing condition} or that \(\Delta\) is \emph{weighted Carleson} with constant \(C\in (0,\infty)\) if every \(v\in \cV(\Delta)\) satisfies
  \begin{equation}\label{eq:def weighted Carleson}
    W\big(\cD(v)\cap \cVv(\Delta)\big) \le C |Q_v|,
  \end{equation}
  where we recall that $\cD(v)$ are the descendants of $v$ and $\cVv(\Delta)$ are the vertically cut vertices.
\end{defn}

\begin{remark}
  Vertical cuts increase $W$ and horizontal cuts decrease it.  More precisely, suppose that $v, w\in \cV(\Delta)$ and $w$ is a child of $v$.  If $v$ is vertically cut, then by Lemma~\ref{lem:soft cut props} (with $\e=\frac{1}{4}$),
  \begin{equation}\label{eq:vertical cut weight}
    W(\{w\})=\alpha(Q_w)^{-4} |Q_w|=16 \alpha(Q_v)^{-4} |Q_w| \ge 16 \alpha(Q_v)^{-4} \cdot \left(\frac{1}{2}-\epsilon\right) |Q_v| \ge 4 W(\{v\}).
  \end{equation}
  When $\epsilon\to 0^+$, $W(\{w\})$ approaches $8 W(\{v\})$.
  If $v$ is horizontally cut, then by Lemma~\ref{lem:soft cut props} (with $\e=\frac{1}{4}$),
  \begin{equation}\label{eq:horizontal cut weight}
    W(\{w\})=\alpha(Q_w)^{-4} |Q_w| \le (\sqrt{2}-\epsilon)^{-4} \left(\frac{1}{2}+\epsilon\right) W(\{v\}) \le \frac{3}{7} W(\{v\}),
  \end{equation}
  and $W(\{w\})$ approaches $\frac{1}{8} W(\{v\})$ when $\epsilon\to 0^+$.
\end{remark}

The next lemma implies that even though only $\cVv(\Delta)$ appears in~\eqref{eq:def weighted Carleson}, this condition formally implies bounds on $\cVh(\Delta)$ as well.

\begin{lemma}\label{lem:vertex sums}
  Let $\Delta$ be a  \(\frac{1}{32}\)--rectilinear foliated patchwork for  $Q$  with $W(\cVv(\Delta))<\infty$, and let $v_0$ be the root of $\Delta$.  Then
  $$W\big(\cVv(\Delta)\big) \lesssim W\big(\cVh(\Delta)\big)\lesssim W\big(\cVv(\Delta)\big)+\alpha(Q)^{-4}|Q|.$$
\end{lemma}

\begin{proof}
  Let $\cT_{\HH}$ (respectively $\cT_{\VV}$) be the set of connected components of the subgraph of $\Delta$ spanned by $\cVh(\Delta)$ (respectively $\cVv(\Delta)$).  Let $T\in \cT_{\HH}$ and let $M(T)$ be the maximal vertex of $T$.  Each vertex of $T$ is horizontally cut, so by \eqref{eq:horizontal cut weight}, we have $W(\cC(v))\le \frac{6}{7}W(\{v\})$ for all $v\in \cV(T)$.  Therefore, $W(\cV(T))\asymp W(\{M(T)\})$, because
  $$W\big(\cV(T)\big)
  =\sum_{n=0}^\infty W\big( \cC^n(M(T))\cap \cV(T)\big)
  \le \sum_{n=0}^\infty \biggl(\frac{6}{7}\biggr)^n W(\{M(T)\})\lesssim W\big(\{M(T)\}\big).$$
  Hence, if we denote  $S_M=\{M(T)\mid T\in \cT_{\HH}\}$, then
  $W(\cVh(\Delta))\asymp W(S_M).$

  Now, take $T\in \cT_{\VV}$.  By \eqref{eq:vertical cut weight}, we have $W(\{w\})\ge 4W(\{v\})$ for all $v\in \cV(T)$ and $w\in \cC(v)$. As  $W(\cV(T))<\infty$, it follows that $T$ must be finite.  Let
  $m(T)=\{w\in \cV(T)\mid \cC(w)\not \subset \cV(T)\}$
  be the lower boundary of $T$ and let $S_m=\bigcup_{T\in \cT_{\VV}}m(T)$.

  For all $v\in \cV(T)$, let $A(v)=\{w\in \cV(T)\mid w\ge v\}$ be the set of ancestors of $v$ in $T$.  By \eqref{eq:vertical cut weight},
  $$
  W\big(A(v)\big) \le \sum_{n=0}^{|A(v)|-1} 4^{-n} W(\{v\})\le  2 W(\{v\}).
  $$
  Every vertex of $T$ is an ancestor of a leaf, so it follows that
  $$W\big(\cV(T)\big)\le W\bigg(\bigcup_{v\in m(T)} A(v)\bigg) \le\sum_{v\in m(T)}   W\big(A(v)\big)\le  2W\big(m(T)\big)\le 2W\big(\cV(T)\big).$$
Therefore,
  $W(\cVv(\Delta))\asymp W(S_m).$

  If $v\in S_M$ and $v\ne v_0$, then $\cP(v)$ is horizontally cut and has a vertically cut child, so $\cP(v)\in S_m$.  In fact, $\cP(S_M\setminus \{v_0\})=S_m$.  Since $W(\{v\})\approx W(\{\cP(v)\})$ for all $v$ and since $\cP$ is a two-to-one map, it follows that $W(S_M\setminus \{v_0\})\asymp W(S_m).$ Therefore,
  $$W\big(\cVh(\Delta)\big)\asymp W(S_M) 
  \lesssim W(S_m) + W(\{v_0\}) \asymp W(\cVv(\Delta)) +\alpha(Q)^{-4}|Q|,$$
  and
  \begin{equation*}
  W\big(\cVv(\Delta)\big)\asymp W(S_m) \asymp W(S_M\setminus \{v_0\}) \le W(S_M).\tag*{\qedhere}
  \end{equation*}
\end{proof}

Suppose that $\Delta=(Q_v)_{v\in \cV(\Delta)}$ is a $\mu$--rectilinear foliated patchwork for $\Gamma=\Gamma_f$. For $\sigma>0$, a set of \emph{$\sigma$--approximating planes} for $\Delta$ is a collection of vertical planes $(P_v)_{v\in \cVh(\Delta)}$ such that for every $v\in \cVh(\Delta)$, if $f_v\from V_0\to \R$ is the affine function such that $\Gamma_{f_v}=P_v$, then
\begin{equation}\label{eq:def lambda approximating}
 \frac{\|f_v-f\|_{L_1(10 Q_v)}}{ |Q_v|}
  \le \sigma\frac{\delta_z(Q_v)}{\delta_x(Q_v)}.
\end{equation}
The following lemma verifies that the choice of right-hand side in~\eqref{eq:def lambda approximating} produces a condition that is invariant under stretch automorphisms and shear automorphisms.
\begin{lemma}\label{lem:foliatedPatchworksInvariant}
  Let $\Delta=(Q_v)_{v\in \cV(\Delta)}$ be a $\mu$--rectilinear foliated patchwork for an intrinsic Lipschitz graph $\Gamma=\Gamma_f$ with a set $(P_v)_{v\in \cVh(\Delta)}$ of $\sigma$--approximating planes and let $r\from \H\to \H$ be a left translation, a stretch automorphism, or a shear map.  Let $\hat{r}=\Pi\circ r\from V_0\to V_0$ be the map induced on $V_0$.  Then $\Delta'=((\hat{r}(Q_v),\hat{r}(R_v)))_{v\in \cV(\Delta)}$ is a $\mu$--rectilinear foliated patchwork for $r(\Gamma)$ and $(r(P_v))_{v\in \cVh(\Delta)}$ is a set of $\sigma$--approximating planes for $\Delta'$.
\end{lemma}
\begin{proof}
  By Lemma~\ref{lem:automs preserve intrinsic lipschitz} and Lemma~\ref{lem:quad transforms}, $r(\Gamma)$ is an intrinsic Lipschitz graph and the elements of $\Delta'$ are $\mu$--rectlinear pseudoquads for $r(\Gamma)$.  It is straightforward to check that Definition~\ref{def:rectilinear foliated patchwork} holds for $\Delta'$. Let $v\in \cVh(\Delta)$ and let $f_v\from V_0\to \R$ be the affine function such that $P_v=\Gamma_{f_v}$.  By Lemma~\ref{lem:curve transforms}, there are functions $\hat{f}$ and $\hat{f}_v$ such that $r(\Gamma)=\Gamma_{\hat{f}}$ and $r(P_v)=\Gamma_{\hat{f}_v}$.

  If $r$ is a left translation or a shear map and $w\in 10Q_v$, then $\hat{r}(w)\in 10\hat{r}(Q_v)$ and   $$\big|\hat{f}\big(\hat{r}(w)\big)-\hat{f}_v\big(\hat{r}(w)\big)\big|=|f(w)-f_v(w)|.$$
  In this case, $\delta_x(Q_v)=\delta_x(\hat{r}(Q_v))$ and $\delta_z(Q_v)=\delta_z(\hat{r}(Q_v))$, so if $P_v$ is a $\sigma$--approximating plane for $Q_v$, then $r(P_v)$ is a $\sigma$--approximating plane for $\hat{r}(Q_v)$.

  If $r=s_{a,b}$ for some $a,b\in \R\setminus \{0\}$, then $\hat{r}=r|_{V_0}$,  $\hat{r}(10Q_v)=10\hat{r}(Q_v)$, and for any $w\in 10Q_v$,
  $$\big|\hat{f}\big(\hat{r}(w)\big)-\hat{f}_v\big(\hat{r}(w)\big)\big|=|b|\cdot |f(w)-f_v(w)|.$$
  In this case, $\delta_x(\hat{r}(Q_v))=|a|\delta_x(Q_v)$ and $\delta_z(\hat{r}(Q_v))=|ab|\delta_z(Q_v)$, so by \eqref{eq:def lambda approximating},
 \begin{equation*}
 \frac{\|\hat{f}_v-\hat{f}\|_{L_1(10 \hat{r}(Q_v))}}{|\hat{r}(Q_v)|}=\frac{|a^2 b^2|\|f_v-f\|_{L_1(10 Q_v)}}{|a^2b|\cdot |Q_v|} \le |b| \sigma \frac{\delta_z(Q_v)}{\delta_x(Q_v)} =\sigma \frac{\delta_z(\hat{r}(Q_v))}{\delta_x(\hat{r}(Q_v))}.\tag*{\qedhere}
 \end{equation*}
\end{proof}

\section{Foliated corona decompositions}\label{sec:decompose here}

An intrinsic graph that admits rectilinear foliated patchworks that satisfy a weighted Carleson condition and have approximating planes is said to have a foliated corona decomposition.



\begin{defn} \label{def:foliated corona}
  Fix $0<\mu_0\le \frac{1}{32}$ and $D\from \R^+\times \R^+\to \R^+$.  We say that an intrinsic Lipschitz graph $\Gamma$ has a \emph{$(D,\mu_0)$--foliated corona decomposition} if for every $0<\mu\le \mu_0$, every $\sigma>0$ and every $\mu$--rectilinear pseudoquad $Q\subset V_0$, there is a $\mu$--rectilinear foliated patchwork $\Delta$ for $Q$ such that $\Delta$ is $D(\mu,\sigma)$--weighted-Carleson and has a set of $\sigma$--approximating planes.
\end{defn}

The following theorem is a more precise formulation of Theorem~\ref{thm:ilg admits fcd intro}.

\begin{thm}\label{thm:corona graph formulated} For every $0<\lambda<1$  there is a function $D_\lambda\from \R^+\times \R^+\to \R^+$ such that for any $0<\mu_0\le \frac{1}{32}$, any intrinsic $\lambda$--Lipschitz graph has a $(D_\lambda,\mu_0)$--foliated corona decomposition.
\end{thm}

Definition~\ref{def:foliated corona} requires the root of the foliated patchwork to be $\mu$--rectilinear; the next lemma shows that intrinsic Lipschitz graphs have many $\mu$--rectilinear pseudoquads.
\begin{lemma}\label{lem:small alpha rectilinear}
  Let $\mu_0>0$, let $0<\lambda<1$, and let $\Gamma=\Gamma_f$ be an intrinsic $\lambda$--Lipschitz graph.  There is an $\alpha_0>0$ with the following property.  Let $Q$ be a pseudoquad for $\Gamma$, let $v$ be a point in the lower boundary of $Q$ and suppose that $vZ^s$ is in the upper boundary.  Let $r=\delta_x(Q)$.  If $\frac{r}{\sqrt{s}} \le \alpha_0$, then there is a parabolic rectangle $R$ such that $(Q,R)$ is $\mu_0$--rectilinear.
\end{lemma}
\begin{proof}
Denote
  $$L\eqdef \frac{\lambda}{\sqrt{1-\lambda^2}}\qquad\mathrm{and}\qquad \alpha_0\eqdef \min\left\{ \sqrt{\frac{\mu_0}{16L}}, \frac{\mu_0 (1-\lambda)}{24}\right\}.$$
Let $g_1,g_2\from \R\to \R$ be the lower and upper bounds of $Q$ and let $I$ be its base.  After a translation, we may suppose that $v=\mathbf{0}$ and $f(v)=0$.  Then $I\subset [-r,r]$, $g_1(0)=0$, $g_2(0)=s$, and $g_1'(0)=-f(\mathbf{0})=0$.  Let $R=I\times [0,s]$; we claim that $(Q,R)$ is a $\mu_0$--rectilinear pseudoquad.

  It suffices to show that for all $t\in [-4r,4r]$ and $i\in \{1,2\}$, we have $|g_i(t)-g_i(0)|\le \mu_0 s$.  By Lemma~\ref{lem:char curve z bounds}, for all $t\in [-4r,4r]$, we have
  $$\big|g_i(t) - \big(g_i(0) + t g_i'(0)\big)\big|\le  8 r^2 L \le \frac{\mu_0 s}{2}.$$
  In particular, $|g_1(t)|\le \mu_0 s$. Lemma~\ref{lem:ilg line distance} implies that
  $$|g_2'(0)|=|f(Z^s)-f(\mathbf{0})|\le \frac{3}{1-\lambda} d(\mathbf{0}, Z^s) = \frac{3\sqrt{s}}{1-\lambda} \le \frac{\mu_0}{8}\cdot \frac{s}{r},$$
  so if $|t|\le 4r$, then
  \begin{equation*}\label{eq:small alpha small boundary variation}
    |g_2(t) - g_2(0)|\le \frac{\mu_0}{8}\cdot \frac{s}{r}\cdot  4r + \frac{\mu_0s }{2} \le \mu_0 s.\tag*{\qedhere}
  \end{equation*}
\end{proof}

\begin{cor}\label{cor:small alpha}
  Continuing with the setting and notation of Lemma~\ref{lem:small alpha rectilinear}, any $\frac{1}{32}$--rectilinear pseudoquad $Q$ such that $\alpha(Q)\le \frac{\alpha_0}{2}$ is $\mu_0$--rectilinear.
\end{cor}
\begin{proof}
  Let $v$ be in the lower boundary of $Q$.  Then there is an $s\ge (1-\frac{1}{16}) \delta_z(Q)$ such that $vZ^s$ is in the upper boundary.  If $\alpha(Q)\le \frac{\alpha_0}{2}$, then
  $$\delta_x(Q) \le \frac{\alpha_0}{2} \sqrt{\delta_z(Q)} \le \alpha_0\sqrt{s},$$
  so Lemma~\ref{lem:small alpha rectilinear} implies that $Q$ is $\mu_0$--rectilinear.
\end{proof}

The following lemma shows that the choice of $\mu_0$ is not important; we can increase $\mu_0$ at the cost of an increase in $D$.
\begin{lemma}\label{lem:increas mu0}
 For any $\lambda>0$ and $0<\mu_0<\mu_0'\le \frac{1}{32}$, and any $D\from \R^+\times \R^+\to \R^+$, there exists $D'\from \R^+\times \R^+\to \R^+$ such that if $\Gamma$ is an intrinsic $\lambda$--Lipschitz graph that has a $(D,\mu_0)$--foliated corona decomposition, then $\Gamma$ also has a $(D',\mu'_0)$--foliated corona decomposition.
\end{lemma}
\begin{proof}
  Fix $0<\mu<\mu_0'$ and $0<\sigma<1$.   Let $\alpha_0>0$ be as in Lemma~\ref{lem:small alpha rectilinear}. Suppose that we are given a $\mu$--rectilinear pseudoquad $Q$.  We wish to construct a rectilinear foliated patchwork for $Q$ with a set of $\sigma$--approximating planes.  If $\alpha(Q)<\frac{\alpha_0}{2}$, then by Corollary~\ref{cor:small alpha}, $Q$ is $\mu_0$--rectilinear, and since $\Gamma$ admits a $(D,\mu_0)$--foliated corona decomposition, the desired patchwork and set of approximating planes for $Q$ exist.

  We thus suppose that $\alpha(Q)\ge \frac{\alpha_0}{2}$ and denote
  $$i_0=\left\lceil \log_2\frac{\alpha(Q)}{\alpha_0} \right\rceil+1.$$
  We will construct a rectilinear foliated patchwork for $Q$ by first cutting $Q$ vertically $i_0$ times into pseudoquads $P_1,\dots, P_{2^{i_0}}$ of width $2^{-i_0}\delta_x(Q_0)$,  height $\delta_z(Q)$, and aspect ratio
  $$\alpha(P_i)=\frac{\delta_x(P_i)}{\sqrt{\delta_z(P_i)}}=2^{-i_0}\alpha(Q) < \frac{\alpha_0}{2}.$$
  By Corollary~\ref{cor:small alpha}, each $P_i$ is $\mu_0$--rectilinear and thus admits a $D(\mu,\sigma)$--weighted Carleson rectilinear foliated patchwork and a set of $\sigma$--approximating planes.  Combining these patchworks, we obtain a rectilinear foliated patchwork $\Delta'$ for $Q_0$.  Let $v_0$ be its root vertex.  Note that for any $0\le m\le i_0$ and any $w\in \cC^m(v_0)$, we have $\alpha(Q_w)=2^{-m}\alpha(Q_0)$.

  It remains to check that $\Delta'$ is weighted Carleson.  Let $p_1,\dots, p_{2^{i_0}}\in \cV(\Delta')$ be the vertices such that $P_j=Q_{p_j}$.  If $v\in \cV(\Delta')$ and $v\le p_j$ for some $j$, then $Q_v$ satisfies the weighted Carleson condition \eqref{eq:def weighted Carleson}  with constant at most $D(\mu,\sigma)$.

  Otherwise, $v$ is an ancestor of some $p_j$ and $Q_v=P_a\cup \dots\cup P_b$ for some $a\le b$.  For each $w\in \cV(\Delta')$, let $A(w)$ be the set of ancestors of $w$.  Every ancestor of $p_j$ except possibly $p_j$ itself is vertically cut, so by \eqref{eq:vertical cut weight}, the weight of $\cP^k(p_j)$ decays exponentially.  Thus
  $$W\big(A(p_j)\big)= \sum_{k=0}^{i_0} W\big(\{\cP^k(p_j)\}\big) \stackrel{\eqref{eq:vertical cut weight}}{\le} \sum_{k=0}^{i_0} 4^{-k} W(\{p_j\}) \le 2 W(\{p_j\}).$$

  For each $w\in \cV(\Delta')$, let $\cDv(w)=\cD(w)\cap \cVv(\Delta')$ be the set of vertically cut descendants of $w$.
 As every element of $\cDv(v)$ is a descendant or an ancestor of some $p_j$ with $a\le j\le b$,
  \begin{multline*}
    W\big(\cDv(v)\big)
    \le \sum_{j=a}^b \left[W\big(\cDv(p_j)\big) + W\big(A(p_j)\big)\right] \le \sum_{j=a}^b \left[ D(\mu,\sigma) |P_j| + 2W(\{p_j\}) \right] \\
    = D(\mu,\sigma) |Q_v| + 2\cdot \alpha(P_a)^{-4} |Q_v|
    \le D(\mu,\sigma) |Q_v| + \alpha_0^{-4} |Q_v|.
  \end{multline*}
  Therefore, $\Delta'$ is $\left(D(\mu,\sigma)+\alpha_0^{-4}\right)$--weighted Carleson.
\end{proof}
\section{Vertical perimeter and foliated corona decompositions}\label{sec:fcd bounds vper}
In this section we will assume Theorem~\ref{thm:corona graph formulated} and prove the following theorem, which bounds the vertical perimeter of half-spaces bounded by intrinsic Lipschitz graphs.
\begin{thm}\label{thm:fcd bounds vper} For any $0<\lambda<1$ and $r>0$, if $\Gamma$ is an intrinsic $\lambda$--Lipschitz graph, then
  $$\big\|\vpfl{B_r(\0)}\big(\Gamma^+\big)\big\|_{L_4(\R)}\lesssim_{\lambda} r^3.$$
\end{thm}
This coincides with the bound~\eqref{eq:use previous corona} needed in Section~\ref{sec:corona intro}.  Combined with the reduction from arbitrary sets to intrinsic Lipschitz graphs described in that section, this completes the proof of Theorem~\ref{thm:XYD}.

\subsection{Vertical perimeter for graphs with foliated corona decompositions}

Theorem~\ref{thm:fcd bounds vper} is a consequence of the following lemma.
\begin{lemma}\label{lem:pvp on pseudoquads} Suppose that $f\from V_0\to \Gamma$ is intrinsic Lipschitz and denote $\Gamma=\Gamma_f$.  Fix $\sigma>0$.  Let $Q\subset V_0$ be a $\frac{1}{32}$--rectilinear pseudoquad.  Let $\Delta$ be a $\frac{1}{32}$--rectilinear foliated patchwork for $Q$ and let $(P_v)_{v\in \cVh(\Delta)}$ be a set of $\sigma$--approximating planes.  Denoting $t_0=-\log_4 \delta_z(Q)$, we have
  \begin{equation}\label{eq:pvp on pseudoquads}
    \big\|\vpP{Q,f}\big\|_{L_4([t_0,\infty))}\lesssim \sigma |Q|^{\frac{3}{4}}W\big(\cV(\Delta)\big)^{\frac{1}{4}}.
  \end{equation}
\end{lemma}
Note that while the intrinsic Lipschitz constant of $f$ appears in Theorem~\ref{thm:fcd bounds vper}, it does not appear in~\eqref{eq:pvp on pseudoquads}.  Indeed, this bound is invariant under scalings and stretch automorphisms; if $\Gamma$, $Q$, and $(\Delta,(Q_v)_{v\in \cV(\Delta)})$ are as in Lemma~\ref{lem:pvp on pseudoquads}, $a, b>0$, $s=s_{a,b}$, and $\hat{s}=\Pi\circ s|_{V_0}=s|_{V_0}$, then, by Lemma~\ref{lem:foliatedPatchworksInvariant}, $\hat{s}(Q)$ is a  pseudoquad in $s(\Gamma)=\Gamma_{\hat{f}}$, where $\hat{f}=b f\circ \hat{s}^{-1}$. Furthermore, $\Delta'=(\hat{s}(Q_v))_{v\in \cV(\Delta)}$ is a foliated patchwork for $\hat{s}(Q)$ and $(s(P_v))_{v\in \cVh(\Delta)}$ is a set of $\sigma$--approximating planes.

By Lemma~\ref{lem:quad transforms}, $\alpha(\hat{s}(Q_v))=\sqrt{\frac{a}{b}}\alpha(Q_v)$ and $|\hat{s}(Q_v)|=a^2b|Q_v|$, so $W(\cV(\Delta'))=b^{3}W(\cV(\Delta))$ and
$$|\hat{s}(Q)|^{\frac{3}{4}}W\big(\cV(\Delta')\big)^{\frac{1}{4}}=(a^2b)^{\frac{3}{4}}|Q|^{\frac{3}{4}}b^{\frac{3}{4}} W\big(\cV(\Delta)\big)^{\frac{1}{4}}=(ab)^{\frac{3}{2}}|Q|^{\frac{3}{4}} W\big(\cV(\Delta)\big)^{\frac{1}{4}}.$$
If~\eqref{eq:pvp on pseudoquads} holds for $f$ and $Q$, then, by Lemma~\ref{lem:pvp is invariant},
\begin{multline*}
  \big\|\vpP{\hat{s}(Q),\hat{f}}\big\|_{L_4([t_0-\log_4(ab),\infty))}=(ab)^{\frac{3}{2}}\big\|\vpP{Q,f}\big\|_{L_4([t_0,\infty))} \\
  \lesssim \sigma (ab)^{\frac{3}{2}} |Q|^{\frac{3}{4}} W\big(\cV(\Delta)\big)^{\frac{1}{4}} = \sigma |\hat{s}(Q)|^{\frac{3}{4}}W\big(\cV(\Delta')\big)^{\frac{1}{4}}.
\end{multline*}
That is, \eqref{eq:pvp on pseudoquads} holds for $\hat{f}$ and $s(Q)$.

To prove Lemma~\ref{lem:pvp on pseudoquads}, we will need some lemmas on partitions and coherent sets.  A collection $\{Q_1,\dots, Q_n\}$ of pseudoquads is a \emph{partition} of $Q$ if $Q=\bigcup_{i=1}^n Q_i$ and if the $Q_i$ overlap only along their boundaries.  A \emph{coherent subtree} of $T$ is a connected subtree such that for every $v\in T$, either all children of $v$ are contained in $T$ or none of them are.  A \emph{coherent subset} of $\cV(\Delta)$ is the vertex set of a coherent subtree.

\begin{lemma}\label{lem:coherent partitions}
  Let $\Delta$ be a rectilinear foliated patchwork for $Q$ and suppose that $S\subset \cV(\Delta)$ is coherent.  Let $M=\max S$ be the maximal element of $S$ and let $\min S$ be the set of minimal elements of $S$.  Denote
  $$F_1=F_1(S)\eqdef \big\{p\in Q_M\mid \textup{there are infinitely many $v\in S$ such that $p\in Q_v$}\big\}.$$
  Then
  \begin{equation}\label{eq:coherent partitions}
    Q_M=F_1\bigcup \bigg(\bigcup_{w\in \min S} Q_w\bigg).
  \end{equation}
  The interiors of  $\{Q_w: w\in \min S\}$ are pairwise disjoint and disjoint from $F_1$.  If $S$ is finite, then $\min S$ is a partition of $Q_M$.
\end{lemma}
\begin{proof}
  Let $v\in \min S$ and let $p\in \inter Q_v$.  If $u\in S$ and $p\in Q_u$, then either $u<v$ or $v\le u$.  The first is impossible by the minimality of $v$, so $v\le u$.  It follows that there are only finitely many $w\in S$ such that $p\in Q_w$ and no such $w$ is minimal except $v$.  That is, $\inter Q_v$ is disjoint from $F_1$ and if $u\in \min S$ and $u\ne v$, then $\inter Q_v$ is disjoint from $\inter Q_u$.

  If $p\in Q_M\setminus F_1$, then the set $\{v\in S\mid p\in Q_v\}$ is finite and thus has a minimal element $v_0$. Let $w$ be a child of $v_0$ such that $p\in Q_w$. The minimality of $v_0$ implies that $w\not \in S$, so $v\in \min(S)$ by the coherence of $S$.  This implies \eqref{eq:coherent partitions}.
\end{proof}

\begin{lemma}\label{lem:hcut partitions}
  Fix $0<\mu\le \frac{1}{32}$ and let $(\Delta,(Q_v)_{v\in \Delta})$ be a $\mu$--rectilinear foliated patchwork for $Q$ with $W(\cVv(\Delta))<\infty$.  For any $0<\sigma\le \delta_z(Q)$, denote
  $S_\sigma=\{v\in \cV(\Delta)\mid \delta_z(Q_v)\ge \sigma\}$
  and let $F_\sigma=\min S_\sigma$.  Then $\{Q_v\}_{v\in F_\sigma}$ is a partition of $Q$ into horizontally cut pseudoquads such that $\sigma \le \delta_z(Q_v)<4\sigma$ for all $v\in F_\sigma$.
\end{lemma}

\begin{proof}
  By Definition~\ref{def:rectilinear foliated patchwork} and Lemma~\ref{lem:soft cut props}, the height of every pseudoquad of $\Delta$ is equal to the height of its sibling and at most the height of its parent.  Therefore, $S_\sigma$ is coherent.  If $v\in S_\sigma$, then
  $$W(\{v\})=\alpha(Q_v)^{-4}|Q_v|\approx \delta_z(Q_v)^3 \delta_x(Q_v)^{-3}\ge \sigma^{3} \delta_x(Q)^{-3},$$
  which is bounded away from $0$, so Lemma~\ref{lem:vertex sums} implies that $S_\sigma$ is finite.  By Lemma~\ref{lem:coherent partitions}, $F_\sigma$ partitions $Q$.

  Suppose $v\in F_\sigma$ and let $w\in \cC(v)$.  By the minimality of $v$, we have $v\in S_\sigma$ and $w\not\in S_\sigma$, so $\delta_z(Q_v)\ge \sigma>\delta_z(Q_w)$.  Since $\delta_z(Q_w)<\delta_z(Q_v)$, $v$ is horizontally cut.  Furthermore, by Lemma~\ref{lem:soft cut props}, $\sigma>\delta_z(Q_w)\ge \frac{1}{4}\delta_z(Q_v)$, so $v$ is a horizontally cut pseudoquad such that $\sigma \le \delta_z(Q_v)<4\sigma$, as desired.
\end{proof}

We will use these partitions to decompose the parametric vertical perimeter of $f$ and prove Lemma~\ref{lem:pvp on pseudoquads}.
\begin{proof}[{Proof of Lemma~\ref{lem:pvp on pseudoquads}}]
  By the remarks after Lemma~\ref{lem:pvp on pseudoquads}, condition~\eqref{eq:pvp on pseudoquads} is invariant under scaling, so we may rescale so that $\delta_z(Q)=1$.  Let $\Delta$ be a $\frac{1}{32}$--rectilinear foliated patchwork for $Q$ and let $(P_v)_{v\in \cVh(\Delta)}$ be a set of $\sigma$--approximating planes.
  Without loss of generality, we suppose that $W(\cV(\Delta))<\infty$. For each $v\in \cVh(\Delta)$, let $f_v\from V_0\to \R$ be the affine function such that $\Gamma_{f_v}=P_v$.  For $i\in \N\cup\{0\}$ let $C_i=F_{2^{-2i-1}} \subset \cVh(\Delta)$ be as in Lemma~\ref{lem:hcut partitions}, so that $\{Q_v\}_{v\in C_i}$ is a partition of $Q$ into horizontally-cut pseudoquads with heights in  $[2^{-2i-1} ,2^{-2i+1})$.  No vertex of $\Delta$ appears in more than one of the $C_i$'s.

  We start by bounding  $\vpP{Q_v,f}(t)$ from above for each $v\in C_i$ for a fixed $i\in \N\cup\{0\}$.  Then we have $2^{-2i} \le 2\delta_z(Q_v)$, so Lemma~\ref{lem:quad translates} implies that $Z^{-2^{-2t}} Q_v\subset 10Q_v$ for any $t\in [i,i+1]$.  Therefore, since $f_v$ is constant on vertical lines,
  \begin{align*}
    \vpP{Q_v,f}(t) &= 2^{t} \int_{Q_v} \big|f(w)-f\big(w Z^{-2^{-2t}}\big)\big|\ud w \\
           &\le 2^t \int_{Q_v} \Big(|f(w)-f_v(w)|+\big|f_v\big(w Z^{-2^{-2t}}\big)-f\big(w Z^{-2^{-2t}}\big)\big|\Big)\ud w \\
           &= 2^t \left(\|f-f_v\|_{L_1(Q_v)}+\|f-f_v\|_{L_1(Z^{-2^{-2t}}Q_v)}\right)\\
           &\stackrel{\mathclap{\eqref{eq:def lambda approximating}}}{\le} \;2^{t+1} |Q_v| \sigma  \frac{\delta_z(Q_v)}{\delta_x(Q_v)}
           \stackrel{\eqref{eq:soft cut area}}{\approx} \sigma \delta_z(Q_v)^{\frac{3}{2}} \approx \sigma  \alpha(Q_v)^{-1} |Q_v|.
  \end{align*}
Since  $\{Q_v\}_{v\in C_i}$ is a partition of $Q$, we have $\vpP{Q,f}(t)=\sum_{v\in C_i}\vpP{Q_v,f}(t)$ for all $t\in \R$.  Thus
  \begin{equation}\label{eq:on i, i+1}
    \big\|\vpP{Q,f}\big\|_{L_4([i,i+1))} \le \sum_{v\in C_i} \big\|\vpP{Q_v,f}\big\|_{L_4([i,i+1])} \lesssim \sum_{v\in C_i} \sigma \alpha(Q_v)^{-1} |Q_v|.
  \end{equation}
Consequently,
  \begin{multline*}
    \big\|\vpP{Q,f}\big\|^4_{L_4([0,\infty))}=\sum_{i=0}^\infty \big\|\vpP{Q,f}\big\|^4_{L_4([i,i+1))}\stackrel{\eqref{eq:on i, i+1}}{\lesssim} \sigma^4\sum_{i=0}^\infty\bigg(\sum_{v\in C_i}  \alpha(Q_v)^{-1} |Q_v|\bigg)^4 \\ \le \sigma^4\sum_{i=0}^\infty \bigg(\sum_{v\in C_i}  |Q_v|\bigg)^3\bigg(\sum_{v\in C_i} \alpha(Q_v)^{-4} |Q_v|\bigg)\stackrel{\eqref{eq:def weight}}{=}\sigma^4 |Q|^3\sum_{i=0}^\infty W(C_i)\le \sigma^4 |Q|^3W\big(\cV(\Delta)\big),
    \end{multline*}
where the third step is an application of H\"older's inequality.
\end{proof}

Finally, we use Lemma~\ref{lem:vpP and vpfl} and Lemma~\ref{lem:pvp is invariant} to prove Theorem~\ref{thm:fcd bounds vper}.

\begin{proof}[{Proof of Theorem~\ref{thm:fcd bounds vper}}]
    After scaling, it suffices to prove the theorem in the case that $r=1$, i.e., that if $\Gamma$ is the intrinsic graph of an intrinsic $\lambda$--Lipschitz function $f\from V_0\to \R$, then
  \begin{equation}\label{eq:goal at r=1}
    \big\|\vpfl{B_1}\big(\Gamma^+\big)\big\|_{L_4(\R)}\lesssim_{\lambda} 1.
    \end{equation}

By the definition \eqref{eq:def vert per},
 $$
\forall t\in \R,\qquad  \vpfl{B_1}\big(\Gamma^+\big)(t)=2^t \big|B_1\cap \big(\Gamma^+\symdiff \Gamma^+Z^{2^{-2t}}\big)\big|\le 2^t |B_1|\lesssim 2^t.
 $$
 Hence,
\begin{equation*}
\forall  a\in \R,\qquad  \big\|\vpfl{B_1}\big(\Gamma^+\big)\big\|_{L_4((-\infty,a])}=\bigg(\int_{-\infty}^a  \vpfl{B_1}\big(\Gamma^+\big)(t)^4\ud t\bigg)^{\frac14}\lesssim \bigg(\int_{-\infty}^a2^{4t}\ud t\bigg)^{\frac14}\lesssim 2^{a},
\end{equation*}
 and therefore  we have the following simple \emph{a priori} bound.
 \begin{equation}\label{eq:trivial a priori vertical}
\forall  a\in \R,\qquad  \big\|\vpfl{B_1}\big(\Gamma^+\big)\big\|_{L_4(\R)}\lesssim 2^a+\big\|\vpfl{B_1}\big(\Gamma^+\big)\big\|_{L_4((a,\infty))}.
 \end{equation}

We will first treat the (trivial) case  $B_5\cap \Gamma=\emptyset$, so that either $B_5\subset \Gamma^+$ or $B_5\subset \Gamma^-$. Without loss of generality, suppose that  $B_5\subset \Gamma^+$. This implies that $B_1\subset \Gamma^+ \cap Z^{2^{-2t}}\Gamma^+$ for any $t\ge 0$, so $\vpfl{B_1}(\Gamma^+)(t)=0$, and therefore in this case~\eqref{eq:goal at r=1} follows from the case $a=0$ of~\eqref{eq:trivial a priori vertical}.

We may thus suppose from now on that $B_{5}\cap \Gamma\neq \emptyset$.  Fix any point $p\in B_{5}\cap \Gamma$.  Then $d(p,\langle Y\rangle)\le 5$ and $p=vY^{f(v)}$ for some $v\in V_0$ with $|f(v)|\le 5$, so by Lemma~\ref{lem:ilg line distance}, we have
  $$|f(\mathbf{0})|\le |f(v)| + |f(v) - f(\mathbf{0})| \le 5 + \frac{2}{1-\lambda} d(p, \langle Y\rangle) \lesssim \frac{1}{1-\lambda}.$$
  Likewise, for any $t\in \R$,
  $$|f(Z^t)|\le |f(\mathbf{0})|+\frac{2}{1-\lambda}d(\mathbf{0},Z^t)\lesssim  1+\frac{\sqrt{|t|}}{1-\lambda}.$$
  For $t\in \R$, let $g_t\from \R\to \R$ be a function such that $g_t(0)=t$ and the graph of $g_t$ is characteristic for $f$. By \eqref{eq:horizontal curve eq}, $g_t'(0)=-f(Z^t)$, so by Lemma~\ref{lem:char curve z bounds} and the estimate above,
\begin{equation}\label{eq:more slowly than t}\max_{x\in [-1,1]} |g_t(x)-t|\le |g_t'(0)| + \frac{\lambda}{2\sqrt{1-\lambda^2}} =
  |f(Z^{t})|+\frac{\lambda}{2\sqrt{1-\lambda^2}}\lesssim_\lambda 1+\sqrt{|t|}.\end{equation}
The right hand side of~\eqref{eq:more slowly than t} grows slower than $|t|$ as $|t|\to \infty$, so there is $t_0=t_0(\lambda) > 1$ such that the pseudoquad $Q$ that is bounded by the lines $x=\pm1$ and $z=g_{\pm t_0}(x)$ is $\frac{1}{32}$--rectilinear and contains the projection $\Pi(B_1)$.

Theorem~\ref{thm:corona graph formulated} applied with the choice of parameters $\mu_0=\frac{1}{32}$ and $\sigma=1$ shows that $Q$ has a foliated patchwork $\Delta$ and a set of $1$--approximating planes that satisfy
$$
  W\big(\cVv(\Delta)\big)\lesssim_{\lambda} |Q|\lesssim_{\lambda} 1.
$$
  By Lemma~\ref{lem:vertex sums}, this implies that
    \begin{equation}\label{eq:asymp lambda}
    W\big(\cV(\Delta)\big)\lesssim W\big(\cVv(\Delta)\big)+\alpha(Q)^{-4}|Q|\lesssim_{\lambda} 1.
    \end{equation}
  By Lemma~\ref{lem:vpP and vpfl} and Lemma~\ref{lem:pvp on pseudoquads}, we conclude as follows.
  \begin{multline*}\big\|\vpfl{B_1}\big(\Gamma^+\big)\big\|_{L_4(\R)}\lesssim \frac{1}{\sqrt{\delta_z(Q)}}+\big\|\vpfl{B_1}\big(\Gamma^+\big)\big\|_{L_4([-\log_4\delta_z(Q) ,\infty))}\\
  \le   \frac{1}{\sqrt{\delta_z(Q)}}+ \big\|\vpP{Q,f}\big\|_{L_4([-\log_4\delta_z(Q) ,\infty))}\lesssim   \frac{1}{\sqrt{\delta_z(Q)}}+ |Q|^{\frac{3}{4}}W\big(\cV(\Delta)\big)^{\frac{1}{4}}\lesssim_{\lambda}1,\end{multline*}
  where the first step is an application of~\eqref{eq:trivial a priori vertical} with $a=-\log_4\delta_z(Q)$, the second step is an application of Lemma~\ref{lem:vpP and vpfl} because $Q\supset \Pi(B_1)$, the third step is an application of Lemma~\ref{lem:pvp on pseudoquads}, and the final step holds due to~\eqref{eq:asymp lambda} and because $|Q|\asymp \delta_z(Q)\asymp_\lambda 1$.
\end{proof}

\section{The subdivision algorithm: constructing a foliated corona decomposition}\label{sec:subdiv algorithm}

In this section, we will formulate an iterative subdivision algorithm (Lemma~\ref{lem:subdivision algorithm} below) and prove that, given certain propositions on the geometry of pseudoquads, this algorithm produces a foliated corona decomposition.  In the following sections, we will prove these geometric propositions.  Together, these arguments establish Theorem~\ref{thm:corona graph formulated}.

Fix $\lambda,\sigma\in (0,1)$. Let $f\from V_0\to \R$, and suppose that $\Gamma=\Gamma_f$ is an intrinsic $\lambda$--Lipschitz graph.  Let $0<\mu\le \frac{1}{32}$.  To show that $\Gamma$ admits a foliated corona decomposition, we must show that for any $\mu$--rectilinear pseudoquad $Q$, there is a $\mu$--rectilinear foliated patchwork $\Delta$ for $Q$ which has a set of $\sigma$--approximating planes and such that $\Delta$ is weighted-Carleson.

In order to describe the subdivision algorithm that produces $\Delta$, we  will introduce the \emph{$R$--extended parametric normalized nonmonotonicity} of $\Gamma$, denoted by $\Omega^P_{\Gamma^+ ,R}$, which is a measure on $V_0$ with density based on how horizontal line segments of length at most $R>0$ intersect $\Gamma$.  If $\Gamma$ is a plane, for instance, then $\Omega^P_{\Gamma^+ ,R}=0$, while $\Omega^P_{\Gamma^+ ,R}$ has positive density when $\Gamma$ is bumpy at scale $R$.

This is in the spirit of the quantitative nonmonotonicity used in \cite{CKN} and \cite{NY18}, but it counts different segments, and, like the parametric vertical perimeter, it is defined in terms of the function $f$. We will give a full definition in Section~\ref{sec:parametric monotonicity} and discuss the relationship between extended nonmonotonicity and quantitative nonmonotonicity in Remarks~\ref{rem:ENM-vs-NM} and \ref{rem:stronger-conclusion}.
In Section~\ref{sec:weighted Carleson}, we will show that there is  $c>0$ depending on the intrinsic Lipschitz constant of $\Gamma$ such that the following kinematic formula (inequality) holds for every measurable subset $U\subset V_0$.
\begin{equation*}\label{eq:a priori bound kinematic Omega}
\sum_{i\in \Z} \Omega^P_{\Gamma^+,2^{-i}}(U)\le c |U|.
\end{equation*}

\begin{defn}\label{def:paramonotone}
  Suppose that $\eta,r,R>0$ and $Q$ is a $\frac{1}{4}$--rectilinear pseudoquad.  We say that $\Gamma$ is \emph{$(\eta, R)$--paramonotone} on $rQ$ if it satisfies the following bound.
  \begin{equation}\label{eq:def paramonotone}
   \frac{\Omega^P_{\Gamma^+,R\delta_x(Q)}(r Q)}{|Q|}\le \frac{\eta}{\alpha(Q)^{4}}.
  \end{equation}
  This condition is invariant under scalings, stretch maps, and shear maps; see the discussion immediately after the proof of Lemma~\ref{lem:Omega scaling} below.
\end{defn}

One of the main results of \cite{CKN} was that for small $\eta>0$, any $\eta$--monotone set is close to a plane in $\H$; this is a ``stability version'' of the characterization of monotone sets in~\cite{CheegerKleinerMetricDiff}.  The following proposition, which we will prove in Sections~\ref{sec:omega control outline}--\ref{sec:l1 and characteristic}, states not only that paramonotone pseudoquads are close to vertical planes in $\H$, but also that their characteristic curves are close to the characteristic curves of their approximating planes.

\begin{prop}\label{prop:Omega control}
  There is a universal constant $r>1$ such that for any $\sigma>0$ and $0<\zeta\le \frac{1}{32}$, there are $\eta, R > 0$ such that if $\Gamma=\Gamma_f$ is the intrinsic Lipschitz graph of $f\from V_0\to \R$, and if $Q$ is a $\frac{1}{32}$--rectilinear pseudoquad for $\Gamma$ such that $\Gamma$ is $(\eta, R)$--paramonotone on $rQ$, then
  \begin{enumerate}
  \item \label{it:omega control plane}
    There is a vertical plane $P\subset \H$ (a $\sigma$--approximating plane) and an affine function $F\from V_0\to \R$ such that $P$ is the intrinsic graph of $F$ and
    \begin{equation}
      \frac{\|F-f\|_{L_1(10 Q)}}{|Q|} \le \sigma \frac{\delta_z(Q)}{\delta_x(Q)}.
    \end{equation}
  \item \label{it:omega control characteristics}
    Let $u\in 4 Q$ and let $g_\Gamma, g_P\from \R\to \R$ be  such that $\{z=g_\Gamma(x)\}$ (respectively $\{z=g_P(x)\}$) is a characteristic curve for $\Gamma$ (respectively $P$) that passes through $u$.  Then $$\|g_P-g_\Gamma\|_{L_\infty(4 I)}\le \zeta \delta_z(Q).$$

  \end{enumerate}
\end{prop}

It is important to observe that the bounds in Proposition~\ref{prop:Omega control} do not depend on the intrinsic Lipschitz constant of $f$.  Indeed, this proposition holds when $\Gamma$ is merely the intrinsic graph of a continuous function. 
This is important because paramonotonicity is invariant under stretch automorphisms; a bound that depended on the intrinsic Lipschitz constant of $\Gamma$ would not be invariant.

Proposition~\ref{prop:Omega control} allows us to construct a $\mu$--rectilinear foliated patchwork and a collection of $\sigma$--approximating planes by recursively subdividing $Q$ according to a greedy algorithm.
\begin{lemma}\label{lem:subdivision algorithm}
  Let $r$ be as in Proposition~\ref{prop:Omega control}.  Fix $0<\mu \le\frac{1}{32}$ and $\sigma >0$.  There are $\eta,R>0$ with the following property.  Let $\Gamma$ be an intrinsic Lipschitz graph and let $Q$ be a $\mu$--rectilinear pseudoquad. There is a $\mu$--rectilinear foliated patchwork $\Delta$ for $Q$, such that for all $v\in \cV(\Delta)$, $Q_v$ is horizontally cut if and only if $\Gamma$ is $(\eta,R)$--paramonotone on $rQ$, and $\Delta$ admits a set of $\sigma$--approximating planes.
\end{lemma}

\begin{proof}
  Let $r, \eta,$ and $R$ be positive constants so that Proposition~\ref{prop:Omega control} is satisfied with $\zeta=\frac{\mu}{4}$.

  We construct $\Delta$ by a greedy algorithm.  Denote the root vertex of $\Delta$ by $v_0$ and let $Q_{v_0}=Q$; by assumption, it is $\mu$--rectilinear.  Suppose by induction that we have already constructed a $\mu$--rectilinear pseudoquad $(Q_v,R_v)$.  Let $v\in \cV(\Delta)$ be a vertex with children $w$ and $w'$.  Let $I=[a,b]$ be the base of $Q_v$ and let $g_1,g_2\from \R \to \R$ be its lower and upper bounds, respectively.

  Suppose that $\Gamma$ is not $(\eta,R)$--paramonotone on $rQ_v$.  The vertical line $x=\frac{a+b}{2}$ cuts $Q_v$ and $R_v$ vertically into two halves.  Let $Q_w$ and $Q_{w'}$ be the halves of $Q_v$ and let $R_w$ and $R_{w'}$ be the halves of $R_v$.  Since $(Q_v,R_v)$ is $\mu$--rectilinear, $(Q_w,R_w)$ and $(Q_{w'},R_{w'})$ are both $\mu$--rectilinear.

  Now suppose that $\Gamma$ is $(\eta,R)$--paramonotone on $rQ_v$.    Proposition~\ref{prop:Omega control} states that there is a $\sigma$--approximating plane $P$ for $Q_v$ such that for every $u\in 4Q_v$, any characteristic curve of $\Gamma$ that passes through $u$ is $\zeta \delta_z(Q)$--close to the characteristic curve of $P$ that passes through $u$.  For $i\in \{1,2\}$, let $u_i=\bigl(\frac{a+b}{2}, g_i(\frac{a+b}{2})\bigr)$, and let $m$ be the midpoint of $u_1$ and $u_2$.

  Let $g_3\from \R\to \R$ be a function whose graph is a characteristic curve for $\Gamma$ that passes through $m$.  Let $Q_w$ and $Q_{w'}$ be the pseudoquads with base $I$ that are bounded by the graphs of $g_1$, $g_3$, and $g_2$.

  The characteristic curves of $P$ that pass through $u_1$, $u_2$, and $m$ are parallel evenly-spaced parabolas; let $h_1,h_2,h_3\from V_0\to \R$ be the corresponding quadratic functions and let $d=h_2-h_3=h_3-h_1$ be the constant distance between them.  Let $R_w$ and $R_{w'}$ be the parabolic rectangles with base $I$ that are bounded by these three parabolas.  By Proposition~\ref{prop:Omega control}, we have $\|g_i-h_i\|_{L_\infty(4I)}\le \zeta \delta_z(Q)$ for $i\in \{1,2,3\}$.  In particular, every $x\in I$ satisfies
  $$|\delta_z(Q)-2d|\le |\delta_z(Q)-(g_2(x)-g_1(x))| + |g_2(x)-g_3(x) - d| + |g_3(x)-g_1(x) - d| \le 3\zeta \delta_z(Q),$$
  so $d\ge \frac{1}{4} \delta_z(Q)$ and $\|g_i-h_i\|_{L_\infty(4I)}\le  4\zeta d=\mu d$ for $i\in \{1,2,3\}$.  That is, $(Q_w,R_w)$ and $(Q_{w'},R_{w'})$ are $\mu$--rectilinear and satisfy Definition~\ref{def:rectilinear foliated patchwork} with $k=h_3$. We construct the desired rectilinear foliated patchwork by repeating this process for every vertex of $\Delta$.
\end{proof}

Pseudoquads that are not paramonotone contribute to the nonmonotonicity of $\Gamma$, so, as in \cite{NY18}, the total number and size of these pseudoquads is bounded by the measure of $\Gamma$.  In Section~\ref{sec:weighted Carleson}, we will use an argument based on the Vitali Covering Lemma to prove that rectilinear foliated patchworks constructed using Lemma~\ref{lem:subdivision algorithm} satisfy a weighted Carleson condition, as stated in the following proposition.
\begin{prop}\label{prop:algorithm is weighted Carleson}
  Let $r>1$ and let $0<\mu\le \frac{1}{32r^2}$.  Let $\eta, R>0$ and let $0<\lambda<1$.  Let $\Gamma$ be an intrinsic $\lambda$--Lipschitz graph, let $\Delta$ be a $\mu$--rectilinear foliated patchwork for $\Gamma$, and suppose that for all $v\in \cV(\Delta)$, the pseudoquad $Q_v$ is horizontally cut if and only if $\Gamma$ is $(\eta,R)$--paramonotone on $rQ_v$.  Let $W\from 2^{\cV(\Delta)}\to [0,\infty]$ be as in~\eqref{eq:def weight}.  Then for any $v\in \cV(\Delta)$,
  \begin{equation}\label{eq:weight bounded by perim}
    W\big(\{w\in \cVv(\Delta)\mid w\le v\}\big)  \lesssim_{\eta, r, R, \lambda} |Q_v|.
  \end{equation}
\end{prop}

With these tools at hand, Theorem~\ref{thm:corona graph formulated} follows directly.
\begin{proof}[{Proof of Theorem~\ref{thm:corona graph formulated} assuming Proposition~\ref{prop:Omega control} and Proposition~\ref{prop:algorithm is weighted Carleson}}]
  Let $r$ be as in Proposition~\ref{prop:Omega control} and write $\mu_0=\frac{1}{32r^2}$.  Fix $0<\mu\le \mu_0$ and $\sigma>0$, and let $\eta, R$ be as in Lemma~\ref{lem:subdivision algorithm}.  Since $\Gamma$ is an intrinsic $\lambda$--Lipschitz graph, Lemma~\ref{lem:subdivision algorithm} produces a $\mu$--rectilinear foliated patchwork $\Delta$ rooted at $Q$ with a set of $\sigma$--approximating planes.  By Proposition~\ref{prop:algorithm is weighted Carleson}, this patchwork is weighted--Carleson with a constant depending on $\eta, r, R,\sigma$, and $\lambda$.  Since $r>1$ is universal and $\eta,R$ depend only on $\mu,\sigma$, we obtain Theorem~\ref{thm:corona graph formulated} by using Lemma~\ref{lem:increas mu0} to increase $\mu_0=\frac{1}{32r^2}$ to $\mu_0=\frac{1}{32}$.
\end{proof}

Observe in passing that since in the above proof the patchwork that established Theorem~\ref{thm:corona graph formulated} was obtained from Proposition~\ref{prop:Omega control}, we actually derived the following more nuanced formulation of Theorem~\ref{thm:corona graph formulated}; it is worthwhile to state it explicitly here because this is how it will be used in forthcoming work of the second named author.

\begin{thm}\label{thm:future-use}
  For every $0<\lambda<1$  there is a function $D_\lambda\from \R^+\times \R^+\to \R^+$, and for every $0<\mu \le\frac{1}{32}$ and $\sigma >0$ there are $\eta=\eta(\mu,\sigma),R=R(\mu,\sigma)>0$ with the following properties.  Suppose that $\Gamma\subset \H$  is an intrinsic $\lambda$--Lipschitz graph over $V_0$ and $Q\subset V_0$ is a $\mu$--rectilinear pseudoquad for $\Gamma$. Then there is a $\mu$--rectilinear foliated patchwork $\Delta$ for $Q$ such that $\Delta$ is $D_\lambda(\mu,\sigma)$--weighted-Carleson and has a set of $\sigma$--approximating planes. Moreover, for all vertices $v\in \cV(\Delta)$, the associated pseudoquad $Q_v$ is horizontally cut if and only if $\Gamma$ is $(\eta,R)$--paramonotone on $rQ$, where $r>1$ is the universal constant in Proposition~\ref{prop:Omega control}.
\end{thm}

\begin{remark}\label{rem:ruled-surfaces}
  While the results in this paper rely only on approximating a intrinsic Lipschitz graph by vertical planes to bound its vertical perimeter, Theorem~\ref{thm:future-use} allows one to glue vertical planes together to approximate an intrinsic Lipschitz graph by ruled surfaces. Indeed, with notation as in Theorem~\ref{thm:future-use}, let $F\subset \cV(\Delta)$ be a finite coherent subset such that every vertex in $F$ is horizontally cut. Let $v_F$ be the maximal element of $F$ and let $m(F)$ be the set of minimal elements of $F$. Then $\{Q_v\mid v\in m(F)\}$ is a partition of $Q_{v_F}$ into a stack of pseudoquads $Q_1,\dots, Q_k$ that are vertically adjacent. The characteristic curves bounding these pseudoquads can be approximated by parabolas, denoted $h_0,\dots, h_k$, and the $\mu$--rectilinearity of $\Delta$ implies that these parabolas do not intersect inside $Q_{v_F}$; see the proof of Lemma~\ref{lem:pseudoquad expansions}. We can then construct a foliation of $Q_{v_F}$ by parabolas by linearly interpolating between the $h_i$'s. Since any parabola is the projection of a horizontal line to $V_0$, this foliation is the set of characteristic curves of a ruled surface $\Sigma\subset \H$. By passing to a limit, one can construct a ruled surface corresponding to any coherent subset of horizontally cut vertices.  This procedure is roughly analogous to the method used in \cite{DavidSemmesSingular} to approximate stopping-time regions in uniformly rectifiable sets in $\R^n$ by Lipschitz graphs. In our setting, we can use linear interpolation instead of using a partition of unity as in \cite{DavidSemmesSingular} or \cite{NY18}.

  By Proposition~\ref{prop:Omega control}, $\Sigma$ approximates $\Gamma$ and the characteristic curves of $\Sigma$ approximate the characteristic curves of $\Sigma$ inside $Q_{v_F}$ (with accuracy depending on the heights of the $Q_i$'s). In fact, if $v\in \cVh(\Delta)$ is a vertex such that every descendant of $v$ is horizontally cut (i.e., $\cD(v)\subset \cVh(\Delta)$), then $\Sigma$ coincides with $\Gamma$ over $Q_v$.
  We omit the details of these approximations because they are not needed in the current work, but complete details will be given in forthcoming work of the second named author where they will be used to analyze intrinsic Lipschitz functions.
\end{remark}

We will prove Proposition~\ref{prop:Omega control} and Proposition~\ref{prop:algorithm is weighted Carleson} in the following  sections.  Specifically, in Section~\ref{sec:parametric monotonicity}, we will define extended nonmonotonicity and extended parametric normalized nonmonotonicity and prove some of their basic properties.  In Section~\ref{sec:weighted Carleson}, we will prove that Proposition~\ref{prop:Omega control} implies Proposition~\ref{prop:algorithm is weighted Carleson}.  Finally, in Sections~\ref{sec:omega control outline}--\ref{sec:l1 and characteristic}, we will prove Proposition~\ref{prop:Omega control}.

\section{Extended nonmonotonicity}\label{sec:parametric monotonicity}

\subsection{Extended nonmonotonicity in $\R$}\label{sec:parametric monotonicity R}

In this section, we define the \emph{extended nonmonotonicity} and \emph{extended parameterized nonmonotonicity} of a set $E\subset \H$.  Like the quantitative  nonmonotonicity that was defined  in \cite{CKN} and the horizontal width that was defined in \cite{FasslerOrponenRigot18}, these measure how horizontal lines intersect $\partial E$.

We first define these quantities on subsets of lines, then define them on subsets of $\H$ by integrating over the space of horizontal lines.  Let $\cL$ be the space of horizontal lines in $\H$.  Let $\cN$ be the Haar measure on $\cL$, normalized so that the measure of the set of lines that intersect the ball of radius $r$ is equal to $r^3$.

Recall that a measurable subset $S\subset \R$ is \emph{monotone}~\cite{CheegerKleinerMetricDiff} if its indicator function is a monotone function (i.e., $S$ is equal to either $\emptyset, \R$, or some ray).  For a measurable set $U\subset \R$, we define the \emph{nonmonotonicity} of $S$ on $U$ by
\[\NM_S(U)\eqdef  \inf\big\{\cH^1\big(U\cap(M\symdiff S)\big) \mid M \text{ is monotone}\big\},\]
where, as usual, $M\symdiff S=(M\setminus S) \cup (S\setminus M)$ is the symmetric difference of $M$ and $S$.

For $S\subset \R$, we say that $S$ has \emph{finite perimeter} if $\partial_{\cH^1} S$ is a finite set, where we recall the notation~\eqref{eq:def partial mu} for  measure theoretical boundary, which in the present setting becomes
$$\partial_{\cH^1} S\eqdef \big\{x\in \R\mid 0<\cH^1\big((x-\epsilon, x+\epsilon)\cap S\big)<2\epsilon \text{ for all }\epsilon>0\big\}.$$

If $S\subset \R$ is a set of finite perimeter, then there is a unique collection of disjoint closed intervals of positive length $\cI(S)=\{I_1(S),I_2(S),\dots\}$ such that $S\symdiff \bigcup \cI(S)$ has measure zero.  For any $R>0$, we define as follows a point measure $\omega_{S,R}$ supported on the boundaries of the intervals in $\cI(S)$ of length at most $R$.
$$\omega_{S,R}\eqdef \sum_{\substack{I\in \cI(S)\\ \cH^1(I)\le R}} \cH^1(I) \cdot (\delta_{\min I}+\delta_{\max I}).$$
Let
$$\widehat{\omega}_{S,R}=\frac{\omega_{S,R}+\omega_{\R\setminus S,R}}{2}.$$

These measures are inspired  by analogous measures $\{\widehat{w}_i\}_{i\in \Z}$ used in \cite{CKN}.  It was shown in~\cite{CKN} that if $\delta>0$ is sufficiently small, then the nonmonotonicity of $S$ at scale $\delta^i$ is bounded in terms of a measure $\widehat{w}_i$ that counts the set of endpoints of intervals in $S$ or $\R\setminus S$ of length between $\delta^i$ and $\delta^{i+1}$.  The main difference between $\widehat{w}_i$ and $\widehat{\omega}_{S,\delta^i}$ is that $\widehat{w}_i$ ignores intervals of length less than $\delta^{i+1}$, but $\widehat{\omega}_{S, \delta^i}$ weights them by their lengths.

For $U\subset \R$, we call $\widehat{\omega}_{S,R}(U)$ the \emph{$R$--extended nonmonotonicity of $S$ on $U$}.  (We will typically use this notation when $R>\diam U$.)  We use the term ``extended'' here because it depends not only on $S\cap U$, but also on the behavior of $S$ outside $U$.  For example, let $U=[a,b]$ and suppose that $S\subset \R$ is a set with locally finite perimeter.  If $\widehat{\omega}_{S,R}(U)=0$ for all $R>0$, then there can be no finite-length interval in $\cI(S)$ or $\cI(\R\setminus S)$ with a boundary point in $U$.  That is, $U\cap \partial_{\cH^1}S$ is empty or, up to a measure-zero set, $S=[c,\infty)$ or $S=(-\infty,c]$ for some $c\in [a,b]$. Similarly, when $S=[a,b]$ and $R > b-a$, if $\widehat{\omega}_{S ,R}(U)$ is much smaller than $b-a$, then either $U\cap \partial_{\cH^1} S$ is almost empty or $U$ is almost monotone on an $R$--neighborhood of $S$.  This follows from the following two lemmas.  The first lemma is based on the bounds in Proposition~4.25 of \cite{CKN} and Lemma~3.4 of \cite{FasslerOrponenRigot18}.
\begin{lemma}\label{lem:omega boundary diam NM}
  Let $a,b\in \R$, let $U \subset [a,b]$, and let $R\ge b-a$.  For any finite-perimeter set $S\subset \R$,
  $$\NM_{S}(U)\le \diam \big( (a,b)\cap \partial_{\cH^1} S\big) \le \widehat{\omega}_{S,R}\big((a,b)\big).$$
\end{lemma}
\begin{proof}
  Let $\delta=\diam \big( (a,b)\cap \partial_{\cH^1} S\big)$.  Consider the following set of closed intervals.
  $$\bJ\eqdef \{I\in \cI(S)\cup \cI(\R\setminus S)\mid I\cap (a,b)\ne \emptyset\}.$$
  This set is finite, so we may label its elements $J_1,\dots, J_n$ in increasing order.  After changing $S$ on a measure-zero subset, the interiors of the $J_i$'s are alternately contained in $S$ and disjoint from $S$.  If $n=1$, then $\NM_{S}(U)=0$ and $\delta=0$, so we suppose that $n\ge 2$.  Then
  $$\delta=\min(J_n)-\max(J_1)=\sum_{i=2}^{n-1} \cH^1(J_i)$$
  and
  $$\widehat{\omega}_{S,R}\big((a,b)\big)\ge 2 \sum_{m=2}^{n-1} \cH^1(J_m)\ge \delta.$$
Regardless of whether $J_1$ and $J_n$ are in or out of $S$, there is a monotone subset $M\subset \R$ such that $\one_M$ agrees with $\one_S$ on $J_1$ and $J_n$.  Then
 \begin{equation*}
 \NM_{S}(U)\le \cH^1\big(U\cap (S\symdiff M)\big) \le \cH^1\big([a,b]\setminus (J_1\cup J_n)\big) = \delta.\tag*{\qedhere}
 \end{equation*}
\end{proof}

A similar reasoning gives the following lower bound.  Recall that $\supp_{\cH^1}$ and $\inter_{\cH^1}$ denote measure-theoretic support and interior, see \eqref{eq:def supp mu}--\eqref{eq:def inter mu}.
\begin{lemma}\label{lem:BWB extended}
 Fix $R,a,b\in \R$ with $a<b$ and  $R\ge b-a$.  Let $S\subset\R$ have locally finite perimeter such that $a,b\in \supp_{\cH^1} (\R\setminus S)$.  For any closed interval $I\subset [a,b]$, either $I \subset \inter_{\cH^1} S$ or
  $$\widehat{\omega}_{S, R}(I) \ge \frac{\cH^1(S\cap I)}{2}.$$
\end{lemma}
\begin{proof}
  Suppose that $I \not \subset \inter_{\cH^1} S$.  Let $I_1,\dots, I_n$ be the intervals in $\cI(S)$ that intersect $I$.  By assumption, each of the intervals $I_1,\dots, I_n$  has at least one endpoint in $I$.  Furthermore, since $a,b\in \supp_{\cH^1} (\R\setminus S)$, we have $I_j\subset [a,b]$ for all $j\in \n$. In particular, $\max_{j\in \n}\ell(I_j)\le R$.  Up to a null set, we have $S\cap I\subset \bigcup_{j=1}^n I_j$, so
  \begin{equation*}
  \widehat{\omega}_{S, R}(I)\ge \sum_{j=1}^n \frac{\cH^1(I_{j})}{2} \ge \frac{\cH^1(S\cap I)}{2}.\tag*{\qedhere}
  \end{equation*}
\end{proof}

These lemmas yield  the following description of sets with small extended nonmonotonicity, which states that points in their measure theoretic boundary  must be either very close to each other, or very far from each other.
\begin{prop}
  Let $S\subset \R$ be a set with locally finite perimeter and fix $c,d\in \R$ with $c<d$.  Let $R\ge d-c$ and suppose that $0<\epsilon<\frac{d-c}{8}$ and $\widehat{\omega}_{S,R}((c,d))<\epsilon$.   Then,
\begin{equation}\label{eq:R-t}
  \forall  t\in  [c+4\epsilon,d-4\epsilon]\cap \partial_{\cH^1}S,\qquad \diam\big( (t-R,t+R)\cap \partial_{\cH^1} S\big)<\epsilon.
  \end{equation}
\end{prop}
\begin{proof}
  Fix $t\in  [c+4\epsilon,d-4\epsilon]\cap \partial_{\cH^1}S$.  We will prove that this implies that
  \begin{equation}\label{eq:included in cd}
    (t-R,t+R)\cap \partial_{\cH^1}S\subset  (c,d).
  \end{equation}
  Equation~\eqref{eq:R-t} is a consequence of the inclusion~\eqref{eq:included in cd}, since by  Lemma~\ref{lem:omega boundary diam NM},
    $$\diam \big((t-R,t+R)\cap \partial_{\cH^1}S\big) \stackrel{\eqref{eq:included in cd}}{\le} \diam \big((c,d)\cap \partial_{\cH^1}S\big)  \le \widehat{\omega}_{S,R}\big((c,d)\big)<\epsilon.$$

  Suppose by way of contradiction that~\eqref{eq:included in cd} fails. So, there is $u\in (t-R,t+R)\cap \partial_{\cH^1}S$ with $u\ge d$ or $u\le c$. We will treat only the case $u\ge d$ since the case $u\le c$ is analogous. Lemma~\ref{lem:BWB extended} applied with $[a,b]=[t,u]$ and $I=[t,d]$, gives
  $$\frac{\cH^1(S\cap [t,d])}{2}\le \widehat{\omega}_{S, R}([t,d])<\epsilon.$$
  If we replace $S$ by $\R\setminus S$, the Lemma~\ref{lem:BWB extended} gives
  $$\frac{\cH^1([t,d]\setminus S)}{2}\le \widehat{\omega}_{S, R}([t,d])<\epsilon.$$
  So $d-t<4\epsilon$, which contradicts the choice of $t$.
\end{proof}

\begin{remark}\label{rem:ENM-vs-NM}
  Despite the name ``extended nonmonotonicity,'' there is no direct comparison between the extended nonmonotonicity of $S$ on $U$ and the nonmonotonicity of $S$ on a neighborhood of $U$. For example, if $R>0$, $0<\epsilon<1$, and $S=[-\epsilon,\epsilon] \cup [R,\infty)$, then $\NM_S(\R)=4\epsilon$, but $\widehat{\omega}_{S,R}$ is the point measure
  $$\widehat{\omega}_{S,R}=\epsilon\delta_{-\epsilon} + \left(\frac{\epsilon}{2}+\frac{R}{2}\right)\delta_{\epsilon} + \frac{R}{2}\delta_{R},$$
  so $\widehat{\omega}_{S,R}([-1,1])$ is large despite $S$ having small nonmonotonicity. Conversely, for any $T\subset \R$ that contains $[-1,1]$, the boundary $\partial_{\cH^1} T$ is disjoint from $(-1,1)$, so $\widehat{\omega}_{T,R}((-1,1))=0$ regardless of the behavior of $T$ on the rest of $\R$.
\end{remark}

\subsection{Extended nonmonotonicity in $\H$}

We have defined $\NM$ and $\widehat{\omega}$ for subsets of $\R$, but the same definitions are valid for subsets of any line $L\in \cL$.  This lets us define the nonmonotonicity of a subset of $\H$ by integrating over horizontal lines.

When $U, E\subset \H$ are measurable sets, we define the nonmonotonicity of $E$ on $U$ by
$$\NM_{E}(U)\eqdef\int_{\cL} \NM_{E\cap L}(U\cap L)\ud \cN(L).$$
(Note that this definition differs from the definition in \cite{CKN}.  Specifically, in \cite{CKN}, this was only defined in the case that $U=B_r(x)$ for some $r\in (0,\infty)$ and $x\in \H$, and was normalized by a factor of $r^{-3}$ to make it scale-invariant.)

\begin{defn}\label{def:extended monotone} Fix $R>0$.
  Let $E\subset \H$ be a set with finite perimeter.  By the kinematic formula (Section~\ref{sec:kinematic}), for almost every $L\in \cL$, the intersection $E\cap L$ is a set with finite perimeter, and we define for $U\subset \H$,
  \begin{equation}\label{eq:first hat}
  \wwidehat{\omega}_{E,R}(U,L)\eqdef \widehat{\omega}_{E\cap L,R}(U\cap L).
  \end{equation}
We then define a measure $\ENM_{E, R}$ on $\H$ by setting
  $$\ENM_{E, R}(U)\eqdef \int_{\cL}\wwidehat{\omega}_{E,R}(U, L)\ud \cN(L).$$
  We call $\ENM_{E, R}(U)$ the \emph{$R$--extended nonmonotonicity} of $E$ on $U$, and for $\nu>0$ we say that $E$ is \emph{$(\nu,R)$--extended monotone} on $U$ if $\ENM_{E,R}(U)\le \nu$.  Like $\widehat{\omega}_{S,R}(\cdot)$, $\ENM_{E,R}(U)$ depends on the behavior of $E$ in an $R$--neighborhood of $U$. If $R\le R'$, then $\ENM_{E, R}\le \ENM_{E, R'}$.
\end{defn}

When we say that a subset $U\subset \H$ is  \emph{convex}, we will always mean that it is convex as a subset of the vector space $\R^3$.  For every $g\in \H$, the map $v\mapsto gv$ is an affine map from $\H$ to itself, so convexity is preserved by left multiplication.

\begin{lemma}\label{lem:compare NM omega convex}
  Let $U\subset \H$ be a measurable bounded set and let $K\subset U$ be convex.   Let $E\subset \H$ be a finite-perimeter set.  Then, for every $R>\diam U$ we have
  $$\NM_{E}(K) \le \ENM_{E,R}(U).$$
\end{lemma}
\begin{proof}
  Let $L\in \cL$ be a horizontal line.  By convexity, the intersection $I=L\cap K$ is an interval and $\ell(I)\le \diam U$.  By Lemma~\ref{lem:omega boundary diam NM},
  $$\NM_{E\cap L}(I)\le \widehat{\omega}_{E\cap L,R}(I) \le \wwidehat{\omega}_{E,R}(U,L).$$
  Integrating both sides of this inequality with respect to $\cN$ yields the desired bound.
\end{proof}


We will also define a parametric version of extended nonmonotonicity that is better adapted to intrinsic Lipschitz graphs.  This is based on a different measure on the space of horizontal lines, denoted $\cN_P$, which we next describe.

Let $W_0=\{x=0\}$ be the $yz$--plane and let $\cL_P$ be the set of horizontal lines that are not parallel to $W_0$.  Each $L\in \cL_P$ intersects $W_0$ in a single point $w(L)$, called the \emph{intercept} of $L$, and has a unique \emph{slope} $m(L)\in \R$ such that $L=w(L)\cdot \langle X+m(L)Y\rangle$.

The map $(m,w)\from \cL_P\to \R\times W_0$ is a bijection, and we define $\cN_P$ to be the pullback of the Lebesgue measure on $\R\times W_0$ under this bijection. This measure is preserved by shear maps and translations. If $a,b>0$ and if $L_{(0,y,z),m}$ is the line with slope $m$ and intercept $(0,y,z)$, then $s_{a,b}(L_{(0,y,z),m})=L_{(0,by,abz),m\frac{b}{a}}$, so for any measurable set $A\subset \cL_P$,
\begin{equation}\label{eq:transform N}
  \cN_P\big(s_{a,b}(A)\big)=b^3 \cN_P(A).
\end{equation}

Let $E \subset \H$.  For any $R>0$, any $U\subset V_0$, and any $L\in \cL_P$, we define
\begin{equation}\label{eq:define hat OmegaP}
  \wwidehat{\omega}^P_{E, R}(U, L) \eqdef \widehat{\omega}_{x(E \cap L),R}\Big(x\big(\Pi^{-1}(U)\cap L\big)\Big).
\end{equation}
This is similar to $\wwidehat{\omega}_{E, R}(\Pi^{-1}(U), L)$ in~\eqref{eq:first hat}, but the projection to the $x$--coordinate that appears in~\eqref{eq:define hat OmegaP} changes the measures and lengths involved by a constant factor.

When $E$ is a finite-perimeter subset of $\H$, we define a measure $\Omega^P_{E,R}$ on $V_0$ by setting for any measurable subset $U\subset V_0$,
\begin{equation}\label{eq:define OmegaP}
  \Omega^P_{E,R}(U)\eqdef \frac{1}{R}\int_{\cL_P} \wwidehat{\omega}^P_{E,R}(U,L)\ud \cN_P(L).
\end{equation}
We call $\Omega^P_{E, R}(U)$ the $R$--\emph{extended parametric normalized nonmonotonicity} of $E$ on $U$.  Note that the definition~\eqref{eq:define OmegaP} includes an $R^{-1}$ factor that does not appear in Definition~\ref{def:extended monotone}; we will see that this normalization allows for  the kinematic formula \eqref{eq:kinematic intro} to hold.

In general, the measure $\Omega^P_{E, R}$ is not necessarily locally finite. Indeed, if $B\subset \H$ is a ball, then the set of lines that pass through $B$ has infinite $\cN_P$--measure.  But when $\Gamma$ is an intrinsic $\lambda$--Lipschitz graph, any line with sufficiently large slope intersects $\Gamma$ exactly once.  If $E=\Gamma^+$ and $L\in \cL_P$ is a line such that $L\cap E$ is nonmonotone, then $L$ has bounded slope; it follows that $\Omega^P_{\Gamma^+,R}(K)$ is finite for any compact $K\subset V_0$. Furthermore, $\Omega^P_{\Gamma^+,R}$ is bounded below by $\ENM_{\Gamma^+,R}$.

\begin{lemma}\label{lem:measure comparison}
  Let $R>0$ and let $x\in \H$.  Suppose that $E$ is a finite-perimeter subset of $\H$ and let $U\subset \H$ be measurable.  Then
  $$\ENM_{E,R}(U) \lesssim R \Omega^P_{E,R}\big(\Pi(U)\big).$$
\end{lemma}
\begin{proof}
  Let $L\in \cL_P$ and let $m=m(L)$ be the slope of $L$, so that the restriction $x|_L$ shrinks lengths by a factor of $\phi(m)=\sqrt{1+m^2}$.  Then
  $$\wwidehat{\omega}^P_{E,R}(\Pi(U),L)=\frac{\wwidehat{\omega}_{E ,\phi(m) \cdot R}(\Pi^{-1}(\Pi(U)),L)}{\phi(m)}\ge \frac{\wwidehat{\omega}_{E,R}(U,L)}{\phi(m)}.$$
 For $w\in W_0$ and $m\in \R$, let $L_{w,m}$ be the line $L_{w,m}=w\cdot \langle X+mY\rangle\in \cL_P$.  Then it follows that
  \begin{align*}
    R \Omega^P_{E,R}\big(\Pi(U)\big)
    =\int_{W_0}\int_\R \wwidehat{\omega}^P_{E,R}(\Pi(U), L_{w,m}) \ud m \ud w
    \ge \int_{W_0}\int_\R \frac{\wwidehat{\omega}_{E,R}(U, L_{w,m})}{\sqrt{1+m^2}} \ud m \ud w.
  \end{align*}

  For $\theta\in \R$, let $R_\theta\from \H\to \H$ be the rotation by angle $\theta$ around the $z$--axis. Since $\cN$ is invariant under translations and rotations, there is $c>0$ such that for any measurable $f\from \cL\to \R$,
  $$\int_{\cL} f(M) \ud \cN(M) = c\int_{W_0} \int_{-\frac{\pi}{2}}^{\frac{\pi}{2}} f\big(R_{\theta}(L_{g,0})\big) \ud \theta \ud g.$$
Any line in $\cL_P$ can be written as $R_\theta(L_{g,0})$ for some $\theta\in \R$ and $g\in W_0$.  Specifically, for $w\in W_0$ and $m\in \R$, let $\theta(m)=\arctan m$ and let $g_m(w)$ be the $W_0$--intercept of $R_{-\theta(m)}(L_{w,m}))$, so that $L_{w,m}=R_{\theta(m)}(L_{g_m(w),0})$.
  Writing $g_m$ in coordinates as $g_m=(0, b_m,c_m)$, its Jacobian is
  $$J_{g_m}(y, z) = \det\begin{pmatrix}
    \frac{\ud  b_m}{\ud  y} & \frac{\ud  b_m}{\ud  z} \\
    \frac{\ud  c_m}{\ud  y} & \frac{\ud  c_m}{\ud  z}
  \end{pmatrix} = \det\begin{pmatrix}
    \cos(\arctan m)& 0 \\
    \frac{d c_m} {d y} & 1
  \end{pmatrix} = \frac{1}{\sqrt{1+m^2}}.$$
Consequently,
  \begin{align*}\int_{W_0} \int_{-\frac{\pi}{2}}^{\frac{\pi}{2}} f\big(R_{\theta}(L_{g,0})\big) \ud \theta \ud g &= \int_{W_0} \int_\R f(L_{w,m}) \frac{\ud \theta}{\ud  m} J_{g_m}(w) \ud m \ud w\\& = \int_{W_0} \int_\R \frac{f(L_{w,m})}{(1+m^2)^{\frac{3}{2}}} \ud m \ud w.\end{align*}
Thus
  \begin{equation*}
    \ENM_{E,R}(U)
    = \int_{\cL} \wwidehat{\omega}_{E,R}(U, L) \ud\cN(L)
    = c \int_{W_0} \int_\R \frac{\wwidehat{\omega}_{E,R}(U, L_{w,m})}{(1+m^2)^{\frac{3}{2}}} \ud m \ud w
    \le c R \Omega^P_{E,R}\big(\Pi(U)\big),
  \end{equation*}
  as desired.
\end{proof}

One advantage of $\Omega^P$ over $\ENM$ is that $\Omega^P$ scales nicely under automorphisms.

\begin{lemma}\label{lem:Omega scaling}
  Fix $a,b\in \R\setminus \{0\}$ and let $g=q\circ \rho_h\circ s_{a,b}\from \H\to \H$  be a composition of a shear map $q$, a left-translation by $h\in \H$, and a stretch map $s_{a,b}$.  Let $\hat{g}\from V_0\to V_0$ be the map induced on $V_0$, i.e., $\hat{g}(x)=\Pi(g(x))$ for all $x\in V_0$.  Let $E\subset \H$ be a set with finite perimeter.  For any measurable $U\subset V_0$ and any $R>0$, if $\Omega^P_{E, R}(U)$ is finite, then
  \begin{equation}\label{eq:b cubed}\Omega^P_{g(E), |a|R}(\hat{g}(U))=|b|^3\Omega^P_{E, R}(U),\end{equation}
  and
  \begin{equation}\label{eq:Omega scaling density}
    \frac{\Omega^P_{g(E), |a|R}(\hat{g}(U))}{|\hat{g}(U)|}=\frac{b^2}{a^2} \cdot \frac{\Omega^P_{E, R}(U)}{|U|}.
  \end{equation}
  In particular, if $g$ is a composition of a scaling, shear, and translation, i.e., when $a=b$ above,  then $g$ preserves the density of $\Omega^P_{E, R}$.
\end{lemma}
\begin{proof} The identity~\eqref{eq:b cubed} is verified by computing as follows, using~\eqref{eq:transform N}.
  \begin{align*}
    \Omega^P_{g(E), |a|R}\big(\hat{g}(U)\big)
    &=\frac{1}{|a|R} \int_{\cL_P} \wwidehat{\omega}^P_{g(E),|a| R}\big(\hat{g}(U),L\big)\ud \cN_P(L)\\
    &=\frac{|b|^3}{|a|R}\int_{\cL_P} \wwidehat{\omega}^P_{g(E),|a| R}\big(\hat{g}(U),g(L)\big)\ud \cN_P(L)\\
    &=\frac{|b|^3}{|a|R} \int_{\cL_P} |a| \wwidehat{\omega}^P_{E, R}(U,L)\ud \cN_P(L)\\
    &=|b|^3 \Omega^P_{E, R}(U),
  \end{align*}
 By Lemma~\ref{lem:v0 transforms} we have $|\hat{g}(U)|=a^2|b|\cdot  |U|$, which implies~\eqref{eq:Omega scaling density}.
\end{proof}

Suppose that $Q$ is a pseudoquad for an intrinsic Lipschitz graph $\Gamma\subset \H$ and that $g$ is as in Lemma~\ref{lem:Omega scaling}.  If $\Gamma$ is $(\eta,R)$--paramonotone on $rQ$ as in Definition~\ref{def:paramonotone}, then the density of $\Omega^P_{\Gamma^+,R\delta_x(Q)}$ is bounded as follows:
$$\frac{\Omega^P_{\Gamma^+,R\delta_x(Q)}(r Q)}{|Q|} \le \frac{\eta}{ \alpha(Q)^{4}}.$$
Let $\hat{Q}=\hat{g}(Q)$ and $\hat{\Gamma}=g(\Gamma)$.  Then \eqref{eq:Omega scaling density} and Lemma~\ref{lem:quad transforms} imply that
$$\frac{\Omega^P_{\hat{\Gamma}^+,R\delta_x(\hat{Q})}(r \hat{Q})}{|\hat{Q}|} \le \frac{\eta b^2}{a^{2}\alpha(Q)^{4}} = \frac{\eta}{ \alpha(\hat{Q})^{4}},$$
so $\hat{\Gamma}$ is $(\eta,R)$--paramonotone on $r\hat{Q}$ if and only if $\Gamma$ is $(\eta,R)$--paramonotone on $rQ$.

In particular, it follows from Lemma~\ref{lem:measure comparison} that if $\Gamma$ is $(\eta,R)$--paramonotone on $rQ$, then
\begin{equation}\label{eq:ENM of paramonotone}
  \ENM_{\Gamma^+,R \delta_x(Q)}\big(\Pi^{-1}(rQ)\big)\lesssim R \delta_x(Q) \Omega^P_{\Gamma^+,R\delta_x(Q)}(r Q) \le \frac{\delta_x(Q) |Q|} {\alpha(Q)^{4}} \eta R \approx_Q \eta R.
\end{equation}
\section{The kinematic formula and the proof of Proposition~\ref{prop:algorithm is weighted Carleson}}\label{sec:weighted Carleson}
In this section, we prove Proposition~\ref{prop:algorithm is weighted Carleson} using two lemmas.  The first bounds the total weight of the vertically cut descendants of a vertex $v\in \cV(\Delta)$ in terms of $\Omega^P$.
\begin{lemma}\label{lem:weight bounded by Omega}
  Let $r$, $\eta$, $R$, $\lambda$, $\Gamma$, and $\Delta$ be as in Proposition~\ref{prop:algorithm is weighted Carleson}.  Then for any $v\in \cV(\Delta)$,
  \begin{equation}\label{eq:weight bounded by Omega}
    W(\{w\in \cVv(\Delta)\mid w\le v\}) \lesssim_{\eta, r,R} \sum_{i=0}^\infty \Omega^P_{\Gamma^+,2^{-i} R\delta_x(Q_v)}(rQ_v).
  \end{equation}
\end{lemma}
The second is a kinematic formula bounding $\Omega^P$ in terms of Lebesgue measure on $V_0$.
\begin{lemma}\label{lem:sums of OmegaP}
  Let $0<\lambda<1$ and let $\Gamma$ be an intrinsic $\lambda$--Lipschitz graph.  For any measurable set $U\subset V_0$,
  \begin{equation}\label{eq:sums of OmegaP}
    \sum_{i\in \Z} \Omega^P_{\Gamma^+,2^{-i}}(U)\lesssim_{\lambda} |U|.
  \end{equation}
\end{lemma}

Proposition~\ref{prop:algorithm is weighted Carleson} follows from Lemma~\ref{lem:weight bounded by Omega} and Lemma~\ref{lem:sums of OmegaP}.
\begin{proof}[{Proof of Proposition~\ref{prop:algorithm is weighted Carleson} assuming Lemma~\ref{lem:weight bounded by Omega} and Lemma~\ref{lem:sums of OmegaP}}]
  Fix $v\in \cV(\Delta)$ and denote $\delta=\delta_x(Q_v)$.  Due to Lemma~\ref{lem:weight bounded by Omega},
  $$W(\{w\in \cVv(\Delta)\mid w\le v\}) \lesssim_{\eta, r,R} \sum_{i=0}^\infty \Omega^P_{\Gamma^+,2^{-i} R\delta}(rQ_v).$$
  Let $k$ be the integer such that $2^{k-1}\le R\delta <2^{k}$.  Then
  \begin{align*}
    \sum_{i=0}^\infty \Omega^P_{\Gamma^+,2^{-i} R\delta}(rQ_v)
     \le \sum_{i=0}^\infty 2 \Omega^P_{\Gamma^+,2^{-i+k}}(rQ_v)
     \stackrel{\eqref{eq:sums of OmegaP}}{\lesssim}_{\lambda} |rQ_v|.
  \end{align*}
  Thus, $W(\{w\in \cVv(\Delta)\mid w\le v\}) \lesssim_{\eta, r,\lambda, R} |Q_v|.$
\end{proof}

We first establish Lemma~\ref{lem:weight bounded by Omega}, which we prove using an argument based on the Vitali Covering Lemma.  The first step is to construct partitions of $Q$ into pseudoquads with dyadic widths.  As in Lemma~\ref{lem:hcut partitions}, we construct these partitions from coherent subtrees.

\begin{lemma}\label{lem:vcut partitions}
  Let $0<\mu\le \frac{1}{32}$ and let $(\Delta,(Q_v)_{v\in \Delta})$ be a $\mu$--rectilinear foliated patchwork for a $\mu$--rectilinear pseudoquad $Q$. Fix $j\in \N\cup\{0\}$.  For $v\in \cV(\Delta)$, let $\cDv(v)\subset \cV(\Delta)$ denote the set of vertically cut descendants of $v$, and let
  $$F_j(v)\eqdef\left\{w\in \cDv(v)\mid \delta_x(Q_w)=2^{-j}\delta_x(Q) \right\}.$$
  Then, for any $w,w'\in F_j(v)$, if $w\ne w'$, then $Q_w$ and $Q_{w'}$ have disjoint interiors.
\end{lemma}
\begin{proof}
  Let $\cD(v)$ be the set of descendants of $v$ and let
  $$R_j=\left\{w\in \cD(v)\mid \delta_x(Q_w)\ge 2^{-j}\delta_x(Q_v) \right\}.$$
  By Lemma~\ref{lem:soft cut props}, this is a coherent set and $F_j(v)=\min R_j$, so Lemma~\ref{lem:coherent partitions} implies that $F_j(v)$ consists of pseudoquads with disjoint interiors.
\end{proof}

Let $v_0$ be the root of $\Delta$ (so $Q=Q_{v_0}$).  For each $j\in \N\cup \{0\}$ we write $F_j=F_j(v_0)$.  Denote $I=x(Q)$ and   $l_j=2^{-j}\ell(I)=2^{-j}\delta_x(Q)$. Let $I_{j,1},\dots, I_{j,2^j}$ be the partition of $I$ into $2^j$ intervals of length $l_j$, so that for any $v\in \cV(\Delta)$, there are $j,m\in \N\cup \{0\}$ such that $x(Q_v)=I_{j,m}$.  We partition $F_j$ into columns as follows.
\begin{equation}\label{eq:def Fjm}
\forall m\in \{1,\ldots, 2^j\},\qquad F_{j,m}\eqdef \{w\in F_j \mid x(Q_w)=I_{j,m}\}.
\end{equation}
Each column satisfies the following version of the Vitali Covering Lemma.
\begin{lemma}\label{lem:column coverings}
  For each $j\in \N\cup\{0\}$ and $m\in \{1,\ldots,2^j\}$, there is a (possibly finite) sequence of vertices $D_{j,m}=\{v_1,v_2,\ldots\}\subset  F_{j,m}$ such that  $rQ_{v_1},rQ_{v_2},\ldots$ are pairwise disjoint and
  $$W(D_{j,m})\approx_r W(F_{j,m}).$$
\end{lemma}

We prove Lemma~\ref{lem:column coverings} using the following expansion property.

\begin{lemma}\label{lem:pseudoquad expansions}
  Let $r>1$ and $0<\mu\le \frac{1}{32r^2}$.  Let $\Delta$ be a $\mu$--rectilinear foliated patchwork.  Let $v,w\in \cV(\Delta)$ be vertices such that $x(Q_{v})=x(Q_{w})$ and suppose that $rQ_{v} \cap rQ_{w}$ is nonempty. If $\delta_z(Q_v)\ge \delta_z(Q_w)$, then $Q_{w}\subset 3r Q_{v}$, and  if $\delta_z(Q_w)\ge \delta_z(Q_v)$, then $Q_{w}\subset 3r Q_{v}$.
\end{lemma}
\begin{proof}
  Write $I=[-1,1]$.  By rescaling and translating, we may suppose without loss of generality  that $x(Q_{v})=x(Q_{w})=I$.  Also,  we may suppose that $Q_v$ is vertically below $Q_w$.  We first construct a stack of pseudoquads of width at least $2$ that connects $Q_v$ and $Q_w$.

  For $u\in \cV(\Delta)$, let $A(u)=\{t\in \cV(\Delta)\mid t\ge u\}$ be the set of ancestors of $u$.  If $u\ne v_0$, let $S(u)$ be the sibling of $u$ and let $P(u)$ be the parent of $u$.  Let
  $$J=A(v)\cup A(w) \cup S\big(A(v)\cup A(w)\big).$$
  Since $A(v)\cup A(w)$ spans a connected subtree of $\Delta$, so does $J$, and $J$ is a coherent subset of $\cV(\Delta)$.  Furthermore, $J$ is finite, so $K=\min J$ is a partition of $Q$.

  If $u\in K$, then $u$ is either an ancestor of $v$ or $w$ or a sibling of such an ancestor.  In either case, $\delta_x(Q_u)\ge 2$, and the base of $Q_u$ either contains $I$ or its interior is disjoint from $I$.  Let $K'=\{u\in K\mid I\subset x(Q_u)\}$.  For each $u\in K'$, $Q_{u}$ intersects the $z$--axis in an interval.  We denote the elements of $K'$ by $u_1,\dots, u_n$, in order of increasing $z$--coordinate.  These pseudoquads form a stack; each pseudoquad $Q_{u_{i}}$ is vertically adjacent to $Q_{u_{i+1}}$.  We suppose that $u_a=v$ and $u_b=w$, with $a<b$.

  Rectilinearity implies that the boundaries of the $Q_{u_i}$'s have similar slopes.  For each $i\in \n$, let $g_i$ be the lower bound of $Q_{u_i}$ and let $g_{i+1}$ be its upper bound.  These may be defined on different domains, but all of their domains contain $I$.  For each $i\in \n$, let $R_{u_i}$ be the parabolic rectangle associated to $Q_{u_i}$ and let  $d_i=\delta_z(Q_{u_i})$, so that there are quadratic functions $h_i\from \R\to \R$ satisfying
  \begin{align*}\label{eq:choice of dh}
    \left\|g_i-\left(h_i-\frac{d_i}{2}\right)\right\|_{L_\infty(I)}  \le \mu d_i\qquad\mathrm{and}\qquad     \left\|g_{i+1}-\left(h_i+\frac{d_i}{2}\right)\right\|_{L_\infty(I)}  \le \mu d_i.
  \end{align*}
   Then
  \begin{equation}\label{eq:gi-hi-close}
    \left\|\left(h_{i+1} - \frac{d_{i+1}}{2}\right) - \left(h_{i} + \frac{d_i}{2}\right)\right\|_{L_\infty(I)}  \le \mu(d_i+d_{i+1}).
  \end{equation}
 Hence,  for any $i,j\in \n$ with $i<j$,
  $$\left\|h_{j}-h_i-\sum_{k=i}^{j-1} \frac{d_k+d_{k+1}}{2}\right\|_{L_\infty(I)}\le \mu \sum_{k=i}^{j-1} (d_k+d_{k+1}).$$
  Since $h_{j}-h_i-\sum_{k=i}^{j-1} \frac{d_k+d_{k+1}}{2}$  is  quadratic, by Lemma~\ref{lem:quadratic interval} it follows that
  $$\left\|h_{j}-h_i-\sum_{k=i}^{j-1} \frac{d_k+d_{k+1}}{2}\right\|_{L_\infty([-r,r])} \le 4r^2 \mu \sum_{k=i}^{j-1} (d_k+d_{k+1}) \le \sum_{k=i}^{j-1} \frac{d_i+d_{i+1}}{8}.$$
Denoting
  $$D=\sum_{k=a}^{b-1} \frac{d_k+d_{k+1}}{2},$$
 it follows that for all $x\in [-r,r]$ we have
 \begin{equation}\label{eq:3454}
 \frac{3}{4}D\le h_b(x)-h_a(x)\le \frac{5}{4}D.
 \end{equation}

Suppose that  $\delta_z(Q_w)\le \delta_z(Q_v)$.  For each $i\in\n$, the definition~\eqref{eq:def rho R} of $r Q$ states
  $$r Q_{u_i}=\left \{(x,z)\in V_0\mid x\in [-r,r]\ \mathrm{and}\  |z-h_i(x)|\le \frac{r^2 d_i}{2}\right\}.$$
  Since $\delta_z(Q_w)\le \delta_z(Q_v)$ and  $rQ_v$ intersects $rQ_w$, there is  $t\in [-r,r]$ such that
  $$h_b(t) - h_a(t) \le \frac{r^2(d_b+d_a)}{2} = \frac{r^2(\delta_z(Q_v)+\delta_z(Q_w))}{2}\le r^2\delta_z(Q_v),$$
  and thus by~\eqref{eq:3454} we have $D\le \frac43r^2\delta_z(Q_v)$.

  Let $(x,z)\in Q_w$. By \eqref{eq:gi-hi-close},
  $$z\in [g_b(x),g_{b+1}(x)]\subset \left[h_b(x)-\frac{3}{4}\delta_z(Q_w), h_b(x)+\frac{3}{4}\delta_z(Q_w)\right],$$
  so
  \begin{equation*}
    |h_a(x) - z|
    \le |h_a(x)-h_b(x)|+|h_b(x)-z|
    \le \frac{5}{4} D + \frac{3}{4}\delta_z(Q_w)
  \le \frac{(3r)^2 \delta_z(Q_v)}{2},
  \end{equation*}
  where the penultimate step uses~\eqref{eq:3454} and the final step uses the upper bound on $D$ that we derived above and the assumption $\delta_z(Q_w)\le \delta_z(Q_v)$. It follows that $(x,z)\in 3rQ_v$ and thus $Q_w\subset 3rQ_v$. If   $\delta_z(Q_v)\le \delta_z(Q_w)$, then the analogous reasoning shows that $Q_v\subset 3rQ_w$.
\end{proof}

\begin{proof}[{Proof of Lemma~\ref{lem:column coverings}}] Similarly to the proof of the Vitali covering lemma, we define inductively a sequence $S_0,S_1\ldots,$ of subsets of $F_{j,m}$ as follows.
  Let $S_0=\emptyset$.  For each $i\in \N$, let $v_i$ be an element of $F_{j,m}\setminus \bigcup_{k=0}^{i-1}S_i$ that maximizes $\delta_z(Q_{v_i})$.  Define
  $$S_{i}=\left\{w\in F_{j,m}\setminus \bigcup_{k=0}^{i-1}S_k\mid rQ_{v_i}\cap rQ_w\ne \emptyset\right\}.$$
  If $\bigcup_{k=1}^{i}S_k=F_{j,m}$,  we stop.  By construction, $rQ_{v_1},rQ_{v_2},\ldots$ are disjoint.  We will show that the set  $D_{j,m}=\{v_1,v_2,\ldots\}\subset F_{j,m}$ satisfies the desired properties.

  We first claim that $F_{j,m}=S_1\cup S_2\cup\ldots$, where this holds by construction if  there are only finitely many $v_i$'s. So,  suppose that there are infinitely many $v_i$'s and let $w\in F_{j,m}$.  There are only finitely many elements of $F_{j,m}$ with height greater than $\delta_z(Q_w)$, so there is  $i\in \N$ such that $\delta_z(Q_{v_i})<\delta_z(Q_w)$.  By the maximality of $\delta_z(Q_{v_i})$, this implies that $w\in S_1\cup\ldots\cup S_{i-1}.$

  We next show that $W(D_{j,m})\approx_r W(F_{j,m})$.  As $D_{j,m}\subset F_{j,m}$, we have  $W(D_{j,m})\le W(F_{j,m})$.  Conversely, if $w\in S_i$, then $rQ_w$ intersects $rQ_{v_i}$ and $\delta_z(Q_{w})\le \delta_z(Q_{v_i})$, so Lemma~\ref{lem:pseudoquad expansions} implies that $Q_w\subset 3rQ_{v_i}$.  Since the elements of $F_{j,m}$ are pairwise disjoint (Lemma~\ref{lem:vcut partitions}) pseudoquads of the same width, we have $\alpha(Q_{w})\ge \alpha(Q_{v_i})$ and
  \begin{multline*}
  W(S_i)=\sum_{w\in S_i} \alpha(Q_w)^{-4}|Q_w| \le \alpha(Q_{v_i})^{-4}  \sum_{w\in S_i} |Q_w|\\= \alpha(Q_{v_i})^{-4} \Big|\bigcup_{w\in S_i} Q_w\Big|  \le \alpha(Q_{v_i})^{-4}\bigl|3rQ_{v_i}\bigr| \approx r^3W(\{v_i\}).
  \end{multline*}
By summing this bound over $j$, we conclude
 \begin{equation*}W(F_{j,m})=W(S_1)+W(S_2)+\ldots \lesssim r^3 \bigl(W(\{v_1\})+W(\{v_2\})+\ldots\bigl) =r^3 W(D_{j,m}).\tag*{\qedhere}\end{equation*}
\end{proof}

We are now ready to prove Lemma~\ref{lem:weight bounded by Omega}.
\begin{proof}[{Proof of Lemma~\ref{lem:weight bounded by Omega}}]
  It suffices to treat the case where $v$ is the root of $\Delta$, so $Q_v=Q$.  Fix $j\in \N\cup \{0\}$ and $m\in \{1,\ldots,2^j\}$. Let $F_{j,m}$ and $D_{j,m}$ be as in Lemma~\ref{lem:column coverings}.

  Since, by definition,   $F_{j,m}$ consists only of vertices that are vertically cut, by hypothesis, $\Gamma$ is not $(\eta, R)$--paramonotone on $Q_w$  for each $w\in F_{j,m}$, i.e.,
  $$\forall w\in F_{j,m},\qquad \Omega^P_{\Gamma^+,R l_j}(r Q_w) > \eta \alpha(Q_w)^{-4}|Q_w|=\eta W(\{w\}).$$
  Let $S_m=r I_{j,m}\times\{0\}\times  \R\subset V_0$.  The sets $\{rQ_w\}_{w\in D_{j,m}}$ are disjoint subsets of $S_m\cap r Q$, so
  $$W(F_{j,m}) \approx_r W(D_{j,m})\le \eta^{-1} \sum_{w\in D_{j,m}}\Omega^P_{\Gamma^+,R l_j}(r Q_w) \le \eta^{-1} \Omega^P_{\Gamma^+,R l_j}(S_m\cap rQ).$$
By summing this bound over $m\in \{1,\ldots,2^j\}$ we get
  \begin{equation*}\label{eq:Fi packing condition}
    W(F_{j}) \lesssim_r \sum_{m=1}^{2^j}\eta^{-1} \Omega^P_{\Gamma^+,R l_j}(S_m\cap rQ) \lesssim_r \eta^{-1} \Omega^P_{\Gamma^+,R l_j}(rQ),
  \end{equation*}
  where the last step holds because the scaled intervals $r I_{j,1},\ldots,r I_{j,2^j}$ have bounded overlap (depending on $r$).
By summing this bound over $j$, we conclude as follows.
  \begin{equation*}W\big(\cVv(\Delta)\big)=\sum_{j=0}^\infty W(F_{j}) \lesssim_{r}\eta^{-1} \sum_{j=0}^\infty \Omega^P_{\Gamma^+,R 2^{-j}\delta_x(Q)}(r Q).\tag*{\qedhere}\end{equation*}
\end{proof}

Next, we prove Lemma~\ref{lem:sums of OmegaP} using the following kinematic formula for intrinsic Lipschitz graphs.  Recall (Section~\ref{sec:Prelim heis}) that for a measurable subset $E\subset \H$, we let $\Per_E$ denote the perimeter measure of $E$; this measure is supported on $\partial E$, and when $E$ is bounded by an intrinsic Lipschitz graph, it differs from $3$--dimensional Hausdorff measure on $\partial E$ by at most a multiplicative constant.  For any horizontal line $L\in \cL$, let $\partial_{\cH^1|_L}E$ be the measure-theoretic boundary of $E$ in $L$ and let $\Per_{E,L}$ be the counting measure on $\partial_{\cH^1|_L}E$.
\begin{lemma}\label{lem:parametric kinematic}
  Fix $0<\lambda<1$. Let $\psi\from V_0\to \R$ be intrinsic $\lambda$--Lipschitz and let $\Gamma=\Gamma_\psi$ be its intrinsic graph.  Let $U\subset V_0$ be a measurable set.  For almost every $L\in \cL_P$, the intersection $L\cap \Gamma^+$ has locally finite perimeter.  If $\cM\subset \cL_P$ is the set of lines that intersect $\Gamma$ at least twice, then
  \begin{equation}\label{eq:lesssim lambda |U|}\int_{\cM} \Per_{\Gamma^+,L}\big(\Pi^{-1}(U)\big)\ud \cN_P(L)\lesssim_\lambda |U|.\end{equation}
\end{lemma}
\begin{proof}
  The measures $\cN_P$ and $\cN$ are absolutely continuous with respect to each other.  Indeed, for each $m>0$, if $D\subset \cL_P$ is a set of lines with slopes that lie in $[-m,m]$, then $\cN_P(D)\approx_m\cN(D)$. By \eqref{eq:kinematic}, there is  $c>0$ such that for any measurable $A\subset \H$,
  $$\Per_{\Gamma^+}(A) = c \int_{\cL} \Per_{\Gamma^+,L}(A) \ud \cN(L).$$
  Since $\Gamma^+$ has locally finite perimeter, this implies that for almost every line $L\in \cL_P$, the intersection $L\cap\Gamma^+$ has locally finite perimeter. For $L\in \cL_P$ let $m(L)$ be the slope of $L$ as in Section~\ref{sec:parametric monotonicity}. Suppose that $p\in L\cap \Gamma$.  By~\eqref{eq:formula for intersection with cone},  if $|m(L)|> \lambda/\sqrt{1-\lambda^2}$, then $L\subset p\cdot \mathrm{Cone}_\lambda$ and thus $L$ intersects $\Gamma$ exactly once.  Consequently, $|m(M)|\le \lambda/\sqrt{1-\lambda^2}$ for every $M\in \cM$, and hence $\cN_P(D)\approx_\lambda \cN_P(D)$ for every measurable $D\subset \cM$. So, by \eqref{eq:kinematic} and Lemma~\ref{lem:intrinsic Lipschitz perimeter},
  \begin{align*}
    \int_{\cM} \Per_{\Gamma^+,L}\big(\Pi^{-1}(U)\big)\ud \cN_P(L)
     \lesssim_\lambda \Per_{\Gamma^+}\big(\Pi^{-1}(U)\big) \approx_\lambda |U|.\tag*{\qedhere}
  \end{align*}
\end{proof}


\begin{proof}[{Proof of Lemma~\ref{lem:sums of OmegaP}}]
  For a finite-perimeter set $S\subset \R$ and $R>0$, let $\cI(S)$ and $$\widehat{\omega}_{S, R}=\frac{\omega_{S,R}+\omega_{\R\setminus S,R}}{2}$$ be as in Section~\ref{sec:parametric monotonicity R}.  Divide $\cI(S)$ according to the length of the intervals as follows.
  $$\forall j\in \Z,\qquad C_j(S)\eqdef\left\{I\in \cI(S)\mid 2^{-j-1}< |I| \le 2^{-j}\right\}.$$
  Let $\cE_j(S)\subset \R$ be the set of endpoints of the intervals in $C_j(S)$.  Let $\lambda_{S,j}$ be the counting measure on $\cE_j(S)$ and let $$\widehat{\lambda}_j(S)\eqdef \frac{\lambda_{S,j}+\lambda_{\R\setminus S,j}}{2}.$$  Then $\sum_{j\in \Z}\widehat{\lambda}_{S,j}\le \Per_{S}$.  (This isn't necessarily an equality as the left hand side is influenced only by bounded intervals while the right hand side could have a contribution from rays.)

  For each $k\in \Z$, the measure $\widehat{\omega}_{S, 2^{-k}}$ is a point measure supported on the set
  $$\bigcup_{j=k}^\infty \big(\cE_j(S) \cup \cE_j(\R\setminus S)\big),
   $$
   that weights each point according to the lengths of the intervals it bounds.  In particular, $$\supp\big(\widehat{\omega}_{S,2^{-k}}-\widehat{\omega}_{S,2^{-k-1}}\big)\subset \cE_k(S)\cup \cE_k(\R\setminus S),
    $$
    and
    $$\forall p\in \cE_k(S)\cup \cE_k(\R\setminus S),\qquad 2^{-k-2}\le \widehat{\omega}_{S,2^{-k}}(p) - \widehat{\omega}_{S,2^{-k-1}}(p)\le 2^{-k}.$$
    Consequently, if we denote
  $$\widehat{\kappa}_{S,k} \eqdef 2^{k}\big(\widehat{\omega}_{S,2^{-k}} - \widehat{\omega}_{S,2^{-k-1}}\big),$$
  then $\widehat{\kappa}_{S,j}\approx \widehat{\lambda}_{S,j}$ and
  $$
  \sum_{j\in \Z} \widehat{\kappa}_{S,j}=\sum_{j\in \Z} 2^{j+1} \widehat{\omega}_{S,2^{-j}} - \sum_{j\in \Z} 2^{j} \widehat{\omega}_{S,2^{-j}}=   \sum_{j\in \Z} 2^j \widehat{\omega}_{S,2^{-j}}.
  $$
It follows that
  \begin{equation}\label{eq:less than per}\sum_{j\in \Z} 2^j \widehat{\omega}_{S,2^{-j}} \approx \sum_{j\in \Z} \widehat{\lambda}_{S,j}\le \Per_{S}.
  \end{equation}

For every measurable $E\subset \H$ and $U\subset V_0$, and every $L\in \cL_P$, we have
  \begin{equation}\label{eq:scale sums omegaP}
    \sum_{j\in \Z} 2^j \wwidehat{\omega}^P_{E, 2^{-j}}(U, L) \stackrel{\eqref{eq:define hat OmegaP}\wedge \eqref{eq:less than per}}{\lesssim} \Per_{x(E\cap L)}\big(x(\Pi^{-1}(U)\cap L)\big) = \Per_{E,L}\big(\Pi^{-1}(U)\big).
  \end{equation}
 Let $\cM\subset \cL_P$ be the set of lines that intersect $\Gamma$ at least twice.  If $L\in \cL_P\setminus \cM$, then $\cI(L\cap \Gamma^+)$ consists of infinite rays, so $\widehat{\omega}_{\Gamma^+, R}(U,L)=0$ for any $U\subset V_0$.  Thus,
  \begin{multline*}
    \sum_{j\in \Z} \Omega^P_{\Gamma^+,2^{-j}}(U)
    \stackrel{\eqref{eq:define OmegaP}}{=}\sum_{j\in \Z} \int_{\cM} 2^j \wwidehat{\omega}^P_{\Gamma^+,2^{-j}}(U,L) \ud \cN_P(L) \\
    \stackrel{\eqref{eq:scale sums omegaP}}{\lesssim} \int_{\cM} \Per_{\Gamma^+,L}(\Pi^{-1}(U))\ud \cN_P(L)
    \stackrel{\eqref{eq:lesssim lambda |U|}}{\lesssim_\lambda} |U|. \tag*{\qedhere}
  \end{multline*}
\end{proof}

\section{Outline of proof of Proposition~\ref{prop:Omega control}}\label{sec:omega control outline}

The rest of this paper is dedicated to the proof of Proposition~\ref{prop:Omega control}.  This is the longest part of the proof of Theorem~\ref{thm:corona graph formulated}, and we will divide it into two pieces.

In the first step (Section~\ref{sec:extended monotone}), we prove the following Proposition~\ref{prop:stability of extended monotone}, which is a stability result for extended-monotone sets (Definition~\ref{def:extended monotone}).  For every $r>0$ and $h\in \H$, let $\overline{B}_r(h)\subset \H$ be the convex hull of $B_r(h)$ (as a subset of $\R^3$); when $h$ is omitted, we take it to be $\mathbf{0}$.  The convex hull of $B_r$ with respect to the horizontal lines or with respect to all lines in $\R^3$ is the same, and $\overline{B}_r\subset B_{2r}$.
\begin{prop}\label{prop:stability of extended monotone}
  Let $E\subset \H$ be a measurable set.  For any $\epsilon>0$, there are $\nu, R>0$ such that if $E\subset \H$ is $(\nu',R')$--extended monotone on $\overline{B}_1$ for some $\nu', R'>0$ that satisfy $R' \ge R$ and $\nu' R'\le \nu R$, then there is a plane $P\subset \H$ such that $$|\overline{B}_{1}\cap (P^+\symdiff E)|<\epsilon.$$  If $\Gamma$ is an intrinsic Lipschitz graph and $E=\Gamma^+$, then we can take $P$ to be a vertical plane.
\end{prop}

Proposition~\ref{prop:stability of extended monotone} is in the spirit of the stability theorem for monotone sets that was proved in~\cite{CKN}, though here we do not need to obtain an explicit dependence of $\nu,R$ on $\e$ (in~\cite{CKN} it was important to get power-type dependence).  The lack of explicit dependence lets us use a compactness argument that was not available in the context of~\cite{CKN}.  At the same time, Theorem~4.3 of \cite{CKN} states that if the nonmonotonicity of $E$ is small on the unit ball $B_1$, then there is a smaller ball $B_{\epsilon^3}$ on which $E$ is $O(\epsilon)$--close to a plane, while  Proposition~\ref{prop:stability of extended monotone} assumes a stronger hypothesis, namely that $\ENM_{E,R}(\overline{B}_1)<\nu$, and obtains the stronger conclusion that $E$ is close to a plane on the same ball $\overline{B}_1$.

\begin{remark}\label{rem:stronger-conclusion}
  The stronger conclusion above is crucial for the covering argument that we used in Section~\ref{sec:weighted Carleson} because of the delicacy of the Vitali-type argument used in Lemma~\ref{lem:column coverings}. We use Lemma~\ref{lem:column coverings} to show that if $\Delta$ is a $\mu$--rectilinear foliated patchwork, $0<\mu<\frac{1}{32r^2}$, and $F\subset \cV(\Delta)$ is a collection of vertices corresponding to pseudoquads of the same width, then there is a large subset $G$ of these pseudoquads such that if $Q,Q'\in G$, then $r Q$ is disjoint from $r Q'$. Lemma~\ref{lem:column coverings} only holds when $\mu=O(r^{-2})$.  If $\mu r^2$ is too large, then a $\mu$--rectilinear foliated patchwork could contain arbitrarily many vertically cut pseudoquads $Q_1,Q_2,\dots,Q_n$ of equal height and width such that $rQ_1,\dots, rQ_n$ all intersect.

  We do not see how a modified subdivision algorithm that uses monotonicity instead of paramonotonicity can ensure that the conditions of Lemma~\ref{lem:column coverings} are satisfied. For example, consider a modification of the subdivision algorithm in Section~\ref{sec:subdiv algorithm} that produces a patchwork $\Delta$ by cutting a pseudoquad $Q$ horizontally or vertically depending on whether $\Gamma$ is $\eta$--monotone (rather than paramonotone) on $rQ$ for some $r>0$. Theorem~4.3 of \cite{CKN} implies that if $\Gamma$ is sufficiently monotone on $rQ$, then $\Gamma$ is $O(r^{-\frac{1}{3}})$--close to a plane on $Q$. Indeed, there are sets that have zero nonmonotonicity on $rQ$, but are only $\epsilon(r)$--close to a plane on $Q$, where $\epsilon(r)\to 0$ as $r\to \infty$. It follows that this modified algorithm can, at best, produce $\mu(r)$--rectilinear foliated patchworks, where $\mu(r)\to 0$ as $r\to \infty$. In particular, since $\mu(r)$ depends on $r$, we cannot choose $\mu$ so that $\mu<\frac{1}{32r^2}$.

  Consequently, we cannot prove the weighted Carleson condition for this modified algorithm. The weighted Carleson condition bounds the number of vertically cut pseudoquads based on the total nonmonotonicity of $\Gamma$, but without Lemma~\ref{lem:column coverings}, a small amount of nonmonotonicity can lead to many vertically cut pseudoquads. That is, if $Q_1,\dots, Q_n$ are pseudoquads in the patchwork such that $rQ_1,\dots, rQ_n$ all intersect, then nonmonotonicity on the intersection $rQ_1\cap \dots\cap rQ_n$ could force the algorithm to cut all of the $Q_i$ vertically.

  Using extended nonmonotonicity rather than nonmonotonicity lets us avoid this problem. The fact that $r$ is a universal constant in Proposition~\ref{prop:Omega control} means that for any $\mu$, there is a subdivision algorithm that produces a $\mu$--rectilinear foliated patchwork by cutting each pseudoquad $Q$ based on whether $\Gamma$ is $(\eta(\mu),R(\mu))$--paramonotone on $rQ$. In particular, we can choose $\mu<\frac{1}{32r^2}$ so that Lemma~\ref{lem:column coverings} applies.
\end{remark}

In the second step, we prove parts \ref{it:omega control plane} and \ref{it:omega control characteristics} of Proposition~\ref{prop:Omega control}.  By Remark~\ref{rem:normalizing rectilinear}, after a stretch, shear, and translation, we may suppose that $Q$ is a rectilinear pseudoquad for $\Gamma$ that is close to $[-1,1]^2$ and $\Gamma$ is $(\eta, R)$--paramonotone on $rQ$.  For any given $c$, if $R$ is sufficiently large, $\eta$ is sufficiently small, and $\Pi(\overline{B}_c)\subset rQ$, then, by Lemma~\ref{lem:measure comparison}, $\Gamma^+$ has small extended nonmonotonicity on $\overline{B}_c$, so $\Gamma^+$ is close to a half-space $P^+$ on $\overline{B}_c$.

Note that even though $\Gamma^+$ is close to a half-space $P^+$ on $\overline{B}_c$, it does not immediately follow that the corresponding intrinsic Lipschitz function $f$ is $L_1$--close to an affine function. Using Remark~\ref{rem:normalizing rectilinear} to normalize $Q$ stretches $\Gamma$ and changes its intrinsic Lipschitz constant. Consequently, even though $f$ is close to an affine function on most of $Q$, it may still take on large values on the rest of $Q$. To show that this does not happen, we must introduce new methods based on analyzing the characteristic curves of $\Gamma$.

For example, a key step in the proof of part~\ref{it:omega control plane} of Proposition~\ref{prop:Omega control} is to show that $\|f\|_{L_1(Q)}$ is bounded.  Since $f$ is intrinsic Lipschitz, $\|f\|_{L_1(Q)}<\infty$, but we need a bound independent of the intrinsic Lipschitz constant.  We obtain such a bound by studying how lines intersect the characteristic curves.  Since $Q$ is $\mu$--rectilinear, the top and bottom boundaries of $Q$ are characteristic curves that are close to the top and bottom edges of $[-1,1]^2$.   If $L$ is a horizontal line such that $\Pi(L)$ crosses $[-1,1]^2$ from top to bottom, then $\Pi(L)$ must also cross the top and bottom boundaries of $Q$.  At these intersection points, the slope of $\Pi(L)$ is less than the slope of the boundary, so the corresponding points of $L$ lie in $\Gamma^+$.  If $\Gamma^+\cap L$ is close to monotone, then most of the interval between these points lies in $\Gamma^+$ and therefore, $f$ is bounded on $Q\cap \Pi(L)$.  By integrating over a family of lines that all cross the top and bottom boundaries, we obtain the desired $L_1$ bound.  Similar arguments based on characteristic curves lead to part~\ref{it:omega control characteristics} of Proposition~\ref{prop:Omega control}, which completes the proof of Proposition~\ref{prop:Omega control}.

\section{Extended-monotone sets are close to half-spaces}\label{sec:extended monotone}

In this section, we will prove Proposition~\ref{prop:stability of extended monotone} by studying limits of $(\epsilon,R)$--extended monotone sets.  Let $U\subset \H$ be measurable and let $E_1,E_2,\dots\subset \H$ be a sequence of measurable sets such that $E_i$ is $(\frac{1}{i},i)$--extended monotone on $U$.  By passing to a subsequence, we may suppose that $\one_{E_i}$ converges weakly to a function $f\in L_\infty(\H)$ taking values in $[0,1]$.  We call $f$ a \emph{$U$--LEM (limit of extended monotones) function}.

One difficulty of studying $f$ is that it need not take values only in $\{0,1\}$.  Indeed, the extended monotonicity $\ENM_{E_i,i}(\overline{B}_1)$ only depends on the intersection of $E_i$ with lines that pass through $\overline{B}_1$.  These lines do not cover all of $\H$, so there are regions of $\H$ where $f$ can take on arbitrary values.

Nevertheless, in Section~\ref{sec:stability locally monotone}, we will show that, after changing $f$ on a measure-zero set,  $f(\overline{B}_1)\subset \{0,1\}$.  This will follow from the fact that, by Lemma~\ref{lem:compare NM omega convex}, $$\lim_{i\to \infty} \NM_{E_i}(\overline{B}_1)=0.$$
We will show that a sequence of sets with nonmonotonicity going to zero on $\overline{B}_1$ converges to a subset which is monotone on $\overline{B}_1$.  If $U$ is an open set, a subset $E\subset \H$ is said to be \emph{monotone} on $U$ if $\NM_E(U)=0$.

Then, in Section~\ref{sec:rectifiable ruled}, we will use techniques from \cite{CheegerKleinerMetricDiff} and \cite{CKN} to characterize sets such that $\NM_{F}(\overline{B}_1)=0$.  A set that is monotone on $\overline{B}_1$ need not be a half-space, but we will show that if $F$ is such a set, then the measure-theoretic boundary $\partial_{\cH^4} F$ is a union of horizontal lines that has an \emph{approximate tangent plane} at every point.  That is, for any $g\in \partial_{\cH^4} F$, the blowups $g\cdot s_{n,n}(g^{-1} \partial_{\cH^4} F)$ converge in the Hausdorff metric to a plane $T_g$ as $n\to \infty$.  In fact, at all but countably many points $g\in \partial_{\cH^4} F$,  there is a unique horizontal line $L_g$ through $g$ that is contained in $\partial_{\cH^4} F$, and $T_g$ is the vertical plane containing $L_g$; in this case, $g$ has an approximate tangent subgroup in the sense of \cite{MSSCCharacterizations}.
At the remaining points, $T_g$ is the horizontal plane centered at $g$. 

Finally, in Section~\ref{sec:stability of extended monotone sets}, we prove Proposition~\ref{prop:stability of extended monotone}.  The proof is somewhat involved, but, as an illustration, we consider the case that $f=\one_E$, where $E$ is precisely $\infty$--extended monotone on $\overline{B}_1$.  That is, for every line $L$, either $\overline{B}_1\cap \partial(L\cap E)=\emptyset$ or $L\cap E$ is a monotone subset of $L$.

We first claim that for every point $b\in {\overline{B}_1}\cap \partial_{\cH^4} E$, if the approximate tangent plane $T_b$ is vertical and $\mathsf{H}_b$ is the horizontal plane centered at $b$, then $\mathsf{H}_b\cap \partial_{\cH^4} E=\mathsf{H}_b\cap T_b$.  Let $T_b^\pm$ be the two half-spaces bounded by $T_b$, labeled so that $T_b^+\cap B_r(b)$ approximates $E\cap B_r(b)$ at small scales.  Let $L_b=\mathsf{H}_b\cap T_b$ be the horizontal line in $\partial_{\cH^4} E$ that passes through $b$ and let $L$ be a line through $b$ that intersects $T_b$ transversally.  Then $E\cap L$ is a monotone set with $b\in \partial_{\cH^1}(E\cap L)$, so $T_b^+\cap L \subset E\cap L$ and $T_b^-\cap L\subset  L\setminus E$.  This holds for every horizontal line through $b$ except $L_b$, so $L_b$ cuts $\mathsf{H}_b$ into two half-planes $P_\pm=T_b^\pm\cap \mathsf{H}_b$ such that $P_+\subset E$ and $P_-\subset \mathsf{H}_b\setminus E$.

When $b'\in L_b$ is close to $b$, the plane $\mathsf{H}_{b'}$ intersects $\mathsf{H}_{b}$ along $L_b$ and the angle between the two planes is small.  As above, there are two half-planes $P_\pm'=T_{b'}^\pm\cap \mathsf{H}_{b'}$ such that $P_+'\subset E$ and $P_-'\subset \mathsf{H}_{b'}\setminus E$.  As $b'$ varies over points close to $b$, the half-plane $P_+'$ varies over half-planes close to $P_+$.  Therefore $P_+$ is in the interior of $E$, $P_-$ is in the exterior, and $\mathsf{H}_b\cap \partial_{\cH^4} E=L_b$.

Suppose that $L_{1}$ and $L_{2}$ are two lines in $\partial_{\cH^4} E$ that intersect $\overline{B}_1$, and suppose by way of contradiction that they are not coplanar.  By the hyperboloid lemma~\cite[Lemma~2.4]{CheegerKleinerMetricDiff} (see Lemma~\ref{lem:hyperboloid lemma}), for any point $q\in L_1$ except possibly a single point, there is a horizontal line $M$ that connects $q$ to a point $r$ in $L_2$.  Then $r\in \mathsf{H}_q\cap \partial_{\cH^4} E=L_1$, so $L_1$ and $L_2$ intersect and are thus coplanar; this is a contradiction.  It follows that $\overline{B}_1\cap \partial_{\cH^4} E$ is contained in a plane.  The proof of Proposition~\ref{prop:stability of extended monotone} runs along the same lines, but it takes some further technical work to apply the weaker hypothesis that $f$ is merely an LEM function.

One of the key tools in the proof is the following  ``hyperboloid lemma,'' which is stated as  Lemma~2.4 in \cite{CheegerKleinerMetricDiff}.  A pair of horizontal  lines $L_1,L_2\in \cL$ are said to be \emph{skew} if $L_1$ and $L_2$ are disjoint and the projections $\uppi(L_1),\uppi(L_2)\subset \mathsf{H}\cong \R^2$ are not parallel.
\begin{lemma}[Cheeger--Kleiner hyperboloid lemma~\cite{CheegerKleinerMetricDiff}]\label{lem:hyperboloid lemma} For any $L_1,L_2\in \cL$ we have
  \begin{enumerate}
  \item
    Suppose that the  projections $\uppi(L_1),\uppi(L_2)$ are parallel but $\uppi(L_1)\ne \uppi(L_2)$.  Then every point in $L_1$ can be joined to $L_2$ by a unique line.  In fact, there is a unique fiber $\uppi^{-1}(p)$ such that every line joining $L_1$ to $L_2$ passes through $\uppi^{-1}(p)$.  Conversely, for every $a\in \uppi^{-1}(p)$, there is a unique line joining $L_1$ to $L_2$ that passes through $a$.
  \item
    If $L_1,L_2$ are skew, then there is a hyperbola $S\subset \mathsf{H}$ with asymptotes $\uppi(L_1)$ and $\uppi(L_2)$ such that every tangent line of $S$ has a unique horizontal lift that intersects $L_1$ and $L_2$.  If $p\in \mathsf{H}$ is the intersection between $\uppi(L_1)$ and $\uppi(L_2)$ and $a\in L_1$ is  such that $\uppi(a)\ne p$, then there is a unique horizontal line that connects $a$ to a point in $L_2$.
  \end{enumerate}
\end{lemma}

\subsection{Stability of locally monotone sets}\label{sec:stability locally monotone}

We begin the proof of Proposition~\ref{prop:stability of extended monotone} by using a compactness argument to prove the following lemma. Throughout what follows, given a measure space $(\mathscr{S},\Sigma,\mu)$ and a measurable subset $\Omega\in \Sigma$ with $\mu(\Omega)>0$, we use the (standard) notation $\fint_\Omega$ to denote the averaging operator on $\Omega$, i.e.,
$$
\forall f\in L_1(\Omega,\mu),\qquad  \fint_\Omega f\ud\mu \eqdef \frac{1}{\mu(\Omega)}\int_\Omega f\ud \mu.
$$
\begin{lemma}\label{lem:stability locally monotone}
  Let $U\subset \H$ be a bounded open set and let $E_1,E_2,\dots \subset \H$ be a sequence of measurable sets such that $\NM_{E_i}(U)<\frac{1}{i}$ for every $i\in \N$.  There is a subsequence $(E_{i_j})_{j\in \N}$ and a set $F\subset U$ that is monotone on $U$ such that $\lim_{j\to \infty}\left|(E_{i_j}\cap U)\symdiff F\right|= 0$.

  It follows that for any $\epsilon>0$, there is a $\delta>0$ such that if $E\subset \H$ is a measurable set and $\NM_{E}(U)<\delta$, then there is a set $F\subset U$ such that $|(E\cap U)\symdiff F|<\epsilon$ and $F$ is monotone on $U$.
\end{lemma}
\begin{proof}
  After passing to a subsequence, we may suppose that the characteristic functions $\one_{E_i}$ converge weakly to a function $f\in L_\infty(U)$ taking values in $[0,1]$.  We claim that $f$ is a characteristic function.

  By Theorem~4.3 of \cite{CKN} (see also \cite[Theorem~63]{NY18}), for every $\epsilon>0$, there are $c(\epsilon)>0$ and $\delta(\epsilon)>0$ such that if $p\in \H$, $\alpha>0$, and $\NM_{E}(B_\alpha(p))<\delta(\epsilon) \alpha^{-3}$, then there is a half-space $P^+$ such that
  \begin{equation}\label{eq:CKN stability}
    \fint_{B_{c(\epsilon) \alpha}(p)} |\one_{P^+}(h)-\one_{E}(h)|\ud \cH^4(h) <\epsilon.
  \end{equation}
  (The hypothesis in \cite{CKN} is that $\NM_{E}(B_\alpha(p))<\delta(\epsilon)$, but our definition of $\NM_{E}(B_\alpha(p))$ differs from the definition in~\cite{CKN} by a normalization factor.)

  By the Lebesgue density theorem, for almost every point $p\in U$, we have
  \begin{equation}\label{eq:lebesgue density}
    \lim_{s\to 0}\fint_{B_s(p)}|f(h)-f(p)|\ud\cH^4(h)=0.
  \end{equation}
  Let $p$ be such a point and let $r>0$ be such that $B_r(p)\subset U$.  By \eqref{eq:CKN stability}, for any $0<s<r$, any $\epsilon>0$, and any sufficiently large $i\in \N$ (depending on $s,\epsilon$), there is a half-space $Q^+_i$ with
  $$\fint_{B_{c(\epsilon) s}(p)} |\one_{Q^+_i}(h)-\one_{E_i}(h)|\ud \cH^4(h) < \epsilon.$$
  Choose a half-space $Q^+$ such that for infinitely many $i\in \N$ we have
  $$\fint_{B_{c(\epsilon) s}(p)} |\one_{Q^+}(h)-\one_{E_i}(h)|\ud \cH^4(h) < 2\epsilon.$$
    Then
  \begin{equation}\label{eq:f close to halfplanes}
    \fint_{B_{c(\epsilon) s}(p)} |\one_{Q^+}(h)-f(h)| \ud \cH^4(h) < 3\epsilon.
  \end{equation}
  Since the function $(x\in [0,1])\mapsto x(1-x)$ is nonnegative and 1--Lipschitz,
  $$\fint_{B_{c(\epsilon) s}(p)} f(h)\big(1-f(h)\big ) \ud\cH^4(h) \le 3\epsilon+\fint_{B_{c(\epsilon)s}(p)} \one_{Q^+}(h)\big(1- \one_{Q^+}(h)\big)\ud\cH^4(h) = 3\epsilon.$$
  This holds for all $0<s<r$, so
  $$\lim_{s\to 0}\fint_{B_s(p)}f(h)(1-f(h)) \ud \cH^4(h) = 0.$$
  By \eqref{eq:lebesgue density}, this implies $f(p)(1-f(p))=0$ and thus $f(p)\in \{0,1\}$.

  Thus $f$ is equivalent to a characteristic function on $U$.  Let $F=f^{-1}(1)$.  By weak convergence, $\lim_{i\to\infty} |U\cap(E_i\symdiff F)|=0$.  For any $i\in \N$,
 \begin{multline*}\NM_{F}(U)=\int_{\cL} \NM_{F\cap L}(U\cap L)\ud \cN(L)\\\le \int_{\cL} \Big(\NM_{E_i\cap L}(U\cap L)+\cH^1\big(U\cap L\cap (E_i\symdiff F)\big)\Big)\ud \cN(L)\lesssim \NM_{E_i}(U)+|U\cap(E_i\symdiff F)|.
 \end{multline*}
  Both terms on the right go to zero as $i\to \infty$, so $\NM_{F}(U)=0$, i.e., $F$ is monotone on $U$.
\end{proof}

\begin{cor}\label{cor:limits of monotone}
  Let $U\subset \H$ be a convex bounded open set and let $f\from \H\to [0,1]$ be a $U$--LEM function.  There is a monotone set $E\subset U$ such that $f|_U=\one_E$ up to a measure-zero set.
\end{cor}
\begin{proof}
  Suppose that $E_1,E_2,\dots\subset \H$ are measurable, $E_i$ is $(\frac{1}{i},i)$--extended monotone on $U$ for all $i\in \N$, and $\one_{E_i}$ converges weakly to $f$.  By Lemma~\ref{lem:compare NM omega convex}, for $i>\diam U$ we have
  $$\NM_{E_i}(U)\le \ENM_{E_i,i}(U) \le \frac{1}{i}.$$
  So, by Lemma~\ref{lem:stability locally monotone}, $f|_U=\one_F$ for some set $F\subset U$ that is monotone on $U$.
\end{proof}

\subsection{Locally monotone sets are bounded by rectifiable ruled surfaces}\label{sec:rectifiable ruled}

Here we will describe sets that are monotone on an open subset of $\H$, which we call \emph{locally monotone sets}.  Note that a locally monotone set need not be a half-space; see Example~9.1 of \cite{CKN}.  Regardless, we use the techniques developed in \cite{CheegerKleinerMetricDiff} and \cite{CKN} to describe such sets.

\begin{prop}\label{prop:local monotone rectifiability}
  Let $E\subset \H$ be a measurable set that is monotone on a convex open set $U\subset \H$.  Then
  \begin{enumerate}
  \item
    $U\cap \partial_{\cH^4} E$ has empty interior.
  \item
    For every $p\in U\cap \partial_{\cH^4} E$, there is a horizontal line $L$ through $p$ with $U\cap L\subset \partial_{\cH^4} E$.  If this line is not unique, then $U\cap \mathsf{H}_p\subset \partial_{\cH^4} E$, and we call $p$ a \emph{characteristic point}.
  \item $\partial_{\cH^4} E$ has an approximate tangent plane $T_p$ at every $p\in U\cap \partial_{\cH^4} E$.  The plane $T_p$ is horizontal if and only if $p$ is a characteristic point, and there are only countably many characteristic points in $U$.
  \item
    If $T_p$ is vertical, then it divides $\H$ into two half-spaces $T_p^+$ and $T_p^-$ such that the following holds.  For $\epsilon, t>0$, let
    $$W_{\epsilon,t}^\pm=\{v\in T_p^\pm \cap \overline{B}_t(p) \mid d(v,T_p)>\epsilon t\}$$
    For any $0<\epsilon<\frac{1}{10}$, there is $r>0$ such that if $0<\alpha<r$, then
    $$
    W_{\epsilon,\alpha}^+ \subset \inter_{\cH^4} (E)\qquad\mathrm{and}\qquad  W_{\epsilon,\alpha}^- \subset \inter_{\cH^4} (\H\setminus E).$$
  \end{enumerate}
\end{prop}

We rely on the following proposition and lemmas, which adapt results from \cite{CheegerKleinerMetricDiff}.
\begin{prop}[generalization of~{\cite[Proposition~5.8]{CheegerKleinerMetricDiff}}]\label{prop:local monotonicity along lines}
  Let $E\subset \H$ be a measurable set that is monotone on a convex open set $U\subset \H$.  Let $L$ be a horizontal line and let $p, q\in L$ be points such that $p\ne q$ and the segment $[p,q]\subset L$ is contained in $U$.  We choose the linear order on $L$ so that $p<q$.  Suppose that $q\in \inter_{\cH^4}(E)$.
  \begin{enumerate}
  \item \label{it:p in supp E}
    If $p\in \supp_{\cH^4} (E)$ and $r\in L\cap U$ satisfies $p<r<q$, then $r\in \inter_{\cH^4}(E)$.
  \item \label{it:p in supp Ec}
    If $p\in \supp_{\cH^4} (\H\setminus E)$ and $r\in L\cap U$ satisfies $p<q<r$, then $r\in \inter_{\cH^4}(E)$.
  \end{enumerate}
\end{prop}
\begin{proof}
  Proposition~5.8 of \cite{CheegerKleinerMetricDiff} proves this result in the case that $U=\H$, generalizing Proposition~4.6 of \cite{CheegerKleinerMetricDiff}, which proves it when $E$ is \emph{precisely monotone} (i.e., $M\cap E$ and $M\cap E^c$ are connected sets for every horizontal line $M$). The reasoning in Proposition~5.8 of~\cite{CheegerKleinerMetricDiff} only uses the fact that for almost every line segment $S$ in a small neighborhood of $[p,\max\{q,r\}]$, the intersection $S\cap E$ is monotone. This holds here, so the conclusion of Proposition~5.8 holds here as well.
  For completeness, we will sketch the argument of \cite{CheegerKleinerMetricDiff}.

  For any $x\in \H$, $v_1,v_2\in \mathsf{H}$, let $\gamma_{x,v_1,v_2}\from [0,2]\to \H$ be the broken geodesic
  $$\gamma_{x,v_1,v_2}(t)=\begin{cases}
    x v_1^t & t\in [0,1] \\
    x v_1 v_2^{t-1} & t\in [1,2].
  \end{cases}$$

  In case \eqref{it:p in supp E}, we have $p<r< q$ with $p\in \supp_{\cH^4}(E)$ and $q\in \inter_{\cH^4}(E)$. Given an $\epsilon>0$, one considers the paths $\gamma_{x,v_1,v_2}$ where $x\in B_\epsilon(p)\cap E$ and $v_1,v_2\in \mathsf{H}$ satisfy $\|v_i-(p^{-1}r)^{\frac{1}{2}}\|<\epsilon$. Then $\gamma_{x,v_1,v_2}(0)$ is close to $p$, $\gamma_{x,v_1,v_2}(2)$ is close to $r$, and $\gamma_{x,v_1,v_2}$ lies in a small neighborhood of $[p,q]$. Further, for any $x$, we can vary $v_1$ and $v_2$ so that $\gamma_{x,v_1,v_2}(2)=xv_1v_2$ covers a neighborhood of $r$.

  Suppose that $E$ is precisely monotone and that $q\in \inter(E)$. Let $x, v_1, v_2$ be as above and let $\lambda_1(t)=xv_1^t$ and $\lambda_2(t)=xv_1v_2^t$ be the two segments of $\gamma_{x,v_1,v_2}$. These are two lines that are close to $L$, so there are $t_1,t_2>1$ such that $\lambda_i(t_i)$ is close to $q$. Since $q\in \inter(E)$, if $\epsilon$ is sufficiently small, then $\lambda_i(t_i)\in E$. Since $\lambda_1(0)=x\in E$ and $\lambda_1(t_1)\in E$, precise monotonicity implies $\lambda_1(1)\in E$, and since $\lambda_2(0)=\lambda_1(1)\in E$ and $\lambda_2(t_2)\in E$, we have $\lambda_2(1)=xv_1v_2\in E$.  If we fix $x$ and let $v_1$ and $v_2$ vary, then $xv_1v_2$ covers a neighborhood of $r$, so $r\in \inter(E)$.

  In our case, $E$ is not precisely monotone and $q\in \inter_{\cH^4}(E)$, but the reasoning above still holds for almost every triple $(x, v_1, v_2)$. Since $\cH^4(B_\epsilon(p)\cap E)>0$, there is an $x\in B_\epsilon(p)\cap E)$ such that $xv_1v_2\in E$ for almost every pair $(v_1,v_2)$. Therefore, $r\in \inter_{\cH^4}(E)$.

  In case \eqref{it:p in supp Ec}, we have $p<q<r$ with $p\in \supp_{\cH^4}(\H\setminus E)$ and $q\in \inter_{\cH^4}(E)$.
  Let $s\in L\cap U$ be such that $p<q<r<s$ and consider $\gamma_{x,v_1,v_2}$ such that
  $x\in B_\epsilon(p)\setminus E$, $\|v_1- p^{-1}s\|<\epsilon$, and $\|v_2- s^{-1} r\|<\epsilon$. That is, $\gamma_{x,v_1,v_2}$ is a path from a neighborhood of $p$ to a neighborhood of $s$ to a neighborhood of $r$. Again, for any $x$, we can vary $v_1$ and $v_2$ so that $\gamma_{x,v_1,v_2}(2)=xv_1v_2$ covers a neighborhood of $r$. If $\epsilon$ is sufficiently small, we have $\gamma_{x,v_1,v_2}([0,2])\subset U$.

  Suppose again that $E$ is precisely monotone and that $q\in \inter(E)$. Let $\lambda_1(t)=xv_1^t$ and $\lambda_2(t)=xv_1v_2^t$. Since $\lambda_1$ and $\lambda_2$ are both close to $L$, if $\epsilon$ is sufficiently small, there are $t_1\in (0,1)$ and $t_2>1$ such that $\lambda_i(t_i)$ is close to $q$ and $\lambda_i(t_i)\in E$. Since $\lambda_1(0)=x\not\in E$ and $\lambda_1(t_1)\in E$, we have $\lambda_1(1)\in E$, and since $\lambda_2(0)=\lambda_1(1)\in E$ and $\lambda_2(t_2)\in E$, we have $\lambda_2(1)=xv_1v_2\in E$. For any fixed $x$, as $v_1$ and $v_2$ vary, $xv_1v_2$ covers a neighborhood of $r$.

  Again, when $E$ is not precisely monotone and $q\in \inter_{\cH^4}(E)$, the reasoning above fails for a null set of triples $(x, v_1, v_2)$. Since $p\in \supp_{\cH^4} (\H\setminus E)$, there is an $x\in B_\epsilon(p)\cap (\H\setminus E)$ such that $xv_1v_2\in E$ for all but a measure zero set of pairs $(v_1,v_2)$, so $r\in \inter_{\cH^4}(E)$.
\end{proof}

By Proposition~\ref{prop:local monotonicity along lines} and the proof of \cite[Lemma~4.8]{CheegerKleinerMetricDiff}, we get the following lemma.
\begin{lemma}[{generalization of \cite[Lemma~4.8]{CheegerKleinerMetricDiff}}]\label{lem:monotone boundary U-convex}
  Let $E\subset \H$ be a measurable set that is monotone on a convex open set $U\subset \H$.  If $L$ is a horizontal line such that $L\cap U$ contains at least two points of $\partial_{\cH^4} E$, then $L\cap U\subset \partial_{\cH^4} E$.
\end{lemma}
\begin{proof}
  Let $I=L\cap U$.  Let $p,q\in I\cap \partial_{\cH^4} E$ be distinct points.  Choose the linear order on $L$ so that $p<q$.  Let $r\in I$ be such that $q<r$.  By part~\eqref{it:p in supp E} of Proposition~\ref{prop:local monotonicity along lines}, if $r\in \inter_{\cH^4} (E)$, then $q\in \inter_{\cH^4} (E)$, which is a contradiction.  Likewise, if $r\in \inter_{\cH^4} (\H\setminus E)$, then $q\in \inter_{\cH^4} (\H\setminus E)$, which is a contradiction, so $r\in \partial_{\cH^4} E$.  Thus $[q,\infty)\cap I \subset \partial_{\cH^4} E$.  By symmetry,
  $I\setminus (p,q)=I\cap \bigl((-\infty,p] \cup [q,\infty)\bigr)\subset \partial_{\cH^4} E$
  for any distinct points $p, q\in I\cap \partial_{\cH^4} E$. Let $r, s \in I\cap [q,\infty)$ be such that $r<s$.  Then $r,s\in I\cap \partial_{\cH^4} E$, so $I\setminus (r,s)\subset \partial_{\cH^4} E$.  Since $(r,s)$ and $(p,q)$ are disjoint, $I\subset \partial_{\cH^4} E$.
\end{proof}

Likewise, the following lemma is based on the proof of Lemma~4.9 of \cite{CheegerKleinerMetricDiff}.
\begin{lemma}[{generalization of \cite[Lemma~4.9]{CheegerKleinerMetricDiff}}]\label{lem:monotone is union of lines}
  Let $E\subset \H$ be a measurable set that is monotone on a convex open set $U\subset \H$.  For every $p\in U\cap \partial_{\cH^4} E$, there is a horizontal line $L$ such that $p\in L$ and $L\cap U\subset \partial_{\cH^4} E$.
\end{lemma}
\begin{proof}
  Let $B\subset U$ be a ball centered at $p$ and let $\mathsf{H}_p$ be the horizontal plane centered at $p$.  Let $B'=B\setminus \{p\}$.  Suppose by way of contradiction that $\mathsf{H}_p\cap B'\cap \partial_{\cH^4} E =\emptyset$.  Since $\mathsf{H}_p\cap B'$ is connected, we have $\mathsf{H}_p\cap B'\subset \inter_{\cH^4} (E)$ or $\mathsf{H}_p \cap B'\subset \inter_{\cH^4} (\H\setminus E)$.  Without loss of generality, we assume that $\mathsf{H}_p\cap B'\subset \inter_{\cH^4} (E)$.

  Let $M$ be a line through $p$ and let $q,r\in M\cap B$ be two points on opposite sides of $p$.  Then $q,r\in \inter_{\cH^4}(E)$, so by part~\eqref{it:p in supp E} of Proposition~\ref{prop:local monotonicity along lines}, we have $p\in \inter_{\cH^4}(E)$.  This is a contradiction, so there exists some point $q$ lying in $\mathsf{H}_p\cap B'\cap  \partial_{\cH^4} E$.  Let $L$ be the line containing $p$ and $q$; then by Lemma~\ref{lem:monotone boundary U-convex}, $L\cap U\subset \partial_{\cH^4} E$, as desired.
\end{proof}

The fact that $U\cap \partial_{\cH^4} E$ has empty interior also follows from the techniques of \cite{CheegerKleinerMetricDiff}.
\begin{lemma}\label{lem:empty interior}
  If $E$ and $U$ are as in Lemma~\ref{lem:monotone is union of lines}, then $U\cap \partial_{\cH^4} E$ has empty interior.
\end{lemma}
\begin{proof}
  The measure-theoretic version of Lemma~4.12 of \cite{CheegerKleinerMetricDiff}, whose proof appears in (part (4) of) the proof of Theorem~5.1 of~\cite{CheegerKleinerMetricDiff}, asserts that if $F\subset \H$ is monotone on $\H$, then $\partial_{\cH^4} F\ne \H$.  That proof  relies on the monotonicity of a configuration of line segments, and it directly shows that there is a large enough universal constant  $r>0$ such that this configuration lies in the ball $B_r(\mathbf{0})$.  Consequently, if $B_r(\mathbf{0})\subset U$, then there is a point $p\in B_r(\mathbf{0})$ such that $p\not \in \partial_{\cH^4} E$.  By rescaling and translation, this is true with $B_r(\mathbf{0})$ replaced by an arbitrary ball, and thus $\inter_{\cH^4}(E)\cup \inter_{\cH^4}(\R\setminus E)$ is dense in $U$.
\end{proof}

Lemma~\ref{lem:empty interior} proves part (1) of Proposition~\ref{prop:local monotone rectifiability}. Lemma~\ref{lem:monotone boundary U-convex} and Lemma~\ref{lem:monotone is union of lines}  imply the first half of part (2) of Proposition~\ref{prop:local monotone rectifiability}.  Before proving the rest of Proposition~\ref{prop:local monotone rectifiability}, we make the following definition.

\begin{defn}
  Let $U\subset \H$ be a convex open set and let $A\subset \H$.  We say that $A$ is \emph{$U$--ruled} if for all $L\in \cL$, if $L\cap U$ intersects $A$ in two points, then $L\cap U\subset A$.  We call such a line $L$ a $U$--\emph{ruling} of $A$.
\end{defn}

Lemmas~\ref{lem:monotone boundary U-convex}--\ref{lem:empty interior} imply that $U\cap \partial_{\cH^4} E$ is $U$--ruled and has empty interior.  We will prove the rest of Proposition~\ref{prop:local monotone rectifiability} by studying  lines in the boundary of such a set. The following lemma is based on Step B3 in Section~8.2 of \cite{CKN}, which shows that the boundary of a monotone set cannot contain skew lines.
\begin{lemma}\label{lem:convex hull of perpendicular lines}
  Let $M_1$ be the line $\langle X\rangle$ and let $M_2$ be the line $Z\langle Y\rangle$.  There exists $r_0>1$ such that any $\overline{B}_{r_0}$--ruled set containing $(M_1\cup M_2)\cap \overline{B}_{r_0}$ has nonempty interior.
\end{lemma}
\begin{proof}
  Let $r_0$ be large enough that $[-2,2]^3\subset \overline{B}_{r_0}$.  Let $E$ be a $\overline{B}_{r_0}$--ruled set with $\overline{B}_{r_0}$--rulings $M_1,M_2\in \cL$. By Lemma~\ref{lem:hyperboloid lemma}, there is a hyperbola $S\subset \mathsf{H}$, asymptotic to the $x$--axis and the $y$--axis, such that every tangent line of $S$ has a unique horizontal lift that intersects $M_1$ and $M_2$.  Indeed, for every $t\ne 0$, the points $X^t\in M_1$ and $ZY^{2/t}\in M_2$ are connected by a horizontal line
  $$\forall u\in \R,\qquad L_t(u)\eqdef X^t \left(-t,\frac{2}{t},0\right)^u = \left( (1-u) t, \frac{2u}{t}, u\right).$$
  For $t\in [-2,-1]\cup [1,2]$ and $u\in [0,1]$, the point $L_t(u)$ lies on a horizontal line segment connecting two points in $E$, so $L_t(u)\in E$.  The resulting family of points
  $$S\eqdef\{L_t(u)\mid t\in [-2,-1]\cup [1,2], u\in [0,1]\}\subset E$$
  consists of two disjoint embedded surfaces.

  Let
  $$w \eqdef  L_{\sqrt{2}}\left(\frac{1}{2}\right) = \left(\frac{\sqrt{2}}{2},\frac{\sqrt{2}}{2},\frac{1}{2}\right)$$
  and let $w'=s_{-1,-1}(w)= L_{-\sqrt{2}}(\frac{1}{2})$. Let $M$ be the horizontal line from $w$ to $w'$. Then $M$ intersects $S$ twice, at $w$ and $w'$, so $M \cap \overline{B}_{r_0}\subset E$. One calculates
  \begin{align*}
    \frac{\ud}{\ud t} L_t(u)\big|_{(t,u)=(\sqrt{2}, \frac{1}{2})}
    &= ( 1-u, - 2u t^{-2}, 0) \big|_{(t,u)=(\sqrt{2}, \frac{1}{2})} = \left(\frac{1}{2},-\frac{1}{2},0\right)\\
    \frac{\ud}{\ud u} L_t(u) \big|_{(t,u)=(\sqrt{2}, \frac{1}{2})} &= ( -t, 2t^{-1}, 1) \big|_{(t,u)=(\sqrt{2}, \frac{1}{2})} = ( -\sqrt{2}, \sqrt{2}, 1),
  \end{align*}
  so $M$ intersects $S$ transversally at $w$ and $w'$. By transversality, any horizontal line $M'$ close to $M$ intersects $S$ near $w$ and $w'$, so $M' \cap \overline{B}_{r_0}\subset E$. These lines cover a neighborhood of $M$, so $E$ contains a nonempty open set.
\end{proof}

As shown in the next lemma, for any pair of skew lines, there is an automorphism of $\H$ that sends them to $M_1$ and $M_2$.  The next lemma uses this fact to show that nearby skew lines in $\partial_{\cH^4} E$ must have nearly parallel projections.  For $\phi\in\R$, let $R_\phi\from \H\to \H$ be the rotation by angle $\phi$ around the $z$--axis.
\begin{lemma}\label{lem:convex hull of skew lines}
  Let $r_0$ be as in Lemma~\ref{lem:convex hull of perpendicular lines}.  Let $L_1,L_2\in \cL$ be skew lines and let $p\in \mathsf{H}$ be the intersection of $\uppi(L_1)$ and $\uppi(L_2)$.  Suppose that the angle between $\uppi(L_1)$ and $\uppi(L_2)$ is $\theta\in (0,\frac{\pi}{2})$. For $i\in \{1,2\}$, let $q_i\in \H$ be the point where $\uppi^{-1}(p)$ intersects $L_i$.  Suppose that
  \begin{equation}\label{eq:r0 hypothesis}
    d(q_1,q_2) \le \frac{\sqrt{\theta}}{r_0\sqrt{2}}.
  \end{equation}
If $L_1,L_2\in \cL$ are $\overline{B}_1(q_1)$--rulings of an $\overline{B}_1(q_1)$--ruled set $S$, then $S$ has nonempty interior.
\end{lemma}

\begin{proof}
  After applying a translation and rotation and possibly replacing $S$ with $s_{1,-1}(S)$, we may suppose that $q_1=\mathbf{0}$, $q_2=Z^{h}$ for some $h>0$ and that $\uppi(L_1)$ and $\uppi(L_2)$ form angles of $\frac{\theta}{2}$ with the $x$--axis.  (We cannot control which line forms a positive angle with the $x$--axis and which line forms a negative angle.)  Let $t=\tan\frac{\theta}{2}\in (0,1)$ so that the lines
  $$\uppi\Big(s_{\sqrt{t}, \frac{1}{\sqrt{t}}}(L_1)\Big)\qquad \mathrm{and}\qquad  \uppi\Big(s_{\sqrt{t}, \frac{1}{\sqrt{t}}}(L_2)\Big)$$
  are perpendicular.  There is an angle $\phi=\pm \frac{\pi}{4}$ such that if
  $$f\eqdef R_\phi\circ s_{\frac{1}{\sqrt{h}}, \frac{1}{\sqrt{h}}}\circ s_{\sqrt{t}, \frac{1}{\sqrt{t}}},$$
  then $f(L_1)=M_1$ and $f(L_2)=M_2$, where $M_1,M_2$ are the lines in Lemma~\ref{lem:convex hull of perpendicular lines}.  Now, by the ball-box inequality and our hypothesis on $d(q_1,q_2)$,
  $$\Lip(f^{-1})=\frac{\sqrt{h}}{\sqrt{t}} \stackrel{\eqref{eq:metric approximation}}{\le} \frac{d(q_1,q_2)}{\sqrt{\tan \theta/2}}\le \frac{d(q_1,q_2)}{\sqrt{\theta/2}} \stackrel{\eqref{eq:r0 hypothesis}}{\le} \frac{1}{r_0}.$$
  Thus, $f^{-1}(\overline{B}_{r_0})\subset \overline{B}_{r_0 \Lip(f^{-1})}\subset \overline{B}_1$, or $\overline{B}_{r_0}\subset f(\overline{B}_1)$. Since $f(S)$ is a $f(\overline{B}_1)$--ruled set and $M_1$ and $M_2$ are $f(\overline{B}_1)$--rulings of $f(S)$, by Lemma~\ref{lem:convex hull of perpendicular lines}, $f(S)$ has nonempty interior and thus $S$ has nonempty interior.
\end{proof}

It follows from Lemma~\ref{lem:empty interior} and Lemma~\ref{lem:convex hull of skew lines} that two lines in $\partial_{\cH^4} E$ with different angles must either intersect or stay at least a definite distance apart.  In the terminology of \cite{CKN}, every pair of rulings of $\partial_{\cH^4} E$ must form a degenerate initial condition.
\begin{lemma}\label{lem:definite angles intersect}
  For any $\epsilon>0$, there is $\delta>0$ such that if $S$ is a $\overline{B}_1$--ruled set with empty interior and $L_1, L_2$ are $\overline{B}_1$--rulings of $S$ that intersect $\overline{B}_\delta$ and such that $\angle(\uppi(L_1),\uppi(L_2))>\epsilon$, then $L_1$ and $L_2$ intersect.
\end{lemma}
\begin{proof}
  We  suppose that $0<\epsilon<1$ and take $\delta=\frac{\epsilon^{\frac32}}{100 r_0}\le \frac{1}{100},$ where $r_0$ is as in Lemma~\ref{lem:convex hull of perpendicular lines}.

  Let $p\in \mathsf{H}$ be the intersection of the projections $\uppi(L_1)$ and $\uppi(L_2)$.  Since $\uppi(\overline{B}_\delta)$ is the ball $B^{\mathsf{H}}_\delta$ of radius $\delta$ in $\mathsf{H}$, the projections intersect $B^{\mathsf{H}}_\delta$ and form an angle of at least $\epsilon$, so
  $$\|p\|\le \frac{\delta}{\sin \frac{\epsilon}{2}}\le \frac{4\delta}{\epsilon}<\frac{1}{4}.$$
  For $i\in \{1,2\}$, let $q_i=\uppi^{-1}(p)\cap L_i$.  By assumption, $L_1$ and $L_2$ intersect $\overline{B}_\delta\subset B_{2\delta}$, so if $b_i\in L_i\cap \overline{B}_\delta$, then
  $$d(\0, q_i)\le d(\0,b_i)+d(b_i,q_i)=d(\0,b_i)+\|\uppi(b_i)-\uppi(q_i)\| \le d(\0,b_i)+\|\uppi(b_i)\|+\|p\| \le 3\delta + \|p\|.$$
  In particular, $d(\0,q_i)\le \frac{1}{2}$.  Hence $\overline{B}_{\frac{1}{2}}(q_1)\subset \overline{B}_1$, so $S$ is a $\overline{B}_{\frac{1}{2}}(q_1)$--ruled set.  Further,
  $$d(q_1,q_2)\le 2\|p\|+6\delta < \frac{20 \delta}{ \epsilon}\le \frac{\sqrt{ \epsilon}}{5 r_0}.$$
Because $S$ has empty interior, Lemma~\ref{lem:convex hull of skew lines} implies that $L_1$ and $L_2$ cannot be skew lines, and must therefore intersect.
\end{proof}

The next lemma completes the proof of part (2) of Proposition~\ref{prop:local monotone rectifiability}.
\begin{lemma}\label{lem:intersecting lines and horizontal planes}
  Suppose that $U$ is a convex open set and that $E\subset \H$ is monotone on $U$.  Let $p\in U$, and let $L_1$ and $L_2$ be two distinct $U$--rulings of $\partial_{\cH^4} E$ that intersect at $p$. Then $U\cap \mathsf{H}_p\subset \partial_{\cH^4} E$, and there is a neighborhood $A$ containing $p$ such that $A\cap \partial_{\cH^4} E=A\cap \mathsf{H}_p$, where we recall that $\mathsf{H}_p$ denotes  the horizontal plane through $p$.
\end{lemma}
\begin{proof}
  Since $U$ is convex, $\partial_{\cH^4} E$ is $U$--ruled.
  After translating and applying an automorphism, we may suppose that $p=\mathbf{0}$ and that $L_1$ and $L_2$ are the $x$--axis and $y$--axis, respectively. Set $\epsilon=\frac{1}{40}$ and let $\delta>0$ satisfy Lemma~\ref{lem:definite angles intersect}. Suppose that $\overline{B}_\delta\subset U$.

  Fix $q\in B_{\frac{\delta}{8}}\cap \partial_{\cH^4} E$.  By Lemma~\ref{lem:monotone is union of lines}, $\partial_{\cH^4} E $ has a $U$--ruling $M_q$ that passes through $q$.  We will show that $M_q$ intersects both $L_1$ and $L_2$ and that any such line passes through $p$.

  For any horizontal line $L$, let $\overline{L}=\uppi(L)$.  Either $\angle(\overline{L_1},\overline{M_q})\ge \frac{\pi}{4}$ or $\angle(\overline{L_2},\overline{M_q})\ge \frac{\pi}{4}$. Therefore, by Lemma~\ref{lem:definite angles intersect}, $M_q$ intersects either $L_1$ or $L_2$.  Suppose by way of contradiction that $M_q$ intersects $L_2$ but not $L_1$.  By Lemma~\ref{lem:definite angles intersect}, this implies that $\angle(\overline{L_1},\overline{M_q})\le \epsilon$.  Let $r$ be the intersection of $\overline{M_q}$ with $\overline{L_2}$ and let $t=d(p,r)>0$ (see Figure~\ref{fig:line intersects L2}).  Straightforward trigonometry shows that $t<\frac{\delta}{4}$.

  \begin{figure}[h]
    \begin{centering}
      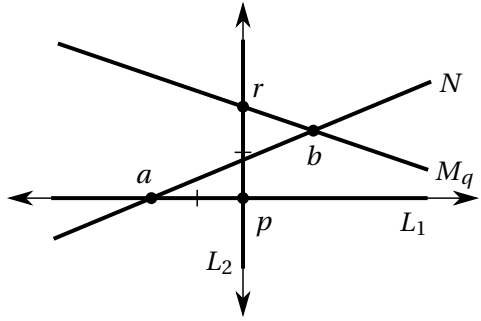
    \end{centering}
    \caption{If line $M_q$ intersects the $y$--axis $L_2$ but not the $x$--axis $L_1$, there must be a line $N$ intersecting $L_1$ and $M_q$ as seen above.  Lines above are projected to $\mathsf{H}$ by $\uppi$. \label{fig:line intersects L2}}
  \end{figure}

  Let $a=pX^{-t}\in L_1$.  By Lemma~\ref{lem:hyperboloid lemma}, there is a unique point $b\in M_q$ such that there is a horizontal line $N$ that passes through $a$ and $b$. Indeed, since $r$, $p$, $a$, and $b$ are the vertices of a quadrilateral $Q$ in $\H$ whose sides are horizontal lines, the projection $\uppi(Q)$ has zero signed area. Since the triangle $\triangle\uppi(p)\uppi(r)\uppi(a)$ has area $\frac{t^2}{2}$, the triangle $\triangle\uppi(b)\uppi(r)\uppi(a)$ must also have area $\frac{t^2}{2}$, so $\uppi(b)$ is the intersection of $\overline{M_q}$ with the line $\langle X+Y\rangle$. Because $\overline{M_q}$ has slope between $-\epsilon$ and $\epsilon$, this implies that $|\uppi(b)-(t,t)|\le 4\epsilon t \le \frac{t}{10}$. In particular, $d(r,b)=|\uppi(r)-\uppi(b)|\le 2t$, $\angle(\overline{L_1},\overline{N})> \epsilon$, and $\angle(\overline{L_2},\overline{N})> \epsilon$. Then $d(p,b)\le d(p,r)+d(r,b)\le 3t<\delta$, so $b\in \overline{B}_\delta$.

  Since $a, b\in U\cap \partial_{\cH^4} E$, $N$ is a $U$--ruling of $\partial_{\cH^4} E$. By Lemma~\ref{lem:definite angles intersect} and the fact that $\angle(\overline{L_2},\overline{N})> \epsilon$, $N$ intersects $L_2$. That is, $L_1$, $L_2$, and $N$ are three distinct lines in $\H$ that intersect pairwise. If three distinct lines intersect pairwise, then they must all intersect at the same point. Otherwise, their projections to $\mathsf{H}$ would contain a non-degenerate triangle that lifts to a horizontal closed curve in $\H$, but this is impossible since the signed area of the projection of a horizontal closed curve must vanish. But $L_1$ intersects $N$ at $a$ and intersects $L_2$ at $p$, where $d(p,a)=t>0$ by construction. This is a contradiction, so $M_q$ intersects $L_1$ and $L_2$. Since $M_q$, $L_1$, and $L_2$ are distinct lines that intersect pairwise, $M_q$ must intersect $L_1$ and $L_2$ at $p$.

  Hence, every point $q\in B_{\frac{\delta}{2}}\cap \partial_{\cH^4} E$ lies on the horizontal plane $\mathsf{H}_p$ through $p$.  The measure-theoretic boundary of $E$ disconnects $B_{\frac{\delta}{2}}$, so $B_{\frac{\delta}{2}}\cap \partial_{\cH^4} E=B_{\frac{\delta}{2}}\cap \mathsf{H}_p$.

  Consequently, any line $L$ through $p$ intersects $U\cap \partial_{\cH^4} E$ in at least two points, so $U\cap L \subset \partial_{\cH^4} E$.  The union of all such lines is $\mathsf{H}_p$, so $U\cap \mathsf{H}_p \subset \partial_{\cH^4} E$
\end{proof}

Finally, we prove parts (3) and (4) of Proposition~\ref{prop:local monotone rectifiability}.

\begin{proof}[{Proof of parts (3) and (4) of Proposition~\ref{prop:local monotone rectifiability}}]
  Due to Lemma~\ref{lem:intersecting lines and horizontal planes}, if $p$ is a characteristic point, then $\partial_{\cH^4} E$ has a horizontal approximate tangent plane at $p$.  Lemma~\ref{lem:intersecting lines and horizontal planes} also implies that if $p$ is a characteristic point, then there is a ball $B$ such that $B$ contains no characteristic points other than $p$.  That is, the characteristic points form a discrete subset of $\H$; since $\H$ is separable, there are only countably many characteristic points.

  Let $p\in U\cap \partial_{\cH^4} E$ be a non-characteristic point, so that there is a unique line $L$ through $p$.  Let $V$ be the vertical plane that contains $L$. Fix $0<\epsilon<\frac{1}{10}$.  We claim that there is $r>0$ such that if $0<\alpha\le r$, then $\overline{B}_{\alpha}(p)\cap \partial_{\cH^4} E$ is contained in the $\epsilon \alpha$--neighborhood of $V$.

  We translate, rotate, and rescale so that $p=\mathbf{0}$, $L$ is the $x$--axis, and $\overline{B}_1$ is a subset of $U$ that contains no characteristic points.  Then $V=V_0$ is the $xz$--plane.  Let $\Pi\from \H \to V_0$ be the projection to $V_0$ along cosets of $\langle Y\rangle$, as in Section~\ref{sec:intrinsic graphs}, so that $\Pi(x,y,z)=(x,0,z-\frac{xy}{2})$.

  For each point $s\in \overline{B}_{1}\cap \partial_{\cH^4} E$, there is a unique $U$--ruling $M_s$ passing through $s$.  By Lemma~\ref{lem:definite angles intersect}, there is  $\delta\in (0,1)$ such that $\angle(M_s,L)<\frac{\epsilon^2}{200}$ for every $s\in \overline{B}_{\delta}\cap \partial_{\cH^4} E$.  Let $r=\min\{\delta, \frac{\epsilon}{80}\}$ and let $0<\alpha\le r$.  Let $q\in \overline{B}_{\alpha}\cap \partial_\mu E$ and suppose by way of contradiction that $d(q,V_0)=|y(q)|> \epsilon \alpha$.  Without loss of generality, we may suppose that $y(q)>\epsilon \alpha$.

  Let $m\in \R$ be the slope of $\uppi(M_q)$, so that $M_q=q\cdot \langle X+mY\rangle$. Let $\gamma(t)=q\cdot (X+mY)^t$ parametrize $M_q$.  Then
  $$|m|=|\sin \angle(M_s,L)|<\frac{\epsilon^2}{200}.$$

  Since $q\in \overline{B}_{\alpha}\subset B_{2\alpha}$, we have $\Pi(q)\in B_{4\alpha}$ and thus $|z(\Pi(q))|\le 16\alpha^2$.  By~\eqref{eq:horizontal curve eq}, for all $t\in \R$,
  $$\frac{\ud}{\ud t} z\Big(\Pi\big(\gamma(t)\big)\Big) = - y\big(\gamma(t)\big) = -y(q) - mt.$$
  Consequently,
  $$\forall t\in \R,\qquad z\Big(\Pi\big(\gamma(t)\big)\Big)=z(q) - y(q)t - m \frac{t^2}{2}.$$
  Letting $s=\frac{20\alpha}{ \epsilon}$, it follows that
  $$z\Big(\Pi\big(\gamma(s)\big)\Big)\le 16\alpha^2 - \alpha \epsilon s + \frac{\epsilon^2}{200}\cdot \frac{s^2}{2} \le -3\alpha^2$$
  and
  $$z\Big(\Pi\big(\gamma(-s)\big)\Big)\ge - 16\alpha^2 + \alpha \epsilon s - \frac{\epsilon^2}{200} \cdot \frac{s^2}{2} \ge 3\alpha^2.$$
So, there is $t$ with $|t|<s\le \frac{1}{4}$ and $z(\Pi(\gamma(t)))=0$, i.e., $\Pi(\gamma(t))\in L$.  The coset $N=\gamma(t) \langle Y\rangle$ is thus a horizontal line that intersects $M_q$ at $\gamma(t)$ and intersects $L$ at $\Pi(\gamma(t))$.  Since
  $$d\big(\mathbf{0},\gamma(t)\big)\le d(\mathbf{0},q)+|t| \le 2\alpha + \frac{1}{4}\le \frac{1}{2},$$
  and
  $$d\Big(\mathbf{0},\Pi\big(\gamma(t)\big)\Big)\le 2 d\big(\mathbf{0},\gamma(t)\big)\le 1,$$
  $\gamma(t)$ and $\Pi(\gamma(t))$ belong to  $\overline{B}_1\cap \partial_{\cH^4} E$, so $N\cap \overline{B}_1\subset \partial_{\cH^4} E$.  Then $M_q$ and $N$ are distinct $U$--rulings of $\partial_{\cH^4} E$ passing through $\gamma(t)$, which contradicts the fact that there are no characteristic points in $\overline{B}_1$.  Therefore, $d(q,V_0)\le \epsilon \alpha$ for all $q\in \overline{B}_{\alpha}\cap \partial_{\cH^4} E$.

  Let $T_p=V_0$ and let $T_p^+$ and $T_p^-$ be the corresponding half-spaces.  The argument above shows that for any $0<\alpha \le r$, the sets $W_{\epsilon,\alpha}^\pm$ are disjoint from $\partial_{\cH^4} E$, so each set is contained in either $\inter_{\cH^4} (E)$ or $\inter_{\cH^4} (\H\setminus E)$.

  Consider $W_{\epsilon,r}^+$ and $W_{\epsilon,r}^-$.  Every line sufficiently close to the $y$--axis intersects both of these sets, so if both are contained in $\inter_{\cH^4} (E)$, then by Proposition~\ref{prop:local monotonicity along lines}, $p\in \inter_{\cH^4} (E)$ as well.  Likewise, if both are contained in $\inter_{\cH^4} (\H\setminus E)$, then $p\in \inter_{\cH^4} (\H\setminus E)$.  Either of these conclusions  is a contradiction, so one of $W_{\epsilon,r}^+,W_{\epsilon,r}^-$ is contained in $\inter_{\cH^4}(E)$ and the other is contained in $\inter_{\cH^4} (E)$.  If necessary, we switch $T_p^+$ and $T_p^-$ so that $W^+_{\epsilon,r}\subset \inter_{\cH^4} (E)$.

  We claim that $W^+_{\epsilon,\alpha}\subset \inter_{\cH^4} (E)$ for every $\alpha\in (0,r]$.  Fix $0<\beta\le r$ with $\frac{\beta}{2}<\alpha<\beta$.  Then $W^+_{\epsilon,\alpha}$ intersects $W^+_{\epsilon,\beta}$, so if $W^+_{\epsilon,\beta}\subset \inter_{\cH^4} (E)$, then $W^+_{\epsilon,\alpha}\subset \inter_{\cH^4} (E)$ as well.  By induction, $W^+_{\epsilon,\alpha}\subset \inter_{\cH^4} (E)$ for all $0<\alpha\le r$.  Likewise, $W^-_{\epsilon,\alpha}\subset \inter_{\cH^4} (\H\setminus E)$ for all $0<\alpha\le r$.
\end{proof}

\subsection{Stability of extended monotone sets}\label{sec:stability of extended monotone sets}

Here we prove Proposition~\ref{prop:stability of extended monotone}.  We show that there are $\nu>0$ and $R>0$ such that if $E$ is a set that is $(\nu,R)$--extended monotone on $\overline{B}_1$, then $E$ is close to a half-space on $\overline{B}_1$.  If $R' \ge R$ and $\nu' R'\le \nu R$, then $(\nu',R')$--extended monotonicity implies $(\nu, R)$--extended monotonicity, so this implies the full proposition.

To prove this, it suffices to show that if $f$ is a $\overline{B}_1$--LEM function, then $f|_{\overline{B}_1}$ is the characteristic function of a half-space.  Suppose that $f$ is a weak limit of a sequence $(\one_{E_i})_i$, where $E_1,E_2,\dots\subset \H$ are sets such that $E_i$ is $(\frac{1}{i},i)$--extended monotone on $\overline{B}_1$.  By Corollary~\ref{cor:limits of monotone}, $f|_{\overline{B}_1}$ is the characteristic function of a locally monotone subset $F\subset \overline{B}_1$, but this result only uses the fact that each $E_i$ is $\frac{1}{i}$--monotone on $\overline{B}_1$.  In this section, we  improve Corollary~\ref{cor:limits of monotone} by using the stronger hypothesis that  the $E_i$ are extended monotone sets.

The first issue is that $\ENM_{E_i,R}(\overline{B}_1)$ only depends on the intersection of $E_i$ with lines through $\overline{B}_1$.  These lines don't cover all of $\H$, so a $\overline{B}_1$--LEM function need not take values in $\{0,1\}$ outside $\overline{B}_1$.  The following lemma shows that it is takes values in $\{0,1\}$ on lines that intersect the boundary of $F$ transversally.  For $p \in \H$ and $V\in \mathsf{H}$ a horizontal vector, the coset $p\langle V\rangle$ is a horizontal line.  Let $p\langle V\rangle^+=\{pV^t\mid t>0\}$ and let $p\langle V\rangle^-=\{pV^t\mid t<0\}$.
\begin{lemma}\label{lem:monotone transverse intersections}
  Let $f$ be a $\overline{B}_1$--LEM function and let $F=f^{-1}(1)\cap \overline{B}_1$ be the corresponding locally monotone set.  Let $p\in \overline{B}_1\cap  \partial_{\cH^4} F $ be a point with a vertical approximate tangent plane $T_p$ and let $V\in \mathsf{H}_p$ be a horizontal vector pointing into $T_p^+$.  Then,
  \begin{equation}\label{eq:ray inclusions}
  p\langle V\rangle^+ \subset \inter_{\cH^4}\big( f^{-1}(1)\big)\qquad\mathrm{and}\qquad p\langle V\rangle^-\subset \inter_{\cH^4} \big(f^{-1}(0)\big).
  \end{equation}
\end{lemma}
\begin{proof}
  Let $E_i\subset \H$ be a sequence of sets such that $E_i$ is $(\frac{1}{i},i)$--monotone on $\overline{B}_1$ and $\one_{E_i}$ converges weakly to $f$.  Let $L=p\langle V\rangle$, $L^\pm=p\langle V\rangle^\pm$ and $\theta=\angle(V, T_p)$.  Let $\epsilon=\frac{\theta}{20}$ and let $W^\pm_{\epsilon,t}$ be as in Proposition~\ref{prop:local monotone rectifiability}.  For  $t>0$, $L^\pm$ intersects $W^\pm_{\epsilon,t}$ in an interval of length at least $\frac{t}{2}$.

  Fix $t>0$ and let $q=pV^t$.  For the first inclusion in~\eqref{eq:ray inclusions}, the goal is to demonstrate that $q\in \inter_{\cH^4}( f^{-1}(1))$.  Let $0<\alpha<\frac{t}{2}$ be a radius such that $\overline{B}_\alpha(p)\subset \overline{B}_1$, $W^+_{\epsilon,\alpha}\subset F$ up to a null set, and $W^-_{\epsilon,\alpha}\subset \H\setminus F$ up to a null set.  For any $\delta>0$, let $\cK_\delta\subset \cL$ be the set of lines of the form $q'\langle V'\rangle$ where $q'\in \overline{B}_{\delta}(q)$ and $V'\in \mathsf{H}$ is a horizontal vector such that $\angle(V,V')<\delta$.  For $K\in \cK_\delta$, let $K^\pm=K\cap T_p^{\pm}$.

  Since the lines $\cK_\delta$ are all close to $L$, there is a $\delta$ depending on $\theta$ and $\alpha$ such that $0<\delta<\min\{\epsilon, \alpha\}$ and every line $K\in \cK_\delta$ intersects both $W^+_{\epsilon,\alpha}$ and $W^-_{\epsilon,\alpha}$ in intervals of length at least $\frac{\alpha}{4}$.  We claim that
  $$\lim_{i\to \infty} \cH^4\big( (\H\setminus E_i)\cap \overline{B}_{\delta}(q) \big)=0,$$
  and thus that $f=1$ almost everywhere on $\overline{B}_{\delta}(q)$.

  For each $i\in \N$ define
  $$\cT_i\eqdef \left\{K\in \cK_\delta \mid \cH^1\big(K \cap \overline{B}_1\cap (E_i\symdiff F)\big) < \frac{\alpha}{8}\right\}.$$
  By Fubini's theorem, for any measurable subset $A\subset \H$ and any horizontal vector $M\in \mathsf{H}$ that is not parallel to $T_p$, we have
  \begin{equation}\label{eq:fubini sets heis}
    \int_{T_p} \cH^1(b\langle M\rangle \cap A) \sin \big(\angle(M,T_p)\big) \ud \cH^3(b) \approx \cH^{4}(A).
  \end{equation}
  Therefore, $\lim_{i\to \infty}\cN(\cT_i)=\cN(\cK_\delta)$, and for almost every $K\in \cT_i$,
  $$\cH^1\big(K^+\cap F \cap \overline{B}_\alpha(p)\big) \ge \cH^1(K^+ \cap W^+_{\epsilon,\alpha}) > \frac{\alpha}{4}.$$
 By the definition of $\cT_i$, this implies that
  \begin{equation}\label{eq:k intersect alpha 8}
    \cH^1\big(K^+\cap E_i \cap \overline{B}_\alpha(p)\big) > \frac{\alpha}{8},
  \end{equation}
  and likewise,
  \begin{equation}\label{eq:k intersect alpha 8 comp}
    \cH^1\big(K^-\cap E_i^c \cap \overline{B}_\alpha(p)\big)> \frac{\alpha}{8}.
  \end{equation}

  Let
  $$\cS_i\eqdef\big\{K\in \cT_i\mid \cH^1\big(K \cap \overline{B}_\delta(q) \cap (\H\setminus E_i)\big)>0\big\}.
    $$
    Suppose that $i\ge d(p,q)+2\delta+2\alpha$ and $K\in \cS_i$.  By \eqref{eq:k intersect alpha 8}, \eqref{eq:k intersect alpha 8 comp}, and the definition of $\cS_i$, there are disjoint intervals $I_1=K^-\cap \overline{B}_\alpha(p)$, $I_2=K^+\cap \overline{B}_\alpha(p)$, and $I_3=K \cap \overline{B}_\delta(q)$ such that: $I_2$ is between $I_1$ and $I_3$; $I_1\cup I_2\cup I_3$ has diameter at most $i$; $\cH^1(I_1\cap (\H\setminus E_i))>\frac{\alpha}{8}$; $\cH^1(I_2\cap E_i)>\frac{\alpha}{8}$; and $\cH^1(I_3 \cap (\H\setminus E_i))>0$.  Lemma~\ref{lem:BWB extended} implies that
  $$\wwidehat{\omega}_{E_i,i}(\overline{B}_1,K) \ge \wwidehat{\omega}_{E_i,i}(\overline{B}_\alpha(p),K) \ge \frac{\cH^1(E_i\cap I_2)}{2} \ge \frac{\alpha}{16}.$$
 Hence,
  $$\frac{\alpha}{16} \cN(\cS_i)\le \int_\cL \wwidehat{\omega}_{E_i, i}(\overline{B}_1, K)\ud \cN(K)= \ENM_{E_i,i}(\overline{B}_1)\le \frac{1}{i},$$
  so $\lim_{i\to \infty} \cN(\cS_i)=0.$

  Let
  $$
  \cR_i\eqdef \big\{K\in \cK_\delta\mid \cH^1\big(K \cap \overline{B}_\delta(q) \cap (\H\setminus  E_i)\big)>0\big\}.
  $$  Then
  $\cN(\cR_i)\le \cN(\cS_i) + \cN(\cK_\delta\setminus \cT_i)$, and so $\lim_{i\to \infty} \cN(\cR_i)=0$.  By \eqref{eq:fubini sets heis},
  $$\cH^4\big(\overline{B}_\delta(q) \cap (\H\setminus E_i)\big) \approx_\delta \int_{\cK_\delta} \cH^1\big(K\cap \overline{B}_\delta(q) \cap (\H\setminus E_i)\big) \ud\cN(K)\le \int_{\cR_i} 2\delta \ud\cN(K),$$
  where the last inequality follows from the fact that $\cH^1(K\cap \overline{B}_\delta(q))\le 2\delta$ for any horizontal line $K$.  We therefore conclude as follows.
  \begin{equation*}
  \lim_{i\to \infty} \cH^4\big(\overline{B}_\delta(q) \cap (\H\setminus E_i)\big) \le \lim_{i\to \infty} 2\delta \cN(\cR_i)=0.
  \tag*{\qedhere}
  \end{equation*}
\end{proof}

By Lemma~\ref{lem:monotone is union of lines}, $\overline{B}_1\cap \partial_{\cH^4} F$ is a union of line segments.  Extended monotonicity implies that these line segments can be extended to lines.

\begin{lemma}\label{lem:extended monotone lines}
  Let $f$ be a $\overline{B}_1$--LEM function and let $F=f^{-1}(1)\cap \overline{B}_1$ be the corresponding locally monotone set.  Let $L$ be a horizontal line.  If an open subinterval $I\subset L$ is contained in $\overline{B}_1 \cap \partial_{\cH^4} F$, then $L\subset \partial_{\cH^4} F$.
\end{lemma}
\begin{proof}
  By Proposition~\ref{prop:local monotone rectifiability}, $\partial_{\cH^4} F$ has  at most countably many characteristic points.  Let $p\in I$ be non-characteristic.  Then the vertical plane $T_p$ containing $L$ is the approximate tangent plane to $\partial_{\cH^4} F$ at $p$.  Recalling that $\mathsf{H}_p$ is the horizontal plane centered at $p$, every horizontal line through $p$, other than $L$ itself, intersects $\partial_{\cH^4} F$ transversally at $p$, so by Lemma~\ref{lem:monotone transverse intersections}, we have $T_p^+\cap \mathsf{H}_p\subset\inter_{\cH^4} (F)$ and $T_p^-\cap \mathsf{H}_p\subset\inter_{\cH^4} (\H\setminus F)$.  Since $L$ lies in the closures of $T_p^+\cap \mathsf{H}_p$ and $T_p^-\cap \mathsf{H}_p$, we have $L\subset \supp_{\cH^4} (F)\cap \supp_{\cH^4} (\H\setminus F)=\partial_{\cH^4} F$.
\end{proof}

Finally, we show that if $\overline{B}_1\cap \partial_{\cH^4} F$ is nonplanar, then we can construct an arrangement of lines that leads to a contradiction.
\begin{lemma}\label{lem:LEM half-spaces}
  Let $f$ be a $\overline{B}_1$--LEM function.  There is a plane $Q\subset \H$ such that $f|_{\overline{B}_1}=\one_{Q^+}$ outside a null set. In fact, the same holds true in a larger set.  Let
  \begin{equation}\label{eq:def S union}
  S\eqdef  (Q \cap \overline{B}_1)\mathsf{H}
  \end{equation}
  be the union of the horizontal lines intersecting $Q \cap \overline{B}_1$.
  Then $f|_{S}=\one_{Q^+}$ outside a null set.
\end{lemma}
\begin{proof}
  Let $F=f^{-1}(1)\cap \overline{B}_1$ be the locally monotone set corresponding to $f$ and suppose by way of contradiction that $\overline{B}_1\cap \partial_{\cH^4} F$ is non-planar.  By part (2) of Proposition~\ref{prop:local monotone rectifiability} and by Lemma~\ref{lem:extended monotone lines}, for every point $p\in \overline{B}_1\cap \partial_{\cH^4} F$, there is a horizontal line $M_p$ through $p$ such that $M_p\subset \partial_{\cH^4} F$.

  Reasoning as in Lemma~4.11 of \cite{CheegerKleinerMetricDiff} shows that there are two $\overline{B}_1$--rulings of $F$ that satisfy one of the cases of Lemma~\ref{lem:hyperboloid lemma}, i.e., they are a pair of skew lines or a pair of lines with distinct parallel projections.  Indeed, suppose that $J$ and $K$ are $\overline{B}_1$--rulings of $F$ with parallel projections.  If $\uppi(J)\ne\uppi(K)$, we are done; otherwise, $J$ and $K$ are contained in a vertical plane $V$.  Let $L$ be a $\overline{B}_1$--ruling of $F$ not in $V$, which exists by the assumed non-planarity.  Then $L$ is skew to $J$ or $K$ or parallel to $V$ with a distinct projection. It remains to treat the case when any two $\overline{B}_1$--rulings of $F$ have nonparallel projections.  Let $J$ and $K$ be two such rulings.  If $J$ and $K$ are disjoint, we are done, so we suppose $J$ and $K$ intersect at a point $p$ and are thus contained in the horizontal plane $\mathsf{H}_p$ centered at $p$.  If $L$ is a $\overline{B}_1$--ruling of $F$ that is not contained in $\mathsf{H}_p$ (it exists by assumed non-planarity), then $L$ intersects $\mathsf{H}_p$ at a single point other than $p$, so $L$ is skew to either $J$ or $K$, as desired.

  This shows that there are two $\overline{B}_1$--rulings $L_1$ and $L_2$ of $F$ that are skew or have distinct parallel projections.  Let $I=L_1\cap \overline{B}_1$ and let $p\in I$ be a noncharacteristic point such that $\uppi(p)\not \in \uppi(L_2)$.  By Lemma~\ref{lem:hyperboloid lemma}, there is a horizontal line $M$ that goes through $p$ and intersects $L_2$ at $q$.  This line is not equal to $L_1$, so it intersects $\partial_{\cH^4} F$ transversally at $p$.  By Lemma~\ref{lem:monotone transverse intersections}, this implies that $q\in \inter_{\cH^4} (f^{-1}(0))$ or $q\in \inter_{\cH^4} (f^{-1}(1))$, but $q\in L_2\subset \partial_{\cH^4} F$, which is a contradiction.  Therefore, $\overline{B}_1\cap \partial_{\cH^4} F$ is planar and there is a plane $Q$ such that $F\cap \overline{B}_1=Q^+\cap \overline{B}_1$ up to a null set.  Since $f$ takes values in $\{0,1\}$ inside  $\overline{B}_1$, this implies the first part of Lemma~\ref{lem:LEM half-spaces}.

  With $S$ as in~\eqref{eq:def S union}, take $w\in Q^+\cap S$.  Then $w$ lies on a horizontal line that intersects $Q\cap \overline{B}_1$ transversally, and Lemma~\ref{lem:monotone transverse intersections} implies that $w\in \inter_{\cH^4} (f^{-1}(1))$.  It follows that $f=1$ almost everywhere in $Q^+\cap S$ and likewise that $f=0$ almost everywhere in $Q^-\cap S$.
\end{proof}

The second part of Proposition~\ref{prop:stability of extended monotone} states that extended monotone intrinsic graphs are close to vertical planes.  This follows from the fact that neighborhoods of the center of a horizontal plane cannot be approximated by intrinsic graphs.
\begin{lemma}\label{lem:LEM vertical half-spaces}
  Let $V_0$ be the $xz$--plane and let $E_1,E_2,\ldots\subset \H$ be a sequence of intrinsic graphs over $V_0$ such that $E_i^+$ is $(\frac{1}{i},i)$--extended monotone on $\overline{B}_1$ and $\one_{E_i^+}$ converges weakly to a function $f\in L_\infty(\H)$ as $i\to \infty$. There is a vertical plane $Q\subset \H$ such that $f|_{\overline{B}_1}=\one_{Q^+}$ outside a null set.  Furthermore, if $S$ is as in~\eqref{eq:def S union}, then  $f|_{S}=\one_{Q^+}$ outside a null set.
\end{lemma}
\begin{proof}
  For any intrinsic graph $\Gamma$ and any $g\in \Gamma^+$, we have $gY^t\in \Gamma^+$ for every $t>0$.  Since $\cH^4$ is right-invariant, this implies that for any measurable set $U\subset \N$ and any $i\in \N$,
  $$\cH^4\big(U\cap E_i^+\big)\le \cH^4\big(U\cap E_i^+ Y^t\big).$$
  Therefore,
  $$\int_U f \ud \cH^4\le \int_{U Y^t} f \ud \cH^4.$$
  Consequently,
  \begin{equation}\label{eq:f monotonicity}
    f(g)\le f(gY^t)\qquad\text{for almost every $(g,t)\in \H\times (0,\infty)$}.
  \end{equation}

  If $f$ is almost-surely constant on $\overline{B}_1$, we can take $Q$ to be a vertical plane that does not intersect $\overline{B}_1$.  We thus suppose that $f|_{\overline{B}_1}$ is not almost-surely constant.  By Lemma~\ref{lem:LEM half-spaces}, there is a plane $Q$ that satisfies $f|_{S}=\one_{Q^+}$ outside a null set, where $S$ is given in~\eqref{eq:def S union}.

  Suppose for contradiction that $Q$ is horizontal.  Let $c\in \H$ be such that $Q = \mathsf{H}_c = c \mathsf{H}$ and let $p\in Q\cap \inter (\overline{B}_1)$ be such that $x(p)\ne x(c)$.  Let $L$ be the horizontal line from $c$ to $p$ and let $V=(x_V,y_V,0)$ be the horizontal vector such that $p=cV$.  Set $q = c V^{-1}=c(-x_V,-y_V,0)$. We claim that there is $\epsilon>0$ such that $\{pY^{\pm\epsilon}, q Y^{\pm \epsilon}\}\subset S$.  Choose $\epsilon>0$ so that $p Y^t\in \overline{B}_1$ and $r_t=c (x_V, y_V+t,0)\in \overline{B}_1\cap Q$ for all $t\in [-2\e,2\e]$.  Then
  \begin{multline*}
    r_t  \left(-2x_V, - 2y_V -\frac{3t}{2}, 0\right)
    \\= c  (x_V,y_V+t,0)  \left(-2x_V, -2y_V -\frac{3t}{2}, 0\right)
    = c  \left(-x_V,-y_V-\frac{t}{2}, \frac{x_V t}{4}\right)
    = q Y^{-\frac{t}{2}}.
  \end{multline*}
  It follows that  $q Y^{-\frac{t}{2}}\in r_t \mathsf{H} \subset S$.  In particular, $q Y^{\pm\epsilon}\in S$. At the same time, $p Y^{\epsilon}$ and $p Y^{-\epsilon}$ are on opposite sides of $Q$; equation \eqref{eq:f monotonicity} implies that $p Y^{\epsilon}\in Q^+$ and $p Y^{-\epsilon}\in Q^-$.  Likewise, $q Y^{\pm \epsilon}\in Q^\pm$.  But since $c$ is between $p$ and $q$, the points $p Y^{\epsilon}$ and $q Y^{\epsilon}$ are on opposite sides of $Q$, which is a contradiction.  Therefore, $Q$ is a vertical plane.
\end{proof}

\begin{proof}[{Proof of Proposition~\ref{prop:stability of extended monotone}}]
  If the first part of the proposition were false, then there would exist  $\epsilon>0$ and a sequence of measurable sets $(E_i)_{i=1}^\infty$ such that for any $i\in \N$, the set $E_i$ is $(\frac{1}{i},i)$--extended monotone on $\overline{B}_1$ and $|\overline{B}_{1}\cap (P^+\symdiff E_i)|>\epsilon$ for every plane $P\subset \H$.  There is a subsequence $(E_{i(j)})_{j=1}^\infty$ whose characteristic functions converge weakly to a $\overline{B}_1$--LEM function $f$.  By Lemma~\ref{lem:LEM half-spaces}, there is a plane $Q\subset \H$ such that $f=\one_{Q^+}$ almost everywhere on $\overline{B}_1$.  Then $\lim_{j\to \infty} |\overline{B}_{1}\cap (Q^+\symdiff E_{i(j)})|= 0$, which is a contradiction.

  Similarly, if the second part of the proposition were false, then there would exist $\epsilon>0$ and a sequence of intrinsic graphs $(E_i)_{i=1}^\infty$ over $V_0$ such that for any $i\in \N$, the epigraph $E_i^+$ is $(\frac{1}{i},i)$--extended monotone on $\overline{B}_1$ and $|\overline{B}_{1}\cap (P^+\symdiff E_i^+)|>\epsilon$ for every vertical plane $P\subset \H$.  Passing to a subsequence, we may suppose that the indicators $\one_{E_i^+}$ converge weakly to a $\overline{B}_1$--LEM function $f$. By Lemma~\ref{lem:LEM vertical half-spaces}, there is a vertical plane $Q\subset \H$ such that $f=\one_{Q^+}$ almost everywhere on $\overline{B}_1$.  Then $\lim_{i\to \infty} |\overline{B}_{1}\cap (Q^+\symdiff E_i^+)|= 0$, which is a contradiction.
\end{proof}

\section{$L_1$ bounds and characteristic curves on monotone intrinsic graphs}\label{sec:l1 and characteristic}

Here we complete the proof of Proposition~\ref{prop:Omega control}, which obtains $L_1$ bounds for paramonotone pseudoquads and bounds their characteristic curves.

Fix $0<\mu\le \frac{1}{32}$ and a $\mu$--rectilinear pseudoquad  $Q$ in an intrinsic Lipschitz graph $\Gamma=\Gamma_f$. Suppose that $\Gamma$ is $(\eta, R)$--paramonotone on $rQ$.  By Remark~\ref{rem:normalizing rectilinear}, we can normalize $Q$ and $\Gamma$ so that the corresponding parabolic rectangle is the square $[-1,1]\times \{0\}\times [-1,1]$; by Lemma~\ref{lem:Omega scaling} and the discussion immediately after its proof, the normalized pseudoquad remains paramonotone.  So, it suffices to prove Proposition~\ref{prop:Omega control} for such pseudoquads.

For $t>0$, denote  $D_t=[-t,t]\times\{0\}\times  [-t^2,t^2]\subset V_0$.  By our choice of normalization, we have $tQ=D_t$.  Furthermore, $D_t\subset B_{5t}$ and $\Pi(\overline{B}_{t})\subset D_t$. We will proceed in several steps.
\begin{enumerate}
\item
  First, we will prove in Lemma~\ref{lem:paramonotone is L1 bounded} that there is a universal constant $\kappa>0$ such that $\|f\|_{L_1(Q)}\le \kappa$ when $\eta$ is sufficiently small.  This relies on Lemma~\ref{lem:supercharacteristic bounds} that bounds the tails of $f$ in regions that are bounded above and below by supercharacteristic curves (projections of horizontal curves in $\Gamma\cup \Gamma^+$).
\item
  Next, we will show that $\Gamma$ is close to a plane on a ball around the origin.  Since $\|f\|_{L_1(Q)}\le \kappa$, the intersections $\Gamma^+\cap B_{\kappa}$ and $\Gamma^-\cap B_{\kappa}$ both have positive measure.  For any $r>0$, we have $\Pi(\overline{B}_{r})\subset rQ$, so $\ENM_{\Gamma^+,R}(\overline{B}_{r})\lesssim \eta R$.  When $\eta R$ is sufficiently small and $r$ and $R$ are sufficiently large, Proposition~\ref{prop:stability of extended monotone} implies that there is a vertical plane $P$ that intersects $B_{\kappa}$ and approximates $\Gamma$ on $B_r$, i.e.,
  $$\cH^4\left((\Gamma^+\symdiff P^+)\cap \overline{B}_{r}\right)<\epsilon.$$
  Furthermore, since $\|f\|_{L_1(Q)}\le \kappa$, the slope and $y$--intercept of $\uppi(P)$ are both at most some universal constant.

  We then apply an automorphism that sends $P$ to $V_0$.  Since the slope and $y$--intercept of $P$ are bounded, there is a universal constant $c>0$ and a map $q\from \H\to \H$ (a composition of a left translation in the $y$--direction and a shear) such that $q(P)=V_0$ and
  $B_{c^{-1}s-c}\subset q(B_s)\subset B_{cs+c}$
  for all $s>c^2$.  We let $\hat{\Gamma}=q(\Gamma)$, $\hat{Q}=\hat{q}(Q)=\Pi(q(Q))$, and let $\hat{f}$ be such that $\hat{\Gamma}=\Gamma_{\hat{f}}$.  Since $q$ preserves $\cH^4$,
  \begin{equation}\label{eq:hat gamma close on ball}
    \cH^4\left((\hat{\Gamma}^+\symdiff V_0^+)\cap \overline{B}_{c^{-1}r-c}\right)<\epsilon.
  \end{equation}
  This inequality controls $\hat{f}$ on $V_0\cap \overline{B}_{c^{-1}r-c}$, and we choose $r$ large enough that $11\hat{Q}\subset \overline{B}_{c^{-1}r-c}$.
\item
  By \eqref{eq:hat gamma close on ball},
  $$\int_{10 \hat{Q}} \min\left\{1,|\hat{f}(p)|\right\} \ud \cH^3(p) \le \epsilon,$$
  so a bound on the tails of $\hat{f}$ would lead to a bound on $\|\hat{f}\|_{L_1(10\hat{Q})}$.  We bound the tails in Lemma~\ref{lem:paramonotone squares L1 close}, by finding supercharacteristic curves above and below $10\hat{Q}$, then applying Lemma~\ref{lem:supercharacteristic bounds} again.  This implies that $\|\hat{f}\|_{L_1(10\hat{Q})}\lesssim \epsilon$ when $\eta$ is sufficiently small, which proves the first part of Proposition~\ref{prop:Omega control}.
\item
  Finally, we bound the characteristic curves of $\hat{\Gamma}$ in Lemma~\ref{lem:close foliation curves squares}, by showing that if $\hat{\Gamma}$ contains characteristic curves that are not nearly parallel to the $x$--axis, then either $\|\hat{f}\|_{L_1}$ is bounded away from zero or $\Omega^P_{\Gamma^+,R}$ is  bounded away from zero.  This completes the proof of Proposition~\ref{prop:Omega control}.
\end{enumerate}

We will use the following notation for horizontal lines. Every horizontal line in $\cL_P$ can be written uniquely as follows for some for some $w=(0,y_0,z_0) \in \H$ and $m\in \R$.
$$
L_{w,m}\eqdef w \langle X+ mY\rangle.
$$
Let $\rho_{L_{w,m}}\from \R\to L_{w,m}$ be the following parametrization, so that $x(\rho_L(t))=t$ for all $t\in \R$.
$$
\forall t\in \R,\qquad \rho_{L_{w,m}}(t)\eqdef w (X+mY)^t.
$$
For every $x\in \R$ define
\begin{equation}\label{eq:def gL}
g_{L_{w,m}}(x)\eqdef z\Big(\Pi\big(\rho_{L_{w,m}}(x)\big)\Big)=-\frac{m}{2} x^2- y_0 x + z_0.
\end{equation}
Note that since $L_{w,m}$ is horizontal, we have $y(\rho_{L_{w,m}}(x))=-g_{L_{w,m}}'(x)$.

\subsection{Bounding the tails of $f$}

We start by showing that if $Q$ is a rectilinear pseudoquad for $\Gamma=\Gamma_f$ such that $\Gamma^+$ is $(\eta, R)$--paramonotone on $rQ$, as in Proposition~\ref{prop:Omega control}, and $Q$ is normalized so that the corresponding parabolic rectangle is a $2\times 2$ square, as in Remark~\ref{rem:normalizing rectilinear}, then there is a universal constant $\kappa$ such that $\|f\|_{L_1(Q)}\le \kappa$ when $r$ and $R$ are sufficiently large and $\eta$ is sufficiently small.

This step relies on the following lemma, which will also be used in step 3.  A \emph{supercharacteristic curve} (respectively \emph{subcharacteristic curve}) for $\Gamma$ is the projection $\Pi(\gamma)$ of a horizontal curve $\gamma\from I\to \H$ such that $x(\gamma(t))=t$ for all $t\in I$ and $\gamma(I)\subset \Gamma\cup \Gamma^+$ (respectively $\gamma(I)\subset \Gamma\cup \Gamma^-$).

Such a curve can be written as a graph of the form $\{z=g(x)\}\subset V_0$. By the argument of Lemma~\ref{lem:char}, $g$ is differentiable almost everywhere and satisfies $g'(x)=y(\gamma(x))$ for almost every $x\in I$; since $g$ is locally Lipschitz, $g'(x)=y(\gamma(x))$ for every $x\in I$. In particular, $g'(x)\le -f(x,0,g(x))$ for all $x\in I$.  We then say that $g$ is a \emph{function with supercharacteristic graph}.

\begin{lemma}\label{lem:supercharacteristic bounds}
  Let $g_1,g_2 \from [-2,2]\to \R$ be functions with supercharacteristic graphs such that $\sup g_1([-2,2]) < \inf g_2([-2,2])$.  For $0\le r\le 2$, let
  $$U_r=\{(x,0,z) \in V_0\mid |x|\le r\ \mathrm{and}\  g_1(x)\le z\le g_2(x)\}.$$
  Denoting $H=\max\{\|g_1\|_{L_\infty([-2,2])}, \|g_2\|_{L_\infty([-2,2])}\}$, for any $t\ge 8H$ we have
  \begin{equation}\label{eq:supercharacteristic tail bound}
    |\{v\in U_1\mid f(v)\ge t\}| \lesssim \frac{1}{t^{2}} \Omega^P_{\Gamma^+,4}(U_2).
  \end{equation}
  Likewise, if $g_1,g_2 \from [-2,2]\to \R$ have subcharacteristic graphs and $U_r$ and $H$ are as above, then for any $t\ge 8H$ we have
  $$|\{v\in U_1\mid f(v)\le -t\}| \lesssim \frac{1}{t^{2}} \Omega^P_{\Gamma^+,4}(U_2).$$
\end{lemma}

Once we prove Lemma~\ref{lem:supercharacteristic bounds}, we will apply it to the case that $Q$ approximates $[-1,1]^2$ and $g_1$ and $g_2$ are the lower and upper bounds of $Q$.

\begin{proof}
  Fix  $t\ge 8H$ and $y_0,m,z_0\in \R$ such that $|y_0-\frac{t}{2}|< \frac{t}{12}$  and $|m|< \frac{t}{12}$.  Let $L=L_{(0,y_0,z_0) ,m}$.  For any $s\in [-2,2]$ we have
  \begin{equation}\label{eq:h derivative bound}
    \left|g_L'(s)+\frac{t}{2}\right| = \left|y\big(\rho_L(s)\big)-\frac{t}{2}\right| < \frac{t}{4},
  \end{equation}
  so $-\frac34 t < g_L'(s) < -\frac{1}{4}t$ on $[-2,2]$.

  We claim that for any almost every such $L$ we have
  \begin{equation}\label{eq:omega detects large values 2}
    \wwidehat{\omega}^P_{\Gamma^+, 4}(U_2, L) \ge \frac{1}{2} \cH^1\Big(x\big(\Gamma^-\cap L\cap \Pi^{-1}(U_1)\big)\Big).
  \end{equation}

  By \eqref{eq:h derivative bound}, we have
  $$
  g_L(-2)=g_L(s)-\int_{-2}^s g'_L(u)\ud u> -H+(s+2)\frac{t}{4}\ge -H+2H=H,
  $$
  and
  $$
  g_L(2)=g_L(s)+\int_s^2g_L'(u)\ud u<H-(2-s)\frac{t}{4}\le H-2H=-H.
  $$
  Hence, $\Pi(L)$ crosses $U_2$ negatively (from top to bottom), as depicted in Figure~\ref{fig:two horizontal curves and a line}.
  The curve $\Pi(L)$ only intersects the top and bottom of $U_2$, not the sides, so we say that $\Pi(L)$ is transverse to the boundary of $U_2$ if $\Pi(L)$ intersects the top and bottom boundaries transversally; that is, if $g_L(u)=g_i(u)$ for some $u\in [-2,2]$ and $i=1,2$, then $g_L'(u) \ne g_i'(u)$.

  Suppose that $\Pi(L)$ is transverse to the boundary of $U_2$ and that $L\cap \Gamma^+$ has finite perimeter; these are true for almost every $L$.  If $\Pi(L)$ does not intersect $U_1$, then the right side of \eqref{eq:omega detects large values 2} is 0 and the inequality holds trivially.  We thus suppose in addition that $L$ intersects $U_1$. In this case, there is some $s\in [-1,1]$ such that $|g_{L}(s)|\le H$.

 Fix $i\in \{1,2\}$ and suppose that $\Pi(L)$ crosses the graph of $g_i$ negatively at $(u,0,g_L(u))$. Let $v=\rho_L(u)$ be the point on $L$  over the intersection.  Then $g_L(u) = g_i(u)$ and $g_L'(u) < g_i'(u)$.  Since the graph of $g_i$ is supercharacteristic, $f(u,0,g_i(u)) \le -g_i'(u)$, and therefore
  $$y(v) = -g_L'(u) > -g_i'(u) \ge f\big(u,0,g_i(u)\big) = f\big(\Pi(v)\big).$$
  That is, $v\in \Gamma^+$.

  \begin{figure}[h]\label{fig:two horizontal curves and a line}
  \begin{centering}
    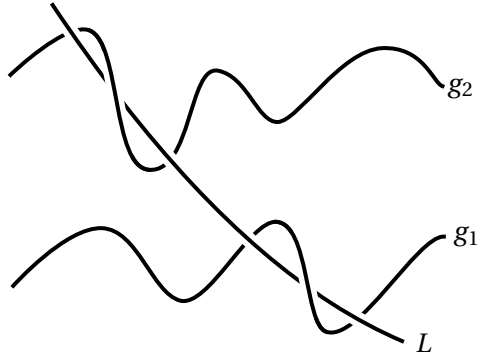
  \end{centering}
  \caption{
    Two characteristic curves $g_1$ and $g_2$ and a horizontal line $L$, projected to $V_0$; the positive $y$--axis points toward the reader.  Since $\Pi(L)$ crosses $U_2$ negatively, the segments of $L$ at the first and last crossings  lie in $\Gamma^+$, so the size of the intersection $L\cap \Gamma^-$ is bounded by $\wwidehat{\omega}^P_{\Gamma^+, 4}(U_2, LZ^t)$.
  }
\end{figure}

  Since $\Pi(L)$ is transverse to the boundary of $U_2$, the intersection $\Pi(L)\cap U_2$ consists of a collection of intervals.  Let $[a_1,b_1],\dots, [a_n,b_n]\subset \R$ be the disjoint intervals such that
  $$x(\Pi(L)\cap U_2)=[a_1,b_1]\cup \dots \cup [a_n,b_n],$$
  and these intervals are in ascending order.  The projection $\Pi(L)$ does not intersect the left or right boundary of $U_2$, so $\Pi(L)$ crosses the graph of $g_1$ or $g_2$ at each $a_i$ or $b_i$.  Since $g_L$ is decreasing and $\sup g_1([-2,2]) < \inf g_2([-2,2])$, the crossings of $g_2$ all have smaller $x$--coordinate than the crossings of $g_1$.

  Consider  $S=x(L \cap \Gamma^+)$.  Since $\Pi(L)$ crosses the graph of $g_2$ negatively at $a_1$ and crosses the graph of $g_1$ negatively at $b_n$, the argument above implies that $a_1,b_n\in S$.  Furthermore, for each $i\in \n$, one of three cases holds.
  \begin{enumerate}
  \item $\Pi(L)$ crosses the graph of $g_2$ negatively at $a_i$ and positively (from bottom to top) at $b_i$.
  \item $\Pi(L)$ crosses the graph of $g_2$ negatively at $a_i$ and crosses the graph of $g_1$ negatively at $b_i$.
  \item $\Pi(L)$ crosses the graph of $g_1$ positively at $a_i$ and negatively at $b_i$.
  \end{enumerate}
  In each case, $a_i\in S$ or $b_i\in S$.  By Lemma~\ref{lem:BWB extended} (applied with $[a,b]=[a_1,b_n]$),
  $$\widehat{\omega}_{S, 4}([a_i,b_i]) = \widehat{\omega}_{\R\setminus S, 4}([a_i,b_i]) \ge \frac{1}{2} \cH^1(x(\Gamma^-\cap L)\cap [a_i,b_i]).$$
  Summing over $i\in \n$, we find that
  \begin{equation}\label{eq:omega detects large values}
    \widehat{\omega}_{S, 4}\Big(\bigcup_{i=1}^n [a_i,b_i]\Big)=\wwidehat{\omega}^P_{\Gamma^+, 4}(U_2, L) \ge \frac{1}{2} \cH^1\Big(x\big(\Gamma^-\cap L \cap \Pi^{-1}(U_2)\big)\Big).
  \end{equation}
This proves \eqref{eq:omega detects large values 2}.

  Next, let $A=U_1 \cap f^{-1} ([t,\infty))$.  By \eqref{eq:h derivative bound}, $y(\rho_L(s)) < t$ for all $s\in [-2,2]$, so if $\Pi(\rho_L(s))\in A$, then $\rho_L(s)\in \Gamma^-$.  Therefore, by \eqref{eq:omega detects large values 2},
  $$\frac{1}{2} \cH^1\big(x(\Pi(L)\cap A)\big) \le \frac{1}{2} \cH^1\Big(x\big(\Gamma^-\cap L \cap \Pi^{-1}(U_1)\big)\Big) \le \wwidehat{\omega}^P_{\Gamma^+, 4}(U_2,L).$$
  By Fubini's Theorem, for any $y_0$ and $m$ as above,
  $$\frac{1}{2}|A|=\frac{1}{2} \int_\R \cH^1\big(x(L_{(0,y_0,z_0),m} \cap A)\big)\ud z_0
    \le \int_\R \wwidehat{\omega}^P_{\Gamma^+, 4}(U_2,L_{(0,y_0,z_0),m}) \ud z_0.$$
  Therefore, recalling the definition \eqref{eq:define OmegaP} of $\Omega^P$, we have
  \begin{align*}
    \Omega^P_{\Gamma^+,4}(U_2)
    &= \frac{1}{4} \int_{\cL} \wwidehat{\omega}^P_{\Gamma^+, 4}(U_2,L) \ud \cN_P(L)\\
    &\ge \frac{1}{4} \int_{-\frac{t}{12}}^{\frac{t}{12}} \int_{\frac{5t}{12}}^{\frac{7t}{12}} \int_{\R} \wwidehat{\omega}^P_{\Gamma^+, 4}(U_2,L_{(0,y_0,z_0),m}) \ud z_0 \ud y_0 \ud m \\
    &\ge \frac{1}{8} \int_{-\frac{t}{12}}^{\frac{t}{12}} \int_{\frac{5t}{12}}^{\frac{7t}{12}} |A|\ud y_0 \ud m\\
    &=\frac{t^2}{288} |A|.
  \end{align*}
  That is,
  $$|\{v\in U_1\mid f(v)\ge t\}| \lesssim \frac{1}{t^{2}} \Omega^P_{\Gamma^+,4}(U_2).$$
  This proves \eqref{eq:supercharacteristic tail bound}.

  We can show that
  $$|\{v\in U_1\mid f(v)\le -t\}| \lesssim \frac{1}{t^{2}} \Omega^P_{\Gamma^+,4}(U_2).$$
  when $g_1,g_2$ have subcharacteristic graphs by either applying a similar argument or by replacing $\Gamma$, $U_r$, etc. by $s_{1,-1}(\Gamma)$, $s_{1,-1}(U_r)$, etc.
\end{proof}

The desired bound on $\|f\|_{L_1(Q)}$ follows by integrating \eqref{eq:supercharacteristic tail bound} with respect to $t$.

\begin{lemma}\label{lem:paramonotone is L1 bounded}
  Let $f\from V_0\to \R$ be a continuous function and let $\Gamma$ be its intrinsic graph.  Let $(Q, [-1,1]\times \{0\}\times [-1,1])$ be a $\frac{1}{32}$--rectilinear pseudoquad for $\Gamma$. Suppose that $\Omega^P_{\Gamma^+,4}(2Q)\le 1$.  There is a universal constant $\kappa>0$ such that $\|f\|_{L_1(Q)}\le \kappa.$
\end{lemma}
\begin{proof}
  Let $g_1$ and $g_2$ be the lower and upper bounds of $Q$ and for $0\le r\le 2$, let $U_r$ be as in Lemma~\ref{lem:supercharacteristic bounds}.  Then $Q=U_1$ and $U_2\subset 2Q$.  Let $H=2$.  Since the graphs of $g_1$ and $g_2$ are supercharacteristic and $U_2\subset 2Q$, Lemma~\ref{lem:supercharacteristic bounds} implies that  for any $t\ge 16$,
  $$|\{v\in Q\mid f(v)\ge t\}| \lesssim t^{-2} \Omega^P_{\Gamma^+,4}(U_2)\le t^{-2} \Omega^P_{\Gamma^+,4}(2Q)\le t^{-2}.$$

  Since the graphs of $g_1$ and $g_2$ are also subcharacteristic, for any $t\ge 16$ we also have
  $$|\{v\in U_1\mid f(v)\le -t\}| \lesssim t^{-2}.$$
   Then
  \begin{equation*}\|f\|_{L_1(Q)}=\int_0^\infty |\{v\in Q\mid |f(v)|\ge t\}|\ud t\lesssim 16 |Q|+\int_{16}^\infty t^{-2}\ud t \lesssim 1.\tag*{\qedhere}\end{equation*}
\end{proof}
\subsection{Constructing the approximating plane}

Now we will use Lemma~\ref{lem:paramonotone is L1 bounded} and the results of Section~\ref{sec:extended monotone} to show that if $Q$ is a paramonotone pseudoquad for $\Gamma_f$, then $f$ is close on $Q$ to an affine function with bounded coefficients.

\begin{lemma}\label{lem:paramonotone approximate plane small} 
  Let $\kappa>0$ be the constant in Lemma~\ref{lem:paramonotone is L1 bounded}, and let $C=4\kappa$.  For any $0<\epsilon<1$ and $r \ge 2 \kappa + 6$, there are $0<\eta<\frac{1}{2}$ and $R>0$ with the following property.

  Let $\Gamma=\Gamma_f$ be an intrinsic graph such that $(Q, [-1,1]\times \{0\}\times [-1,1])$ is a $\frac{1}{32}$--rectilinear pseudoquad for $\Gamma$.  Let $g_1$ and $g_2$ be the lower and upper bounds of $Q$, respectively.  If $Q$ is $(\eta, R)$--paramonotone on $r Q$, then there is a vertical plane $P\subset \H$ such that
  \begin{equation}\label{eq:plane symdiff close}
    \cH^4\Big(\overline{B}_{r}\cap \big(P^+\symdiff \Gamma^+\big)\Big)<\epsilon.
  \end{equation}
  Moreover, $P$ is the graph of an an affine function $F\from V_0\to \R$ of the form $F(w)=a+bx(w)$, whose coefficients satisfy $\max \{|a|,|b|\}\le C$.
\end{lemma}
\begin{proof}
  We have $\delta_x(Q)=2$ and $\alpha(Q)=\sqrt{2}$. Also, $2\le |Q|\le 6$. Hence, recalling \eqref{eq:def paramonotone}, if $\Gamma$ is $(\eta, R)$--paramonotone on $rQ$, then assuming $R\ge 2$ and $\eta R<1$ we have
  $$\Omega^P_{\Gamma^+,4}(2Q)\le \frac{R}{2}\Omega^P_{\Gamma^+,2R}(rQ) \le \frac{R}{2}\eta \alpha(Q)^{-4} |Q| \le R\eta<1,$$
  so by Lemma~\ref{lem:paramonotone is L1 bounded} we have $\|f\|_{L_1(Q)} < \kappa$.

  Since $\Pi(\overline{B}_{r}) \subset r Q$, \eqref{eq:ENM of paramonotone} implies that
  $$\ENM_{\Gamma^+,2R}(\overline{B}_{r})\lesssim \eta R.$$
  By Proposition~\ref{prop:stability of extended monotone}, when $R$ is sufficiently large and $\eta R$ is sufficiently small, there is a half-space $P^+$ bounded by a vertical plane such that
  $$\cH^4\Big(\overline{B}_{r}\cap \big(P^+\symdiff \Gamma^+\big)\Big)<\epsilon.$$
  If necessary, we may rotate $P$ infinitesimally around the $z$--axis so that it is not perpendicular to $V_0$.  Then $P$ is the graph of an affine function $F\from V_0\to \R$.  Let $a,b\in \R$ be such that  $F(w)=a+b x(w)$ for all $w\in V_0$.

  For all $w\in V_0$, let $\bar{f}(w)$ (respectively $\bar{F}(w)$) be the element of $[-2\kappa,2\kappa]$ that is closest to $f(w)$ (respectively $F(w)$).  Since $r \ge 2\kappa+6$, the intrinsic graphs of $\bar{F}$ and $\bar{f}$ over $Q$ both lie in $\overline{B}_{r}$. Therefore,
  $$\left\|\bar{F}-\bar{f}\right\|_{L_1(Q)}\le \cH^4\left(\overline{B}_{r}\cap (\Gamma_{\bar{f}}^+ \symdiff \Gamma_{\bar{F}}^+)\right) \le \cH^4\left(\overline{B}_{r}\cap (\Gamma_{f}^+ \symdiff \Gamma_{F}^+)\right) \le \epsilon,$$
  and thus
  \begin{equation}\label{eq:barF less 2kappa}
    \left\|\bar{F}\right\|_{L_1(Q)}\le \epsilon + \left\|\bar{f}\right\|_{L_1(Q)} \le \epsilon + \|f\|_{L_1(Q)} \le 2\kappa.
  \end{equation}
  The map $F$ is affine, and $[-1,1]\times\{0\}\times  [-\frac{1}{2},\frac{1}{2}]\subset Q$, so $|\{q\in Q\mid |F(q)|>2\kappa \}| > 1$ if $|a|>2\kappa$ or $|b|>4\kappa$,  which implies that $\|\bar{F}\|_{L_1(Q)}>2\kappa$ in contradiction to \eqref{eq:barF less 2kappa}. So, $\max\{|a|,|b|\}\le 4\kappa$.
\end{proof}

We will next use Lemma~\ref{lem:paramonotone approximate plane small} to construct a new intrinsic Lipschitz graph $\hat{\Gamma}$ that is close to $V_0$ on a ball around $\mathbf{0}$.  Let $0<\epsilon<1$ and $r>0$ be numbers to be chosen later. Let $\eta,R,C,\Gamma,f,Q$ be as in Lemma~\ref{lem:paramonotone approximate plane small},  so that there is a vertical plane $P$ approximating $Q$ that is the graph of an affine function $F(w)=a+bx(w)$ with $\max\{|a|,|b|\}\le C$.

Let $q=q_{a,b}\from \H\to \H$ be the map given by
$$
\forall(x,y,z)\in \H,\qquad q(x,y,z)\eqdef Y^{-a} (x,y-bx,z)=\left(x,y-a-bx,z+\frac{ax}{2}\right).$$
This is a shear map that preserves the $x$--coordinate and sends $P$ to $V_0$.  Let $\hat{q}\from V_0\to V_0$ be the map that $q$ induces on $V_0$, i.e.,
\begin{equation}\label{eq:def q}
 \forall x,z\in \R,\qquad  \hat{q}(x,0,z)=\Pi\big(q(x,0,z)\big)=\left(x,0,z+ax+\frac{b}{2}x^2\right).
\end{equation}
Let $\hat{\Gamma}=q(\Gamma)$ and $\hat{Q}=\hat{q}(Q)$.  By Lemma~\ref{lem:curve transforms}, $\hat{Q}$ is a pseudoquad for $\hat{\Gamma}$ that contains $\mathbf{0}$ and $\hat{\Gamma}=\Gamma_{\hat{f}}$, where $$\hat{f}(v)=f(\hat{q}^{-1}(v))-a-bx(v)=f(\hat{q}^{-1}(v))-F(\hat{q}^{-1}(v)).$$ Since $a,b\in [-C,C]$, there is a universal constant $c>0$ such that for all $s>c^2$,
\begin{equation}\label{eq:hatq d nesting}
  D_{c^{-1}s-c}\subset \hat{q}(D_s)= s\hat{Q} \subset D_{cs+c},
\end{equation}
where we recall $D_s=[-s,s]\times \{0\}\times [-s^2,s^2]$, and
\begin{equation}\label{eq:hatq b nesting}
  B_{c^{-1}s-c}\subset q(B_s)\subset B_{cs+c}.
\end{equation}

Bounds on $\Gamma$ and $Q$ correspond directly to bounds on $\hat{\Gamma}$ and $\hat{Q}$.  For example, shear maps preserve $\cH^4$, so
\begin{equation}\label{eq:compare hat Gamma ch4}
  \cH^4\left(B_{c^{-1}r-c}\cap (V_0^+\symdiff \hat{\Gamma}^+)\right)
  \le \cH^4\left(q(B_{r}) \cap (V_0^+\symdiff \hat{\Gamma}^+)\right)
  = \cH^4\big(B_{r} \cap (P^+\symdiff \Gamma^+)\big)
  <\epsilon.
\end{equation}
In particular, when $r$ is sufficiently large,
\begin{equation}\label{eq:f-F-bulk}
  \|\min \{|f-F|,\frac{r}{2}\}\|_{L_1(10Q)}\le \cH^4\big(B_{r} \cap (P^+\symdiff \Gamma^+)\big) < \epsilon.
\end{equation}
Maps induced by shears preserve the Lebesgue measure $\cH^3$ on $V_0$, so by \eqref{eq:hatq d nesting},
\begin{equation}\label{eq:compare L1 hats}
  \|f-F\|_{L_1(10Q)}=\|\hat{f}\|_{L_1(10\hat{Q})} \le \|\hat{f}\|_{L_1(D_{11c})},
\end{equation}
and by Lemma~\ref{lem:Omega scaling}, $\hat{\Gamma}$ is $(\eta, R)$--paramonotone on $r\hat{Q}$.

\subsection{Bounding $\|f - F\|_{L_1(10Q)}$}
Next, we bound $\|f - F\|_{L_1(10Q)}$.  Lemma~\ref{lem:paramonotone is L1 bounded}, Lemma~\ref{lem:paramonotone approximate plane small}, and \eqref{eq:f-F-bulk} imply that $\|f - F\|_{L_1(Q)}$ and that $\|\min \{|f-F|,\frac{r}{2}\}\|_{L_1(10Q)}$ can be made arbitrarily small. It remains to show that $|f - F|$ does not have large tails on $10Q$. We previously used Lemma~\ref{lem:supercharacteristic bounds} to bound the tails of $f$ on $Q$, but this used the fact that $Q$ is bounded above and below by characteristic curves. We will have to do more work to find supercharacteristic curves above and below $10Q$.
In fact, we will show the following bound on $\hat{f}$, then use \eqref{eq:compare L1 hats} to show a similar bound on $|f-F|$.
\begin{lemma}\label{lem:paramonotone squares L1 close}
  For any $\delta>0$, there is  $\beta=\beta(\delta)>0$ with the following property.  Let $\hat{\Gamma}=\Gamma_{\hat{f}}$ be an intrinsic Lipschitz graph.  Let $\tau>0$ and suppose that
  \begin{equation}\label{eq:paramonotone D1 symdiff condition}
    \cH^4\big(B_{144\tau}\cap (\hat{\Gamma}^+\symdiff V_0^+)\big)<\beta \tau^4,
  \end{equation}
  and that the density of $\Omega^P_{\hat{\Gamma}^+,48\tau}$ on $D_{24\tau}$ is bounded by
  $$\tau^{-3}\Omega^P_{\hat{\Gamma}^+,48\tau}(D_{24\tau})<\beta.$$
  Then $\|\hat{f}\|_{L_1(D_{8\tau})}\le \delta \tau^4$.
\end{lemma}
\begin{proof}
  Recall that by Lemma~\ref{lem:Omega scaling}, the density of $\Omega^P_{\hat{\Gamma}^+,48\tau}$ is invariant under scaling, so, after rescaling, it is enough to treat the case $\tau=1$.  Let
  $$\cU\eqdef \left\{L_{(0,y_0,z_0),m}\mid z_0\in [200,201]\ \wedge\  y_0\in [1,2]\ \wedge\  m \in \left[-\frac{y_0}{20}, -\frac{y_0}{21}\right]\right\}.$$
  We claim that there is some $L\in \cU$ such that the segment $\Pi(\rho_L([-16,16]))$ is a supercharacteristic curve above $D_8$.  A similar construction will produce a second supercharacteristic curve below $D_8$, so we can use Lemma~\ref{lem:supercharacteristic bounds} to bound $\hat{f}$ from above.

  We clip $\hat{f}$ between $-24$ and $24$ and call the result $h$; that is, for all $w\in V_0$, let $h(w)$ be the element of $[-24,24]$ that is closest to $\hat{f}(w)$.  For $L\in \cL_P$ and $t\in \R$, let $h_L(t)=h(\Pi(\rho_L(t)))$.  Define
  \begin{align*}
  \cU_1&\eqdef \left\{L\in \cU\mid \text{$\Pi(\rho_L([-16,16]))$ is supercharacteristic}\right\}\\
  \cU_2&\eqdef \biggl\{L\in \cU\mid \int_{-24}^{24} |h_L(t)|\ud t>\frac{1}{24}\biggr\}\\
  \cU_3&\eqdef \left\{L\in \cU\mid \wwidehat{\omega}^P_{\hat{\Gamma}^+,48}(D_{24},L)\ge 1\right\}.
  \end{align*}
  We claim that almost every $L\in \cU$ is contained in  $\cU_1\cup\cU_2\cup \cU_3$.

  Let $L\in \cU$ and suppose that $x(L\cap \hat{\Gamma}^+)$ is a subset of $\R$ with locally finite perimeter.  This is true for almost every $L$.  Suppose that $L\not\in \cU_1\cup \cU_2$.  Then $\Pi(\rho_L([-16,16]))$ is not supercharacteristic, so there is some $a\in [-16,16]$ such that $\rho_L(a)\in \hat{\Gamma}^-$.  Let $p$ be the intersection point of $L$ with $V_0$; by our choice of parameters, $x(p)\in [20,21]$.  Also, since $m < -\frac{1}{24}$, we have $y(\rho_L(t)) > \frac{1}{24}$ for $t\le 19$.  Since $L\not \in \cU_2$, there are $b_1\in [16,17]$ and $b_2\in [18,19]$ such that for $i\in \{1,2\}$ we have $h_L(b_i)\le \frac{1}{24} < y(\rho_L(b_i))$ and thus $\rho_L(b_i)\in \hat{\Gamma}^+$.  Similarly, $y(\rho_L(t)) < -\frac{1}{24}$ for all $t\ge 22$, so there is $c \in [22,23]$ such that $h_L(c) > y(\rho_L(c))$  and $\rho_L(c)\in \hat{\Gamma}^-$.  There is an element of $\partial_{\cH^1}x(L\cap \hat{\Gamma}^+)$ in $(a,b_1)$ and another in $(b_2,c)$.  Since $a,b_1,b_2,c\in [-24,24]$, Lemma~\ref{lem:omega boundary diam NM} implies that
  $$\wwidehat{\omega}^P_{\hat{\Gamma}^+,48}(D_{24},L)\ge b_2-b_1\ge 1$$
  and thus $L\in \cU_3$.

  Therefore, $\cU_1\cup \cU_2 \cup \cU_3$ contains all of $\cU$ except a null set.  We will next show that $\cN_P(\cU_2)$ and $\cN_P(\cU_3)$ are bounded by multiples of $\beta$.

  Suppose $L=L_{(0,y_0,z_0),m}$.  As in \eqref{eq:def gL}, let $g_L(t)=z(\Pi(\rho_L(t)))=-\frac{m}{2} t^2- y_0 t + z_0$.  For every $t\in [-24,24]$, we have
  \begin{equation}\label{eq:z bounds on L}
    |g_L(t)-200|\le 1+\frac{m}{2} t^2 + y_0 |t|\le 1+\frac{24^2}{20}+48\le 100,
  \end{equation}
  so $\Pi(\rho_L([-24,24]))\subset D_{24}$.  Furthermore, $D_{24}\subset B_{120}$, so for all $v\in D_{24}$ and $t\in [-24,24]$, we have $vY^t\in B_{144}$.  Thus
  \begin{equation}\label{eq:clipped f bounds}
    \|h\|_{L_1(D_{24})}\le \cH^4\big(B_{144}\cap (\hat{\Gamma}^+\symdiff V_0^+)\big)<\beta.
  \end{equation}
  Therefore, for any $y_0\in [1,2]$ and $m \in \left[-\frac{y_0}{20}, -\frac{y_0}{21}\right]$,
  $$\int_{200}^{201} \int_{-24}^{24} |h_L(t)|\ud t \ud z_0 \le \|h\|_{L_1(D_{24})}<\beta.$$
  It follows that $\{ z_0\in [200,201] \mid L_{(0,y_0,z_0),m}\in \cU_2\}$ has measure at most $24 \beta$ and thus
  $$\cN_P(\cU_2)\le \int_1^2\int_{-\frac{y_0}{20}}^{-\frac{y_0}{21}} 24 \beta \ud m \ud y_0\le 24 \beta.$$
  To bound $\cN_P(\cU_3)$, observe that
  $$\cN_P(\cU_3)\le \int_{\cL_P} \wwidehat{\omega}^P_{\hat{\Gamma}^+,48}(D_{24},L)\ud \cN_P(L)=48 \Omega^P_{\hat{\Gamma}^+,48}(D_{24})<48 \beta.$$

  It follows that if $\beta$ is sufficiently small, then
  $$\cN_P(\cU_1)\ge \cN_P(\cU)-\cN_P(\cU_2)-\cN_P(\cU_3)>0.$$
  Therefore $\cU_1$ is nonempty.  That is, there exists a line $L\in \cU$ with parametrization $\rho_L$ such that $S_2=\Pi(L)\cap \{-16\le x\le 16\}$ is a supercharacteristic curve.  By \eqref{eq:z bounds on L}, $S_2$ is above $D_8$ and $S_2\subset D_{24}$.  By symmetry, there  also exists a line $L'$ and a supercharacteristic curve $S_1=\Pi(L')\cap \{-16\le x\le 16\}$ that lies below $D_8$ and satisfies $S_1\subset D_{24}$.

  By Lemma~\ref{lem:supercharacteristic bounds} applied to a rescaling of $\hat{\Gamma}$, there is some $C>24$ such that for any $t>C$,
  $$|\{v\in D_8\mid \hat{f}(v)\ge t\}| \lesssim t^{-2} \Omega^P_{\hat{\Gamma}^+,48}(D_{24})\le t^{-2} \beta.$$
  Applying another symmetry, the analogous reasoning shows that for any $t>C$,
  $$|\{v\in D_8\mid \hat{f}(v)\le -t\}| \lesssim t^{-2}\beta.$$

  Then, for all sufficiently small $\beta$,
  \begin{align*}
    \|\hat{f}\|_{L_1(D_8)}
    &=\|h\|_{L_1(D_8)}+\int_{24}^\infty \left|\left\{v\in D_8\mid |\hat{f}(v)|\ge t\right\}\right| \ud t\\
    &\lesssim \|h\|_{L_1(D_8)}+C \left|\left\{v\in D_8\mid |\hat{f}(v)|\ge 24\right\}\right| + \int_{C}^\infty t^{-2}\beta\ud t\\
    &\le \beta+C \left|\left\{v\in D_8\mid |\hat{f}(v)|\ge 24\right\}\right|+\beta,
  \end{align*}
  where we use the fact $C>24$ to go from the first line to the second.
  But
  $$\left|\left\{v\in D_8\mid |\hat{f}(v)|\ge 24\right\}\right|=\bigl|\bigl\{v\in D_8\mid |h(v)|= 24\bigr\}\bigr|\le \frac{\|h\|_{L_1(D_8)}}{24}\le \frac{\beta}{24},$$
  so $\|\hat{f}\|_{L_1(D_8)}\lesssim \beta$.  This proves Lemma~\ref{lem:paramonotone squares L1 close}, for $\beta$ at most a constant multiple of $\delta$.
\end{proof}

We will use the following corollary in the proof of Proposition~\ref{prop:Omega control}.
\begin{cor}\label{cor:omega L1 bounds}
  Let $c$ be the universal constant in \eqref{eq:hatq d nesting}--\eqref{eq:compare L1 hats} and let $\kappa$ be the universal constant in Lemma~\ref{lem:paramonotone is L1 bounded}.  Denote
  \begin{align*}
    \tau \eqdef \frac{1}{8} \max\{100, 11c\}\qquad\mathrm{and}\qquad
    r  \eqdef  \max\big\{2 \kappa+6, 144 c \tau +c^2\big \}.
  \end{align*}
  For any $\lambda>0$, there are $\eta, R>0$ with the following property. Let $\Gamma=\Gamma_f$ be an intrinsic Lipschitz graph and $(Q, [-1,1]\times \{0\}\times [-1,1])$ a $\frac{1}{32}$--rectilinear pseudoquad for $\Gamma$. Suppose that $\Gamma$ is $(\eta, R)$--paramonotone on $rQ$, and $P$, $F$, and $\hat{\Gamma}=\Gamma_{\hat{f}}$ are as in Lemma~\ref{lem:paramonotone approximate plane small} and the remarks immediately  after its proof.  Then
  $$\|F-f\|_{L_1(10 Q)} \le \lambda |Q|\qquad \mathrm{and}\qquad \|\hat{f}\|_{L_1(D_{100})} \le \lambda.$$
\end{cor}
\begin{proof}
  Set $\delta = \lambda \tau^{-4}$.  Let $\beta=\beta(\delta)$ be as in Lemma~\ref{lem:paramonotone squares L1 close}.  By Lemma~\ref{lem:paramonotone approximate plane small}, there are $\eta_0$, $R_0$ such that if $\Gamma$ is $(\eta_0, R_0)$--paramonotone on $rQ$, then
  $$\cH^4\left(\overline{B}_{144 c \tau +c^2} \cap (P^+\symdiff \Gamma^+)\right)<\beta \tau^4.$$
  By \eqref{eq:compare hat Gamma ch4}, this implies that
  $$\cH^4\left(\overline{B}_{144 \tau} \cap (V_0^+\symdiff \hat{\Gamma}^+)\right)<\beta \tau^4.$$

  We take $R>R_0$ and $\eta R<\eta_0 R_0$, so that $(\eta, R)$--para\-monotonicity implies $(\eta_0,R_0)$--para\-monotonicity.
  Then, by \eqref{eq:hatq d nesting} and the paramonotonicity of $Q$,
  \begin{multline*}
  \Omega^P_{\hat{\Gamma}^+,48\tau}(D_{24\tau}) \le \frac{R \delta_x(\hat{Q})}{48\tau}\Omega^P_{\hat{\Gamma}^+,R \delta_x(\hat{Q})}(D_{24\tau}) \\\le \frac{R}{24\tau}\Omega^P_{\Gamma^+,R \delta_x(Q)}(D_{24c\tau+c^2}) \le \frac{R}{24\tau} |Q| \eta\alpha(Q)^{-4} \lesssim \eta.\end{multline*}
  If $\eta$ is sufficiently small, then Lemma~\ref{lem:paramonotone squares L1 close} implies that
  $$\|\hat{f}\|_{L_1(D_{\max\{100, 11c\}})} < \lambda.$$
  By \eqref{eq:compare L1 hats}, this implies that $\|F-f\|_{L_1(10 Q)} \le \lambda |Q|$.
\end{proof}

\subsection{Characteristic curves are close to lines}

Finally, in this section we will show  that the characteristic curves of $\hat{\Gamma}$ are close to horizontal lines and prove Proposition~\ref{prop:Omega control}.  The key argument is that when characteristic curves fail to be horizontal, configurations like those in Figure~\ref{fig:horizontal curve and line} produce nonmonotonicity.

  \begin{figure}[h]
  \begin{centering}
    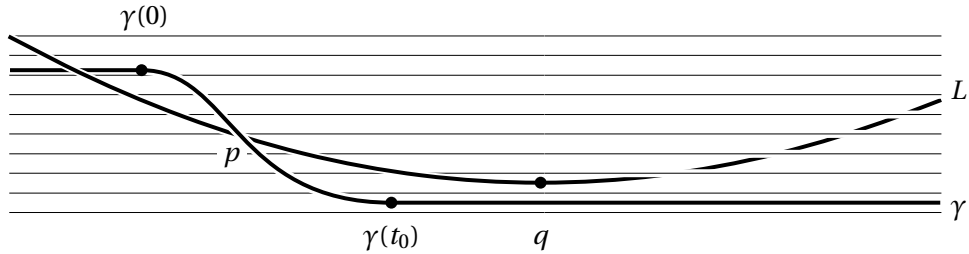
  \end{centering}
  \caption{\label{fig:horizontal curve and line} A characteristic curve $\gamma$ and a horizontal line $L$, projected to $V_0$.  The projection of $L$ crosses $\gamma$ positively at $p$, so $L$ passes behind $\hat{\Gamma}$ at $p$, and $L$ intersects $V_0$ (shown as parallel horizontal lines) at $q$.    If $\hat{f}$ is zero away from $\gamma$, then $L$ intersects $\hat{\Gamma}$ at least three times (twice near $p$ and once at $q$) and the contribution to $\wwidehat{\omega}^P$ is at least $\frac{x(q)-x(p)}{2}$}
\end{figure}

\begin{lemma}\label{lem:close foliation curves squares}
  For any $A>0$, there are $\delta=\delta(A), \theta=\theta(A)>0$ with the following property.  Let $\hat{\Gamma}=\Gamma_{\hat{f}}$ be an intrinsic Lipschitz graph.  Suppose that $$\Omega^P_{\hat{\Gamma}^+,16}(D_{8})< \theta \qquad \mathrm{and}\qquad  \left\|\hat{f}\right\|_{L_1(D_{8})}<\delta.$$  Let $\gamma\from \R\to V_0$ be a characteristic curve through $\mathbf{0}$ and write $\gamma(t)=(t,0,g(t))$ for $t\in \R$.  Then $|g(t)|<A$ for all $t\in [-1,1]$.
\end{lemma}

\begin{proof}
 We may suppose that $0<A<1$.  Choose $\delta=\frac{A^2}{96}$ and $\theta=\frac{A^3}{10^5}$.  Our goal is to show  that if $\|\hat{f}\|_{L_1(D_{8})}<\delta$ and if there is  $t_0\in [-1,1]$ with $|g(t_0)|\ge A$, then $\Omega^P_{\hat{\Gamma}^+,16}(D_8)\ge \theta.$
  After applying a symmetry, we may suppose that $t_0>0$ and that $g(t_0)\le -A$, as in Figure~\ref{fig:horizontal curve and line}.

 Take $z_0\in (-\frac{A}{2}, 0)$, $y_0\in [\frac{A}{4}, \frac{A}{2}]$, $m\in [-\frac{y_0}{5}, -\frac{y_0}{6}]$, and $w=(0,y_0,z_0)$.  Let $L=L_{w,m}$.  Suppose that $\Pi(L)$ and $\gamma$ intersect transversally and $L\cap \hat{\Gamma}^-$ has finite perimeter; these hold for almost every tuple $(y_0,z_0,m)$.  We will show that if
  \begin{equation}\label{eq:f L1 bound}
    \int_0^8 \big|\hat{f}\big(t,0,g_L(t)\big)\big| \ud t < \frac{A}{24},
  \end{equation}
  then $\wwidehat{\omega}^P_{\hat{\Gamma}^+,16}(D_8,L)\ge 1$, where $g_L=z(\Pi(\rho_L))$.

  Suppose that \eqref{eq:f L1 bound} holds.  For $t\in [-8,8]$, we have
  $$|g_L(t)|\le |z_0|+\frac{|m|}{2}t^2+|y_0 t|< 1 + \frac{64}{20}+4<64,$$
  so $\Pi(\rho_L([-8,8]))\subset D_8$.  The graphs of $g_L$ and $g$ intersect as depicted in Figure~\ref{fig:horizontal curve and line}.  That is, $g_L(0)=z_0<g(0)$, $g_L$ is decreasing on $[0,5]$, and $g_L(0)-g_L(1)=\frac{m}{2} + y_0<\frac{A}{2}$, so
  $$g_L(t_0) \ge g_L(1)>g_L(0)-\frac{A}{2} > -A\ge g(t_0).$$
  It follows that the graph of $g_L$ crosses $\gamma$ positively at some point $p=(a,0,g(a))$, where $a\in [0,t_0]$.  Since $g$ is characteristic,
  $$\hat{f}\big(a,0,g(a)\big)=-g'(a)>-g_L'(a)=y\big(\rho_L(a)\big),$$
  so $\rho_L(a)\in \hat{\Gamma}^-$.

  Let $q$ be the point where $L$ intersects $V_0$.  Then $x(q)=-\frac{y_0}{m}\in [5,6]$.  Since $m\le -\frac{A}{24}$, we have $y(\rho_L(t))\ge \frac{A}{24}$ for $t\le 4$ and $y(\rho_L(t))\le -\frac{A}{24}$ for $t\ge 7$.  By \eqref{eq:f L1 bound}, there are $b_1\in [1,2]$ and $b_2\in [3,4]$ such that
  $$\hat{f}\big(b_i,0,g_L(b_i)\big)<\frac{A}{24}\le y\big(\rho_L(b_i)\big).$$
  This implies $\rho_L(b_i)\in \hat{\Gamma}^+$.  Similarly, there is  $c\in [7,8]$ such that $y(\rho_L(c)) < \hat{f}(c,0,g_L(c))$ and thus $\rho_L(c)\in \hat{\Gamma}^-$.  There is an element of $\partial_{\cH^1}\left( x(L\cap \hat{\Gamma}^+)\right)$ in $(a,b_1)$ and another in $(b_2,c)$, and by Lemma~\ref{lem:omega boundary diam NM},
  $$\wwidehat{\omega}^P_{\hat{\Gamma}^+,16}(D_{8},L)\ge b_2-b_1\ge 1,$$
  as desired.

  Therefore, for almost every $(m,y_0,z_0)$ as above, regardless of whether~\eqref{eq:f L1 bound} holds,
  \begin{equation}\label{eq:sum version}
  \wwidehat{\omega}^P_{\hat{\Gamma}^+,16}(D_8, L)+\frac{24}{A}\int_0^8 \big|\hat{f}\big(t,0,g_{L}(t)\big)\big| \ud t\ge 1,
  \end{equation}
  since we showed that at least one of the summands on the left hand side of~\eqref{eq:sum version} is at least 1.  By integrating~\eqref{eq:sum version} with respect to $z_0$, we see that for almost every $(m,y_0)$ that satisfy  $y_0\in [\frac{A}{4}, \frac{A}{2}]$ and $m\in [-\frac{y_0}{5}, -\frac{y_0}{6}]$, we have
  \begin{multline*}
    \int_{-\frac{A}{2}}^{0}\wwidehat{\omega}^P_{\hat{\Gamma}^+,16}(D_8, L)\ud z_0 \ge \frac{A}{2} - \frac{24}{A} \int_{-\frac{A}{2}}^{0} \int_0^8 |\hat{f}(x,0,g_L(x))| \ud x \ud z_0 \\ \ge \frac{A}{2}-\frac{24}{A} \|\hat{f}\|_{L_1(D_8)}\ge \frac{A}{2}-\frac{24\delta}{A}=\frac{A}{4}.
  \end{multline*}
  By integrating this bound over $m$ and $y_0$ as above, we conclude as follows.
  \begin{align*}
    \Omega^P_{\hat{\Gamma}^+,16}(D_8)
    \ge \frac{1}{16} \int_{\frac{A}{4}}^{\frac{A}{2}}\int_{-\frac{y_0}{5}}^{-\frac{y_0}{6}}\int_{-\frac{A}{2}}^{0} \wwidehat{\omega}^P_{\hat{\Gamma}^+,16}(D_8, L) \ud z_0\ud m \ud y_0
    \ge  \frac{A^3}{10^5}.\tag*{\qedhere}
  \end{align*}
\end{proof}

Part \ref{it:omega control characteristics} of Proposition~\ref{prop:Omega control} follows from Lemma~\ref{lem:close foliation curves squares}.
\begin{cor}\label{cor:omega char bounds}
  For every $0<\zeta<1$ there are $\delta=\delta(\zeta)>0$ and $\theta=\theta(\zeta)>0$ with the following property.  Let $\hat{\Gamma}=\Gamma_{\hat{f}}$ be an intrinsic Lipschitz graph such that
  $$\Omega^P_{\hat{\Gamma}^+,128}(D_{100})< \theta\qquad\mathrm{and}\qquad \left\|\hat{f}\right\|_{L_1(D_{100})}<\delta.$$
  Let $\hat{Q}$ be a pseudoquad for $\hat{\Gamma}$ with $x(\hat{Q})=[-1,1]$ such that $\mathbf{0}\in \hat{Q}$ and $\delta_z(\hat{Q})=2$.  For $u\in 4 \hat{Q}$, if $g_u\from \R\to \R$ is such that $\{z=g_u(x)\}$ is a characteristic curve for $\hat{\Gamma}$ that passes through $u$, then $\|g-z(u)\|_{L_\infty([-4,4])}\le \zeta$.  That is, $\hat{Q}$ satisfies part \ref{it:omega control characteristics} of Proposition~\ref{prop:Omega control} for $P=V_0$.
\end{cor}

\begin{proof}
  For $p\in V_0$ and $t>0$, denote $D_t(p)=pD_t$.
  Let $A=\frac{\zeta}{64}$ and let $\delta, \theta>0$ be constants satisfying Lemma~\ref{lem:close foliation curves squares} for this choice of $A$.

  Let $p\in D_{36}$ so that $D_{64}(p)\subset D_{100}$.
  Then $\Omega^P_{\hat{\Gamma}^+,8\cdot 16}(D_{8^2}(p)) < \theta$ and $\|\hat{f}\|_{L_1(D_{8^2}(p))} < \delta$, so by Lemma~\ref{lem:Omega scaling}, the rescaling $s_{1/8,1/8}(p^{-1}\hat{\Gamma})$ satisfies Lemma~\ref{lem:close foliation curves squares}. Hence, if $\gamma=\{z=g_p(x)\}$ is a characteristic curve for $\hat{\Gamma}$ that passes through $p$, then $$\|g_p-z(p)\|_{L_\infty([x(p)-8,x(p)+8])}\le 64 A =\zeta.$$

  Let $g_1$ and $g_2$ be the lower and upper bounds of $\hat{Q}$, respectively.  Then $g_1(0)\in [-3,0]$ and $g_2(0)\in [0,3]$, so $\|g_1-g_1(0)\|_{L_\infty([-8,8])}\le \zeta$ and $\|g_2-g_2(0)\|_{L_\infty([-8,8])}\le \zeta$.  Therefore, $4\hat{Q}\subset D_{36}$.  If $u\in 4 \hat{Q}$ and $\{z=g_u(x)\}$ is a characteristic curve, then
  \begin{equation*}
  \|g_u-z(u)\|_{L_\infty([-4,4])} \le \|g_u-z(u)\|_{L_\infty([x(u)-8,x(u)+8])} \le \zeta.\tag*{\qedhere}
  \end{equation*}
\end{proof}

Finally, we combine the results of this section to prove Proposition~\ref{prop:Omega control}.
\begin{proof}[{Proof of Proposition~\ref{prop:Omega control}}]
  By Lemma~\ref{lem:curve transforms} and Lemma~\ref{lem:Omega scaling}, if $Q$ is a pseudoquad of $\Gamma$ and $h$ is a composition of a shear map, a translation, and a stretch map, then $Q$ and $\Gamma$ satisfy Proposition~\ref{prop:Omega control} if and only if $\hat{h}(Q)=\Pi(h(Q))$ and $h(\Gamma)$ do.  So, by Remark~\ref{rem:normalizing rectilinear}, it suffices to prove Proposition~\ref{prop:Omega control}  for rectilinear pseudoquads of the form $(Q,[-1,1]\times \{0\}\times [-1,1])$.

  Let $r$ be as in Corollary~\ref{cor:omega L1 bounds}; we may suppose $r>100$.  Let $\delta=\delta(\zeta),\theta=\theta(\zeta)>0$ as in Corollary~\ref{cor:omega char bounds}.  Then we can choose $R_0=R_0(\lambda,\zeta)>0$ and $\eta_0=\eta_0(\lambda,\zeta)>0$
  so that if $\Gamma$ is $(\eta_0, R_0)$--paramonotone on $rQ$ and $P$, $F$, and $\hat{\Gamma}=\Gamma_{\hat{f}}$ are as above, then
  \begin{equation}
    \|F-f\|_{L_1(10 Q)} \le \lambda |Q| \qquad \mathrm{and} \qquad \|\hat{f}\|_{L_1(D_{100})} \le \delta.\label{eq:final L1 bounds}
  \end{equation}

  Denote $R=\max\{R_0, 128\}$ and $\eta=\min\{\frac{\theta}{R}, \eta_0 \frac{R_0}{R}\}$.  Since $R\ge R_0$, $\eta R \le \eta_0 R_0$, and $\Gamma$ is $(\eta, R)$--paramonotone on $rQ$, it is also $(\eta, R)$--paramonotone, so $Q$ satisfies \eqref{eq:final L1 bounds}, which implies part \ref{it:omega control plane} of Proposition~\ref{prop:Omega control}.  Furthermore,
  $$\Omega^P_{\hat{\Gamma}^+,128}(D_{100})\le \frac{R}{128} \Omega^P_{\hat{\Gamma}^+,R}(rQ) \le \frac{R}{128} |Q| \alpha(Q)^{-4}\eta<\theta.$$
  Thus $\hat{\Gamma}$ satisfies the hypotheses of Corollary~\ref{cor:omega char bounds}, so $\hat{Q}$ satisfies part \ref{it:omega control characteristics} of Proposition~\ref{prop:Omega control}.  As $\hat{Q}$ is the image of $Q$ under a shear map,  part \ref{it:omega control characteristics} of Proposition~\ref{prop:Omega control} holds for $Q$ as well.
\end{proof}

\medskip

\noindent{\bf Acknowledgements.} We thank  Alexandros Eskenazis for a discussion that led to Remark~\ref{rem:J}. We are also grateful to the anonymous referees for their careful reading of this article and their many helpful corrections and suggestions.

Our former colleague Louis Nirenberg passed away as this project was being completed. Over the years, he made significant efforts (partially in collaboration with A.~N.) to answer the question that we resolve here, though in hindsight those attempts were doomed to fail because they aimed to prove~\eqref{eq:our main in intro q version} with $p=2$, which we now know does not hold. His deep mathematical insights,  his contagious joie de vivre, and his kindness are dearly missed.

\bibliographystyle{alphaabbrvprelim}
\bibliography{corona}

\def\cprime{$'$} \def\cprime{$'$} \def\cprime{$'$}
\begin{thebibliography}{GNRS04}
\expandafter\ifx\csname urlstyle\endcsname\relax
  \providecommand{\doi}[1]{doi:\discretionary{}{}{}#1}\else
  \providecommand{\doi}{doi:\discretionary{}{}{}\begingroup
  \urlstyle{rm}\Url}\fi

\bibitem[Amb01]{Amb01}
L.~Ambrosio.
\newblock Some fine properties of sets of finite perimeter in {A}hlfors regular
  metric measure spaces.
\newblock \emph{Adv. Math.}, 159(1):51--67, 2001.

\bibitem[ANT13]{AusNaoTes}
T.~Austin, A.~Naor, and R.~Tessera.
\newblock Sharp quantitative nonembeddability of the {H}eisenberg group into
  superreflexive {B}anach spaces.
\newblock \emph{Groups Geom. Dyn.}, 7(3):497--522, 2013.

\bibitem[ASCV06]{AVSCIntrinsic}
L.~Ambrosio, F.~Serra~Cassano, and D.~Vittone.
\newblock Intrinsic regular hypersurfaces in {H}eisenberg groups.
\newblock \emph{J. Geom. Anal.}, 16(2):187--232, 2006.

\bibitem[Ass83]{Ass83}
P.~Assouad.
\newblock Plongements lipschitziens dans {${\bf R}^{n}$}.
\newblock \emph{Bull. Soc. Math. France}, 111(4):429--448, 1983.

\bibitem[Bad09]{Bad09}
N.~Badr.
\newblock Real interpolation of {S}obolev spaces.
\newblock \emph{Math. Scand.}, 105(2):235--264, 2009.

\bibitem[Bal92]{Bal92}
K.~Ball.
\newblock Markov chains, {R}iesz transforms and {L}ipschitz maps.
\newblock \emph{Geom. Funct. Anal.}, 2(2):137--172, 1992.

\bibitem[Bal13]{Bal13}
K.~Ball.
\newblock The {R}ibe programme.
\newblock \emph{Ast\'erisque}, (352):Exp. No. 1047, viii, 147--159, 2013.
\newblock S{\'e}minaire Bourbaki. Vol. 2011/2012. Expos{\'e}s 1043--1058.

\bibitem[Bau07]{Bau07}
F.~Baudier.
\newblock Metrical characterization of super-reflexivity and linear type of
  {B}anach spaces.
\newblock \emph{Arch. Math. (Basel)}, 89(5):419--429, 2007.

\bibitem[BC05]{BC05}
B.~Brinkman and M.~Charikar.
\newblock On the impossibility of dimension reduction in {$l_1$}.
\newblock \emph{J. ACM}, 52(5):766--788 (electronic), 2005.

\bibitem[BCL94]{BCL94}
K.~Ball, E.~A. Carlen, and E.~H. Lieb.
\newblock Sharp uniform convexity and smoothness inequalities for trace norms.
\newblock \emph{Invent. Math.}, 115(3):463--482, 1994.

\bibitem[BCSC15]{BigolinCaravennaSerraCassano}
F.~Bigolin, L.~Caravenna, and F.~Serra~Cassano.
\newblock Intrinsic {L}ipschitz graphs in {H}eisenberg groups and continuous
  solutions of a balance equation.
\newblock \emph{Ann. Inst. H. Poincar\'e Anal. Non Lin\'eaire}, 32(5):925--963,
  2015.

\bibitem[BL76]{BL76}
J.~Bergh and J.~L{\"o}fstr{\"o}m.
\newblock \emph{Interpolation spaces. {A}n introduction}.
\newblock Springer-Verlag, Berlin-New York, 1976.
\newblock Grundlehren der Mathematischen Wissenschaften, No. 223.

\bibitem[Bla03]{Bla03}
S.~Blach{\`e}re.
\newblock Word distance on the discrete {H}eisenberg group.
\newblock \emph{Colloq. Math.}, 95(1):21--36, 2003.

\bibitem[Boc55]{Boc55}
S.~Bochner.
\newblock \emph{Harmonic analysis and the theory of probability}.
\newblock University of California Press, Berkeley and Los Angeles, 1955.

\bibitem[Bou85]{Bou85}
J.~Bourgain.
\newblock On {L}ipschitz embedding of finite metric spaces in {H}ilbert space.
\newblock \emph{Israel J. Math.}, 52(1-2):46--52, 1985.

\bibitem[Bou86]{Bou86}
J.~Bourgain.
\newblock The metrical interpretation of superreflexivity in {B}anach spaces.
\newblock \emph{Israel J. Math.}, 56(2):222--230, 1986.

\bibitem[BR96]{BR96}
A.~Bella\"{\i}che and J.-J. Risler, editors.
\newblock \emph{Sub-{R}iemannian geometry}, volume 144 of \emph{Progress in
  Mathematics}.
\newblock Birkh\"{a}user Verlag, Basel, 1996.
\newblock ISBN 3-7643-5476-3.
\newblock \doi{10.1007/978-3-0348-9210-0}.

\bibitem[CD14]{CD14}
J.~A. Ch\'{a}vez-Dom\'{\i}nguez.
\newblock Lipschitz factorization through subsets of {H}ilbert space.
\newblock \emph{J. Math. Anal. Appl.}, 418(1):344--356, 2014.

\bibitem[CDH10]{CDH10}
I.~Chatterji, C.~Dru\c{t}u, and F.~Haglund.
\newblock Kazhdan and {H}aagerup properties from the median viewpoint.
\newblock \emph{Adv. Math.}, 225(2):882--921, 2010.

\bibitem[CK10a]{CK10}
J.~Cheeger and B.~Kleiner.
\newblock Differentiating maps into {$L^1$}, and the geometry of {BV}
  functions.
\newblock \emph{Ann. of Math. (2)}, 171(2):1347--1385, 2010.

\bibitem[CK10b]{CheegerKleinerMetricDiff}
J.~Cheeger and B.~Kleiner.
\newblock Metric differentiation, monotonicity and maps to {$L\sp 1$}.
\newblock \emph{Invent. Math.}, 182(2):335--370, 2010.

\bibitem[CKN09]{CKN09}
J.~Cheeger, B.~Kleiner, and A.~Naor.
\newblock A {$(\log n)^{\Omega(1)}$} integrality gap for the sparsest cut
  {SDP}.
\newblock In \emph{2009 50th {A}nnual {IEEE} {S}ymposium on {F}oundations of
  {C}omputer {S}cience ({FOCS} 2009)}, pages 555--564. IEEE Computer Soc., Los
  Alamitos, CA, 2009.
\newblock \doi{10.1109/FOCS.2009.47}.

\bibitem[CKN11]{CKN}
J.~Cheeger, B.~Kleiner, and A.~Naor.
\newblock Compression bounds for {L}ipschitz maps from the {H}eisenberg group
  to {$L\sb 1$}.
\newblock \emph{Acta Math.}, 207(2):291--373, 2011.

\bibitem[Cla36]{Cla36}
J.~A. Clarkson.
\newblock Uniformly convex spaces.
\newblock \emph{Trans. Amer. Math. Soc.}, 40(3):396--414, 1936.

\bibitem[Czu17]{Czu17}
A.~Czuro\'{n}.
\newblock Property {$F\ell_q$} implies property {$F\ell_p$} for
  {$1<p<q<\infty$}.
\newblock \emph{Adv. Math.}, 307:715--726, 2017.

\bibitem[DFO20]{DFO20}
D.~{Di Donato}, K.~Fässler, and T.~Orponen.
\newblock Metric rectifiability of $\mathbb{H}$-regular surfaces with
  {H}\"older continuous horizontal normal, 2020.
\newblock Preprint, available at \url{https://arxiv.org/abs/1906.10215}.

\bibitem[DL97]{DL97}
M.~M. Deza and M.~Laurent.
\newblock \emph{Geometry of cuts and metrics}, volume~15 of \emph{Algorithms
  and Combinatorics}.
\newblock Springer-Verlag, Berlin, 1997.
\newblock ISBN 3-540-61611-X.
\newblock \doi{10.1007/978-3-642-04295-9}.

\bibitem[DLP13]{DLP13}
J.~Ding, J.~R. Lee, and Y.~Peres.
\newblock Markov type and threshold embeddings.
\newblock \emph{Geom. Funct. Anal.}, 23(4):1207--1229, 2013.

\bibitem[DS91]{DavidSemmesSingular}
G.~David and S.~Semmes.
\newblock Singular integrals and rectifiable sets in {${\bf R}\sp n$}: {B}eyond
  {L}ipschitz graphs.
\newblock \emph{Ast\'erisque}, (193):152, 1991.

\bibitem[DS93]{DSAnalysis}
G.~David and S.~Semmes.
\newblock \emph{Analysis of and on uniformly rectifiable sets}, volume~38 of
  \emph{Mathematical Surveys and Monographs}.
\newblock American Mathematical Society, Providence, RI, 1993.
\newblock ISBN 0-8218-1537-7.
\newblock \doi{10.1090/surv/038}.

\bibitem[Enf70]{Enf70}
P.~Enflo.
\newblock Uniform structures and square roots in topological groups. {I}, {II}.
\newblock \emph{Israel J. Math. 8 (1970), 230-252; ibid.}, 8:253--272, 1970.

\bibitem[Fig76]{Fig76}
T.~Figiel.
\newblock On the moduli of convexity and smoothness.
\newblock \emph{Studia Math.}, 56(2):121--155, 1976.

\bibitem[FJ09]{FJ09}
J.~D. Farmer and W.~B. Johnson.
\newblock Lipschitz {$p$}-summing operators.
\newblock \emph{Proc. Amer. Math. Soc.}, 137(9):2989--2995, 2009.

\bibitem[FOR20]{FasslerOrponenRigot18}
K.~F\"{a}ssler, T.~Orponen, and S.~Rigot.
\newblock Semmes surfaces and intrinsic {L}ipschitz graphs in the {H}eisenberg
  group.
\newblock \emph{Trans. Amer. Math. Soc.}, 373(8):5957--5996, 2020.

\bibitem[FSC07]{FSSC-AreaFormula}
B.~Franchi, R.~Serapioni, and F.~S. Cassano.
\newblock Regular submanifolds, graphs and area formula in heisenberg groups.
\newblock \emph{Advances in Mathematics}, 211(1):152 -- 203, 2007.

\bibitem[FSSC01]{FSSCRectifiability}
B.~Franchi, R.~Serapioni, and F.~Serra~Cassano.
\newblock Rectifiability and perimeter in the {H}eisenberg group.
\newblock \emph{Math. Ann.}, 321(3):479--531, 2001.

\bibitem[FSSC06]{FSSC06}
B.~Franchi, R.~Serapioni, and F.~Serra~Cassano.
\newblock Intrinsic {L}ipschitz graphs in {H}eisenberg groups.
\newblock \emph{J. Nonlinear Convex Anal.}, 7(3):423--441, 2006.

\bibitem[FSSC11]{FSSCDifferentiability}
B.~Franchi, R.~Serapioni, and F.~Serra~Cassano.
\newblock Differentiability of intrinsic {L}ipschitz functions within
  {H}eisenberg groups.
\newblock \emph{J. Geom. Anal.}, 21(4):1044--1084, 2011.

\bibitem[GNRS04]{GNRS04}
A.~Gupta, I.~Newman, Y.~Rabinovich, and A.~Sinclair.
\newblock Cuts, trees and {$l_1$}-embeddings of graphs.
\newblock \emph{Combinatorica}, 24(2):233--269, 2004.

\bibitem[Gro93]{Gro93}
M.~Gromov.
\newblock Asymptotic invariants of infinite groups.
\newblock In \emph{Geometric group theory, {V}ol. 2 ({S}ussex, 1991)}, volume
  182 of \emph{London Math. Soc. Lecture Note Ser.}, pages 1--295. Cambridge
  Univ. Press, Cambridge, 1993.

\bibitem[Gro96]{Gro96}
M.~Gromov.
\newblock Carnot-{C}arath\'eodory spaces seen from within.
\newblock In \emph{Sub-{R}iemannian geometry}, volume 144 of \emph{Progr.
  Math.}, pages 79--323. Birkh\"auser, Basel, 1996.

\bibitem[Han56]{Han56}
O.~Hanner.
\newblock On the uniform convexity of {$L^p$} and {$l^p$}.
\newblock \emph{Ark. Mat.}, 3:239--244, 1956.

\bibitem[HN19]{HN19}
T.~Hyt\"{o}nen and A.~Naor.
\newblock Heat flow and quantitative differentiation.
\newblock \emph{J. Eur. Math. Soc. (JEMS)}, 21(11):3415--3466, 2019.

\bibitem[Jam78]{Jam78}
R.~C. James.
\newblock Nonreflexive spaces of type {$2$}.
\newblock \emph{Israel J. Math.}, 30(1-2):1--13, 1978.

\bibitem[JL84]{JL84}
W.~B. Johnson and J.~Lindenstrauss.
\newblock Extensions of {L}ipschitz mappings into a {H}ilbert space.
\newblock In \emph{Conference in modern analysis and probability ({N}ew
  {H}aven, {C}onn., 1982)}, volume~26 of \emph{Contemp. Math.}, pages 189--206.
  Amer. Math. Soc., Providence, RI, 1984.
\newblock \doi{10.1090/conm/026/737400}.

\bibitem[JMS09]{JMS09}
W.~B. Johnson, B.~Maurey, and G.~Schechtman.
\newblock Non-linear factorization of linear operators.
\newblock \emph{Bull. Lond. Math. Soc.}, 41(4):663--668, 2009.

\bibitem[JNGV20]{JulNicVit}
A.~Julia, S.~Nicolussi~Golo, and D.~Vittone.
\newblock Area of intrinsic graphs and coarea formula in {C}arnot groups, 2020.
\newblock Available at \url{http://arxiv.org/abs/2004.02520}.

\bibitem[Jon90]{Jon90}
P.~W. Jones.
\newblock Rectifiable sets and the traveling salesman problem.
\newblock \emph{Invent. Math.}, 102(1):1--15, 1990.

\bibitem[JS09]{JS09}
W.~B. Johnson and G.~Schechtman.
\newblock Diamond graphs and super-reflexivity.
\newblock \emph{J. Topol. Anal.}, 1(2):177--189, 2009.

\bibitem[Kal08]{Kal08-survey}
N.~J. Kalton.
\newblock The nonlinear geometry of {B}anach spaces.
\newblock \emph{Rev. Mat. Complut.}, 21(1):7--60, 2008.

\bibitem[Kal12]{Kal12}
N.~J. Kalton.
\newblock The uniform structure of {B}anach spaces.
\newblock \emph{Math. Ann.}, 354(4):1247--1288, 2012.

\bibitem[KP62]{KP61}
M.~I. Kadec and A.~Pe{\l}czy\'{n}ski.
\newblock Bases, lacunary sequences and complemented subspaces in the spaces
  {$L_{p}$}.
\newblock \emph{Studia Math.}, 21:161--176, 1961/62.

\bibitem[KSC04]{KirSC04}
B.~Kirchheim and F.~Serra~Cassano.
\newblock Rectifiability and parameterization of intrinsic regular surfaces in
  the {H}eisenberg group.
\newblock \emph{Ann. Sc. Norm. Super. Pisa Cl. Sci. (5)}, 3(4):871--896, 2004.

\bibitem[Kwa72]{Kwa72}
S.~Kwapie\'{n}.
\newblock Isomorphic characterizations of inner product spaces by orthogonal
  series with vector valued coefficients.
\newblock \emph{Studia Math.}, 44:583--595, 1972.

\bibitem[Laa00]{Laa00}
T.~J. Laakso.
\newblock Ahlfors {$Q$}-regular spaces with arbitrary {$Q>1$} admitting weak
  {P}oincar\'{e} inequality.
\newblock \emph{Geom. Funct. Anal.}, 10(1):111--123, 2000.

\bibitem[Laa02]{Laa02}
T.~J. Laakso.
\newblock Plane with {$A_\infty$}-weighted metric not bi-{L}ipschitz embeddable
  to {${\Bbb R}^N$}.
\newblock \emph{Bull. London Math. Soc.}, 34(6):667--676, 2002.

\bibitem[LMN05]{LMN05}
J.~R. Lee, M.~Mendel, and A.~Naor.
\newblock Metric structures in {$L_1$}: dimension, snowflakes, and average
  distortion.
\newblock \emph{European J. Combin.}, 26(8):1180--1190, 2005.

\bibitem[LN04]{LN04}
J.~R. Lee and A.~Naor.
\newblock Embedding the diamond graph in {$L_p$} and dimension reduction in
  {$L_1$}.
\newblock \emph{Geom. Funct. Anal.}, 14(4):745--747, 2004.

\bibitem[LN05]{LN05}
J.~R. Lee and A.~Naor.
\newblock Extending {L}ipschitz functions via random metric partitions.
\newblock \emph{Invent. Math.}, 160(1):59--95, 2005.

\bibitem[LN06]{LN06}
J.~R. Lee and A.~Naor.
\newblock ${L}_p$ metrics on the {H}eisenberg group and the {G}oemans-{L}inial
  conjecture.
\newblock In \emph{Proceedings of 47th Annual IEEE Symposium on Foundations of
  Computer Science (FOCS 2006)}, pages 99--108. 2006.
\newblock Available at
  \url{https://web.math.princeton.edu/~naor/homepage\%20files/L_pHGL.pdf}.

\bibitem[LN14a]{LN14}
V.~Lafforgue and A.~Naor.
\newblock A doubling subset of {$L_p$} for {$p>2$} that is inherently infinite
  dimensional.
\newblock \emph{Geom. Dedicata}, 172:387--398, 2014.

\bibitem[LN14b]{LafforgueNaor}
V.~Lafforgue and A.~Naor.
\newblock Vertical versus horizontal {P}oincar\'e inequalities on the
  {H}eisenberg group.
\newblock \emph{Israel J. Math.}, 203(1):309--339, 2014.

\bibitem[LNP09]{LNP06}
J.~R. Lee, A.~Naor, and Y.~Peres.
\newblock Trees and {M}arkov convexity.
\newblock \emph{Geom. Funct. Anal.}, 18(5):1609--1659, 2009.

\bibitem[LP68]{LP68}
J.~Lindenstrauss and A.~Pe{\l}czy\'{n}ski.
\newblock Absolutely summing operators in {$L_{p}$}-spaces and their
  applications.
\newblock \emph{Studia Math.}, 29:275--326, 1968.

\bibitem[LP01]{LP01}
U.~Lang and C.~Plaut.
\newblock Bilipschitz embeddings of metric spaces into space forms.
\newblock \emph{Geom. Dedicata}, 87(1-3):285--307, 2001.

\bibitem[LTJ80]{LT80}
D.~R. Lewis and N.~Tomczak-Jaegermann.
\newblock Hilbertian and complemented finite-dimensional subspaces of {B}anach
  lattices and unitary ideals.
\newblock \emph{J. Functional Analysis}, 35(2):165--190, 1980.

\bibitem[Man72]{Man72}
P.~Mankiewicz.
\newblock On {L}ipschitz mappings between {F}r\'{e}chet spaces.
\newblock \emph{Studia Math.}, 41:225--241, 1972.

\bibitem[Mau74]{Mau74}
B.~Maurey.
\newblock \emph{Th\'{e}or\`emes de factorisation pour les op\'{e}rateurs
  lin\'{e}aires \`a valeurs dans les espaces {$L^{p}$}}.
\newblock Soci\'{e}t\'{e} Math\'{e}matique de France, Paris, 1974.
\newblock With an English summary, Ast\'{e}risque, No. 11.

\bibitem[Mau03]{Mau03}
B.~Maurey.
\newblock Type, cotype and {$K$}-convexity.
\newblock In \emph{Handbook of the geometry of {B}anach spaces, {V}ol.\ 2},
  pages 1299--1332. North-Holland, Amsterdam, 2003.
\newblock \doi{10.1016/S1874-5849(03)80037-2}.

\bibitem[MdlS20]{MS20}
A.~Marrakchi and M.~de~la Salle.
\newblock Isometric actions on ${L}_p$-spaces: dependence on the value of $p$,
  2020.
\newblock Available at \url{https://arxiv.org/abs/2001.02490}.

\bibitem[MM16]{MM16}
K.~Makarychev and Y.~Makarychev.
\newblock Metric extension operators, vertex sparsifiers and {L}ipschitz
  extendability.
\newblock \emph{Israel J. Math.}, 212(2):913--959, 2016.

\bibitem[MN04]{MN04}
M.~Mendel and A.~Naor.
\newblock Euclidean quotients of finite metric spaces.
\newblock \emph{Adv. Math.}, 189(2):451--494, 2004.

\bibitem[MN08]{MN-SOCG}
M.~Mendel and A.~Naor.
\newblock Markov convexity and local rigidity of distorted metrics [extended
  abstract].
\newblock In \emph{Computational geometry ({SCG}'08)}, pages 49--58. ACM, New
  York, 2008.
\newblock \doi{10.1145/1377676.1377686}.

\bibitem[MN13a]{MN13}
M.~Mendel and A.~Naor.
\newblock Markov convexity and local rigidity of distorted metrics.
\newblock \emph{J. Eur. Math. Soc. (JEMS)}, 15(1):287--337, 2013.

\bibitem[MN13b]{MN13-bary}
M.~Mendel and A.~Naor.
\newblock Spectral calculus and {L}ipschitz extension for barycentric metric
  spaces.
\newblock \emph{Anal. Geom. Metr. Spaces}, 1:163--199, 2013.

\bibitem[MN14]{MN14}
M.~Mendel and A.~Naor.
\newblock Nonlinear spectral calculus and super-expanders.
\newblock \emph{Publ. Math. Inst. Hautes \'{E}tudes Sci.}, 119:1--95, 2014.

\bibitem[Mon02]{Mon02}
R.~Montgomery.
\newblock \emph{A tour of subriemannian geometries, their geodesics and
  applications}, volume~91 of \emph{Mathematical Surveys and Monographs}.
\newblock American Mathematical Society, Providence, RI, 2002.
\newblock ISBN 0-8218-1391-9.

\bibitem[Mon05]{Mon05}
F.~Montefalcone.
\newblock Some relations among volume, intrinsic perimeter and one-dimensional
  restrictions of {BV} functions in {C}arnot groups.
\newblock \emph{Ann. Sc. Norm. Super. Pisa Cl. Sci. (5)}, 4(1):79--128, 2005.

\bibitem[MSSC10]{MSSCCharacterizations}
P.~Mattila, R.~Serapioni, and F.~Serra~Cassano.
\newblock Characterizations of intrinsic rectifiability in {H}eisenberg groups.
\newblock \emph{Ann. Sc. Norm. Super. Pisa Cl. Sci. (5)}, 9(4):687--723, 2010.

\bibitem[MTX06]{MTX06}
T.~Mart{\'{\i}}nez, J.~L. Torrea, and Q.~Xu.
\newblock Vector-valued {L}ittlewood-{P}aley-{S}tein theory for semigroups.
\newblock \emph{Adv. Math.}, 203(2):430--475, 2006.

\bibitem[Nao10]{Nao10}
A.~Naor.
\newblock {$L_1$} embeddings of the {H}eisenberg group and fast estimation of
  graph isoperimetry.
\newblock In \emph{Proceedings of the {I}nternational {C}ongress of
  {M}athematicians. {V}olume {III}}, pages 1549--1575. Hindustan Book Agency,
  New Delhi, 2010.

\bibitem[Nao12]{Nao12}
A.~Naor.
\newblock An introduction to the {R}ibe program.
\newblock \emph{Jpn. J. Math.}, 7(2):167--233, 2012.

\bibitem[Nao14]{Nao14}
A.~Naor.
\newblock Comparison of metric spectral gaps.
\newblock \emph{Anal. Geom. Metr. Spaces}, 2:1--52, 2014.

\bibitem[Nao18]{Nao18}
A.~Naor.
\newblock Metric dimension reduction: {A} snapshot of the {R}ibe program.
\newblock In \emph{Proceedings of the 2018 {I}nternational {C}ongress of
  {M}athematicians, Rio de Janeiro. {V}olume {I}}, pages 767--846. 2018.

\bibitem[Nao19]{Nao19}
A.~Naor.
\newblock An average {J}ohn theorem, 2019.
\newblock To appear in {\em Geom.\ Topol.}. Available at
  \url{https://arxiv.org/abs/1905.01280}.

\bibitem[NN12]{NN12}
A.~Naor and O.~Neiman.
\newblock Assouad's theorem with dimension independent of the snowflaking.
\newblock \emph{Rev. Mat. Iberoam.}, 28(4):1123--1142, 2012.

\bibitem[NP11]{NP11}
A.~Naor and Y.~Peres.
\newblock {$L_p$} compression, traveling salesmen, and stable walks.
\newblock \emph{Duke Math. J.}, 157(1):53--108, 2011.

\bibitem[NPS18]{NPS18}
A.~Naor, G.~Pisier, and G.~Schechtman.
\newblock Impossibility of dimension reduction in the nuclear norm [extended
  abstract].
\newblock In \emph{Proceedings of the {T}wenty-{N}inth {A}nnual {ACM}-{SIAM}
  {S}ymposium on {D}iscrete {A}lgorithms}, pages 1345--1352. SIAM,
  Philadelphia, PA, 2018.
\newblock \doi{10.1137/1.9781611975031.88}.

\bibitem[NPSS06]{NPSS06}
A.~Naor, Y.~Peres, O.~Schramm, and S.~Sheffield.
\newblock Markov chains in smooth {B}anach spaces and {G}romov-hyperbolic
  metric spaces.
\newblock \emph{Duke Math. J.}, 134(1):165--197, 2006.

\bibitem[NR03]{NR03}
I.~Newman and Y.~Rabinovich.
\newblock A lower bound on the distortion of embedding planar metrics into
  {E}uclidean space.
\newblock \emph{Discrete Comput. Geom.}, 29(1):77--81, 2003.

\bibitem[NY17]{NY-STOC}
A.~Naor and R.~Young.
\newblock The integrality gap of the {G}oemans-{L}inial {SDP} relaxation for
  sparsest cut is at least a constant multiple of {$\sqrt{\log n}$}.
\newblock In \emph{S{TOC}'17---{P}roceedings of the 49th {A}nnual {ACM}
  {SIGACT} {S}ymposium on {T}heory of {C}omputing}, pages 564--575. ACM, New
  York, 2017.

\bibitem[NY18]{NY18}
A.~Naor and R.~Young.
\newblock Vertical perimeter versus horizontal perimeter.
\newblock \emph{Ann. of Math. (2)}, 188(1):171--279, 2018.

\bibitem[Ost12]{Ost12}
M.~I. Ostrovskii.
\newblock Embeddability of locally finite metric spaces into {B}anach spaces is
  finitely determined.
\newblock \emph{Proc. Amer. Math. Soc.}, 140(8):2721--2730, 2012.

\bibitem[Ost13]{Ost13}
M.~I. Ostrovskii.
\newblock \emph{Metric embeddings}, volume~49 of \emph{De Gruyter Studies in
  Mathematics}.
\newblock De Gruyter, Berlin, 2013.
\newblock ISBN 978-3-11-026340-4; 978-3-11-026401-2.
\newblock \doi{10.1515/9783110264012}.
\newblock Bilipschitz and coarse embeddings into Banach spaces.

\bibitem[Pan89]{Pan89}
P.~Pansu.
\newblock M\'etriques de {C}arnot-{C}arath\'eodory et quasiisom\'etries des
  espaces sym\'etriques de rang un.
\newblock \emph{Ann. of Math. (2)}, 129(1):1--60, 1989.

\bibitem[Pan13]{Pan13}
P.~Pansu.
\newblock Difficult\'{e} d'approximation (d'apr\`es {K}hot, {K}indler,
  {M}ossel, {O}'{D}onnell,{$\dots$}).
\newblock \emph{Ast\'{e}risque}, (352):Exp. No. 1045, vii, 83--120, 2013.
\newblock S\'{e}minaire Bourbaki. Vol. 2011/2012. Expos\'{e}s 1043--1058.

\bibitem[Pis75]{Pis75}
G.~Pisier.
\newblock Martingales with values in uniformly convex spaces.
\newblock \emph{Israel J. Math.}, 20(3-4):326--350, 1975.

\bibitem[Pis86a]{Pis86}
G.~Pisier.
\newblock \emph{Factorization of linear operators and geometry of {B}anach
  spaces}, volume~60 of \emph{CBMS Regional Conference Series in Mathematics}.
\newblock Published for the Conference Board of the Mathematical Sciences,
  Washington, DC; by the American Mathematical Society, Providence, RI, 1986.
\newblock ISBN 0-8218-0710-2.
\newblock \doi{10.1090/cbms/060}.

\bibitem[Pis86b]{Pis86-varenna}
G.~Pisier.
\newblock Probabilistic methods in the geometry of {B}anach spaces.
\newblock In \emph{Probability and analysis ({V}arenna, 1985)}, volume 1206 of
  \emph{Lecture Notes in Math.}, pages 167--241. Springer, Berlin, 1986.
\newblock \doi{10.1007/BFb0076302}.

\bibitem[PX87]{PX87}
G.~Pisier and Q.~H. Xu.
\newblock Random series in the real interpolation spaces between the spaces
  {$v_p$}.
\newblock In \emph{Geometrical aspects of functional analysis (1985/86)},
  volume 1267 of \emph{Lecture Notes in Math.}, pages 185--209. Springer,
  Berlin, 1987.
\newblock \doi{10.1007/BFb0078146}.

\bibitem[Rao99]{Rao99}
S.~Rao.
\newblock Small distortion and volume preserving embeddings for planar and
  {E}uclidean metrics.
\newblock In \emph{Proceedings of the {F}ifteenth {A}nnual {S}ymposium on
  {C}omputational {G}eometry ({M}iami {B}each, {FL}, 1999)}, pages 300--306
  (electronic). ACM, New York, 1999.
\newblock \doi{10.1145/304893.304983}.

\bibitem[Rib76]{Rib76}
M.~Ribe.
\newblock On uniformly homeomorphic normed spaces.
\newblock \emph{Ark. Mat.}, 14(2):237--244, 1976.

\bibitem[Rig19]{RigotQuantitative}
S.~Rigot.
\newblock Quantitative notions of rectifiability in the {H}eisenberg groups,
  2019.
\newblock Available at \url{http://arxiv.org/abs/1904.06904}.

\bibitem[Tao19]{Tao19}
T.~Tao.
\newblock Embedding the {H}eisenberg group into a bounded dimensional
  {E}uclidean space with optimal distortion, 2019.
\newblock To appear in {\em Rev. Mat. Iberoam.}, preprint available at
  \url{https://arxiv.org/abs/1811.09223}.

\bibitem[Tes08]{Tes08}
R.~Tessera.
\newblock Quantitative property {A}, {P}oincar\'e inequalities,
  {$L^p$}-compression and {$L^p$}-distortion for metric measure spaces.
\newblock \emph{Geom. Dedicata}, 136:203--220, 2008.

\bibitem[Xie16]{Xie16}
X.~Xie.
\newblock Some examples of quasiisometries of nilpotent {L}ie groups.
\newblock \emph{J. Reine Angew. Math.}, 718:25--38, 2016.

\bibitem[Xu20]{Xu18}
Q.~Xu.
\newblock Vector-valued {L}ittlewood-{P}aley-{S}tein theory for semigroups
  {II}.
\newblock \emph{Int. Math. Res. Not. IMRN}, (21):7769--7791, 2020.

\end{thebibliography}

\appendix
\section{On the implicit dependence on $p$  in~\cite{LafforgueNaor}}\label{sec:littlewood paley}

A version of Theorem~\ref{thm:LN-XYD} was stated in~\cite{LafforgueNaor} with an implicit dependence on the exponent $p$.  In this section, we explain how the arguments in~\cite{LafforgueNaor} can be used to derive the explicit dependence on $p$ that we needed in Section~\ref{sec:dim reduction}.

Let $(E,\|\cdot\|_E)$ be a Banach space and fix $q\in [2,\infty]$. The $q$--uniform convexity constant of $X$, denoted $K_q(E)$, is defined~\cite{Bal92,BCL94} as the infimum over $K\in (0,\infty]$ such that
\begin{equation}\label{def:uniform convexity q}
\forall\, x,y\in E,\qquad \Big(\|x\|_E^q+\frac{1}{K^q}\|y\|_E^q\Big)^{\frac{1}{q}}\le \bigg(\frac12 \|x+y\|_E^q+\frac12\|x-y\|_E^q\bigg)^{\frac{1}{q}}.
\end{equation}

Setting $x=0$ in~\eqref{def:uniform convexity q} shows that necessarily $K\ge 1$. By convexity, \eqref{def:uniform convexity q} always holds  when $K=\infty$ or when $q=\infty$ and $K=1$. Thus, \eqref{def:uniform convexity q} quantifies the extent to which the norm $\|\cdot\|_E$ is strictly convex. An equivalent (but somewhat less convenient to work with) formulation of this fact (see~\cite{Fig76,BCL94}) is that $K_q(E)$ is bounded above and below by universal constant multiples of the infimum over those $C>0$ such that the sharpened triangle inequality $\|u+v\|_E\le 2-C^{-q}\|u-v\|_E^q$ holds for any two unit vectors $u,v\in E$.

Theorem~\ref{thm:LN-XYD} is the special case $E=\R$, $q=2$ and $1<p\le 2$ of the following theorem.

\begin{thm}\label{thm:constants that we get LP} For any $p>1$ and $q\ge 2$, if $(E,\|\cdot\|_E)$ is a Banach space with $K_q(E)<\infty$, then every smooth and compactly supported function  $f\from \H\to E$ satisfies
\begin{equation}\label{eq:quote LN with K and p}
\bigg(\int_0^\infty \|D_{\vv}^tf\|_{L_p(\cH^4;E)}^{\max\{p,q\}}\frac{\ud t}{t}\bigg)^{\frac1{\max\{p,q\}}}\lesssim \max\Big\{(p-1)^{\frac{1}{q}-1},K_{q}(E)\Big\} \|\nabla_\H f\|_{L_p(\cH^4;\ell_p^2(E))},
\end{equation}
where we use the (standard) notation $\nabla_\H f\eqdef(\XX f,\YY f)\in E\times E$ for the horizontal gradient.
\end{thm}

Theorem~\ref{thm:constants that we get LP} is due to~\cite{LafforgueNaor}, except that it is stated there with a factor that depends in an unspecified way on $p,q,E$ in place of the quantity $\max\{K_{q}(E),1/(p-1)^{1-1/q}\}$. This is because the proof of~\cite{LafforgueNaor} uses the vector-valued Littlewood--Paley--Stein inequality of~\cite{MTX06}, for which  explicit bounds on the relevant constants were not available in the literature at the time when~\cite{LafforgueNaor} was written. However, such bounds were subsequently derived in~\cite{HN19} (using in part an argument of~\cite{LafforgueNaor} itself), so we will next briefly explain how to obtain Theorem~\ref{thm:constants that we get LP} by incorporating this input into~\cite{LafforgueNaor}.

Let $\{h_t\}_{t>0}$ and $\{p_t\}_{t>0}$ be the heat and Poisson kernels on $\R$, respectively, i.e.,
$$
\forall\, s>0,\qquad h_t(s)\eqdef \frac{1}{2\sqrt{\pi t}} e^{-\frac{s^2}{4t}}\qquad\mathrm{and}\qquad p_t(s)\eqdef \frac{t}{\pi(s^2+t^2)}.
$$
It will be convenient to denote the time derivatives $\frac{\partial}{\partial t} h_t, \frac{\partial}{\partial t} p_t$ by $\dot{h}_t,\dot{p}_t$, respectively, i.e.,
$$
\forall\, s>0,\qquad \dot{h}_t(s)=\frac{s^2-2t}{8\sqrt{\pi}t^{\frac52}}e^{-\frac{s^2}{4t}}\qquad\mathrm{and}\qquad \dot{p}_t(s)=\frac{s^2-t^2}{\pi(s^2+t^2)^2}.
$$
By a straightforward evaluation of the integral in~\eqref{eq:subordination} below, one checks the following standard identity  (semigroup subordination; see e.g.~\cite[Section~4.4]{Boc55}).
\begin{equation}\label{eq:subordination}
\forall\, s>0,\qquad \dot{p}_t(s)=\frac{1}{\sqrt{\pi}}\int_0^\infty \frac{e^{-\frac{t^2}{4u}}}{\sqrt{u}}\dot{h}_u(s)\ud u.
\end{equation}

Fix $\phi\in L_q(\R;E)$ and $p\ge 1$. The following bound holds for any $t>0$.
\begin{align}\label{eq:point-wise LP}
\begin{split}
\big\|t\dot{p}_t*\phi\big\|_{L_q(\R,E)}^p&=2^p \bigg\|\int_0^\infty \frac{te^{-\frac{t^2}{4u}}}{2u\sqrt{\pi u}} u\dot{h}_u*\phi \ud u\bigg\|_{L_q(\R,E)}^p\\ &\le \frac{2^{p-1}t}{\sqrt{\pi}}\int_0^\infty u^{-\frac{3}{2}} e^{-\frac{t^2}{4u}} \|u\dot{h}_u*\phi\|_{L_q(\R;E)}^p \ud u.
\end{split}
\end{align}
The first step of~\eqref{eq:point-wise LP} is the representation~\eqref{eq:subordination}, and the second step of~\eqref{eq:point-wise LP} is Jensen's inequality, because $\int_0^\infty t\exp(-t^2/(4u))/(2u\sqrt{\pi u})\ud u=1$. Integration of~\eqref{eq:point-wise LP} gives
\begin{align}\label{eq:poisson bounded by heat}
\begin{split}
\int_0^\infty \|t\dot{p}_t*\phi\|_{L_q(\R;E)}^p\frac{\ud t}{t}&\le \frac{2^{p-1}}{\sqrt{\pi}}\int_0^\infty \bigg(\int_0^\infty e^{-\frac{t^2}{4u}} \ud t \bigg) u^{-\frac{3}{2}}\|u\dot{h}_u*\phi\|_{L_q(\R;E)}^p \ud u\\&=2^{p-1} \int_0^\infty \|u\dot{h}_u*\phi\|_{L_q(\R;E)}^p\frac{\ud u}{u}.
\end{split}
\end{align}
Now, if $q\ge 2$ and $K_q(E)<\infty$, then it was proved\footnote{\cite{HN19} states~\eqref{eq:quote HT} with the factor $K_q(E)$ in the right hand side replaced by a parameter $\mathfrak{m}_q(E)$ that is called~\cite{Pis86-varenna} the martingale cotype $q$ constant of $E$. There is no need to state the definition of $\mathfrak{m}_q(E)$ here because it will not have a role in the ensuing discussion; it suffices to recall that by the martingale inequality of~\cite{Pis75} we have $\mathfrak{m}_q(E)\lesssim K_q(E)$. So, \eqref{eq:poisson bounded by heat} is a formal consequence of~\cite{HN19}, but the above formulation is essentially (namely, up to $O(1)$--renorming) equivalent to that of~\cite{HN19}. For the reverse direction use the fact that there is a  norm $|||\cdot|||$ on $E$ that satisfies $\|x\|_E\asymp  |||x|||$ for all $x\in E$ and such that $K_q(E,|||\cdot|||)\lesssim \mathfrak{m}_q(E)$. This renorming statement is essentially due to the deep work~\cite{Pis75}, except that it is derived in~\cite{Pis75} with the weaker property  $\|x\|_E\le |||x|||\lesssim \mathfrak{m}_q(E)\|x\|_E$. The existence of such a norm which is $O(1)$--equivalent to $\|\cdot\|_E$  follows by combining~\cite{LNP06} and~\cite{MN13}, though we checked (details omitted) that one could adapt the reasoning in~\cite{Pis75} so as to obtain a proof of this fact which avoids any reference to the nonlinear considerations of~\cite{LNP06,MN13}. Alternatively, Gilles Pisier has recently showed us (private communication) a derivation of this $O(1)$--renorming result from the {\em statement} of~\cite[Theorem~3.1]{Pis75}.} in~\cite{HN19} that
\begin{equation}\label{eq:quote HT}
\bigg(\int_0^\infty \|t\dot{h}_t*\phi\|_{L_q(\R;E)}^q\frac{\ud t}{t}\bigg)^{\frac{1}{q}}\lesssim K_q(E) \|\phi\|_{L_q(\R;E)}.
\end{equation}
In combination with~\eqref{eq:poisson bounded by heat} we therefore see that also
\begin{equation}\label{eq:possin with const}
\bigg(\int_0^\infty \|t\dot{p}_t*\phi\|_{L_q(\R;E)}^q\frac{\ud t}{t}\bigg)^{\frac{1}{q}}\lesssim K_q(E)\|\phi\|_{L_q(\R;E)}.
\end{equation}

\begin{remark} The reason why we passed from the vector-valued Littlewood--Paley--Stein inequality~\eqref{eq:quote HT} for the heat semigroup to its counterpart~\eqref{eq:possin with const}  for the Poisson semigroup is that at the time when~\cite{LN14} was written this was known (with $K_q(E)$ in~\eqref{eq:possin with const} replaced by an unspecified constant factor) for the Poisson semigroup due to~\cite{MTX06}, while the validity of~\eqref{eq:quote HT} was an open question. For this reason, \cite{LN14} worked with the Poisson semigroup, so it is simplest to use~\eqref{eq:possin with const} when we refer below to steps in~\cite{LN14}. However, one could  repeat the reasoning of~\cite{LN14} mutatis mutandis while working directly with the heat semigroup and using~\eqref{eq:quote HT}. The above subordination argument is standard, but we included the quick derivation to verify that the constants are universal.
\end{remark}

The case $p=q$ of Theorem~\ref{thm:constants that we get LP} follows by substituting~\eqref{eq:possin with const} into~\cite{LafforgueNaor}. Specifically, we are asserting  that the implicit constant in~\cite[Theorem~2.1]{LafforgueNaor}  is $O(K_q(E))$ when $p=q$. To check this, note that in the proof of~\cite[Theorem~2.1]{LafforgueNaor} the only loss of a factor that is not a universal constant occurs in~\cite[equation~(18)]{LafforgueNaor}, which  is an instantiation of~\cite[inequality~(15)]{LafforgueNaor}; the latter inequality is the same as~\eqref{eq:possin with const} when $p=q$, except that the constant factor in the right hand side is now specified to be $O(K_q(E))$.

The case $p>q$ of Theorem~\ref{thm:constants that we get LP} follows from the case $p=q$.  When $p>q$, we have $K_{q}(E)\ge K_{p}(E)$ (for justification of this monotonicity, see~\cite{BCL94} or~\cite[Section~6.2]{MN14}) and $(p-1)^{1-1/q}\le (p-1)^{1-1/p}$ (since $p>q\ge 2$), so the constant on the right hand side of~\eqref{eq:quote LN with K and p} increases as $q$ decreases. We thus suppose from now that $1<p<q$.

For $M>1$, let $\beta_M\from \H\to [0,1]$ be a smooth bump function that is $O(1)$--Lipschitz (with respect to the Carnot--Carath\'eodory metric $d$), satisfies $\beta_M(h)=1$ for all $h\in B_M$, and has $\supp(\beta_M)\subset B_{M+1}$.

For  a smooth compactly supported $f\from \H\to E$, consider  $F_M\from \H\to L_p(\cH^4;E)$ given by
\begin{equation}\label{eq:def F bump}
\forall\, g,h\in \H,\qquad F_M(h)(g)\eqdef\beta_M(h)f(gh).
\end{equation}

We have $(q-1)^{\frac{1}{q}-1}\le 1 \le K_q(E)$, so the case $p=q$ of Theorem~\ref{thm:constants that we get LP} with $E$ replaced by  $L_p(\cH^4;E)$ gives
\begin{align}\label{eq:in Bochner space}
\begin{split}
\bigg(\int_0^\infty \|D_{\vv}^tF_M\|_{L_q(\cH^4;L_p(\cH^4;E))}^{q}&\frac{\ud t}{t}\bigg)^{\frac1{q}}\lesssim K_q\big(L_p(\cH^4;E)\big)\|\nabla_\H F_M\|_{L_q(\cH^4;\ell_q^2(L_p(\cH^4;E)))}\\
&\!\!\!\!\lesssim \max\Big\{(p-1)^{\frac{1}{q}-1},K_{q}(E)\Big\}\|\nabla_\H F_M\|_{L_q(\cH^4;\ell_q^2(L_p(\cH^4;E)))},
\end{split}
\end{align}
where the last step uses the fact that, by inequality (4.4) in~\cite{Nao14}\footnote{Formally,  \cite[inequality~(4.4)]{Nao14} is the dual of~\eqref{eq:q modulus of Bochner}; see~\cite[Lemma~5]{BCL94} for the relevant duality.}, we have
\begin{equation}\label{eq:q modulus of Bochner}
K_q\big(L_p(\cH^4;E)\big)\lesssim \max\Big\{(p-1)^{\frac{1}{q}-1},K_{q}(E)\Big\}.
\end{equation}

To bound the final term in~\eqref{eq:in Bochner space} from above, note that by the left invariance of $\nabla_\H$,
$$
\nabla_\H F_M(h)(g)=\big(\XX\beta_M(h)f(gh),\YY\beta_M(h)f(gh)\big)+\beta_M(h)\nabla_\H f(gh).
$$
Hence, for all $h\in \H$,
$$
\|\nabla_\H F_M(h)\|_{\ell_q^2(L_p(\cH^4; E))}\lesssim \|f\|_{L_\infty(\cH^4;E)}\1_{B_{M+1}(\0)\setminus B_M(\0)}(h)+ \|\nabla_\H f\|_{L_p(\cH^4;E)}\1_{B_{M+1}(\0)}(h).
$$
So,
\begin{equation}\label{eq:nabla with bump}
\|\nabla_\H F_M\|_{L_q(\cH^4;\ell_q^2(L_p(\cH^4;E)))}\lesssim M^{\frac{3}{q}} \|f\|_{L_\infty(\cH^4;E)}+M^{\frac{4}{q}}\|\nabla_\H f\|_{L_p(\cH^4;E)}.
\end{equation}

In order to bound the left hand side of~\eqref{eq:in Bochner space} from below, note  that by~\eqref{eq:metric approximation}, if $0<t<\frac{M^2}{16}$ and $h\in B_{M-4\sqrt{t}}(\0)$, then  $hZ^t\in B_M$, and therefore $\beta(h)=\beta(hZ^t)=1$. Hence,
$$
\forall\, h\in B_{M-4\sqrt{t}}(\0),\qquad \|D^t_\vv F_M(h)\|_{L_p(\cH^4;E)}=\|D^t_\vv f\|_{L_p(\cH^4;E)}.
$$
Consequently,
$$
\|D^t_\vv F_M\|_{L_q(\cH^4;L_p(\cH^4;E))}\ge \cH^4\big(B_{M-4\sqrt{t}}(\0)\big)^{\frac{1}{q}}\|D^t_\vv f\|_{L_p(\cH^4;E)}\asymp \Big(M-4\sqrt{t}\Big)^{\frac{4}{q}}\|D^t_\vv f\|_{L_p(\cH^4;E)}.
$$
Hence, for every $0<T<\frac{M}{4}$ we have
$$
\bigg(\int_0^{T^2} \|D_{\vv}^tF_M\|_{L_q(\cH^4;L_p(\cH^4;E))}^{q}\frac{\ud t}{t}\bigg)^{\frac1{q}}\gtrsim (M-4T)^{\frac{4}{q}}\bigg(\int_0^{T^2} \|D_{\vv}^tf\|_{L_p(\cH^4;E)}^{q}\frac{\ud t}{t}\bigg)^{\frac1{q}}.
$$
Combining this with~\eqref{eq:in Bochner space} and~\eqref{eq:nabla with bump}, letting $M\to \infty$ and then $T\to \infty$, gives Theorem~\ref{thm:constants that we get LP}.

\begin{remark} In the setting of the proof of Theorem~\ref{thm:constants that we get LP}, the Hardy--Littlewood--Stein (Poisson semigroup) $\mathcal{G}$--function of a function $\phi\in L_q(\R;E)$ is the function $\mathcal{G}_q(\phi)\from \R\to \R$ that  is defined by
\begin{equation}\label{eq:G function def}
\forall\, x\in E,\qquad \mathcal{G}_q(\phi)(x)\eqdef \bigg(\int_0^\infty \|t\dot{p}_t*\phi(x)\|_{E}^q\frac{\ud t}{t}\bigg)^{\frac{1}{q}}.
\end{equation}
By~\cite{MTX06}, if $K_q(E)<\infty$, then  for every $1<p<\infty$,
\begin{equation}\label{quote MTX for $G$ function}
\|\mathcal{G}_q(\phi)\|_{L_p(\R)}\lesssim_{p,q,K_q(E)} \|\phi\|_{L_p(\R;E)}.
\end{equation}
If the implicit constant  in~\eqref{quote MTX for $G$ function} were $O(\max\{K_{q}(E),1/(p-1)^{1-1/q}\})$  for $1<p<q$ (this is so when $p\ge q$ by~\eqref{eq:possin with const} and Jensen's inequality), then  Theorem~\ref{thm:constants that we get LP} would follow by direct substitution into~\cite{LafforgueNaor} without the need to consider the above averaging argument using the auxiliary function $F_M$ in~\eqref{eq:def F bump}. However, it seems that the interpolation argument~\cite{MTX06} does not yield this dependence. Determining the optimal dependence on  $p,q,K_q(E)$ in the $\mathcal{G}$--function bound~\eqref{quote MTX for $G$ function} remains an interesting open question.

The same question for the heat semigroup variant of~\eqref{quote MTX for $G$ function}, i.e., with $\dot{p}_t$ replaced by $\dot{h}_t$ in~\eqref{eq:G function def}, is  a  bigger mystery. That such an inequality for the vector-valued heat semigroup Hardy--Littlewood--Stein $\mathcal{G}$--function holds with any dependence on $p,q,K_q(E)$ was  established recently in~\cite{Xu18}, but as $p\to 1^+$  the dependence of~\cite{Xu18} seems suboptimal. Obtaining the analogue of~\eqref{eq:quote HT} for the $n$--dimensional heat semigroup (in which case $\phi$ is a mapping from $\R^n$ to $E$) would be very interesting. In~\cite{Xu18}, this is achieved with a constant that is independent of $n$ but has a much worse dependence on $K_q(E)$.
\end{remark}

A substitution of Theorem~\ref{thm:constants that we get LP} into the reasoning of~\cite{LafforgueNaor} yields the following restatement of the nonembedding result of~\cite{LafforgueNaor},  with explicit dependence on $K_q(E)$.

\begin{thm}\label{thm:nonembedding KE} For $q\ge 2$, if $E$ is a Banach space with $K_q(E)<\infty$, then for every $n\in \N$, the word-ball in $\H$ of radius $n$ has $E$--distortion
$$
\cc_E(\mathcal{B}_n)\gtrsim \frac{(\log n)^{\frac{1}{q}}}{K_q(E)}.
$$
\end{thm}
Since by~\cite{BCL94}, the Schatten--von Neumann trace class $\mathsf{S}_r$ has $K_2(\mathsf{S}_r)=\sqrt{r-1}$ when $1<r\le 2$, Theorem~\ref{thm:nonembedding KE} implies the lower bound on $\cc_{\mathsf{S}_r}(\mathcal{B}_n)$ that we used in Section~\ref{sec:dim reduction} (recall that the behavior as $r\to 1^+$ was important for that application). This also shows that the following question about a possible strengthening of Theorem~\ref{thm:nonembedding KE} would imply the distortion lower bound~\eqref{eq:fourth dist Sp} that we asked about in Section~\ref{sec:dim reduction}. In fact, a positive answer to this question would be a remarkable geometric result, which, as we explained in Section~\ref{sec:dim reduction}, would have strong implications; at present, we do not have sufficient evidence to conjecture that the answer is indeed positive in such great generality.

\begin{question} Can the conclusion of Theorem~\ref{thm:nonembedding KE} be improved to
$
\cc_E(\mathcal{B}_n)\gtrsim \left(\frac{\log n}{K_q(E)}\right)^{\frac{1}{q}}
$?
\end{question}

\end{document}